\documentclass[final]{amsart} 
\usepackage{mathtools, amsthm, amssymb}
\usepackage[margin=1in]{geometry}
\usepackage{tikz-cd}

\usepackage{hyperref}

\usepackage{thmtools}
\usepackage{cleveref} % don't use `capitalise` option, since \lcnamecref sometimes causes errors when used with thms? Instead just use Cref

\usepackage{microtype}

\newcommand{\resp}{{\sfcode`\.1000 resp.}}
\newcommand{\ie}{{\sfcode`\.1000 i.e.}}
\newcommand{\eg}{{\sfcode`\.1000 e.g.}}
\newcommand{\cf}{{\sfcode`\.1000 cf.}}

\DeclareFontFamily{U}{min}{}
\DeclareFontShape{U}{min}{m}{n}{<-> udmj30}{}
\newcommand\yo{\!\text{\usefont{U}{min}{m}{n}\symbol{'210}}\!}

\newcommand\pt{\mathrm{pt}}

\DeclareMathOperator\SH{\mathbf{SH}}

\DeclareMathOperator\Fun{Fun}

\DeclareMathOperator\Hom{Hom}
\DeclareMathOperator\Op{Op}

\DeclareMathOperator\cofib{cofib}

\DeclareMathOperator\Spec{Spec}

\DeclareMathOperator\GL{GL}

\DeclareMathOperator\Psh{Psh}
\DeclareMathOperator\Shv{Shv}

\DeclareMathOperator\aff{\mathbb{A}}
\DeclareMathOperator\proj{\mathbb{P}}

\newcommand\LAd{\mathrm{LAd}}
\newcommand\RAd{\mathrm{RAd}}

\newtheorem{thm}{Theorem}[subsection]
\newtheorem*{pseudothm}{Pseudo-Theorem}
\newtheorem{prp}[thm]{Proposition}
\newtheorem{lem}[thm]{Lemma}
\newtheorem{cor}[thm]{Corollary}

\theoremstyle{remark}
\newtheorem{rmk}[thm]{Remark}

\theoremstyle{definition}
\newtheorem{defn}[thm]{Definition}
\newtheorem{nota}[thm]{Notation}
\newtheorem{exa}[thm]{Example}
\newtheorem{cnstr}[thm]{Construction}
\newtheorem{setng}[thm]{Setting}
\newtheorem{lang}[thm]{Language}

\newcommand\id{\mathrm{id}}

\newcommand\comps{\mathbb{C}}
\newcommand\reals{\mathbb{R}}
\newcommand\ints{\mathbb{Z}}

\newcommand\fin{\mathrm{fin}} %%

\newcommand\lred{\mathrm{lred}}
\newcommand\nice{\mathrm{nice}}

\newcommand\open{\mathrm{open}}

\newcommand\Nis{\mathrm{Nis}}
\newcommand\NisC{\mathrm{Nis^\an}}

\DeclareMathOperator\B{\mathbf{B}\!}

\newcommand\an{\mathrm{an}}
\newcommand\hol{\mathrm{hol}}
\newcommand\alg{\mathrm{alg}}
\newcommand\diff{\mathrm{diff}}

\newcommand\htpy{\mathrm{htpy}}
\newcommand\shf{\mathrm{shf}}
\newcommand\op{\mathrm{op}}
\newcommand\univ{\mathrm{univ}}
\newcommand\slice{\mathrm{slice}}

\DeclareMathOperator\QC{QCoh}
\DeclareMathOperator\Coh{Coh}

\DeclareMathOperator\PF{PF}

\newcommand\Mfld{\mathrm{Mfld}}
\newcommand\AlgStk{\mathrm{AlgStk}}
\newcommand\An{\mathrm{An}}

\newcommand\Fin{\mathbf{Fin}}
\newcommand\Ab{\mathbf{Ab}}
\newcommand\spaces{\mathcal{S}}

\newcommand\Cat{\mathbf{Cat}}
\newcommand\PrL{{\mathbf{Pr}^{\mathbf{L}}}}
\newcommand\PrR{{\mathbf{Pr}^{\mathbf{R}}}}

\newcommand\Mod{\mathrm{Mod}}
\newcommand\Alg{\mathrm{Alg}}
\newcommand\CAlg{\mathrm{CAlg}}
\newcommand\Mon{\mathrm{Mon}}

\DeclareMathOperator\Mul{Mul}

\DeclareMathOperator\Pic{Pic}

\DeclareMathOperator\Sp{Sp}

\DeclarePairedDelimiter\realise\lvert\rvert
\DeclarePairedDelimiter\cls\lbrack\rbrack

\newcommand\qadm{\mathrm{qadm}}

\newcommand\stackslink[1]{\href{https://stacks.math.columbia.edu/tag/#1}{#1}}
\newcommand\stackscite[1]{\cite[Tag \stackslink{#1}]{stacks-project}}
\newcommand\kerodonlink[1]{\href{https://kerodon.net/tag/000E}{#1}}
\newcommand\kerodoncite[1]{\cite[Tag \kerodonlink{#1}]{kerodon}}

\title[Universal Properties and Constructions of Pullback Formalisms]{Universal Properties and Constructions of Pullback Formalisms in Terms of Invariance and Stability}
\author{Roy Magen}
\address{Institute of Mathematics and Informatics, Bulgarian Academy of Sciences \\ Bulgaria, Sofia 1113, Acad. G. Bonchev Str., Bl. 8}
\thanks{Supported by the Simons Foundation, grant SFI-MPS-T-Institutes-00007697, and the Ministry of Education and Science of the Republic of Bulgaria, grant DO1-239/10.12.2024}

\begin{document}

\begin{abstract}
	In this article, we introduce fundamental notions and results about pullback formalisms, building on \cite{UnivFF}. Our main application is producing a pullback formalism $\SH^\hol$ that encodes a version of motivic homotopy theory for complex analytic stacks, and establishing some of its properties.

	The notions introduced in this article will be used in later articles \cite{6FF} and \cite{Gluing} in which we also establish more properties of $\SH^\hol$, notably the gluing property of Morel and Voevodsky, the structure of a 6-functor formalism, and a realization map from the motivic homotopy theory of algebraic stacks of \cite{SixAlgSt} that is compatible with Grothendieck's six operations, generalizing Ayoub's results on Betti realization for schemes.
\end{abstract}
\maketitle

\tableofcontents

% FIX: I have sometimes been a bit sloppy with size conditions
% for example, the result that gives $\Psh(\mathcal C)$ a qadm structure asks for $\mathcal C$ to be small, but I often want to apply it without assuming $\mathcal C$ is small.
%	maybe if I replace $\Psh(\mathcal C)$ with the full subcategory generated under small colimits by representable presheaves?
% Anyway, as long as I only consider quasi-small pullback contexts, it shouldn't cause any problems:
%	Any application probably only uses a small collection of objects in the pullback context
%	If the pullback context is quasi-small, this means that these objects are contained in a full anodyne pullback subcontext
%		Wait, do I actually need "locally small" for this? Know that the collection of objects will be small, but the mapping spaces in $\mathcal C_S$ are not the same as the mapping spaces in $\mathcal C$, so maybe the latter will be large even if the former are small?
%		It would probably be fine to add the assumption that quasi-small pullback contexts are locally small
%	Most (if not all) of the constructions I consider restrict well to full anodyne pullback subcontexts
% Also note that at some point I removed all "locally small" hypotheses in section about pseudotopologies. Why did I put them in there in the first place?
%	If it turns out that I do need this for some things, I could probably just get around adding it back by requiring all quasi-small pullback contexts to be locally small?

\section*{Introduction}

\subsection*{Invariants of Geometric Objects}

In his unpublished note \cite{voe-view} from 2000, Voevodsky explains how invariants of algebraic varieties form a hierarchy: there are element-level, set-level, and category-level invariants. As examples of element-level invariants, he mentions the \emph{number} of rational points of a variety or its Betti numbers or Euler characteristic. Examples of set-level invariants are the \emph{set} of rational points or the cohomology \emph{groups} of that variety. Finally, an example of a category-level invariant is the \emph{category} of coherent sheaves on a variety. It is often the case that understanding invariants of a higher level is useful for proving theorems about invariants on a lower level: the Weil conjectures about element-level invariants follow from the existence of a certain cohomology theory -- a set-level invariant. Similarly, the properties of $\ell$-adic cohomology used to prove the Weil conjectures follow from the behavior of the category-level invariant of derived categories of constructible sheaves.

This perspective makes sense in settings other than that of algebraic varieties. In topology, Betti numbers and Euler characteristics are element-level invariants that can be computed using cohomology groups which are set-level invariants. Various properties of cohomology theories, such as Atiyah duality and long exact sequences, can be seen as coming from a study of the category-level invariant given by the category of spectra. Equivariant homotopy theory gives us another example of the use of category-level invariants, in which the relationships between various categories of genuine equivariant spectra lead to important properties of equivariant cohomology theories, such as the Wirthm\"{u}ller isomorphism.

The procedure for passing from category-level invariants to set-level invariants is usually (if not always) achieved by considering mapping objects between objects of the categories. For example, sheaf cohomology theories are generally computed as (derived) hom-groups between sheaves, and given a spectrum $E$, to compute the $E$-cohomology of a topological space $X$, we consider the mapping spectrum of maps from the suspension spectrum of $X$ to $E$. 

\subsubsection*{Morphisms of category-level invariants and Betti realization}
The study of morphisms of category-level invariants allows allows us to compare different types of set-level invariants.

Motivic homotopy theory was developed, at least partially, as an algebro-geometric analog of the usual homotopy theory of CW-complexes. As mentioned in Joseph Ayoub's ICM address \cite{AyoubICM}, one of the goals of motivic homotopy is to provide a bridge between so-called ``transcendental'' invariants of an algebraic variety, such as its Betti cohomology, and its algebro-geometric invariants, such as its Chow groups and $K$-theory. Indeed, there is a Betti realization functor that can be seen as a morphism from the category-level invariant of motivic spectra, to the category-level invariant of spectra. Ayoub showed in \cite{AyoubBetti} that this morphism is actually compatible with the additional structure that these category-level invariants satisfy, in particular, the structure of Grothendieck's six operations. This has proven to be a crucial tool not only for using topology to study algebraic geometry, but the reverse as well: Voevodsky used Betti realization in his proof of the Bloch-Kato conjecture in $K$-theory to compute the motivic Steenrod algebra in terms of the topological one, and in \cite{C2equiv}, Behrens and Shah show how to use Betti realization to compute $C_2$-equivariant homotopy groups in terms of motivic homotopy groups over $\reals$.

\subsubsection*{6-functor formalisms and pullback formalisms}

Grothendieck's six operations, first considered in the context of \'{e}tale cohomology, have emerged as a recurring and exceedingly useful type of behavior for category-level invariants. Roughly speaking, given a category $\mathcal C$ of geometric objects, and an assignment of a category $D(X)$ for each object $X$ of $\mathcal C$, we say that the category-level invariant $D$ is a 6-functor formalism if the categories $D$ are closed monoidal (so they are equipped with a tensor product and internal hom functor), and for each map $f : X \to Y$, there is an adjunction $f^* \dashv f_*$ between $D(Y), D(X)$, and for certain $f$, there is another adjunction $f_! \dashv f^!$ between $D(X), D(Y)$. These operations are then required to satisfy various relations. It turns out that the properties of 6-functor formalisms lead to familiar results in many contexts, in particular we have base change formulae, projection formulae, and duality theorems. Many important duality theorems such as Poincar\'{e} duality, Serre duality, and Wirthm\"{u}ller isomorphisms can be seen as properties of 6-functor formalisms.

In his thesis \cite{Mann6FF}, which constructs a 6-functor formalism for $p$-torsion \'{e}tale sheaves in rigid-analytic geometry, Lucas Mann builds on ideas of \cite{LZ6FF} to introduce a formal definition of 6-functor formalism that makes sense in a general context (see \cite[Definition A.5.7]{Mann6FF}). % PERF: say more

Mann proves some powerful results about establishing the structure of a 6-functor formalism on categorical invariants, namely \cite[Proposition A.5.10]{Mann6FF}, as well as results about extending 6-functor formalisms, such as \cite[Lemma A.5.11, Proposition A.5.16]{Mann6FF}, and later with Heyer in \cite[\S3.4]{HM6FF}. These results are useful for establishing the structure of a 6-functor formalism on a given category-level invariant, but they do not give us tools to construct the category-level invariants themselves, and furthermore, the author is not aware of any general results about the compatibility of morphisms of category-level invariants with the six operations, although we are often interested in establishing such compatibilities, such as in the case of Betti realization where this was established in \cite{AyoubBetti}.

In this article, we will take a different approach to studying category-level invariants, in which the basic data with which we equip our categories tells us which maps behave like ``smooth'' maps. Pullback formalisms are category-level invariants in which these smooth maps have the appropriate cohomological properties. This type of structure on category-level invariants was introduced in \cite{UnivFF}, and variants of it have been studied in \cite{SixAlgSt} and \cite{TwAmb}. In fact, this structure was already foreshadowed in earlier works on motivic homotopy theory, such as \cite{A1htpysch}, \cite{voe-crit}, and \cite{Ayoub6I, Ayoub6II}. Indeed, motivic homotopy theory naturally admits the structure of a pullback formalism, and the latter two works show that it also admits the structure of a 6-functor formalism by giving general criteria for when a pullback formalism on schemes admits the structure of a 6-functor formalism.
% These works can be seen as providing criteria for how to deduce the structure of a 6-functor formalism from the structure of a pullback formalism. Indeed, as we shall see in \Cref{S:apps}, the construction of motivic homotopy theory in various contexts naturally produces the structure of a pullback formalism.

One way to compare Mann's framework with the framework of pullback formalisms is by the approach to Poincar\'{e} duality. In Mann's framework, we are given a category $\mathcal C$, along with a certain class of morphisms $E$ in $\mathcal C$ for which the exceptional operations should be defined. One notes (\cite[Remark 3.2.4]{HM6FF}) that this is not enough structure to express any form of Poincar\'{e} duality, for which it would be necessary to define a notion of ``smooth map''. Nevertheless, for any 6-functor formalism, it is possible to define a notion of cohomological smoothness for maps such that the cohomologically smooth maps satisfy Poincar\'{e} duality.

Pullback formalisms, on the other hand, are defined in terms of a prescribed notion of ``smooth maps''. Although pullback formalisms do not provide the existence and properties of the exceptional operations, we may still study which smooth maps have Poincar\'{e} duality, as is done in \cite{6FF}.

% PERF: better title?
\subsubsection*{The merits of pullback formalisms}

One of the reasons one may be led to study pullback formalisms is because constructions of motivic homotopy theory tend to naturally produce this type of category-level invariant.

Indeed, in this article we will provide general tools for constructing pullback formalisms with certain universal properties. This means that we are not only able to establish that certain category-level invariants admit the structure of a pullback formalism, but we are also able to produce the underlying category-level invariant, as well as morphisms from it \emph{that are compatible with the structure of pullback formalisms}.

In \Cref{S:alg,S:diff,S:hol} we will use our results to study algebraic, differentiable, and complex-analytic versions of both stable and unstable motivic homotopy theory by constructing the corresponding category-level invariants as pullback formalisms, and characterizing them by universal properties. Our general results about pullback formalisms allow us to then prove various other results about these category-level invariants.

Some tools of this type have already been established in \cite{UnivFF}. The results established in the present article expand on those results and are of a sufficient generality to be suitable for the study of pullback formalisms on various types of stacks. This stacky context presents new challenges that are not present in the case of trivial stabilizer groups, as we discuss in \Cref{rmk:stacky challenges}.

In \cite{6FF}, we will concern ourselves with the passage from pullback formalisms to 6-functor formalisms. Indeed, the criteria for producing 6-functor formalisms considered in \cite{voe-crit}, \cite{Ayoub6I, Ayoub6II}, and later in \cite{SixAlgSp, SixAlgSt} are specific to the case of algebraic geometry, and using our general study of pullback formalisms initiated in this article, we will be able to establish more general criteria not only for equipping pullback formalisms with the structure of a 6-functor formalism, but also for establishing that morphisms of pullback formalisms (such as those arising from the universal properties established in this article) are actually compatible with the induced structure of 6-functor formalisms.

In conjunction with the universal properties studied in the present article, this will allow us to produce stacky and equivariant versions of Ayoub's results from \cite{AyoubBetti}. These results can be seen as providing morphisms of 6-functor formalisms, and go from the usual equivariant motivic homotopy categories to the equivariant \emph{complex analytic} motivic homotopy categories constructed in \Cref{S:hol}.

Although we do not establish in this article that our version of stable motivic homotopy theory for complex analytic stacks is a 6-functor formalism, there is a lot that can be said already:

\begin{pseudothm}
	Let $\mathcal C$ be an appropriate category of reduced complex analytic stacks with finite stabilizers. Then there is a presheaf of categories $\SH^\hol$ on $\mathcal C$ such that
	\begin{enumerate}

		\item $\SH^\hol$ has \'{e}tale excision, and descent for open covers.

		\item For every $S \in \mathcal C$, $\SH^\hol(S)$ is a symmetric monoidal presentable $\infty$-category. In particular, it has all colimits, $\otimes$ preserves colimits in each parameter, and there is a good notion of hom-objects $\underline\Hom_S : \SH^\hol(S)^\op \times \SH^\hol(S) \to \SH^\hol(S)$.

		\item For every map $f : X \to Y$ in $\mathcal C$, the functor $f^* : \SH^\hol(Y) \to \SH^\hol(X)$ is a symmetric monoidal and preserves colimits, and if $f$ is a representable submersion, then it admits a left adjoint $f_\sharp$ that satisfies the \emph{projection formula}: for any $M \in \SH^\hol(X)$ and $N \in \SH^\hol(Y)$, the canonical map
			\[
				f_\sharp(M \otimes f^* N) \to f_\sharp M \otimes N
			\]
			is an equivalence. Equivalently, for any $N, N' \in \SH^\hol(Y)$, the map
			\[
				f^* \underline\Hom_Y(N,N') \to \underline\Hom_X(f^* N, f^* N')
			\]
			is an equivalence.

		\item $\SH^\hol$ has base change for representable submersions: if $f : X \to Y$ is a representable submersion in $\mathcal C$, then for any Cartesian square
			\[
				\begin{tikzcd}
					X' \ar[d,"p"'] \ar[r,"f'"] & Y' \ar[d,"q"] \\
					X \ar[r,"f"'] & Y
				\end{tikzcd}
			,\]
			the natural map
			\[
				f'_\sharp p^* \to q^* f_\sharp
			\]
			is an equivalence.

		\item $\SH^\hol$ is stable: for any $S \in \mathcal C$, and vector bundle $V$ on $S$, it makes sense to consider the Thom space of $V$ as an object of $\SH^\hol(S)$, and it is $\otimes$-invertible.

		\item Cohomology theories encoded by $\SH^\hol$ are homotopy invariant: for any $S \in \mathcal C$, we can use the objects $E$ of $\SH^\hol(S)$ to encode cohomology theories on $\mathcal C_{/S}$: indeed, we may define the $n$th $E$-cohomology group $E^n(X)$ of $X \to S$ to be the $n$th homotopy group of the mapping space $\SH^\hol(X)(1, p^* E)$, where $p$ is the map $X \to S$, and $1$ is the monoidal unit of $\SH^\hol(X)$. These cohomology theories are homotopy invariant: $E^n(X) \to E^n(X \times \comps)$ is an isomorphism.

		\item Cohomology theories encoded by $\SH^\hol$ are excisive for \'{e}tale neighbourhoods: given a representable embedding $Z \to S$, and $E \in \SH^\hol(S)$, there is a good notion of $E$-cohomology groups with support in $Z$ on objects of $\mathcal C_{/S}$, which we denote by $E^\bullet_Z$. For any $X \in \mathcal C_{/S}$, this fits into a long exact sequence
			\[
				\cdots \to E^n_Z(X) \to E^n(X) \to E^n(X \setminus (X \times_S Z)) \to E^{n - 1}_Z(X) \to \cdots
			,\]
			and $E^n_Z$ is excisive in the following sense: if $\tilde S \to S$ is a representable \'{e}tale neighbourhood of $Z$, then for any $X \in \mathcal C_{/S}$, the map
			\[
				E^n_Z(X) \to E^n_Z(\tilde S \times_S X)
			\]
			is an isomorphism.

	\end{enumerate}
	In fact, $\SH^\hol$ is \textbf{initial} in the category of presheaves satisfying these properties and whose morphisms commute with the $(-)_\sharp$ functors, and are objectwise colimit-preserving and symmetric monoidal.

	Finally, if $G$ is a finite group such that $\B G \in \mathcal C$, then the $\infty$-category $\SH^\hol(\B G)$ is given by formally adjoining $\wedge$-inverses of $G$-representation spheres to the $\infty$-category of pointed homotopy invariant sheaves of spaces on the site of complex manifolds with $G$-action.
\end{pseudothm}

% PERF: should I actually do this, or just say that you can?
% it's actually kind of annoying to do...
% First of all, I think it's actually kind of hard to show that quasi-admissible maps are sent to quasi-admissible maps without using the material on sites with pasting, since otherwise it's hard to show that certain classes of quasi-admissible maps are local
% It's also kind of annoying to show that the acyclic pseudosieves are sent to acyclic pseudosieves, so it's more convenient to produce the realization functors once we've shown that this is implied by localization
% So it's better to do this in a different paper, once we've discussed localization, and where we have room in the appendix for the section about pasting
Furthermore, we can also show that the restriction of $\SH^\hol$ to a suitable category of reduced complex algebraic stacks admits a morphism from the presheaf $\SH$ constructed in \cite{SixAlgSt}, and this morphism commutes with the $(-)_\sharp$ functors, and is objectwise colimit-preserving and symmetric monoidal. Similarly, the restriction of the presheaf $\mathrm{SH}$ constructed in \cite{TwAmb} to a suitable subcategory of $\mathcal C$ also admits a morphism from $\SH^\hol$ satisfying the same properties. Although we can prove this directly using the properties of $\SH^\hol$ given above, we will prefer to postpone the construction of these realization functors until after we establish some results related to the so-called ``gluing property'' defined in \cite{6FF}, and further explored in \cite{Gluing}.

Indeed, the main results we would like to prove about $\SH^\hol$ will be shown in \cite{6FF} and \cite{Gluing}. One of the results we will prove about $\SH^\hol$ in \cite{Gluing} is the gluing property introduced in \cite[Theorem 3.2.21]{A1htpysch}, which implies the so-called ``localization theorem'' that says that if $Z \to S$ is a representable embedding in $\mathcal C$, then there is a natural semi-orthogonal decomposition of $\SH^\hol(S)$ in terms of $\SH^\hol(Z)$ and $\SH^\hol(S \setminus Z)$. We will use this property to show that $\SH^\hol$ has the structure of a 6-functor formalism, that it has Poincar\'e duality, and that the map from the presheaf $\SH$ of \cite{SixAlgSt} respects Grothendieck's six operations, generalizing Ayoub's result in \cite{AyoubBetti}. 

\subsection*{Outline} \label{S:outline}

We record the most important results and definitions of this article here
\begin{description}

	\item[Definitions] \Cref{defn:pullback context}, \Cref{defn:PF}, \Cref{defn:pseudotop}, \Cref{defn:local qadm}, \Cref{defn:D-cover}.

	\item[General results] \Cref{thm:Huniv is univ}, \Cref{thm:PF descent}, \Cref{thm:D-topology}, \Cref{cor:dense slice}, \Cref{cor:describe invariance localization of Huniv}, \Cref{thm:ptd PF sheafy stabilization}.

	\item[Applications] \Cref{prp:univ prop of Halg}, \Cref{thm:SHalg}, \Cref{prp:univ prop of Hdiff}, \Cref{thm:SHdiff}, \Cref{prp:univ prop of Hhol}, \Cref{thm:SHhol}.

\end{description}

Now we give an outline of the article:
\begin{description}

	\item[Basic notions]
		This article begins in \Cref{S:pb} with the basic notions and results about pullback formalisms. In particular, we give the definition of a \emph{pullback context} in \Cref{defn:pullback context}, which is a category $\mathcal C$ equipped with a notion of ``quasi-admissible'' maps, and is the structure that allows us to define pullback formalisms on $\mathcal C$. Pullback formalisms on $\mathcal C$ are then defined in \Cref{defn:PF} to be the presheaves $\mathcal C^\op \to \CAlg(\PrL)$ that have base change and projection formula properties for quasi-admissible maps, and the morphisms between them are transformations that have some compatibility with the quasi-admissible maps.
		% Some general results about pullback formalisms and morphisms between them are established, mostly having to do with establishing that a presheaf is a pullback formalism, and that transformations between them are morphisms of pullback formalisms.
		The last important notion in \Cref{S:pb} is that of the universal pullback formalism, $H^\univ$, which we consider in \Cref{S:univ PF}. Intuitively, $H^\univ$ is the presheaf on $\mathcal C$ given by sending $S \in \mathcal C$ to the category $\Psh(\mathcal C_S)$, where $\mathcal C_S \subseteq \mathcal C_{/S}$ is the full subcategory consisting of quasi-admissible maps to $S$. \Cref{thm:Huniv is univ} then shows that $H^\univ$ is an initial object in the category of pullback formalisms on $\mathcal C$.\footnote{It is worth mentioning that the hard work in the proof of \Cref{thm:Huniv is univ} is done in \cite[Example 2.6.2.8]{internal-cats}, or in \cite[Theorem 3.26]{UnivFF} when $\mathcal C$ is a $1$-category.}

	\item[Notions of generalized descent]
		In \Cref{S:generalized descent}, we will study notions of generalized descent. In particular, we define the notion of pseudotopology in \Cref{S:pseudotopology}, which is a simple and flexible generalization of a Grothendieck topology that also allows us to encode conditions such as homotopy invariance, or preservation of colimits. It turns out that pseudotopologies give precisely the correct level of generality for encoding generalized invariance properties of cohomology theories -- see \Cref{rmk:pseudotops are good}.

		The definition of a pseudotopology on a small category $\mathcal C$ is very simple: it is a collection of maps in $\Psh(\mathcal C)$ whose codomains are representable, and that are stable under base change along maps between representable presheaves (in the case of Grothendieck topologies, we ask that the maps are monomorphisms, and impose further closure conditions on them).

		Although pseudotopologies are not as well-behaved as Grothendieck topologies, they are not so pathological -- \Cref{prp:pseudotop LCL} shows that although the version of sheafification that exists for pseudotopologies is not left exact, it is an accessible locally Cartesian localization that preserves finite products.

		% Using pseudotopologies, we then study local versions of quasi-admissible maps in \Cref{S:local qadm}. In particular, if $\mathcal C \to \tilde{\mathcal C}$ is a functor, $\mathcal C$ is a pullback context and $\tau$ is a pseudotopology on $\tilde{\mathcal C}$, in \Cref{defn:local qadm} we define which maps in $\tilde{\mathcal C}$ ``behave like quasi-admissible maps'' for $\tau$-local presheaves on $\tilde{\mathcal C}$ that restrict to pullback formalisms on $\mathcal C$.

		% Some of the main applications of this \namecref{S:local qadm} are \Cref{prp:PF on Shv} and \Cref{prp:locality of local qadm str}, which allows us to give categories of sheaves useful quasi-admissibility structures, and \Cref{prp:cover qadm} which is useful for producing ``locally quasi-admissible'' maps, and for checking descent.
		
		In \Cref{S:acyclic}, we consider a weaker notion of (generalized) descent for presheaves of categories, which can roughly be seen as replacing the condition that certain maps are equivalences by the condition that they are fully faithful. Seen as a condition on systems of coefficients, this is shown to correspond to the condition that the associated cohomology theories have (generalized) descent.

		Finally, in \Cref{S:qadm covers}, we study descent properties of pullback formalisms. The main results of this \namecref{S:qadm covers} are \Cref{thm:PF descent} and \Cref{thm:D-topology} which give powerful tools for checking descent for pullback formalisms. Given a pullback formalism $D$, the latter result says that if $\{X_i \to S\}_i$ is a family of quasi-admissible maps, then $D$ has descent along this family if and only if the cohomology theories it defines do, and this is equivalent to the seemingly much weaker condition that the functors $\{D(S) \to D(X_i)\}_i$ are jointly conservative. Furthermore, any other pullback formalism admitting a map from $D$ also descends along this family.

		% In \Cref{S:acyclic}, we consider the fundamental notions of $D$-acyclic maps, and invariance properties of $D$. In fact, the $D$-acyclic maps in $\mathcal C$ (or $\Psh(\mathcal C)$) are maps along which the cohomology theories encoded by $D$ are invariant.

	\item[Localizing pullback formalisms]
		\Cref{S:invertibility} is broadly concerned with constructions that freely invert maps in the categories $D(S)$ for $D$ a pullback formalism on $\mathcal C$, and $S \in \mathcal C$. The results of \Cref{S:localizing pf} are very general results to do with localizing along such maps, and can mostly be ignored in favour of the results they are used to prove in \Cref{S:imposing invariance}.

		One of the main results of this section is \Cref{cor:describe invariance localization of Huniv}, which describes the universal pullback formalism whose associated cohomology theories have the invariance properties specified by a prescribed pseudotopology $\tau$. Indeed, this pullback formalism, denoted $H^\tau$, is given by sending $S \in \mathcal C$ to the category $\Psh^\tau(\mathcal C_S)$ of $\tau$-local presheaves, \ie, presheaves that are invariant with respect to $\tau$. When $\tau$ is a Grothendieck topology, these are $\tau$-sheaves, and when $\tau$ expresses a homotopy invariance condition, these are homotopy invariant sheaves. In fact, \Cref{cor:describe invariance localization of Huniv} is proven by noting that $H^\tau$ is obtained by freely imposing $\tau$-invariance on $H^\univ$, and using \Cref{prp:invariance localization}, which describes how to do this for general pullback formalisms.

	\item[Stabilization]
		\Cref{S:stabilize PF} is concerned with the process of ``stabilizing'' pullback formalisms, analogously to the way that the category of spectra is obtaining by stabilizing the category of spaces. This process is decomposed into two operations: producing pullback formalisms of pointed objects, which is considered in \Cref{S:pointing pf}, and formally adjoining $\otimes$-inverses, which is considered in \Cref{S:tensor inversion}. In particular, if for each $S \in \mathcal C$ we are given a collection of objects $A_S$ of $\mathcal C$, we are interested in describing a pullback formalism $D[A^{\otimes -1}]$ obtained from $D$ by freely adjoining $\otimes$-inverses of these objects. Indeed, using \cite[\S2]{robalo}, it is possible to construct $D[A^{\otimes -1}]$ as a presheaf $\mathcal C^\op \to \CAlg(\PrL)$, but it is not clear that $D[A^{\otimes -1}]$ is a pullback formalism. If we assume some good interaction of the collections $\{A_S\}_{S \in \mathcal C}$ with the quasi-admissible maps in $\mathcal C$, then it is possible to show (using \Cref{prp:PF stabilization}) that $D \to D[A^{\otimes -1}]$ is a morphism of pullback formalisms.

		In many cases, this result is sufficient. In particular, it is strong enough to produce the stable motivic homotopy theory of schemes. However, as discussed in more detail in \Cref{rmk:stacky challenges}, it is often the case when $\mathcal C$ is a category of stacks, that we can only check the condition on quasi-admissible maps locally, and a more complicated procedure is required. In particular, the pullback formalism $D[A^{\otimes -1}]$ can no longer be constructed simply by freely adjoining $\otimes$-inverses of objects in $A_S$ to $D[A^{\otimes -1}]$, but we can still give an effective description of it, as shown in \Cref{thm:PF sheafy stabilization}.

		When viewing the category of spectra as a stabilization of the category of spaces, what we mean is that we can construct the category of spectra by formally adjoining $\wedge$-inverses to spheres in the category of pointed spaces. Indeed, often when applying the procedure above of formally adjoining $\otimes$-inverses, we will first want to consider pointed versions of our pullback formalisms, which are provided by \Cref{prp:pointing}. \Cref{thm:ptd PF sheafy stabilization} is a convenient result that combines this with our general stabilization result \Cref{thm:PF sheafy stabilization} mentioned earlier, and is the result we will end up using most regularly in our applications.

	\item[``Motivic homotopy theories'' on various types of stacks]
		\Cref{S:apps} consists of applications of our general results. In it, we produce examples of pullback formalisms ``in the style of motivic homotopy theory'' in the context of algebraic, differentiable, and complex analytic stacks. The algebraic and differentiable contexts recover the constructions of \cite{SixAlgSt} and \cite[\S II]{TwAmb} respectively, and the complex analytic case is new. Note that each case includes unstable and stable versions of the corresponding motivic homotopy theory. As mentioned in the beginning, possibly the most interesting new result of \Cref{S:apps} is \Cref{thm:SHhol}, which establishes the stable motivic homotopy theory of complex analytic stacks.

		% While \Cref{thm:SHhol} is interesting on its own, the author believes that it should be possible to generalize it significantly. In particular, we hope to be able to relax the restriction to finite stabilizers, and the reducedness assumption. Indeed, the results of \Cref{S:apps} should partially be seen as ``proofs of concept'', showing that the general framework we have established in the previous sections is effective at producing results of this type. In \cite{Gluing} and \cite{6FF}, we will show that this general framework is also useful for producing general results about 6-functor formalisms and localization.

\end{description}

\subsection*{Acknowledgements}

The author would like to thank Andrew Blumberg and Johan de Jong for their support as PhD advisors, and Elden Elmanto for his encouragement and advice. The author would also like to thank Martin Gallauer for some technical guidance regarding \Cref{thm:Huniv is univ}.

\subsection*{Notations and Conventions}

Throughout this article, we will make heavy use of the machinery of $\infty$-categories as developed in \cite{htt} and \cite{ha}. Therefore, all of our language will be implicitly $\infty$-categorical:
\begin{enumerate}

	\item We say ``category'' to mean ``$\infty$-category''. Note that then functors, adjoints, and (co)limits must all be understood in the context of $\infty$-categories.

	\item Following \cite[Remark 3.0.0.5]{htt}, we will write $\Cat$ to denote the category of small categories, and $\widehat{\Cat}$ to denote the category of all categories.
		% PERF: should actually just be the category of very large categories?

	\item We write $\spaces$ for the category of small spaces/$\infty$-groupoids/anima (see \cite[\S1.2.16]{htt}), and $\Sp$ for the category of spectra (see \cite[\S1.4.3]{ha}). Write $\widehat{\spaces}$ for the category of all spaces (not necessarily small).

	\item Unless otherwise specified, presheaves and sheaves are always implicitly assumed to take values in $\spaces$. Given a category $\mathcal C$, we write $\Psh(\mathcal C)$ to denote that category of presheaves on $\mathcal C$, and if $\mathcal C$ is equipped with a Grothendieck topology that is understood from context, we write $\Shv(\mathcal C)$ to denote the category of sheaves on $\mathcal C$.

	\item Given a category $\mathcal C$, we write $\mathcal C(-,-)$ for the hom functor $\mathcal C^\op \times \mathcal C \to \widehat{\spaces}$. $\mathcal C$ is locally small if this functor takes values in $\spaces$.

	\item The very large categories $\PrL, \PrR$ of presentable categories are defined in \cite[Definition 5.5.3.1]{htt}. These are the categories of presentable categories and left adjoint functors or right adjoint functors respectively. Note that $\PrL$ is equipped with the structure of a symmetric monoidal category as in \cite[Proposition 4.8.1.15]{ha}.

	\item For any symmetric monoidal category $\mathcal C$, we write $\CAlg(\mathcal C)$ for the category of commutative algebra objects in $\mathcal C$ -- see \cite[Definition 2.1.3.1]{ha}. In particular, $\CAlg(\Cat)$ is the category of symmetric monoidal categories, and $\CAlg(\PrL)$ is the category of symmetric monoidal presentable categories where the monoidal product preserves small colimits in each variable.

	% \item A ``zero object'' of a category $\mathcal C$ is an object that is both initial and terminal. We will say that a category is pointed if it has a zero object. See \cite[Definition 1.1.1.1]{ha}.

		% TODO: universes and size issues

\end{enumerate}

We will also use the following notations and conventions:
\begin{enumerate}

	\item The abbreviation ``qcqs'' will be used to mean ``quasi-compact and quasi-separated'' in any context where these adjectives make sense.

	\item Whenever we say "limit-preserving" or "colimit-preserving", we are referring only to \emph{small} limits and colimits.

	\item If $f : X \to Y$ and $g : X' \to Y$ are maps in a category $\mathcal C$, and the fibred product $X \times_Y X'$ exists, we will sometimes write $f^{-1}(g) : X \times_Y X' \to X$ for the base change of $g$ along $f$ in $\mathcal C$.

	\item Given some implicitly specified ambient category, we will write $\pt$ for a terminal object of that category.

	\item Given a category $\mathcal C$, we will write $\yo : \mathcal C \to \Psh(\mathcal C)$ for the Yoneda embedding of $\mathcal C$. We generally do not include $\mathcal C$ in the notation as it is often clear from context which category's Yoneda embedding we are considering.

	\item Following \cite[\S1.2.8]{htt}, we will use the symbol $\star$ to denote the join of simplicial sets.

	\item Following \cite[Notation 1.2.8.4]{htt}, for any simplicial set $K$, we write $K^\triangleleft$ and $K^\triangleright$ for the simplicial sets obtained by adjoining an initial or terminal cone point respectively to $K$. We will also write $-\infty, \infty$ to denote the cone points of $K^\triangleleft, K^\triangleright$ respectively, so we can write $K^\triangleleft = \{-\infty\} \star K$ and $K^\triangleright = K \star \{\infty\}$.

\end{enumerate}

\section{Pullback Contexts and Formalisms} \label{S:pb}

\subsection{General Motivation: Cohomology Theories with Local Coefficients} \label{S:mot coeff}

We are interested in making a general study of ``coefficients for cohomology theories'' in the context of some general category of geometric objects $\mathcal C$. In particular, this should apply equally well for CW complexes, algebraic stacks, complex analytic stacks, and differentiable stacks. For example, quasi-coherent sheaves give us a notion of coefficients for cohomology theories on schemes, and sheaves of spectra give us coefficients for cohomology theories on topological spaces.

Before making any reference to geometric data, we note that given a category $\mathcal C$, a \emph{system of coefficients for cohomology theories} on $\mathcal C$ should give us the following data:
\begin{enumerate}

	% PERF: clarify that tensor product preserves colimits in each variable
	\item For each $S \in \mathcal C$, we should have some symmetric monoidal presentable category $D(S)$ of ``local coefficients'' on $S$.

	\item For any $S \in \mathcal C$, and $M \in D(S)$, we should have a cohomology object $D(S; M)$ of $S$ with coefficients in $M$. Initially we will only consider the structure of a space (homotopy type) on $D(S; M)$, which gives us ``cohomology sets'' in nonnegative degrees, and which are groups in degree $\geq 1$, and abelian groups in degree $\geq 2$, but eventually we can consider enhancements so that $D(S; M)$ is a spectrum, allowing us to define abelian cohomology groups $D^n(S;M)$ in \emph{all} degrees $n$.

	\item The association $S \mapsto D(S)$ should be contravariantly functorial in $\mathcal C$, so if $f : X \to Y$ is a map in $\mathcal C$, we should have a symmetric monoidal colimit-preserving functor $f^* : D(Y) \to D(X)$.

\end{enumerate}

\begin{exa}
	Here we consider some examples of coefficient systems:
	\begin{enumerate}

		\item Let $\mathcal C$ be the category of schemes. Given a scheme $S$, we can think of $\mathcal O_S$-modules as coefficients for cohomology theories on $S$-schemes. Indeed, we can let $D$ be the contravariant functor sending $S$ to the derived category of $\mathcal O_S$-modules. Given a scheme $X$ over $S$, and a $\mathcal O_S$-module $\mathcal M$, we define the cohomology groups $D^n(X; \mathcal M)$ to be $H^n(X, p^* \mathcal M)$, where $p : X \to S$ is the structure map. When $\mathcal M = \mathcal O_S$, this is simply $H^n(X, \mathcal O_X)$.

		\item Let $\mathcal C$ be the category of topological spaces. Local systems are often used as local coefficients for cohomology theories on topological spaces, so for a topological space $S$, we can set $D(S)$ to be the category of local systems on $S$, which can be thought of as $\Fun(S, \Ab)$, where $\Ab$ is the category of abelian groups. In this case, given a local system $L : S \to \Ab$ in $D(S)$, and a map $p : X \to S$, we define $D^n(X;L)$ to be $H^n(X;L \circ p)$, where $L \circ p : X \to \Ab$ is the pullback of the local system $L$ to a local system on $X$. Note that if $S$ is a point, then local systems on $L$ are simply abelian groups, and $D^n(X;L)$ is simply the $n$th singular cohomology group of $X$ with coefficients in the abelian group $L$.

			In fact, the definition $D(S) = \Fun(S, \Ab)$ is not so convenient for our $\infty$-categorical setting. It makes more sense to set $D(S) = \Fun(S, \Sp)$, where $\Sp$ is the category of spectra. In this case, the case that $S$ is a point also includes any cohomology theory represented by a spectrum, such as $K$-theory and complex cobordism.

			Furthermore, the cohomology groups $D^n(X;L) = H^n(X;L \circ p)$ are given as
			\[
				D^n(X;L) = \pi_n \varprojlim_X (L \circ p)
			.\]

		% \item We can make a global equivariant version of the previous result. Let $\mathcal C$ be the category of orbispaces defined in \cite{htpy-th-of-orbispaces}, which is a subcategory of the category of presheaves on the global orbit category $\Orb$, whose objects are compact Lie groups. By \cite[Proposition 10.4]{laxlimglo}, there is a functor $\Sp_{-} : \Orb^\op \to \CAlg(\PrL)$ which sends a compact Lie group to the category $\Sp_G$ of genuine $G$-spectra, so this defines a limit-preserving functor $\Psh(\Orb)^\op \to \CAlg(\PrL)$, which then restricts to a functor $D : \mathcal C^\op \to \CAlg(\PrL)$. We then have $D(\pt) = \Sp$, and a choice of spectrum $M \in \Sp$ recovers the previous example when evaluated on objects of $\Top \subseteq \mathcal C$.

	\end{enumerate}
	
\end{exa}

Since the assignment $S \mapsto D(S)$ is contravariantly functorial, we have that for any $M \in D(S)$, and $p : X \to S$, we can define a cohomology space $D(X; M)$ as $D(X; p^* M)$.

Given $S \in \mathcal C$, since $D(S)$ is presentable, and $- \otimes -$ preserves colimits in each variable, we have a functor $\underline\Hom_S(-,-) : D(S)^\op \times D(S) \to D(S)$ that preserves limits in each variable, and such that for any $M \in D(S)$, the functor $- \otimes M$ is left adjoint to $\underline\Hom_S(M,-)$. Note that given a map $f : X \to Y$ in $\mathcal C$, for any $M,N \in D(Y)$, there is a natural transformation
\[
	f^* \underline\Hom_Y(M,N) \to \underline\Hom_X(f^* M, f^* N)
.\]

Equipped with this structure, we have that the following two conditions are equivalent:
\begin{enumerate}

	\item For any $M,N \in D(S)$, there is an equivalence $D(S; \underline\Hom_S(M,N)) \simeq D(S)(M,N)$, and this is natural in $M,N$.

	\item The functor $D(S; -) : D(S) \to \spaces$ is corepresented by the monoidal unit of $D(S)$.

\end{enumerate}

Thus, we see that under very modest assumptions, the data of a ``system of coefficients for cohomology theories'' on $\mathcal C$ as above corresponds exactly to the structure of a functor $D : \mathcal C^\op \to \CAlg(\PrL)$: if we write $1$ for the monoidal unit of any monoidal category, for $S \in \mathcal C$, and $M \in D(S)$, we can define
\[
	D(S; M) \coloneqq D(S)(1, M)
.\]

More generally, we can define the following notation:
\begin{nota} \label{nota:cohomology}
	If $\mathcal C$ is a category, and $D : \mathcal C^\op \to \CAlg(\PrL)$ is a functor, for any map $p : X \to S$ in $\mathcal C$, and $M \in D(S)$, we write
	\[
		D(X; M) \coloneqq D(S)(1, p^* M)
	,\]
	which we think of as the ``cohomology of $X$ with coefficients in $M$''.
\end{nota}

Note that in particular, if $D$ takes values in stable categories, this allows us to enhance the cohomology spaces $D(X; M)$ with the structure of spectra, allowing us to obtain a familiar notion of cohomology groups
\[
	H^n_D(X;M) \coloneqq \pi_n D(X; M)
.\]

% Note that this already gives us most of Grothendieck's six operations, and is completely independent of the geometric structure on $\mathcal C$. In order to equip $D$ with the formalism of six operations, we would also need that for certain maps $f : X \to Y$, we have an additional pair of ``exceptional'' adjoint functors $f_! \dashv f^!$ between $D(X), D(Y)$. The various operations are then required to satisfy various relations which imply important results about the cohomology groups, namely duality results such as Poincar\'{e} and Serre duality, and the Lefschetz trace formula.

We also have a good notion of morphisms of coefficient systems: if $\phi : D \to D'$ is a transformation of presheaves $\mathcal C^\op \to \CAlg(\PrL)$, then for any map $p : X \to S$ in $\mathcal C$, and $M \in D(S)$, we get a map
\[
	D(X;M) \to D'(X; \phi(M))
\]
as a map of mapping spaces
\[
	D(X)(1, p^* M) \to D'(X)(\phi(1), \phi(p^* M)) \simeq D'(X)(1, p^* \phi(M))
.\]
If $D,D'$ take values in stable categories, the transformation $\phi$ is automatically objectwise exact, so we can enhance this to a morphism on mapping spectra, and in particular, we get maps 
\[
	H^n_D(X;M) \to H^n_{D'}(X;\phi(M))
.\]

This allows us to compare cohomology theories associated to different types of coefficients.

\begin{exa}
	Suppose $\mathcal C$ is the category of finite type $\comps$-schemes. We can let $D$ be the presheaf that sends a scheme $X$ to the derived category of coherent modules $D_{\Coh}(X)$ on $X$, and let $D'$ be the presheaf that sends a scheme $X$ to the derived category of coherent sheaves on its analytification $X^\an$. As seen in \cite[XII.1.3]{SGA1}, analytification induces a transformation $\phi : D \to D'$, and for any coherent sheaf $M$ on $X$, the transformation
	\[
		H^n_D(X;M) \to H^n_{D'}(X;\phi(M)) \text{ is }
		H^n(X;M) \to H^n(X^\an, M^\an)
	,\]
	allowing us to compare coherent cohomology on $X$ with coherent cohomology on its analytification.
\end{exa}

\subsection{Basic Notions}

We will be interested in coefficient systems that have a ``projection formula'' and ``base change'' for certain maps. In particular, if $D$ is a coefficient system on a category $\mathcal C$, we will want that for certain maps $f : X \to Y$ in $\mathcal C$, the pullback functor $f^* : D(Y) \to D(X)$ has a left adjoint $f_\sharp : D(X) \to D(Y)$ that satisfies
\begin{description}
	
	\item[Projection formula] For any $M \in D(X)$ and $N \in D(Y)$, we have an equivalence
		\[
			f_\sharp(M \otimes f^* N) \to f_\sharp M \otimes N
		.\]

	\item[Base change] If $b : Y' \to Y$ is any map, then there is a Cartesian square
		\[
			\begin{tikzcd}
				X' \ar[d, "a"'] \ar[r, "f'"] & Y' \ar[d, "b"] \\
				X \ar[r, "f"'] & Y
			\end{tikzcd}
		\]
		and an equivalence
		\[
			f'_\sharp a^* \to b^* f_\sharp
		.\]

\end{description}

% NOTE: see 08ET

The geometric data specifying which maps $f$ should have these properties is called a quasi-admissibility structure.

\subsubsection{Pullback contexts and quasi-admissibility structures}

\begin{defn} \label{defn:pullback context}
	A \emph{quasi-admissibility structure} for a category $\mathcal C$ is a collection $Q$ of morphisms satisfying certain conditions. Before stating the conditions we will fix the following notation: for an object $S \in \mathcal C$, the category $\mathcal C_S$ is the full subcategory of $\mathcal C_{/S}$ consisting of maps to $S$ that are in $Q$. Now, $Q$ is said to be a quasi-admissibility structure if and only if
	\begin{enumerate}

		\item Every equivalence is in $Q$.

		\item If $\sigma : S' \to S$ is in $Q$, then for any $f : X \to S$ in $\mathcal C$, there is a base change $f^{-1}(\sigma) : X' \to S'$ of $\sigma$ along $f$, and $f^{-1}(\sigma)$ is in $Q$.

		\item $Q$ is closed under composites. In particular, if $f : X \to Y$ is in $Q$, the functor $f^{-1} : \mathcal C_Y \to \mathcal C_X$ has a left adjoint $\mathcal C_X \to \mathcal C_Y$.

	\end{enumerate}
	A \emph{pullback context} is a category $\mathcal C$ equipped with a quasi-admissibility structure $Q$. We will often simply say that $\mathcal C$ is a pullback context, and call the elements of $Q$ \emph{quasi-admissible morphisms or maps}.

	We will make reference both to small and presentable pullback contexts.

	Say a pullback context $\mathcal C$ is \emph{quasi-small} if for all objects $S \in \mathcal C$, the category $\mathcal C_S$ is small.
\end{defn}

\begin{rmk}
	The terminology of quasi-admissible maps is to be thought of in analogy with the \emph{admissible} maps of \cite[Definition 1.2.1]{dag5} and \cite[Definition 20.2.1.1]{SAG}. In our case, we do not require that quasi-admissible maps are closed under retracts, or that $f : X \to Y$ is quasi-admissible if $g : Y \to Z$ and $g \circ f$ are quasi-admissible. Note that if $g$ also has quasi-admissible diagonal, then a general argument shows that $f$ must be quasi-admissible in this case.
	% This observation is used in \Cref{prp:automatic admissible descent}.

	The collection of smooth morphisms on schemes is the prototypical example of a quasi-admissibility structure, and the collection of \'{e}tale morphisms of schemes is a prototypical example of an admissibility structure. Indeed, a smooth morphism is \'{e}tale if and only if its diagonal is also smooth (see \cite[Tags \stackslink{02GE} and \stackslink{02GK}]{stacks-project}).

	In \cite{Gluing}, we will further explore the notion of admissibility structures.
\end{rmk}

\begin{rmk} \label{rmk:gen qadm}
	If $Q$ is a collection of maps in $\mathcal C$ that is stable under base change, then $\mathcal C$ admits a quasi-admissibility structure in which the quasi-admissible maps are composites of equivalences and maps in $Q$.
	\begin{proof}
		It is clear that this collection is closed under composition and contains all equivalences. Furthermore, it is stable under base change since equivalences and maps in $Q$ are stable under base change.
	\end{proof}
\end{rmk}

\begin{nota}
	Given a functor $F : \mathcal C' \to \mathcal C$, and a presheaf $D : \mathcal C^\op \to \widehat{\Cat}$, we will often denote the presheaf $D \circ F^\op$ on $\mathcal C'$ by $F^* D$.
\end{nota}

\begin{defn}
	A functor between pullback contexts is a morphism of pullback contexts if it preserves quasi-admissible morphisms, and base changes along quasi-admissible morphisms.

	Say a morphism of pullback contexts $F : \mathcal C \to \mathcal D$ is anodyne if for any $S \in \mathcal C$, the induced map
	\[
		\mathcal C_S \to \mathcal D_{F(S)}
	\]
	is an equivalence.
\end{defn}

\begin{rmk}
	Let $\mathcal D$ be a pullback context. A subcategory $\mathcal C$ has at most one quasi-admissibility structure making the inclusion a morphism of pullback contexts. This is the quasi-admissibility structure in which a map is quasi-admissible in $\mathcal C$ if and only if it is quasi-admissible in $\mathcal D$. This is indeed a quasi-admissibility structure if and only the inclusion induces an equivalence of $\mathcal C$ with a replete subcategory of $\mathcal D$, and if $\mathcal C \to \mathcal D$ preserves quasi-admissible base changes. 

	In this case, we will say that $\mathcal C$ is a \emph{pullback subcontext} of $\mathcal D$.
\end{rmk}

\begin{rmk} \label{rmk:full anodyne subctx}
	If $\mathcal D$ is a pullback context, a full subcategory $\mathcal C$ is an anodyne pullback subcontext if and only if for any object of $X \in \mathcal D$, if $X$ admits a quasi-admissible map to an object of $\mathcal C$, then $X \in \mathcal C$.
\end{rmk}

\begin{rmk}
	Quasi-admissibility structures are closed under arbitrary intersections.
\end{rmk}

We end with a remark explaining how to handle size issues in our results.
\begin{rmk}[Handling size issues]
	For technical reasons, some of our results will require that the pullback contexts we consider are \emph{small}. In practice, almost every pullback context $\mathcal C$ we consider will at least be quasi-small and locally small, which means that for any small collection of objects in $\mathcal C$, there is a small full anodyne pullback subcontext of $\mathcal C$ containing them.

	Since many of our constructions behave well with respect to taking full anodyne pullback subcontexts, this will not pose a problem for us. See \Cref{rmk:abuse notation for univ PF} and \Cref{rmk:stabilization compatible with restriction}.
\end{rmk}

\subsubsection{Pullback formalisms}

% PERF: remove "fundamental class" terminology?
\begin{defn} \label{defn:PF}
	Let $\mathcal C$ be a pullback context.
	\begin{itemize}

		\item Say $D : \mathcal C^\op \to \widehat{\Cat}$ \emph{respects quasi-admissibility} if it sends every quasi-admissible map to a right adjoint functor. In this case, for any quasi-admissible map $X \to S$ in $\mathcal C$, we will define the \emph{fundamental class} of $X$ over $S$ to be the object
			\[
				\cls{X} = \cls{X;S} = \cls{X;S}_D \coloneqq p_\sharp(1) \in D(S)
			,\]
			where $p : X \to S$ is the structure map, and $1$ is the monoidal unit of $D(X)$.

		\item Say $D : \mathcal C^\op \to \CAlg(\widehat{\Cat})$ \emph{satisfies the projection formula for quasi-admissible maps} if for every quasi-admissible morphism $f$ in $\mathcal C$, $D(f)$ has a linear left adjoint (see Definition \ref{defn:proj form}).

		\item Say $D : \mathcal C^\op \to \widehat{\Cat}$ that respects quasi-admissibility has \emph{quasi-admissible base change} if $D$ has left base change for quasi-admissible maps (Definition \ref{defn:base change}).

		\item Say $D : \mathcal C^\op \to \CAlg(\widehat{\Cat})$ that satisfies the projection formula for quasi-admissible maps is \emph{geometrically generated} at $S \in \mathcal C$ if the category $D(S)$ is generated under colimits by objects $\cls{X}$ for $X \in \mathcal C_S$. If this holds for all $S \in \mathcal C$, say $D$ is geometrically generated.

	\end{itemize}

	Say a morphism $D \to D'$ in $\Psh_{\widehat{\Cat}}(\mathcal C)$ \emph{respects quasi-admissibility} if it is left adjointable at quasi-admissible maps in the sense of \Cref{defn:compat}, that is, for any quasi-admissible map $X \to Y$, the square
	\[
		\begin{tikzcd}
			D(Y) \ar[d] \ar[r] & D(X) \ar[d] \\
			D'(Y) \ar[r] & D'(X)
		\end{tikzcd}
	\]
	is left-adjointable. Define the category $\PF(\mathcal C)$ of pullback formalisms to be the subcategory of $\Psh_{\CAlg(\PrL)}(\mathcal C)$ of presheaves with quasi-admissible base change that have the projection formula for quasi-admissible maps, and where morphisms must respect quasi-admissibility.
\end{defn}

\begin{rmk}
	Let $\{F_i : \mathcal C_i \to \mathcal C\}_i$ be a family of morphisms of pullback contexts such that for any map $f$ in $\mathcal C$, there is an index $i$ and map $\tilde f$ in $\mathcal C_i$ such that $F_i(\tilde f) = f$, and if $f$ is quasi-admissible then there is an $i$ so that $\tilde f$ can be chosen to be quasi-admissible. Then for any $D : \mathcal C^\op \to \CAlg(\PrL)$, $D$ is a pullback formalism if and only if $D \circ F_i$ is a pullback formalism for all $i$, and any transformation $D \to D'$ is respects quasi-admissibility if and only if $(D \to D') \circ F_i$ is respects quasi-admissibility for all $i$.
	\begin{proof}
		The ``only if'' direction is clear from the definition of morphisms of pullback contexts. For the converse, note that the hypotheses guarantee that any Cartesian square in $\mathcal C$ that has parallel quasi-admissible sides is the image of such a square under $F_i$ for some $i$.
	\end{proof}
\end{rmk}

\begin{exa} \label{exa:slice}
	If $\mathcal C$ is any category that admits all pullbacks, then for any map $X \to Y$ in $\mathcal C$, the slice projection along $X \to Y$ has a right adjoint $\mathcal C_{/Y} \to \mathcal C_{/X}$ given by base change along $X \to Y$. In fact, the slice projection is a linear left adjoint of this functor, where the monoidal structures are Cartesian, and for any Cartesian square
	\[
		\begin{tikzcd}
			X' \ar[d] \ar[r] & Y' \ar[d] \\
			X \ar[r] & Y
		\end{tikzcd}
	\]
	in $\mathcal C$, the resulting square of base change functors
	\[
		\begin{tikzcd}
			\mathcal C_{/Y} \ar[d] \ar[r] & \mathcal C_{/X} \ar[d] \\
			\mathcal C_{/Y'} \ar[r] & \mathcal C_{/X'}
		\end{tikzcd}
	\]
	is left adjointable.

	In particular, if $\mathcal C$ is a presentable category with universal colimits, so that the base change functors $\mathcal C_{/Y} \to \mathcal C_{/X}$ are morphisms in $\CAlg(\PrL)$, then this defines a pullback formalism on $\mathcal C$ with respect to the quasi-admissibility structure of all maps.

	More generally, if $\mathcal C$ is a pullback context, then following \cite[Notation 6.1.3.4]{htt} and \kerodoncite{05S7}, we have that $\mathcal O_{\mathcal C}^\qadm \subseteq \Fun(\Delta^1, \mathcal C) \to \mathcal C$ is a Cartesian fibration, where the projection to $\mathcal C$ is given by evaluation at $1 \in \Delta^1$, and $\mathcal O_{\mathcal C}^\qadm$ is the full subcategory of quasi-admissible maps in $\mathcal C$. This classifies a functor $\mathcal C^\op \to \widehat{\Cat}$ that sends a map $f : X \to Y$ in $\mathcal C$ to the base change functor $f^{-1} : \mathcal C_Y \to \mathcal C_X$, so if $f$ is quasi-admissible, this has a linear left adjoint $\mathcal C_X \to \mathcal C_Y$ given by postcomposition by $f$. This defines a presheaf $\mathcal C^\op \to \widehat{\Cat}$ that has quasi-admissible base change, and \cite[Corollary 2.4.1.9]{ha} shows that this lifts to a presheaf $H^\slice : \mathcal C^\op \to \CAlg(\widehat{\Cat})$ that has the projection formula for quasi-admissible maps with respect to Cartesian monoidal structures. This recovers the ``geometric pullback formalism'' of \cite[\S3.1]{UnivFF}.
\end{exa}

\begin{nota}
	Given a pullback formalism $D$,
	\begin{enumerate}

		\item if $D$ is clear from context, we will write $f^*$ for $D(f)$, and $f_*$ for its right adjoint. If $f$ is quasi-admissible, we will write $f_\sharp$ for the left adjoint.

		\item We will always write $1$ for the monoidal unit of $D(S)$ for any $S \in \mathcal C$.

		% \item If $X \in \mathcal C_S$, we will write
		% 	\[
		% 		\cls X = \cls X_D = \cls{X;S} = \cls{X;S}_D
		% 	\]
		% 	for the fundamental class of $p^*$, where $p : X \to S$ is the quasi-admissible structure map, \ie,
		% 	\[
		% 		\cls{X} = p_\sharp(1)
		% 	.\]
			% (See Definition \ref{defn:fundamental class} for more descriptions.)

		\item Given a section $s : S \to X$ in $\mathcal C_S$, write $\cls{s;S} = \cls{s;S}_D$ for the map $1 \to \cls{X}$ given by $s^*(1 \to p^* p_\sharp(1))$, where $p : X \to S$ is the quasi-admissible structure map.

		\item Given a map $f : X \to Y$ in $\mathcal C_S$, if $p : X \to S$, $q : Y \to S$ are the quasi-admissible structure maps, write $\cls{f;S}_D = \cls{f;S}$ for the map $\cls{X;S}_D \to \cls{Y;S}_D$ adjunct to
			\[
				1 \xrightarrow{\cls{(\id, f); X}_D} \cls{X \times_S Y;X}_D \simeq p^* \cls{Y;S}_D
			.\]

	\end{enumerate}
\end{nota}

\begin{rmk}
	Let $D$ be a pullback formalism on a pullback context $\mathcal C$. Then for any $f : X \to S$ and quasi-admissible $\sigma : S' \to S$,
	\begin{enumerate}

		\item $f^* \cls{S'} \simeq \cls{f^{-1}(S')}$

		\item The counit $\sigma_\sharp \sigma^* \to \id$ is equivalent to $\cls{\sigma;S} \otimes -$
			
		\item If $f$ is quasi-admissible, then $\cls{X \times_S S'; S} \simeq \cls{X} \otimes \cls{S'}$

	\end{enumerate}
\end{rmk}

\subsection{Propagating pullback formalisms along morphisms}

% PERF: maybe generalize away from pullback formalisms?

\Cref{prp:map from PF is to PF} and \Cref{prp:sub PF is PF} will be useful for establishing that the sources and targets of certain transformations are pullback formalisms.

\begin{prp} \label{prp:map from PF is to PF}
	Let $\mathcal C$ be a pullback context, and let $D^\natural \to D$ be a map in $\Psh_{\CAlg(\PrL)}(\mathcal C)$ that respects quasi-admissibility. If the image of $D^\natural(S) \to D(S)$ generates $D(S)$ under small colimits for all $S \in \mathcal C$, and $D^\natural$ is a pullback formalism, then $D$ is a pullback formalism.

	Furthermore, for any pullback formalism $E$ on $\mathcal C$, the square
	\[
		\begin{tikzcd}
			\PF(\mathcal C)(D, E) \ar[d] \ar[r] & \PF(\mathcal C)(D^\natural, E) \ar[d] \\
			\Psh_{\CAlg(\PrL)}(\mathcal C)(D, E) \ar[r] & \Psh_{\CAlg(\PrL)}(\mathcal C)(D^\natural, E)
		\end{tikzcd}
	\]
	is Cartesian.
	\begin{proof}
		This follows easily from \Cref{lem:vert surj adj}. In detail, let $f : X \to Y$ be a quasi-admissible map in $\mathcal C$. Then the square
		\[
			\begin{tikzcd}
				D^\natural(Y) \ar[d] \ar[r] & D^\natural(X) \ar[d] \\
				D(Y) \ar[r] & D(X)
			\end{tikzcd}
		\]
		is left adjointable. To see that $D^\natural$ is a pullback formalism, we must show the following:
		\begin{enumerate}

			\item The square
				\[
					\begin{tikzcd}
						D(Y) \times D(Y) \ar[d, "\otimes"'] \ar[r] & D(X) \times D(Y) \ar[d, "- \otimes f^*"] \\
						D(Y) \ar[r] & D(X)
					\end{tikzcd}
				\]
				is left adjointable. Indeed, if we can show that the outer square in
				\[
					\begin{tikzcd}
						D^\natural(Y) \times D^\natural(Y) \ar[d] \ar[r] & D^\natural(X) \times D^\natural(Y) \ar[d] \\
						D(Y) \times D(Y) \ar[d, "\otimes"'] \ar[r] & D(X) \times D(Y) \ar[d, "- \otimes f^*"] \\
						D(Y) \ar[r] & D(X)
					\end{tikzcd}
				\]
				is left adjointable, then this follows from \Cref{lem:vert surj adj}, since the outer square is left adjointable by \cite[Lemma F.6 (3)]{TwAmb} because it is equivalent to the outer square in the diagram
				\[
					\begin{tikzcd}
						D^\natural(Y) \times D^\natural(Y) \ar[d, "\otimes"'] \ar[r] & D^\natural(X) \times D^\natural(Y) \ar[d, "- \otimes f^*"] \\
						D^\natural(Y) \ar[d] \ar[r] & D^\natural(X) \ar[d] \\
						D(Y) \ar[r] & D(X)
					\end{tikzcd}
				.\]

			\item If $Y' \to Y$ is any map, then
				\[
					\begin{tikzcd}
						D(Y) \ar[d] \ar[r] & D(X) \ar[d] \\
						D(Y') \ar[r] & D(X \times_Y Y')
					\end{tikzcd}
				\]
				is left adjointable. Indeed, if we can show that the outer square in
				\[
					\begin{tikzcd}
						D^\natural(Y) \ar[d] \ar[r] & D^\natural(X) \ar[d] \\
						D(Y) \ar[d] \ar[r] & D(X) \ar[d] \\
						D(Y') \ar[r] & D(X \times_Y Y')
					\end{tikzcd}
				\]
				is left adjointable, then this follows from \Cref{lem:vert surj adj}, since the outer square is left adjointable by \cite[Lemma F.6 (3)]{TwAmb} because it is equivalent to the outer square in the diagram
				\[
					\begin{tikzcd}
						D^\natural(Y) \ar[d] \ar[r] & D^\natural(X) \ar[d] \\
						D^\natural(Y') \ar[d] \ar[r] & D^\natural(X \times_Y Y') \ar[d] \\
						D(Y') \ar[r] & D(X \times_Y Y')
					\end{tikzcd}
				.\]

		\end{enumerate}

		Finally, we will show that if $E$ is a pullback formalism on $\mathcal C$, then
		\[
			\begin{tikzcd}
				\PF(\mathcal C)(D, E) \ar[d] \ar[r] & \PF(\mathcal C)(D^\natural, E) \ar[d] \\
				\Psh_{\CAlg(\PrL)}(\mathcal C)(D, E) \ar[r] & \Psh_{\CAlg(\PrL)}(\mathcal C)(D^\natural, E)
			\end{tikzcd}
		\]
		is Cartesian. We must show that if $D \to E$ is a transformation such that $D^\natural \to D \to E$ respects quasi-admissibility, then $D \to E$ respects quasi-admissibility, so if $X \to Y$ is quasi-admissible, we must show that the bottom square in
		\[
			\begin{tikzcd}
				D^\natural(Y) \ar[d] \ar[r] & D^\natural(X) \ar[d] \\
				D(Y) \ar[d] \ar[r] & D(X) \ar[d] \\
				E(Y) \ar[r] & E(X)
			\end{tikzcd}
		\]
		is left adjointable, given that the top square and outer rectangle are left adjointable, but this follows from \Cref{lem:vert surj adj}.
	\end{proof}
\end{prp}

\begin{prp} \label{prp:sub PF is PF}
	Let $\mathcal C$ be a pullback context, and suppose $D^\natural : \mathcal C^\op \to \widehat{\Cat}$ respects quasi-admissibility. Let $\phi : D \to D^\natural$ be a map in $\Psh_{\widehat{\Cat}}(\mathcal C)$ such that
	\begin{enumerate}

		\item $D(S) \to D^\natural(S)$ is fully faithful for all $S \in \mathcal C$,
			
		\item for any quasi-admissible map $f : X \to Y$ in $\mathcal C$, the functor $f_\sharp : D^\natural(X) \to D^\natural(Y)$ sends $D(X) \subseteq D^\natural(X)$ to $D(Y) \subseteq D^\natural(Y)$,

	\end{enumerate}
	Then $D \to D^\natural$ respects quasi-admissibility, and if $D^\natural$ is a pullback formalism, and for every $S \in \mathcal C$, the essential image of $D(S) \to D^\natural(S)$ is closed under tensor products, then $\phi$ lifts to a morphism of pullback formalisms. In fact, if $D^\natural$ has quasi-admissible base change, then so does $D$, and if $D^\natural$ has the projection formula for quasi-admissible maps, then so does $D$.
	\begin{proof}
		To show that $D \to D^\natural$ respects quasi-admissibility, we must show that for any quasi-admissible $X \to Y$, the square
		\[
			\begin{tikzcd}
				D(Y) \ar[d] \ar[r] & D(X) \ar[d] \\
				D^\natural(Y) \ar[r] & D^\natural(X)
			\end{tikzcd}
		\]
		is left adjointable. Since the vertical arrows are colimit-preserving functors between presentable categories, they admit right adjoints, so this square is left adjointable if and only if its transpose is right adjointable, but the vertical arrows are fully faithful, so the dual of \Cref{lem:compatible localizations} says that this holds if and only if $f^* : D^\natural(Y) \to D^\natural(X)$ sends ``$D(Y)$-coequivalences'' to ``$D(X)$-coequivalences'', or equivalently, if and only if $f_\sharp : D^\natural(X) \to D^\natural(Y)''$ sends $D(X)$ to $D(Y)$, which holds since $f$ is quasi-admissible.

		Finally, since all $-^*$, $-_\sharp$, and $\otimes$ functors of $D$ are restrictions of the same ones for $D^\natural$, the base change and projection formula properties of $D$ follow from the same properties for $D^\natural$.
	\end{proof}
\end{prp}

\subsection{The Universal Pullback Formalism} \label{S:univ PF}

In this \lcnamecref{S:univ PF}, we will construct a pullback formalism, $H^\univ$, which exists for all quasi-small pullback contexts. Often, when defining new pullback formalisms, we begin with $H^\univ$ and then modify it using results from \Cref{S:invertibility} and \Cref{S:stabilize PF}.

Intuitively, this should be the ``universal'' pullback formalism in the sense that it is ``free'' in the category of pullback formalisms. Roughly speaking, given a pullback context $\mathcal C$ and an object $S \in \mathcal C$, $H^\univ(S)$ is determined by the fact that it is a presentable category that receives functors $p_\sharp : H^\univ(X) \to H^\univ(S)$ for any quasi-admissible $p : X \to S$. Since $H^\univ(X)$ is symmetric monoidal, we know that $H^\univ(S)$ must contain $p_\sharp(1)$ for all quasi-admissible $p : X \to S$. Indeed, we will see in \Cref{cnstr:Huniv} that $H^\univ(S) \simeq \Psh(\mathcal C_S)$.

Another perspective on $H^\univ$ is that for any pullback formalism $D$ on a pullback context $\mathcal C$, the objects of $D(S)$ should represent ``cohomology theories'' on $\mathcal C_S$, and these should have some compatibilities with the functorial structure of $D$. In particular, we should have functors $D(S) \to \Psh(\mathcal C_S)$ that are right adjoint to the colimit-preserving functors $\Psh(\mathcal C_S) \to D(S)$ induced by functors $\mathcal C_S \to D(S)$ that send $X \in \mathcal C_S$ to $\cls{X;S}_D$. Indeed, by \cite[Example 2.6.2.8]{internal-cats}, $H^\univ$ is always the initial pullback formalism for any quasi-small pullback context  -- see \Cref{thm:Huniv is univ}.

% \begin{prp} \label{prp:free FF}
% 	If $\mathcal C$ is a small pullback context, then the inclusion $\PF(\mathcal C) \to \Psh_{\CAlg(\PrL)}(\mathcal C)$ admits a left adjoint.
% 	\begin{proof}
% 		This is follows from Proposition 2.12 of \cite{UnivFF}. In this article, the authors assume that $\mathcal C$ is an ordinary category that has all finite limits, but the proof works in our more general setting.
% 		% CHECK
% 		% Alternatively: try to reduce to finitely complete case by passing to limit-preserving pullback formalisms on $\Psh(\mathcal C)$
% 	\end{proof}
% \end{prp}

We will see that in \Cref{rmk:describe Huniv} that $H^\univ$ can be constructed by postcomposing the presheaf $H^\slice$ of \Cref{exa:slice} with the functor $\Psh : \CAlg(\Cat) \to \CAlg(\PrL)$, but we will first construct it as a subformalism of another presheaf $\bar H^\univ$ that will also be useful in various situations.

% In order to give a proper definition of $H^\univ$, we will realize it as a subformalism of another pullback formalism $\hat H^\univ$ that is built using \Cref{exa:slice}. This other pullback formalism $\hat H^\univ$ will turn out to be independently useful for certain purposes.

\begin{cnstr} \label{cnstr:Huniv}
	If $\mathcal C$ is any category, we can consider the quasi-admissibility structure on $\mathcal C$ in which the quasi-admissible maps are those maps that admit all base changes, and we define $\bar H^\univ_{\mathcal C}$ to be the resulting presheaf $H^\slice : \mathcal C^\op \to \CAlg(\widehat{\Cat})$ considered in \Cref{exa:slice}. We will often write simply $\bar H^\univ$ for $\bar H^\univ_{\mathcal C}$ when $\mathcal C$ is clear from context.

	If $\mathcal C$ is a presentable category with universal colimits, then $\bar H^\univ$ takes values in $\CAlg(\PrL)$, so it is a pullback formalism.

	If $\mathcal C$ is a pullback context, define a collection $Q$ of maps in $\Psh_{\widehat{\spaces}}(\mathcal C)$ consisting of those transformations $F \to G$ such that for any $S \in \mathcal C$, and map $\yo(S) \to G$, the base change $F \times_G \yo(S) \to \yo(S)$ is of the form $\varinjlim \yo X \to \yo(S)$ for some small diagram $X : K \to \mathcal C_S$. By construction, $Q$ is stable under base change, so following \kerodoncite{05S7} and \cite[Notation 6.1.3.4]{htt}, this defines a full subpresheaf of $\bar H^\univ_{\Psh_{\widehat{\spaces}(\mathcal C)}}$, and we write $H^\univ_{\mathcal C}$ for the restriction of this presheaf to $\mathcal C \subseteq \Psh_{\widehat{\spaces}}(\mathcal C)$. We will often simply write $H^\univ$ for $H^\univ_{\mathcal C}$ when $\mathcal C$ is clear from context.

	Note that maps in $Q$ are stable under postcomposition by quasi-admissible maps in $\mathcal C$, so by \Cref{prp:sub PF is PF}, we have that $H^\univ \to \bar H^\univ|_{\mathcal C}$ respects quasi-admissibility, that $H^\univ$ has quasi-admissible base change and the projection formula for quasi-admissible maps, and that for any $S \in \mathcal C$, $H^\univ_{\mathcal C}(S) \subseteq \bar H^\univ_{\Psh_{\widehat{\spaces}}}(\yo(S)) \simeq \Psh_{\widehat{\spaces}}(\mathcal C)_{/\yo(S)}$ is the full subcategory generated under small colimits by $\mathcal C_S \subseteq \Psh_{\widehat{\spaces}}(\mathcal C_{/S}) \simeq \Psh_{\widehat{\spaces}}(\mathcal C)_{/\yo(S)}$.

	In particular, if $\mathcal C$ is a quasi-small pullback context, then $H^\univ(S) \simeq \Psh(\mathcal C_S)$, $H^\univ$ takes values in $\CAlg(\PrL)$, and $H^\univ$ is a pullback formalism.

\end{cnstr}

\begin{rmk} \label{rmk:describe Huniv}
	Let $f : X \to Y$ be a map in a quasi-small pullback context $\mathcal C$. Then $f^* : H^\univ(Y) \to H^\univ(X)$ is equivalent to the colimit-preserving functor $\Psh(\mathcal C_Y) \to \Psh(\mathcal C_X)$ induced by the base change functor $f^{-1} : \mathcal C_Y \to \mathcal C_X$. If $f$ is quasi-admissible, this is also equivalent to precomposition by $\mathcal C_X \to \mathcal C_Y$, and the functor $f_\sharp : H^\univ(X) \to H^\univ(Y)$ is the colimit-preserving functor $\Psh(\mathcal C_X) \to \Psh(\mathcal C_Y)$ induced by $\mathcal C_X \to \mathcal C_Y$.
	\begin{proof}
		Since $\yo : \mathcal C \to \Psh_{\widehat{\spaces}}(\mathcal C)$ preserves base changes, the functor $\yo(f)^* : \bar H^\univ(\yo(Y)) \to \bar H^\univ(\yo(X))$ restricts to the base change functor $f^{-1} : \mathcal C_Y \to \mathcal C_X$. Thus, the first statement follows from the fact that $f^* : H^\univ(Y) \to H^\univ(X)$ is a colimit-preserving functor $\Psh(\mathcal C_Y) \to \Psh(\mathcal C_X)$ that restricts to $f^{-1} : \mathcal C_Y \to \mathcal C_X$.

		Similarly, when $f$ is quasi-admissible, the same argument shows that $f_\sharp : H^\univ(X) \to H^\univ(Y)$ is the colimit-preserving functor $\Psh(\mathcal C_X) \to \Psh(\mathcal C_Y)$ induced by $\mathcal C_X \to \mathcal C_Y$, and \cite[Proposition 5.2.6.3]{htt} shows that $f^*$ is the functor $\Psh(\mathcal C_X) \to \Psh(\mathcal C_Y)$ given by precomposing by $\mathcal C_X \to \mathcal C_Y$.
	\end{proof}
\end{rmk}

We record the following consequence of the result given in \cite[Example 2.6.2.8]{internal-cats} that generalizes \cite[Theorem 3.26]{UnivFF}.
\begin{thm} \label{thm:Huniv is univ}
	For any pullback context $\mathcal C$, the presheaf $H^\slice$ of \Cref{exa:slice} is initial in the subcategory of $\Fun(\mathcal C^\op, \CAlg(\widehat{\Cat}))$ where the morphisms are transformations that respect quasi-admissibility, and the objects are those presheaves that have quasi-admissible base change, and the projection formula for quasi-admissible maps.

	Furthermore, if $\mathcal C$ is a quasi-small pullback context, then $H^\univ$ is the initial pullback formalism on $\mathcal C$.
	\begin{proof}
		The first statement follows immediately from \cite[Example 2.6.2.8]{internal-cats}. For the second statement, note that \Cref{rmk:describe Huniv} allows us to identify $H^\univ$ with the composite of $H^\slice$ with the functor $\Psh : \CAlg(\Cat) \to \CAlg(PrL)$, so the result follows from \cite[Remark 4.8.1.9]{ha} after we observe that for any simplicial set $K$, $\Psh$ also induces a left adjoint of the inclusion $\Fun^\LAd(K, \widehat{\Cat}) \to \Fun^\LAd(K, \widehat{\Cat}(\text{all simplicial sets}))$ (recall \cite[Definition 4.8.1.1]{ha}), which restricts to a functor $\Fun^\LAd(K, \Cat) \to \Fun^\LAd(K, \PrL)$.
	\end{proof}
\end{thm}

\begin{rmk}
	For any morphism $F : \mathcal C \to \mathcal D$ of pullback contexts, there is an induced map
	\[
		H^\univ_{\mathcal C} \to H^\univ_{\mathcal D} \circ F
	.\]
	If $F$ is anodyne, then this morphism is an equivalence.
\end{rmk}

\begin{rmk} \label{rmk:geometric stuff in PF}
	Let $\phi : H^\univ \to D$ be a morphism of pullback formalisms. For any $S \in \mathcal C$, we can say the following about $\phi : H^\univ(S) \to D(S)$:
	\begin{itemize}

		\item $\phi : H^\univ(S) \to D(S)$ sends the point presheaf on $\mathcal C_S$ to the monoidal unit, and sends products to tensor products.

		\item For any quasi-admissible $p : X \to S$, we have that $\phi : H^\univ(S) \to D(S)$ sends $\mathcal C_S(-, p) \in \Psh(\mathcal C_S) = H^\univ(S)$ to $\cls{X}$, since $\mathcal C_S(-,X)$ is $p_\sharp \mathcal C_X(-, X) = p_\sharp(\pt)$. Indeed, $\phi$ sends the map $\mathcal C_S(-,X) \to \pt$ to the map $p_\sharp(1) \to 1$.

		\item For $X,Y \in \mathcal C_S$, $\phi$ sends $\mathcal C_S(-, X \times_S Y)$ to $\cls{X} \otimes \cls{Y}$.

		% \item For any quasi-admissible $p : X \to S$, and $P \in \Psh(\mathcal C_S)$, we have that the counit $p_\sharp p^* P \to P$ is the projection map $\mathcal C_S(-,X) \times P \to P$. This follows from the projection formula.

		\item For $f : X \to Y$ in $\mathcal C_S$, $\phi : H^\univ(S) \to D(S)$ sends $\mathcal C_S(-,f)$ to $\cls{f;S} : \cls{X} \to \cls{Y}$ in $D(S)$.

		% \item For quasi-admissible $p : X \to S$ and section $s : S \to X$, the corresponding global section $\pt \to \mathcal C_S(-,X)$ corresponding to $s \in \mathcal C_S(S,X)$ is sent to the map $1 \to p_\sharp(1)$ given by $s^*(1 \to p^* p_\sharp(1))$. This is because $\pt \to \mathcal C_S(-, X)$ is the base change along $s$ of the diagonal map $X \to X \times_S X$.
		% \item For quasi-admissible $p : X \to S$ and $q : Y \to S$, given a map $f : X \to Y$ in $\mathcal C_S$, we have that $f$ induces a section $(\id, f) : X \to X \times_S Y$ of $X \times_S q : X \times_S Y \to X$, so as above we obtain a $1 \to (X \times_S q)_\sharp(1)$, and the map $\mathcal C_S(-,f) : \mathcal C_S(-, X) \to \mathcal C_S(-, Y)$ is sent to the map $p_\sharp(1) \to q_\sharp(1)$ adjunct to
		% 	\[
		% 		1 \to p^* q_\sharp(1) \simeq (X \times_S q)_\sharp(1)
		% 	.\]

	\end{itemize}
\end{rmk}

\section{Generalized Descent and Invariance Conditions} \label{S:generalized descent}

In this section, we will be interested in studying various generalized notions of descent. In \Cref{S:pseudotopology}, we define the notion of \emph{pseudotopology}, which generalizes the notion of a Grothendieck topology, and can be seen as encoding the condition that certain diagrams behave like colimiting diagrams. We will use the very general notion of descent encoded by pseudotopologies to study a notion of ``locally quasi-admissible maps'' in \Cref{S:local qadm}. In this article, the main application of this notion will be for describing when certain maps between \emph{presheaves} behave like quasi-admissible maps. Indeed, we will use this notion to define \emph{quasi-admissible pseudotopologies} in \Cref{S:imposing invariance}, which will be used in that section to describe how to freely enforce invariance properties that we will discuss in \Cref{S:acyclic}. Indeed, \Cref{S:acyclic} discusses a weaker notion of (generalized) descent for systems of local coefficients which corresponds to (generalized) descent conditions of the associated cohomology theories.

Finally, in \Cref{S:qadm covers}, we consider ordinary descent conditions in the setting of pullback formalisms, and we will see in \Cref{thm:PF descent} and \Cref{thm:D-topology} that it is very easy to show descent for pullback formalisms along quasi-admissible families.

\subsection{Pseudotopologies} \label{S:pseudotopology}

% WARN: reason for assuming locally small is to ensure that representable presheaves are in $\Psh$, ie they take values in (small) spaces,
% maybe don't really need to assume this, since can still consider maps to representable presheaves, regardless of size
% Also, note that in quasi-small case, the categories $\mathcal C_S$ are small, so we are taking $\tau$-local presheaves on a small category $\mathcal C_S$ where $\tau$ is a small pseudotopology (good)

In this \namecref{S:pseudotopology}, we will study pseudotopologies, which give a generalization of the structure of a Grothendieck topology that allows us to work with various ``generalized descent properties'' that cannot be expressed in terms of Grothendieck topologies. When we study acyclicity properties in the setting of pullback formalisms, we will also see that pseudotopologies give the correct level of generality with which to study invariance properties of cohomology theories -- see \Cref{rmk:pseudotops are good}, and \Cref{lem:char pseudotops}.

This greater generality comes at a cost. Indeed, the corresponding notion of ``sheafification'' may not preserve finite limits, but \Cref{prp:pseudotop LCL} shows that it is still quite well-behaved.

In \Cref{S:imposing invariance}, we will see how to use pseudotopologies to freely impose invariance conditions on the cohomology theories coming from pullback formalisms.

Let us now consider the various generalized descent properties we may want to consider:
\begin{description}

	\item[Homotopy invariance]
		Homotopy invariance properties can be seen as descent along projections $X \times I \to X$, where $I$ is some object we think of as an interval. For example, the fact that singular cohomology of topological spaces is homotopy invariant can be expressed by stating that for any topological space $X$, the map
		\[
			H^\bullet(X;\ints) \to H^\bullet(X \times [0,1]; \ints)
		\]
		is an isomorphism. This is also the sort of homotopy invariance statement we like to use when talking about $\aff^1$-homotopy: a cohomology theory $H^\bullet$ on varieties is said to be $\aff^1$-invariant if for any variety $X$, the map
		\[
			H^\bullet(X) \to H^\bullet(\aff^1_X)
		\]
		is an isomorphism.

	\item[Excision and \v{C}ech descent]
		If $F$ is a presheaf on a category $\mathcal C$, then excision properties of $F$ are often expressed by saying that $F$ sends certain squares in $\mathcal C$ to Cartesian squares. By viewing a square either as a cone over the cospan diagram $K \coloneqq (\bullet \to \bullet \gets \bullet)$ or a cone under the span diagram $K^\op = (\bullet \gets \bullet \to \bullet)$ this can be seen as saying that $F$ sends certain $K^\triangleright$-indexed diagrams to limiting $(K^\op)^\triangleleft = (K^\triangleright)^\op$-indexed diagrams.

		By considering more general simplicial sets $K$, we can recover various other examples. For example, the previous case, where we consider the condition that $F$ sends certain maps to equivalences, corresponds to the case that $K = \pt$. The example of \v{C}ech descent can be seen as the case where $K = \Delta$. % PERF: details

	\item[Descent along sieves]
		A sieve on an object $X$ of a category $\mathcal C$ can be seen as a monomorphism $\mathcal U \to \yo(X)$ in $\Psh(\mathcal C)$. If the sieve is generated by some family $\{X_i \to X\}_i$ of maps to $X$, then $\mathcal U \to \yo(X)$ is often seen as a map from a colimit
		\[
			\varinjlim_{n \in \Delta} \check{C}_n(\coprod_i \yo(X_i) \to \yo(X)) \to \yo(X)
		,\]
		where we write $\check{C}_\bullet(f)$ to denote the \v{C}ech nerve of a map $f$. Thus, we can view descent along $\mathcal U \to \yo(X)$ in terms of the previous example for $K = \Delta$, as we mentioned earlier. On the other hand, the sieve can also be seen as a full subcategory of $\mathcal C_{/X}$, in which case $\mathcal U \to \yo(X)$ can be seen as the map from the colimit
		\[
			\varinjlim_{\substack{X' \to X \\ \text{in $\mathcal U$}}} \yo(X') \to \yo(X)
		,\]
		and we can see that descent along the sieve can be seen as the special case of the previous example where $K$ is the full subcategory of $\mathcal C_{/X}$ corresponding to $\mathcal U \to \yo(X)$.

		Descent along this sieve does not depend on how we choose to present $\mathcal U \to \yo(X)$ as a map from a colimit, and indeed, there is a ``coordinate-free''' way to express this: we can view the presheaf $F$ on $\mathcal C$ as a limit-preserving presheaf on $\Psh(\mathcal C)$, and the condition that $F$ has descent along the sieve $\mathcal U \to \yo(X)$ is equivalent to the condition that it sends this map $\mathcal U \to \yo(X)$ to an equivalence.

	\item[Sending colimits to limits] We can view the condition that a presheaf $F$ on $\mathcal C$ preserves certain limits as a generalized descent property. Indeed, if $X : K \to \mathcal C$ is a diagram that admits a colimit, then the condition that $F$ sends this colimit to a limit can be seen as saying that if $\bar X : K^\triangleright \to \mathcal C$ is a colimiting extension of $X$, then we can consider the condition that $F \circ \bar X^\op$ is a limiting $(K^\op)^\triangleleft = (K^\triangleright)^\op$-indexed diagram. This can also be expressed in terms of our ``coordinate-free'' formulation by asking that if we view $F$ as a limit-preserving presheaf on $\mathcal C$, then it sends the map
		\[
			\varinjlim \yo X \to \yo \varinjlim X
		\]
		to an equivalence.

\end{description}

Thus, we are lead to consider the notion of descent along arbitrary maps in $\Psh(\mathcal C)$ whose codomain is a representable presheaf.

\begin{defn} \label{defn:pseudotop}
	% Given a category $\mathcal C$, equipped with a $\jmath : \mathcal C \to \tilde{\mathcal C}$ to a presentable category, we make the following definitions:
	Given a category $\mathcal C$, we make the following definitions:

	\begin{itemize}

		% \item A \emph{$\jmath$-pseudosieve} on $X \in \mathcal C$ is a map to $\jmath(X)$. When $\jmath = \yo$ is the Yoneda embedding, we will simply say ``pseudosieve''.
		\item A \emph{pseudosieve} on $X \in \mathcal C$ is a map in $\Psh_{\widehat{\spaces}}(\mathcal C)$ to $\yo(X)$, where $\widehat{\spaces}$ is the category of all spaces (not necessarily small).\footnote{If $\mathcal C$ is not locally small, then $\yo(X)$ might not take values in small spaces.}

		% \item A (small) \emph{pseudotopology} on $\jmath$, consists of, for each $X \in \mathcal C$, the specification of a (small) class of $\jmath$-pseudosieves on $X \in \mathcal C$, called acyclic pseudosieves, such that for any acyclic $\tilde Y \to \jmath(Y)$, and map $X \to Y$, the base change of $g$ along $\jmath(X \to Y)$ is an acyclic pseudosieve on $X$.
		\item A \emph{pseudotopology} $\tau$ on $\mathcal C$, consists of, for each $X \in \mathcal C$, the specification of a collection of pseudosieves on $X \in \mathcal C$, called ($\tau$-)acyclic pseudosieves, such that for any acyclic pseudosieve $\tilde Y \to \yo(Y)$, and map $X \to Y$, the base change of $g$ along $\yo(X \to Y)$ is an acyclic pseudosieve on $X$. We may also view $\tau$ simply as the collection of maps in $\Psh(\mathcal C)$ that are $\tau$-acyclic pseudosieves, so we sometimes say that a map $P \to \yo(X)$ is \emph{in $\tau$} to mean that it is a $\tau$-acyclic pseudosieve.

		\item A pseudotopology $\tau$ on $\mathcal C$ is \emph{small} if for every $X \in \mathcal C$, the collection of $\tau$-acyclic pseudosieves on $S$ is small, and every $\tau$-acyclic pseudosieve on $X$ is a map $P \to \yo(X)$ where $P$ is small colimit of representable presheaves.

		\item Given a pseudotopology $\tau$ on $\mathcal C$, say that a diagram $X : K^\triangleright \to \mathcal C$ is $\tau$-acyclic if the map
			\[
				\varinjlim_{p \in K} \yo(X(p)) \to \yo(X(\infty))
			\]
			is a $\tau$-acyclic pseudosieve.
		
		% \item A \emph{saturated} pseudotopology is a pseudotopology in which every equivalence to $\jmath(X)$ is an acyclic pseudosieve on $X$, and if $P \to \jmath(X)$ is an acyclic pseudosieve, and $Q \to \jmath(X)$ is any map such that for all $\jmath(Y) \to P$, the base change $Q \times_{\jmath(X)} \jmath(Y) \to \jmath(Y)$ is acyclic, then $Q \to \jmath(X)$ is acyclic.
		\item A \emph{saturated} pseudotopology is a pseudotopology in which every equivalence to a representable presheaf is an acyclic pseudosieve on $X$, and for any acyclic diagram $X : K^\triangleright \to \mathcal C$, if $p$ is any map to $\yo(X(\infty))$ such that for all $a \in K$, the base change of $p$ along $\yo(X(a \to \infty))$ is an acyclic pseudosieve, then $p$ is an acyclic pseudosieve.

		% \item We will say ``pseudotopology on $\mathcal C$'' to refer to pseudotopologies on the Yoneda embedding of $\mathcal C$.

		% \item We say an object of $\tilde{\mathcal C}$ is $\tau$-local if it is local with respect to the $\tau$-local equivalences.
		\item Given a pseudotopology $\tau$ on $\mathcal C$, we say an object of $\Psh_{\widehat{\spaces}}(\mathcal C)$ is \emph{$\tau$-local} if it is local with respect to the $\tau$-acyclic pseudosieves. More generally, for any category $\mathcal V$, say a presheaf $F : \mathcal C^\op \to \mathcal V$ is $\tau$-local if for any $\tau$-acyclic diagram $X : K^\triangleright \to \mathcal C$, $F \circ X^\op$ is limiting. We define $\Psh^\tau_{\mathcal V}(\mathcal C) \subseteq \Psh_{\mathcal V}(\mathcal C)$ to be the full subcategory of $\tau$-local presheaves.

		% \item Define the collection of \emph{$\tau$-local equivalences} in $\tilde{\mathcal C}$ to be the strongly saturated class generated by base changes of the acyclic pseudosieves. Also see Definition \ref{defn:local equivalences for pseudotop}.
		\item Given a pseudotopology $\tau$ on $\mathcal C$, we define the collection of \emph{$\tau$-local equivalences} in $\Psh(\mathcal C)$ to be the maps $P \to Q$ in $\Psh_{\widehat{\spaces}}(\mathcal C)$ such that for any $\tau$-local presheaf $R$ on $\mathcal C$, the map
			\[
				\Psh_{\widehat{\spaces}}(\mathcal C)(P \to Q, R)
			\]
			is an equivalence.

			When $\mathcal C$ is locally small and $\tau$ is small, the collection of $\tau$-local equivalences in $\Psh(\mathcal C)$ is the strongly saturated class generated by the $\tau$-acyclic pseudosieves. Also see \Cref{defn:local equivalences for pseudotop}.
			% NOTE: don't need to ask for base changes of acyclic pseudosieves, since any base change of a pseudosieve will be a colimit of acyclic pseudosieves.

		% \item A pseudotopology $\tau$ on $\mathcal C$ is said to be \emph{subcanonical} if all representable presheaves on $\mathcal C$ are $\tau$-local.

	\end{itemize}
\end{defn}

\begin{rmk}[Pseudotopologies are general enough] \label{rmk:pseudotops are good}
	It may seem that by requiring acyclic pseudosieves to be stable under base change, we only allow ourselves to encode invariance conditions that are stable under base change, such as descent, but this is not the case.

	For example, although homotopy equivalences of topological spaces are not stable under base change, they are ``generated'' by maps of the form $X \times [0,1] \to X$, which are stable under base change. Indeed, if a presheaf $H^\bullet$ on topological spaces sends such maps to equivalences, then it also sends the sections $X \times \{t\} \to X \times [0,1]$ to equivalences, so it sends homotopic maps to equivalent maps, which implies that $H^\bullet$ sends homotopy equivalences to equivalences.

	In particular, the invariance conditions encoded by pseudotopologies can encode both descent conditions, and homotopy invariance conditions.

	In fact, \Cref{lem:qadm acyclic} shows that ``quasi-admissible invariance properties'' of cohomology theories coming from pullback formalisms are stable under base change are stable under base change, so that these properties are perfectly encoded by quasi-admissible pseudotopologies (see \Cref{defn:qadm pseudotop}).
\end{rmk}

\begin{rmk}
	Let $X$ be an object of a category $\mathcal C$, and let $f,g : P \to \yo(X)$ be homotopic maps in $\Psh(\mathcal C)$. It follows that $f$ and $g$ are base changes of each other along $\id_{\yo(X)}$, so for any pseudotopology $\tau$ on $\mathcal C$, $f$ is $\tau$-acyclic if and only if $g$ is $\tau$-acyclic.
\end{rmk}

We note the following familiar examples of pseudotopologies:
\begin{exa}
	Let $\mathcal C$ be a category.
	\begin{enumerate}

		\item Any Grothendieck topology on $\mathcal C$ can be seen as a saturated pseudotopology by setting the acyclic pseudosieves to be the covering sieves.

		\item If $\mathcal A$ is a collection of maps in $\mathcal C$ that is stable under base change, then we can define a pseudotopology on $\mathcal C$ where the acyclic pseudosieves are maps of the form $\yo(f)$ for $f \in \mathcal A$.

		\item Any union of pseudotopologies on $\mathcal C$ is a pseudotopology.

	\end{enumerate}
	
\end{exa}

\begin{exa} \label{exa:colim pseudotop}
	Let $\mathcal C$ be a category with (small) universal colimits. Then we can define a pseudotopology on $\mathcal C$ as follows: the acyclic pseudosieves are those maps of the form $\varinjlim \yo X \to \yo \varinjlim X$, where $X : K \to \mathcal C$ is a small diagram. Note that in this case, the acyclic diagrams are precisely the small colimiting diagrams, so the local presheaves for this pseudotopology are the limit-preserving ones.

	More generally, if $\mathcal C$ does not necessarily have universal colimits, following \cite[Lemma 6.1.3.3(5)]{htt}, following \kerodoncite{05SE}, we say that a diagram $X : K \to \mathcal C$ admits a universal colimit if there is an extension $\bar X : K^\triangleright \to \mathcal C$ such that for any Cartesian transformation $\bar X' \to \bar X$, $\bar X'$ is a colimiting diagram. We then define the pseudotopology on $\mathcal C$ by setting the acyclic pseudosieves to be the maps of the form $\varinjlim \yo X \to \yo \varinjlim X$, for small diagrams $X : K \to \mathcal C$ that admit universal colimits.
\end{exa}

\begin{defn} \label{defn:pseudotop acyclic map}
	Given a pseudotopology $\tau$ on a category $\mathcal C$, a map $P \to Q$ in $\Psh_{\widehat{\spaces}}(\mathcal C)$ is said to be $\tau$-acyclic if for any $\yo(Y) \to Q$, the map $P \times_Q \yo(Y) \to \yo(Y)$ is a $\tau$-acyclic pseudosieve.
	% NOTE: this does not conflict with the previous notion of acyclic pseudosieve
	% i.e. a pseudosieve is a $\tau$-acyclic map if and only if it is a $\tau$-acyclic pseudosieve
\end{defn}

\begin{rmk} \label{rmk:pseudotop acyclic map is local equiv}
	In the setting of \Cref{defn:pseudotop acyclic map}, every $\tau$-acyclic map in $\Psh(\mathcal C)$ is a small colimit of $\tau$-acyclic pseudosieves, so it is a $\tau$-local equivalence.
\end{rmk}

The following result describes what sorts of classes of morphisms may appear as collections of acyclic maps for a pseudotopology.
\begin{lem} \label{lem:char pseudotops}
	Let $\mathcal C$ be a small category. Then there is a function from the collection of pseudotopologies on $\mathcal C$ to the collection of classes of morphisms in $\Psh(\mathcal C)$ that are stable under base change, which sends a pseudotopology $\tau$ to the corresponding class of $\tau$-acyclic maps.

	This function is injective, and its image is the collection of local classes of morphisms in $\Psh(\mathcal C)$ (in the sense of \cite[Definition 6.1.3.8]{htt}).

	% In fact, for any pullback formalism $D$ on $\mathcal C$, there is a unique saturated pseudotopology $\tau_D$ such that the collection of quasi-admissible $D$-acyclic maps is the collection of quasi-admissible $\tau_D$-acyclic maps.
	% PERF: make this statement somewhere
	\begin{proof}
		For any pseudotopology $\tau$, \Cref{defn:pseudotop acyclic map} of $\tau$-acyclic maps immediately implies that the $\tau$-acyclic maps are stable under base change, so any pseudotopology is sent to a class of morphisms that is stable under base changes.

		Note that the function sending a class of morphisms $\mathcal W$ in $\Psh(\mathcal C)$ to the collection of morphisms in $\mathcal W$ whose codomain is representable defines a retraction of the function we are considering, and the corresponding idempotent map on the class of collections of morphisms in $\Psh(\mathcal C)$ fixes precisely those $\mathcal W$ such that $P \to Q$ is in $\mathcal W$ if and only if $P \times_Q \yo(Y) \to \yo(Y)$ is in $\mathcal W$ for all maps $\yo(Y) \to Q$. Thus, Lemma \ref{lem:locality in psh} shows that the image of this function consists precisely of the local classes of morphisms in $\Psh(\mathcal C)$.

		% The final statement follows from \Cref{prp:PF on Psh}, \Cref{lem:acyclic qadm is local}, and \Cref{lem:props of acyclic}. % PERF: details?
	\end{proof}
\end{lem}

If $\tau$ is a pseudotopology on a category $\mathcal C$, we can think of the inclusion $\Psh^\tau(\mathcal C) \to \Psh(\mathcal C)$ as a generalization of the inclusion of sheaves into presheaves for a Grothendieck topology. Unlike the case of Grothendieck topologies, when this inclusion admits a left adjoint, it will often not preserve finite limits, but the following result shows that it is still reasonably well-behaved. In particular, a map preserves finite limits if and only if it preserves finite products and base changes, and now we see that the localization $\Psh(\mathcal C) \to \Psh^\tau(\mathcal C)$ preserves finite products and base changes along maps between $\tau$-local presheaves. \cf{} \cite[Proposition 3.4]{sixopsequiv}.
\begin{prp} \label{prp:pseudotop LCL}
	If $\tau$ is a small pseudotopology on a small category $\mathcal C$, then $\Psh^\tau(\mathcal C)$ is a presentable category with universal colimits, and the inclusion $\Psh^\tau(\mathcal C) \subseteq \Psh(\mathcal C)$ has an accessible left adjoint which is a locally Cartesian localization\footnote{See \Cref{S:LCL}.} that preserves finite products.
	\begin{proof}
		Since $\mathcal C$ is small, and $\tau$ is small, it follows that the collection of $\tau$-acyclic pseudosieves is a small collection of maps in $\Psh(\mathcal C)$.

		Note that base change along any map in $\Psh(\mathcal C)$ sends $\tau$-acyclic pseudosieves to $\tau$-acyclic maps, which are $\tau$-local equivalences by \Cref{rmk:pseudotop acyclic map is local equiv}. Thus, we may conclude by \Cref{prp:crit for LCL}.
	\end{proof}
\end{prp}

Before considering pseudotopologies on pullback contexts, we will make the following definitions of locality for classes of morphisms:
\begin{defn}
	Let $\tau$ be a pseudotopology on a category $\mathcal C$, and let $Q$ be a collection of maps in $\mathcal C$.
	\begin{enumerate}

		\item A map is a \emph{transfinite $\tau$-composite of maps in $Q$} if it is of the form $F(\emptyset) \to F(\lambda)$ for some ordinal $\lambda$ and diagram $F : \lambda + 1 \to \tilde{\mathcal C}$ such that for every limit ordinal $\lambda' \leq \lambda$, $F$ restricts to a $\tau$-acyclic diagram on $\lambda' + 1 \cong (\lambda')^\triangleright$, and for every $\alpha \in \lambda$, the map $F(\alpha) \to F(\alpha + 1)$ is in $Q$. Say $Q$ is \emph{closed under transfinite $\tau$-composites} if every such map is in $Q$. Say a collection $\tilde Q \supseteq Q$ is \emph{generated by $Q$ under transfinite $\tau$-composites} if every collection containing $Q$ that is closed under transfinite $\tau$-composites also contains $\tilde Q$.

		\item A map is \emph{(strictly) $\tau$-locally on the target in $Q$} if it is of the form $f(\infty) : X(\infty) \to Y(\infty)$ for some (Cartesian) transformation of small $\tau$-acyclic diagrams $f : X \to Y : K^\triangleright \to \mathcal C$ such that for all $a \in K$, $f(a) \in Q$. Say $Q$ is \emph{(strictly) $\tau$-local on the target} if every such map is in $Q$. Say a collection $\tilde Q \supseteq Q$ is \emph{generated by $Q$ under (strict) $\tau$-locality on the target} if every collection containing $Q$ that is (strictly) $\tau$-local on the target also contains $\tilde Q$.

		\item A map is \emph{$\tau$-locally on the source in $Q$} if it is of the form $X(0) \to X(1)$ for some small diagram $X : K \star \Delta^1$ such that $X|_{K \star \{0\}}$ is $\tau$-acyclic, and $X|_{K \star \{1\}}$ sends all edges to $Q$. Say $Q$ is \emph{$\tau$-local on the source} if every such map is in $Q$. Say a collection $\tilde Q \supseteq Q$ is \emph{generated by $Q$ under $\tau$-locality on the source} if every collection containing $Q$ that is $\tau$-local on the source also contains $\tilde Q$.

	\end{enumerate}
	We may drop $\tau$ from the notation when it is the pseudotopology of universal colimits from \Cref{exa:colim pseudotop}.
\end{defn}

\subsubsection{Weakly quasi-admissible pseudotopologies}

\begin{defn} \label{defn:weak yo-qadm}
	% PERF: should this actually go before defn of Huniv, since use these maps to define it?
	Given a pullback context $\mathcal C$, a map $F \to G$ in $\Psh_{\widehat{\spaces}}(\mathcal C)$ is \emph{weakly $\yo$-quasi-admissible} if for every $S \in \mathcal C$, and map $\yo(S) \to G$, the base change $F \times_G \yo(S) \to \yo(S)$ is equivalent to a map of the form
	\[
		\varinjlim \yo X \to \yo(S)
	,\]
	where $X : K \to \mathcal C_S$ is some small diagram.
\end{defn}

\begin{lem} \label{lem:weakly yo-qadm morphs}
	If $F : \mathcal C \to \mathcal D$ is a morphism of locally small pullback contexts, then the evident colimit-preserving extension $F_! : \Psh(\mathcal C) \to \Psh(\mathcal D)$ of $F$ preserves weakly $\yo$-quasi-admissible maps, and base changes along weakly $\yo$-quasi-admissible maps to representable presheaves.

	If $F$ is anodyne, we also have that for any $P \in \Psh(\mathcal C)$, the induced functor $\Psh(\mathcal C)_{/P} \to \Psh(\mathcal C)_{/F_! P}$ induces an equivalence on full subcategories of weakly $\yo$-quasi-admissible maps.
	\begin{proof}
		Note that by \Cref{lem:locality in psh}, the collection of weakly $\yo$-quasi-admissible maps is a local class in the sense of \cite[Definition 6.1.3.8]{htt}, so by \cite[Lemma 6.1.3.7(1)]{htt}, it suffices to show all of the statements for weakly $\yo$-quasi-admissible maps to representable presheaves.

		The fact that $F_!$ preserves weakly $\yo$-quasi-admissible maps to representable presheaves follows from the fact that $F$ preserves quasi-admissible maps, and $F_!$ preserves small colimits.

		For the statement about base changes, first note that by universality of colimits and the fact that $F_!$ preserves small colimits, it suffices to show that $F_!$ preserves base changes along $\yo$-quasi-admissible maps of maps between representable presheaves. Using the description of $\yo$-quasi-admissible maps to representable presheaves, this follows from universality of colimits, the fact that $F_!$ preserves small colimits, and the fact that $F$ preserves quasi-admissible base changes.

		Finally, note that for any $S \in \mathcal C$, the restriction of $\Psh(\mathcal C)_{/\yo(S)} \to \Psh(\mathcal D)_{/\yo F(S)}$ to full subcategories of weakly $\yo$-quasi-admissible maps is the evident colimit-preserving functor $\Psh(\mathcal C_S) \to \Psh(\mathcal D_{F(S)})$ extending $\mathcal C_S \to \mathcal D_{F(S)}$, so this is an equivalence when $F$ is anodyne.
	\end{proof}
\end{lem}

% NOTE: just need that locally small for some choice of universe
\begin{lem} \label{lem:pull back pseudotop}
	Let $F : \mathcal B \to \mathcal C$ be a morphism of locally small pullback contexts, and let $F_! : \Psh(\mathcal B) \to \Psh(\mathcal C)$ be the evident colimit-preserving extension of $F$.

	If $\tau$ is a pseudotopology on $\mathcal C$ such that all $\tau$-acyclic pseudosieves are weakly $\yo$-quasi-admissible, then there is a pseudotopology $\tau_{\mathcal B}$ on $\mathcal B$ such that for any $S \in \mathcal B$, a pseudosieve $P \to \yo(S)$ is $\tau_{\mathcal B}$-acyclic if and only if $F_! P \to \yo(F(S))$ is $\tau$-acyclic.
	\begin{proof}
		To show that $\tau_{\mathcal B}$ is a pseudotopology, we need to show that if $F_!(P \to \yo(S))$ is $\tau$-acyclic, then for any $T \to S$ in $\mathcal B$, $F_!(P \times_{\yo(S)} \yo(T) \to \yo(T))$ is $\tau$-acyclic. Indeed, \Cref{lem:weakly yo-qadm morphs} says that the latter map is the base change of $F_!(P \to \yo(S))$ along $\yo F(T) \to \yo F(S)$, which is $\tau$-acyclic since $\tau$ is a pseudotopology, and $F_!(P \to \yo(S))$ is $\tau$-acyclic. 
	\end{proof}
\end{lem}

\begin{prp} \label{prp:quasi-top basis}
	Let $\mathcal C$ be a locally small pullback context equipped with a pseudotopology $\tau$. Let $\mathcal C' \subseteq \mathcal C$ be a small full anodyne pullback subcontext such that for any object $S \in \mathcal C$, there is a family of quasi-admissible maps from objects of $\mathcal C'$ that generate a $\tau$-acyclic sieve on $S$.

	If every $\tau$-acyclic pseudosieve is a sieve generated by quasi-admissible maps, then the restriction functor $\Psh(\mathcal C) \to \Psh(\mathcal C')$ restricts to an equivalence $\Psh^\tau(\mathcal C) \to \Psh^{\tau'}(\mathcal C')$, where $\tau'$ is the pseudotopology on $\mathcal C'$ induced by $\tau$, as in \Cref{lem:pull back pseudotop}.

	\begin{proof}
		Note that a quasi-topology, as defined in \cite[C.1]{quadratic-refinement-GLV-trace}, is precisely a pseudotopology in which all acyclic pseudosieves are sieves. Thus, by hypothesis, we have that $\tau$ is a quasi-topology.

		By \Cref{lem:weakly yo-qadm morphs}, we have that the colimit-preserving functor $u_! : \Psh(\mathcal C') \to \Psh(\mathcal C)$ that extends $u$ preserves base changes along weakly $\yo$-quasi-admissible base changes. In particular, $u_!$ preserves diagonals of weakly $\yo$-quasi-admissible maps. Since $u$ is fully faithful, so is $u_!$, so $u_!$ is conservative. This shows that $u_!$ preserves and reflects weakly $\yo$-quasi-admissible monomorphisms (where we use that fact that $u$ is anodyne to reflect weakly $\yo$-quasi-admissible maps using \Cref{lem:weakly yo-qadm morphs}). Hence, since every $\tau$-acyclic pseudosieve is a weakly $\yo$-quasi-admissible sieve, the same is true of $\tau'$, and in particular, $\tau'$ is a quasi-topology.

		We will verify the hypotheses of \cite[Lemma C.3]{quadratic-refinement-GLV-trace}.
		Hypotheses (a1) of \cite[Lemma C.3]{quadratic-refinement-GLV-trace} follows from the fact that $\mathcal C' \to \mathcal C$ is a morphism of pullback contexts, so it preserves fibred products of quasi-admissible maps. Hypothesis (a2) follows from the definition of $\tau'$ and the fact that $u_!$ preserves monomorphisms.

		If $R \to \yo(u(X))$ is a $\tau$-acyclic pseudosieve, then it is weakly $\yo$-quasi-admissible, so since $u$ is anodyne, \Cref{lem:weakly yo-qadm morphs} says that it is of the form $u_!(R' \to \yo(X))$ for some pseudosieve $R' \to \yo(X)$, which is $\tau'$-acyclic by the definition of $\tau'$. Since $u$ is fully faithful, we have that $u_!$ is fully faithful, so the unit $\id \to u^* u_!$ is an equivalence, whence $u^* R \to \yo(X)$ is equivalent to the $\tau'$-acyclic pseudosieve $R' \to \yo(X)$. This shows hypothesis (b) of \cite[Lemma C.3]{quadratic-refinement-GLV-trace}.

		Finally, to show hypothesis (c), we note that every object $S$ of $\mathcal C$ admits a $\tau$-acyclic sieve $\mathcal U$ generated by quasi-admissible maps $u(S') \to S$ with $S' \in \mathcal C'$. Since $u : \mathcal C' \to \mathcal C$ is anodyne, it follows that if $u(S') \to S \gets u(S'')$ are quasi-admissible maps, then $u(S') \times_S u(S'')$ is in the essential image of $u$ since it admits a quasi-admissible map to $u(S')$.

		Therefore, we may apply \cite[Lemma C.3]{quadratic-refinement-GLV-trace} to see that the restriction functor $u^* : \Psh(\mathcal C) \to \Psh(\mathcal C')$ restricts to an equivalence $\Psh^\tau(\mathcal C) \to \Psh^{\tau'}(\mathcal C')$, as desired.
	\end{proof}
\end{prp}

\subsection{Locality of Quasi-Admissibility Structures} \label{S:local qadm}

Throughout this section, $\jmath : \mathcal C \to \tilde{\mathcal C}$ is a functor, and $\tau$ is a pseudotopology on $\tilde{\mathcal C}$. If $\jmath$ is a morphism of pullback contexts, and we write $\PF(\tilde{\mathcal C}; \tau)$ to denote the category of $\tau$-local pullback formalisms on $\tilde{\mathcal C}$, then we have a commutative square
\begin{equation} \label{eqn:local PF/sq}
	\begin{tikzcd}
		\PF(\tilde{\mathcal C}; \tau) \ar[d] \ar[r] & \PF(\mathcal C) \ar[d] \\
		\Psh^\tau_{\CAlg(\PrL)}(\tilde{\mathcal C}) \ar[r] & \Fun(\mathcal C^\op, \CAlg(\PrL))
	\end{tikzcd},
\end{equation}
where the vertical arrows are the evident subcategory inclusions, and the horizontal arrows are induced by $\jmath$. We will be concerned with the question of when this square is Cartesian, that is, when is it the case that a $\tau$-local presheaf $D : \tilde{\mathcal C}^\op \to \CAlg(\PrL)$ is a pullback formalism if and only if $\jmath^* D$ is a pullback formalism, and a transformation $D \to D'$ of $\tau$-local pullback formalisms respects quasi-admissibility if and only if it does after restricting along $\jmath$.

In fact, if $\tau'$ is a pseudotopology on $\mathcal C$ such that $\jmath$ sends every $\tau'$-acyclic pseudosieve to a $\tau$-acyclic pseudosieve, then we have that \eqref{eqn:local PF/sq} is Cartesian if and only if
\begin{equation} \label{eqn:local PF/sq2}
	\begin{tikzcd}
		\PF(\tilde{\mathcal C}; \tau) \ar[d] \ar[r] & \PF(\mathcal C; \tau') \ar[d] \\
		\Psh^\tau_{\CAlg(\PrL)}(\tilde{\mathcal C}) \ar[r] & \Psh^{\tau'}_{\CAlg(\PrL)}(\mathcal C)
	\end{tikzcd}
\end{equation}
is Cartesian.

These squares should become Cartesian when, in some sense, the quasi-admissible maps of $\tilde{\mathcal C}$ $\tau$-locally come from the quasi-admissible maps of $\mathcal C$.

\begin{defn} \label{defn:local qadm}
	Suppose $\mathcal C$ is a pullback context, and let $f$ be a map in $\tilde{\mathcal C}$.

	Say $f$ is \emph{$\tau$-locally $\jmath$-quasi-admissible} if for any $\tau$-local presheaf $D : \mathcal C^\op \to \CAlg(\PrL)$ such that $\jmath^* D$ is is a pullback formalism, we have that $D(f)$ has a linear left adjoint, $D$ has left base change for $f$, and if $\phi : D \to D'$ is a transformation of $\tau$-local presheaves such that $\jmath^* \phi$ is a morphism of pullback formalisms, then $\phi$ is left adjointable at $f$.

	When $\tau$ is the pseudotopology of universal colimits from \Cref{exa:colim pseudotop}, we may omit ``$\tau$-locally'' and simply say $f$ is \emph{$\jmath$-quasi-admissible}.

	% NOTE: in $\Psh(\mathcal C)$, for a pseudosieve to be $\yo$-qadm suff to just check left base change against maps between representables
\end{defn}

\begin{rmk} \label{rmk:local PF sq}
	If $\jmath$ is a morphism of pullback contexts such that every quasi-admissible map in $\tilde{\mathcal C}$ is $\tau$-locally $\jmath$-quasi-admissible, then the square \eqref{eqn:local PF/sq} is Cartesian, and therefore \eqref{eqn:local PF/sq2} is also Cartesian.
\end{rmk}

\begin{rmk} \label{rmk:local qadm str}
	By \cite[Lemma F.6(3)]{TwAmb}, the collection of maps in $\tilde{\mathcal C}$ that are $\tau$-locally $\jmath$-quasi-admissible is closed under composition. Thus, the collection of maps in $\tilde{\mathcal C}$ all of whose base changes exist and are $\tau$-locally $\jmath$-quasi-admissible is a quasi-admissibility structure on $\mathcal C$. These are the universally $\tau$-locally $\jmath$-quasi-admissible maps.
\end{rmk}

Before coming to our main results, we establish the following \namecref{lem:Psh qadm source-local} that shows that at least when $\jmath$ is a Yoneda embedding, the $\jmath$-quasi-admissible maps are in some sense dominated by maps that are representable by quasi-admissible maps.
\begin{lem} \label{lem:Psh qadm source-local}
	Let $\mathcal C$ be a small pullback context. Then for any $S \in \mathcal C$, every $\yo$-quasi-admissible map to $\yo(S)$ is of the form
	\[
		\varinjlim \yo X \to \yo(S)
	\]
	for $K$ a simplicial set, and $X : K \to \mathcal C_S$ a diagram.

	In fact, if $Q'$ is the collection of maps $f$ in $\Psh(\mathcal C)$ such that for any morphism $\phi$ of pullback formalisms on $\mathcal C$, we have that $\phi$ is left adjointable at $f$, then the weakly $\yo$-quasi-admissible maps (see \Cref{defn:weak yo-qadm}) are stable under postcomposition by maps in $Q'$.
	\begin{proof}
		By \Cref{lem:locality in psh}, the collection of weakly $\yo$-quasi-admissible maps is a local class in $\Psh(\mathcal C)$ in the sense of \cite[Definition 6.1.3.8]{htt}, so by \cite[Lemma 6.1.3.7]{htt}, we have that it defines a limit-preserving presheaf $D : \Psh(\mathcal C)^\op \to \widehat{\Cat}$ that is equipped with a transformation $D \to \bar H^\univ$ such that for any $Y \in \Psh(\mathcal C)$, the functor $D(Y) \to \bar H^\univ(Y) \simeq \Psh(\mathcal C)_{/Y}$ is the inclusion of the full subcategory of weakly $\yo$-quasi-admissible maps to representable presheaves.

		Note that for any $S \in \mathcal C$, the evident colimit-preserving functor $\Psh(\mathcal C_S) \to \Psh(\mathcal C)_{/\yo(S)}$ is fully faithful with essential image given by $D(S)$. In particular, $D(S)$ is a presentable category. Since $\Psh(\mathcal C)$ has universal colimits, it follows that $D|_{\mathcal C}$ lands in $\PrL \subseteq \widehat{\Cat}$. Since $D$ preserves small limits, and $\PrL \to \widehat{\Cat}$ preserves small limits by \cite[Proposition 5.5.3.13]{htt}, we have that $D$ lands in $\PrL \subseteq \widehat{\Cat}$, and $D \to \bar H^\univ$ is a transformation of limit-preserving $\PrL$-valued presheaves.
		In fact, since the weakly $\yo$-quasi-admissible are stable under base change, and colimits are universal, we have that $D$ lands in $\CAlg(\PrL)$ (with the Cartesian monoidal structures), and for each $S \in \mathcal C$, the inclusion $D(S) \to \bar H^\univ(S)$ preserves finite products, so we can consider $D \to \bar H^\univ$ as a transformation of presheaves $\Psh(\mathcal C)^\op \to \CAlg(\PrL)$. Note that $D$ is still limit-preserving by \cite[Corollary 3.2.2.5]{ha}.

		By \Cref{prp:sub PF is PF}, we have that the transformation $(D \to \bar H^\univ)|_{\mathcal C}$ is a morphism of pullback formalisms. Thus, by the definition of $Q'$, we have that $D \to \bar H^\univ$ is left adjointable at maps in $Q'$, so for any map $X \to Y$ in $Q'$, the slice projection $\Psh(\mathcal C)_{/X} \to \Psh(\mathcal C)_{/Y}$ sends $D(X) \subseteq \Psh(\mathcal C)_{/X}$ to $D(Y) \subseteq \Psh(\mathcal C)_{/Y}$, \ie, postcomposition by $X \to Y$ preserves weakly $\yo$-quasi-admissible maps.

		To see that this implies the first statement, note that every identity map is weakly $\yo$-quasi-admissible, so that the above fact implies that $Q'$ is contained in the collection of weakly $\yo$-quasi-admissible maps, and we conclude since every $\yo$-quasi-admissible map is in $Q'$.
	\end{proof}
\end{lem}

\subsubsection{Main applications:}

We will now present the statements of the main applications of this \namecref{S:local qadm}. The proofs will be presented later in this section.

% Next, we have that $\tau$-locally $\jmath$-quasi-admissible maps have strong closure properties.
\begin{prp} \label{prp:locality of local qadm str}
	The collection of (universally) $\tau$-locally $\jmath$-quasi-admissible maps (see \Cref{rmk:local qadm str}) is stable under transfinite $\tau$-composites, and if $\tilde{\mathcal C}$ admits all pullbacks, it is also strictly $\tau$-local on the source, and the collection of universally $\tau$-locally $\jmath$-quasi-admissible maps is strictly $\tau$-local on the target. 
\end{prp}

\begin{prp} \label{prp:local PF dense subctx}
	Suppose $\jmath$ is the inclusion of a full anodyne pullback subcontext, and that every object of $\tilde{\mathcal C}$ is equivalent to $X(\infty)$ for some $\tau$-acyclic diagram $X : K^\triangleright \to \tilde{\mathcal C}$ that sends $K$ to $\mathcal C$. Then every quasi-admissible map is $\tau$-locally $\jmath$-quasi-admissible. In particular, the square \eqref{eqn:local PF/sq2} is Cartesian.
\end{prp}

% The following result establishes that in many cases $\jmath$ sends quasi-admissible maps to $\jmath$-quasi-admissible maps.
\begin{prp} \label{prp:PF on Shv}
	Suppose that $\mathcal C$ is a locally small pullback context, and $\jmath$ is of the form $L \circ \yo$, where $L : \Psh(\mathcal C) \to \tilde{\mathcal C}$ is a locally Cartesian localization.

	If $\jmath$ preserves base changes along quasi-admissible maps, then $\jmath$ sends quasi-admissible maps to $\jmath$-quasi-admissible maps.
	% In fact, if $Q$ is a collection of maps in $\tilde{\mathcal C}$ that is stable under base change and generated by the image of the quasi-admissible maps in $\mathcal C$ under transfinite composites, strict locality on the target, and locality on the source, then every map in $Q$ is universally $\jmath$-quasi-admissible. In particular, if $Q$ is a quasi-admissibility structure on $\tilde{\mathcal C}$, the square \eqref{eqn:local PF/sq} is Cartesian.
\end{prp}

\subsubsection{Auxiliary results}

\begin{lem}[Transfinite composites of locally quasi-admissible maps] \label{lem:local qadm transf composites}
	The $\tau$-locally $\jmath$-quasi-admissible maps of $\tilde{\mathcal C}$ are closed under transfinite $\tau$-composites.
	\begin{proof}
		Let $\lambda$ be an ordinal, and let $F : \lambda + 1 \to \tilde{\mathcal C}$ be a diagram such that for every limit ordinal $\lambda' \leq \lambda$, the restriction $F|_{\lambda' + 1}$ is $\tau$-acyclic.
		\begin{enumerate}

			\item Suppose that $\phi : D \to D'$ is a transformation of $\tau$-local presheaves $D,D' : \tilde{\mathcal C}^\op \to \PrL$, and that $\phi$ is left adjointable at $F(\alpha \to \alpha + 1)$ for all $\alpha \in \lambda$. Then $\phi$ is left adjointable at $F(\emptyset) \to F(\lambda)$ by \Cref{cor:adjointability horiz composites}.

			\item Suppose that for all $\alpha \leq \beta \leq \lambda$, the map $F(\alpha) \to F(\beta)$ admits all base changes along maps to $F(\beta)$. Then any map to $F(\lambda)$ is of the form $F'(\lambda) \to F(\lambda)$ for some Cartesian transformation $F' \to F$. For any limit ordinal $\lambda' \leq \lambda$, $F'|_{\lambda' + 1}$ is a $\tau$-acyclic since it is a base change of the $\tau$-acyclic diagram $F|_{\lambda' + 1}$. Thus, if $D : \tilde{\mathcal C}^\op \to \PrL$ is a $\tau$-acyclic presheaf that has left base change for $F(\alpha \to \alpha + 1)$ for all $\alpha \in \lambda$, then \Cref{prp:transf comp base change} says that $D$ has left base change for $F(\emptyset) \to F(\lambda)$.

			\item Suppose that $D : \tilde{\mathcal C}^\op \to \CAlg(\PrL)$ is a $\tau$-local presheaf such that for all $\alpha \in \lambda$, $D$ sends $F(\alpha \to \alpha + 1)$ to a functor that has a linear left adjoint. It follows that $D$ sends all of these to functors with $D(F(\lambda))$-linear left adjoints, so \Cref{thm:Mod LAd} says that $D$ sends $F(\emptyset) \to F(\lambda)$ to a functor that has a linear left adjoint.

		\end{enumerate}
		This shows that if $F(\alpha \to \alpha + 1)$ is $\tau$-locally $\jmath$-quasi-admissible for all $\alpha \in \lambda$, then $F(\emptyset) \to F(\lambda)$ is also $\tau$-locally $\jmath$-quasi-admissible.
	\end{proof}
\end{lem}

\begin{lem}[Locality on the target for locally quasi-admissible maps] \label{lem:local qadm target}
	If $f : X \to Y$ is any Cartesian transformation of small $\tau$-acyclic diagrams $X,Y : K^\triangleright \to \tilde{\mathcal C}$, and $\phi : D \to D'$ is a transformation of $\tau$-local presheaves $\tilde{\mathcal C}^\op \to \widehat{\Cat}$ that have left (right) base change for $f(b)$ against $Y(a \to b)$ for all $a \to b$ in $K$, we have the following:
	\begin{enumerate}

		\item If $\phi$ is left (right) adjointable at $f(a)$ for all $a \in K$, then $\phi$ is left (right) adjointable at $f(\infty)$.

		\item If $Y' \to Y$ is a transformation of $\tau$-acyclic diagrams such that for all $a \in K^\triangleright$, the base change $Y'(a) \times_{Y(a)} f(a)$ exists, and for any map $a \to b$ in $K$, $D$ has left (right) base change for $f(b)$ against $Y'(b) \to Y(b)$, and for $Y'(b) \times_{Y(b)} f(b)$ against $Y'(a \to b)$, then $D$ has left (right) base change for $f(\infty)$ against $Y'(\infty) \to Y(\infty)$.

		\item If $D$ lifts to a presheaf $\tilde{\mathcal C}^\op \to \CAlg(\widehat{\Cat})$ that sends $f(a)$ to a functor that has a linear left (right) adjoint for all $a \in K$, then $D$ sends $f(\infty)$ to a functor that has a linear left (right) adjoint.

	\end{enumerate}
	
	\begin{proof}
		We first make a general observation: if
		\[
			\begin{tikzcd}
				\mathcal Y \ar[d] \ar[r] & \mathcal X \ar[d] \\
				\mathcal Y' \ar[r] & \mathcal X'
			\end{tikzcd}
		\]
		is a diagram in $\Fun(K^\op, \widehat{\Cat})$, and this diagram evaluates to a left (right) adjointable square for all $a \in K$, and for all $a \to b$ in $K$, the squares
		\[
			\begin{tikzcd}
				\mathcal Y(b) \ar[d] \ar[r] & \mathcal X(b) \ar[d] \\
				\mathcal Y(a) \ar[r] & \mathcal X(a)
			\end{tikzcd}
			\text{ and }
			\begin{tikzcd}
				\mathcal Y'(b) \ar[d] \ar[r] & \mathcal X'(b) \ar[d] \\
				\mathcal Y'(a) \ar[r] & \mathcal X'(a)
			\end{tikzcd}
		\]
		are left (right) adjointable, then \cite[Corollary 4.7.4.18(2)]{ha} says that the limit of the square is left (right) adjointable.

		Now we address the listed statements:
		\begin{enumerate}

			\item We apply the above observation to the square
				\[
					\begin{tikzcd}
						D Y \ar[d] \ar[r] & D X \ar[d] \\
						D' Y \ar[r] & D' X
					\end{tikzcd}
				.\]

			\item We have a Cartesian square
				\[
					\begin{tikzcd}
						X' \ar[d] \ar[r, "f'"] & Y' \ar[d] \\
						X \ar[r, "f"'] & Y
					\end{tikzcd}
				\]
				in which all horizontal edges are Cartesian transformations. In particular, $X'$ is $\tau$-acyclic since it is a base change of the $\tau$-acyclic diagram $Y'$. By our assumptions on the base change properties of $D$, the square
				\[
					\begin{tikzcd}
						DY \ar[d] \ar[r] & DX \ar[d] \\
						DY' \ar[r] & DX'
					\end{tikzcd}
				\]
				satisfies the hypotheses of the general observation, and we conclude since $D$ is $\tau$-local.

			\item We may view $D \circ f^\op$ as a transformation of limiting diagrams in $\Mod_{D(Y(\infty))} \widehat{\Cat}$, and for all $a \in K$, $f(a)^*$ has a $D(Y(\infty))$-linear left (right) adjoint. Thus, for all $M \in D(Y(\infty))$, the square
				\[
					\begin{tikzcd}
						D Y^\op \ar[d, "\otimes M"'] \ar[r] & D X^\op \ar[d, "\otimes M"] \\
						D Y^\op \ar[r] & D X^\op \\
					\end{tikzcd}
				\]
				is left (right) adjointable when restricted to $K^\op$, so we may apply the general observation again.

		\end{enumerate}
	\end{proof}
\end{lem}

\begin{lem}[Locality on the source for locally quasi-admissible maps] \label{lem:local qadm source}
	% If $Q$ is a collection of maps in $\tilde{\mathcal C}$ that is stable under base change, and all of whose elements are $\tau$-locally on the source in the collection of (strongly) $\tau$-locally $\jmath$-quasi-admissible maps, then $Q$ is contained in the collection of universally (strongly) $\tau$-locally $\jmath$-quasi-admissible. In particular, if all pullbacks exist in $\tilde{\mathcal C}$, then the collection of (strongly) $\tau$-locally $\jmath$-quasi-admissible maps is $\tau$-local on the source.

	If $X : K^\triangleright \to \tilde{\mathcal C}$ is a $\tau$-acyclic diagram, then we have the following:
	\begin{enumerate}

		\item If $\phi : D \to D'$ is a transformation of $\tau$-local presheaves $D,D' : \tilde{\mathcal C}^\op \to \PrL$ ($\PrR$), and $\phi$ is left (right) adjointable at $X(a \to b)$ for all $a \to b$ in $K$, then $\phi$ left (right) adjointable at $X(a \to \infty)$ for all $a \in K$, and $\phi$ is left (right) adjointable at a map $X(\infty) \to Y$ if and only if it is left (right) adjointable at $X(a) \to Y$ for all $a \in K$.

		\item If $D$ is a $\tau$-local presheaf $D : \tilde{\mathcal C}^\op \to \CAlg(\PrL)$, and $D X(a \to b)$ has a linear left adjoint for all edges $a \to b$ in $K$, then $D X(a \to \infty)$ has a linear left adjoint for all $a \in K$, and for any map $f : X(\infty) \to Y$, $D(f)$ has a linear left adjoint if and only if $D(f \circ X(a \to \infty))$ has a linear left adjoint for all $a \in K$.

		\item For any $\tau$-local presheaf $D : \tilde{\mathcal C}^\op \to \PrL$ ($\PrR$), if $D$ has left (right) base change for $X(a \to b)$ for all $a \to b$ in $K$, then $D$ has left (right) base change for $X(a \to \infty)$ for all $a \in K$, and $D$ has left (right) base change for a map $X(\infty) \to Y$ if and only if all base changes of this map exist, and for all $a \in K$, $D$ has left (right) base change for $X(a) \to Y$.

	\end{enumerate}
	
	\begin{proof}\hfill
		\begin{enumerate}

			\item This follows from \Cref{prp:Pr adj lims}.

			\item This follows from \Cref{thm:Mod LAd}.

			\item This follows from \Cref{prp:source locality of base change}.

		\end{enumerate}
	\end{proof}
\end{lem}

\begin{prp} \label{prp:local PF}
	Suppose that $\jmath$ preserves base changes along quasi-admissible maps, and that for any $S \in \mathcal C$, every map to $\jmath(S)$ is of the form $Y(0) \to Y(1)$ for some small diagram $Y : K \star \Delta^1$ such that $Y|_{K \star \{0\}}$ is $\tau$-acyclic, and $Y|_{K \star \{1\}}$ factors through $\jmath$ by a diagram that sends $1$ to $S$.

	If $f$ is a quasi-admissible map in $\mathcal C$ such that all base changes of $\jmath(f)$ exist, then $\jmath(f)$ is $\tau$-locally $\jmath$-quasi-admissible.

	% In fact, if $Q$ is a collection of maps in $\tilde{\mathcal C}$ that is stable under base change and generated by the image of the quasi-admissible maps in $\mathcal C$ under transfinite $\tau$-composites, strict $\tau$-locality on the target, and $\tau$-locality on the source, then every map in $Q$ is universally $\tau$-locally $\jmath$-quasi-admissible. In particular, if $Q$ is a quasi-admissibility structure on $\tilde{\mathcal C}$, the square \eqref{eqn:local PF/sq2} is Cartesian.
	\begin{proof}
		First let us show that $\jmath(f)$ is a $\tau$-locally $\jmath$-quasi-admissible map.
		Let $D : \tilde{\mathcal C}^\op \to \CAlg(\PrL)$ be a $\tau$-local presheaf such that $\jmath^* D$ is a pullback formalism.
		\begin{itemize}

			\item If $\phi : D \to D'$ is a transformation of $\tau$-local presheaves such that $\jmath^* \phi$ is a morphism of pullback formalisms, then it is clear that $\phi$ is left adjointable at the image under $\jmath$ of any quasi-admissible map.

			\item Similarly, since $\jmath^* D(f)$ has a linear left adjoint, it follows that $D(\jmath(f))$ has a linear left adjoint.

			\item We have assumed that every map to the codomain of $\jmath(f)$ is of the form $\tilde Y(0 \to 1)$ for $\tilde Y : K \star \Delta^1 \to \tilde{\mathcal C}$ a small diagram such that $\tilde Y|_{K \star \{0\}}$ is $\tau$-acyclic, and $\tilde Y|_{K \star \{1\}}$ factors through $\jmath$ by a diagram $Y$ that sends $1$ to the codomain of $f$. Since all base changes of $\jmath(f)$ exist, we have that $\jmath(f) \simeq \tilde f(1)$ for some Cartesian transformation $\tilde f : \tilde X \to \tilde Y$. Since $\tilde f$ is Cartesian, and $\tilde Y|_{K \star \{0\}}$ is $\tau$-acyclic, it follows that $\tilde X|_{K \star \{0\}}$ is also $\tau$-acyclic, and since $\jmath$ preserves base changes along quasi-admissible maps, we have that $\tilde f|_{K \star \{1\}} \simeq \jmath f^\flat$, where $f^\flat$ is some Cartesian transformation $X \to Y$ that evaluates to $f$ at $1 \in K \star \{1\}$. Since $\jmath D^*$ has quasi-admissible base change, it follows from \Cref{prp:source locality of base change} that $D$ has left base change for $\tilde f(b)$ against $\tilde Y(a \to b)$ for all $a \to b$ in $K \star \Delta^1$.

		\end{itemize}
		In particular, since $D$ has left base change for $\jmath(f) \simeq \tilde f(1)$ against $\tilde Y(0 \to 1)$, we have that $\jmath(f)$ is a $\tau$-locally $\jmath$-quasi-admissible map.

		% It follows from \Cref{lem:local qadm transf composites}, \Cref{lem:local qadm target}, and \Cref{lem:local qadm source}, that the collection of universally $\tau$-locally $\jmath$-quasi-admissible maps is stable under transfinite $\tau$-composites, strictly $\tau$-local on the target, and $\tau$-local on the source, so we conclude that $Q$ is contained in this collection.
	\end{proof}
\end{prp}

The next result is given by restating \Cref{prp:local PF} in the case where $\tau$ is the pseudotopology of universal colimits given in \Cref{exa:colim pseudotop}.
\begin{cor} \label{cor:limit-preserving PF}
	Suppose that $\jmath$ preserves base changes along quasi-admissible maps, and that for any $S \in \mathcal C$, every map to $\jmath(S)$ is of the form $\varinjlim \jmath X \to \jmath(S)$ for some small diagram $X : K \to \mathcal C_{/S}$ such that $\jmath X$ admits a universal colimit.

	If $f$ is a quasi-admissible map in $\mathcal C$ such that all base changes of $\jmath(f)$ exist, then $\jmath(f)$ is a $\jmath$-quasi-admissible map.

	% In fact, if $Q$ is a collection of maps in $\tilde{\mathcal C}$ that is stable under base change and generated by the image of the quasi-admissible maps in $\mathcal C$ under transfinite composites, strict locality on the target, and locality on the source, then every map in $Q$ is universally $\jmath$-quasi-admissible. In particular, if $Q$ is a quasi-admissibility structure on $\tilde{\mathcal C}$, the square \eqref{eqn:local PF/sq2} is Cartesian.
\end{cor}

\subsubsection{Proofs of main applications:}

\begin{proof}[Proof of \Cref{prp:locality of local qadm str}]
	Note that since $\tau$-acyclic pseudosieves are stable under base change, \Cref{lem:local qadm transf composites} shows that this collection is closed under transfinite $\tau$-composites, and when $\tilde{\mathcal C}$ has all pullbacks, \Cref{lem:local qadm source}. Finally \Cref{lem:local qadm target} shows that the universally $\tau$-locally $\jmath$-quasi-admissible maps are strictly $\tau$-local on the target.
\end{proof}

\begin{proof}[Proof of \Cref{prp:local PF dense subctx}]
	Follows immediately from \Cref{prp:local PF}. 
\end{proof}

\begin{proof}[Proof of \Cref{prp:PF on Shv}]
	Note that by \cite[Proposition 1.4]{LCC}, $\tilde{\mathcal C}$ has (small) universal colimits.\footnote{The argument is only stated for presentable categories, and $\Psh(\mathcal C)$ might not be presentable if $\mathcal C$ is not small, but the same argument (from \cite[Proposition 6.1.3.15]{htt}) works to show that if a category $\mathcal D$ has universal $K$-indexed colimits, then any locally Cartesian localization of $\mathcal D$ also has universal $K$-indexed colimits.}

	Every object of $\tilde{\mathcal C}$ is a (universal) colimit of objects in the essential image of $\jmath$, and for any $S \in \mathcal C$, every map to $\jmath(S)$ is of the form $\varinjlim \jmath X \to \jmath(S)$ for some small diagram $X : K \to \mathcal C_{/S}$ by \Cref{lem:LCL implies ess surj on slices}.

	Thus, the result follows from \Cref{cor:limit-preserving PF}.
\end{proof}

\subsection{Invariance and Acyclicity} \label{S:acyclic}
 
% PERF: improve this intro

In this \namecref{S:acyclic}, we will study invariance properties of cohomology theories with local coefficients.
As we will see in \Cref{lem:qadm acyclic}, \emph{quasi-admissible} invariance conditions are particularly well-behaved in the setting of pullback formalisms. This good behaviour turns out to be quite useful:
\begin{itemize}

	\item In \Cref{thm:D-topology}, we use this good behaviour to establish some powerful tools for working with descent properties of pullback formalisms.
		% NOTE: notably used in the statement about descent for the associated cohomology theories.

	\item In \Cref{prp:invariance localization}, we use this good behaviour to explain how to freely impose quasi-admissible invariance properties on pullback formalisms. \Cref{cor:describe invariance localization of Huniv} then gives a useful description of the universal pullback formalism with prescribed quasi-admissible invariance properties.

\end{itemize}

% In particular, we will introduce notions of invariance and acyclicity (Definition \ref{defn:acyclic}) for cohomology theories with local coefficients, and establish general results about acyclic maps (Lemmas \ref{lem:props of acyclic}, \ref{lem:qadm acyclic}, \ref{lem:acyclic qadm is local}). One important result that already emerges from this general discussion of acyclicity is \Cref{thm:D-topology}, which provides a powerful tool for working with descent properties of pullback formalisms.
%
% In \Cref{S:pseudotopology}, we will consider a structure, called a ``pseudotopology'', that allows us to effectively encode and impose invariance conditions of cohomology theories with local coefficients in general.

Recalling \Cref{nota:cohomology}, for a system of coefficients $D$ on a category $\mathcal C$ to be invariant with respect to a map $p : V \to S$ in $\mathcal C$, we need that for any $M \in D(S)$, the natural map
\[
	D(S; M) \to D(V; M)
\]
is an equivalence. This says that the functor $p^* : D(V) \to D(S)$ induces an equivalence of mapping spaces
\[
	D(S)(1, M) \to D(V)(p^* 1, p^* M)
,\]
which leads us to the following definition:
\begin{defn} \label{defn:acyclic}
	Let $D : \mathcal C^\op \to \widehat{\Cat}$ be a presheaf on a category $\mathcal C$. Say a map in $\mathcal C$ is \emph{$D$-acyclic} if it is sent to a fully faithful functor by $D$. When $\mathcal C$ is small, $D$ can be thought of as a limit-preserving presheaf on $\Psh(\mathcal C)$, and we extend the notion of $D$-acyclic maps to maps in $\Psh(\mathcal C)$.

	Say a diagram $X : K^\triangleright \to \mathcal C$ is $D$-acyclic if
	\[
		D(X(\infty)) \to \varprojlim_{p \in K} D(X(p))
	\]
	is fully faithful.

	If $f$ is any map in $\Psh(\mathcal C)$, we say that $D$ is \emph{$f$-invariant} if $f$ is $D$-acyclic.
\end{defn}

\begin{exa} \label{exa:acyclic}
	In \Cref{defn:acyclic}, by allowing the map $f$ to be a map in $\Psh(\mathcal C)$, and not just a map in $\mathcal C$, we can use acyclicity and invariance to express more sophisticated conditions.
	\begin{itemize}

		\item If $f$ is the inclusion of the empty presheaf into the presheaf represented by an object $X$ of $\mathcal C$, then $D$ is $f$-invariant if and only if $D(X)$ is empty or contractible. In particular, if $D$ takes values in $\CAlg(\PrL)$, then this means that $D(X)$ is contractible (since it cannot be empty).

		\item If $f$ is the natural map $\yo(X) \coprod \yo(Y) \to \yo(X \coprod Y)$, for some objects $X,Y \in \mathcal C$ admitting a coproduct, then $D$ is $f$-invariant if and only if the natural functor
			\[
				D(X \coprod Y) \to D(X) \times D(Y)
			\]
			is fully faithful. In particular, for any $M \in D(X \coprod Y)$, we have an equivalence
			\[
				D(X \coprod Y; M) \to D(X; M) \times D(Y; M)
			.\]

		\item We may view any square
			\[
				\begin{tikzcd}
					U' \ar[d] \ar[r] & X' \ar[d] \\
					U \ar[r] & X
				\end{tikzcd}
			\]
			in $\mathcal C$ as a diagram $K^\triangleright \to \mathcal C$, where $K$ is the category $\bullet \gets \bullet \to \bullet$. This diagram is $D$-acyclic if and only if the functor
			\[
				D(X) \to D(X') \times_{D(U')} D(U)
			\]
			is fully faithful. In particular, for any $M \in D(X)$, the square
			\[
				\begin{tikzcd}
					D(X; M) \ar[d] \ar[r] & D(X'; M) \ar[d] \\
					D(U; M) \ar[r] & D(U'; M)
				\end{tikzcd}
			\]
			is Cartesian.

		\item Let $\mathcal U$ be a sieve on an object $X \in \mathcal C$. The map $\mathcal U \to \yo(X)$ is $D$-acyclic if and only if the map
			\[
				D(X) \to \varprojlim_{\substack{X' \to X \\ \text{in $\mathcal U$}}} D(X')
			\]
			is fully faithful, so for any $M \in D(X)$, the map
			\[
				D(X; M) \to \varprojlim_{\substack{X' \to X \\ \text{in $\mathcal U$}}} D(X'; M)
			\]
			is an equivalence. This corresponds to a descent condition for the cohomology theory.

			In fact, suppose that $\mathcal C$ is a pullback context, $D$ is a pullback formalism, and $\mathcal U$ is generated by a family of quasi-admissible maps $\{X_i \to X\}_i$. It turns out that the map $\mathcal U \to \yo(X)$ is $D$-acyclic if and only if the monoidal unit is the geometric realization of the simplicial diagram in $D(S)$ that sends $[n] \in \Delta^\op$ to
			\[
				\coprod_{i_0, \dotsc, i_n} \cls{X_{i_0} \times_X \dotsb \times_X X_{i_n}}
			.\]

	\end{itemize}
\end{exa}

We record some general closure properties of acyclic maps:
\begin{lem} \label{lem:props of acyclic}
	Let $D$ be a $\widehat{\Cat}$-valued presheaf on some category.
	\begin{enumerate}

		\item $D$-acyclic maps are stable under cobase change if $D$ preserves base changes.

		\item If $f : X \to Y$ is $D$-acyclic, then for any $g : Y \to Z$, $g$ is $D$-acyclic if and only if $g \circ f$ is $D$-acyclic.

		\item For any simplicial set $K$, if $D$ preserves $K$-indexed limits, then any $K$-indexed colimit of $D$-acyclic maps is $D$-acyclic.

	\end{enumerate}
	In particular, if $\mathcal C$ is a small category, and $D \in \Psh_\PrL(\mathcal C)$, then $D$ can be thought of as a limit-preserving presheaf on $\Psh(\mathcal C)$, and the $D$-acyclic maps in $\Psh(\mathcal C)$ are stable under cobase change and colimits, and satisfy that a map $g : Y \to Z$ is $D$-acyclic if and only if there is some $D$-acyclic map $f : X \to Y$ such that $g \circ f$ is $D$-acyclic.
	\begin{proof}\hfill
		\begin{enumerate}

			\item Since $D$ sends cobase changes to base changes of categories, this follows from the fact that fully faithful functors are stable under base change.

			\item This follows from the fact that if $\beta$ is fully faithful, then $\beta \circ \alpha$ is fully faithful if and only if $\alpha$ is fully faithful.

			\item Since $D$ sends $K$-indexed colimits to limits of categories, this follows from the fact that limits of fully faithful functors are fully faithful.

		\end{enumerate}
	\end{proof}
\end{lem}

\begin{rmk}
	In the setting of Lemma \ref{lem:props of acyclic}, note that the collection of $D$-acyclic maps \emph{almost} satisfies the definition of a strongly saturated class of \cite[Definition 5.5.4.5]{htt}, except that it does not quite satisfy the full ``2-out-of-3'' property. It is possible to have $g \circ f$ and $g$ be $D$-acyclic without $f$ being $D$-acyclic.
\end{rmk}

Quasi-admissible invariance properties for pullback formalisms are particularly well-behaved:
\begin{lem} \label{lem:qadm acyclic}
	Let $\mathcal C$ be a pullback context, and let $D : \mathcal C^\op \to \CAlg(\widehat{\Cat})$ be a presheaf that has the projection formula for quasi-admissible maps.
	\begin{enumerate}

		\item A quasi-admissible map $V \to S$ is $D$-acyclic if and only if $\cls{V;S} \simeq 1$.
			% In particular, $V \to S$ is $D$-acyclic if and only if $D(S;M) \to D(V;M)$ is an equivalence for all $M \in D(S)$.

		\item If $D' : \mathcal C^\op \to \CAlg(\widehat{\Cat})$ also has the projection formula for quasi-admissible maps, and there is a transformation $D \to D'$ that respects quasi-admissibility, then every quasi-admissible $D$-acyclic map is $D'$-acyclic.
			%and satisfy 2 out of 3

		\item If $D$ also has quasi-admissible base change, then quasi-admissible $D$-acyclic maps are stable under base change.
			% In particular, a quasi-admissible map $V \to S$ is $D$-acyclic if and only if for any $X \in \mathcal C_{/S}$, the map $D(X;M) \to D(X \times_S V; M)$ is an equivalence for all $M \in D(X)$.

	\end{enumerate}
	\begin{proof}\hfill
		\begin{enumerate}

			\item This follows immediately from Lemma \ref{lem:proj form implies ff equiv triv fund class}.

			\item This follows immediately from the first point since $D \to D'$ respects quasi-admissibility.

			\item Let
				\[
					\begin{tikzcd}
						\tilde T \ar[d, "q"'] \ar[r, "\tilde f"] & \tilde S \ar[d, "p"] \\
						T \ar[r, "f"'] & S
					\end{tikzcd}
				\]
				be a Cartesian square. Then
				\[
					f^* p_\sharp(1) \simeq q_\sharp \tilde f^*(1) \simeq q_\sharp 1
				,\]
				so if $p_\sharp(1) \simeq 1$, then
				\[
					q_\sharp 1 \simeq f^* 1 \simeq 1
				,\]
				so we may conclude by the first point.

		\end{enumerate}
	\end{proof}
\end{lem}

\begin{rmk} \label{rmk:describe qadm acyclic}
	If $\mathcal C$ is a pullback context, and $D : \mathcal C^\op \to \CAlg(\widehat{\Cat})$ is a presheaf that has the projection formula for quasi-admissible maps, then \Cref{lem:qadm acyclic} implies, in particular, that a quasi-admissible map $V \to S$ is $D$-acyclic if and only if for any $M \in D(S)$, the map $D(S; M) \to D(V; M)$ is an equivalence, and that if $D$ also has quasi-admissible base change, then this is equivalent to the condition that for any $X \in \mathcal C_{/S}$ and $M \in D(X)$, the map $D(X;M) \to D(X \times_S V; M)$ is an equivalence.
\end{rmk}

\begin{lem} \label{lem:acyclic qadm is local}
	Let $\mathcal C$ be a presentable pullback context in which colimits are universal, and the class of quasi-admissible maps is local in the sense of \cite[Definition 6.1.3.8]{htt}, and let $D : \mathcal C^\op \to \CAlg(\PrL)$ be a \emph{limit-preserving pullback formalism}. Then the collection of quasi-admissible $D$-acyclic maps is a local class.
	\begin{proof}
		Let $\mathcal W$ be the collection of quasi-admissible $D$-acyclic maps.

		By \Cref{lem:qadm acyclic}, we have that $\mathcal W$ is stable under base change. Now, by \cite[Lemma 6.1.3.5(3)]{htt}, we must show that if $\bar \alpha : \bar p \to \bar q$ is a transformation of colimiting diagrams $\bar p, \bar q : K^\triangleright \to \mathcal C$ such that $\bar \alpha$ is Cartesian on $K \subseteq K^\triangleright$, and for $x \in K$, $\bar \alpha(x) \in \mathcal W$, then $\bar \alpha$ is Cartesian and $\bar \alpha(\infty) \in \mathcal W$. Since quasi-admissibility is local, we have that $\bar \alpha$ is Cartesian, and $\bar \alpha(\infty)$ is quasi-admissible, so we just need to show that $\bar \alpha(\infty)$ is $D$-acyclic, but this follows from \Cref{lem:props of acyclic}, since $\bar \alpha(\infty)$ is a colimit of $D$-acyclic maps.
	\end{proof}
\end{lem}

\subsection{Acyclic Sieves and Quasi-admissible Coverings} \label{S:qadm covers}

We will now consider more ``classical'' descent properties of pullback formalisms.

\begin{thm} \label{thm:PF descent}
	Let $\Cat^L$ be the subcategory of $\widehat{\Cat}$ consisting of categories that admit all small colimits, and functors between them that have right adjoints.

	Let $\mathcal C$ be a pullback context, and let $D : \mathcal C^\op \to \Cat^L$ be a presheaf that has quasi-admissible base change.

	For any small family $\{X_i \to Y\}_i$ of quasi-admissible maps, if $\mathcal U \subseteq \yo(Y)$ is the sieve generated by this family, then by viewing $D$ as a limit-preserving presheaf $\Psh(\mathcal C)^\op \to \widehat{\Cat}$,\footnotemark we have the following:
	\footnotetext{We do not require that $D$ sends presheaves that are not representable to objects of $\Cat^L$.}
	\begin{enumerate}

		\item The functor
			\[
				D(Y) \to D(\mathcal U)
			\]
			has a fully faithful left adjoint. In particular, $D$ descends along $\mathcal U$ if and only if the functors $\{D(Y) \to D(X_i)\}_i$ are jointly conservative.

		\item For any transformation $D' \to D$ of presheaves $\mathcal C^\op \to \Cat^L$ that respects quasi-admissibility, and every index $i$, the left square in
			\[
				\begin{tikzcd}
					D'(Y) \ar[d] \ar[r] & D'(\mathcal U) \ar[d] \ar[r] & D'(X_i) \ar[d] \\
					D(Y) \ar[r] & D(\mathcal U) \ar[r] & D(X_i)
				\end{tikzcd}
			\]
			is left adjointable, and the right square is left adjointable if the top right horizontal arrow admits a left adjoint.

		\item For any map $q : Y' \to Y$, and index $i$, all of the squares in
			\[
				\begin{tikzcd}
					D(Y) \ar[d] \ar[r] & D(\mathcal U) \ar[d] \ar[r] & D(X_i) \ar[d] \\
					D(Y') \ar[r] & D(q^{-1} \mathcal U) \ar[r] & D(X_i \times_Y Y')
				\end{tikzcd}
			\]
			are left adjointable.

		\item If $D$ lifts to a pullback formalism $\mathcal C^\op \to \CAlg(\PrL)$, then for every index $i$, both functors $D(Y) \to D(\mathcal U) \to D(X_i)$ have linear left adjoints, and if the functors $\{D(Y) \to D(X_i)\}_i$ are jointly conservative, then any pullback formalism admitting a morphism from $D$ descends along every base change of $\mathcal U$.

	\end{enumerate}
	\begin{proof}
		The first three statements follow immediately from \Cref{prp:nerve adj}, which also implies that the functors $D(Y) \to D(\mathcal U) \to D(X_i)$ have $D(Y)$-linear left adjoints. Since $D(Y) \to D(\mathcal U)$ has a fully faithful adjoint, it is essentially surjective, so it follows that these functors actually have linear left adjoints.

		By the previous parts, any pullback formalism $D'$ satisfies that for any $q : Y' \to Y$, $D'(Y') \to D'(q^{-1} \mathcal U)$ has a fully faithful left adjoint $\zeta'$, so $D'$ descends along $q^{-1} \mathcal U$ if and only if $D'(Y') \to D'(q^{-1} \mathcal U)$ is conservative. Note that by the previous parts, if $\phi : D \to D'$ is a morphism of pullback formalisms, then we have that the square
		\[
			\begin{tikzcd}
				D(Y) \ar[d] \ar[r] & D(\mathcal U) \ar[d] \\
				D'(Y') \ar[r] & D'(q^{-1} \mathcal U)
			\end{tikzcd}
		\]
		is left adjointable, so if we write $\zeta$ for the left adjoint of the top horizontal arrow, we have that $\phi$ sends the map $\zeta(1) \to 1$ to the map $\zeta'(1) \to 1$. Since we know that $D$ descends along $\mathcal U$, we have that $\zeta(1) \to 1$ is an equivalence, so $\zeta'(1) \to 1$ is an equivalence, so by \Cref{lem:proj form implies ff equiv triv fund class}, we have that $D'(Y') \to D'(q^{-1} \mathcal U)$ is fully faithful, so it is conservative, as desired.
	\end{proof}

\end{thm}

We will be interested in descent along sieves of the following form:
\begin{defn} \label{defn:D-cover}
	Let $D : \mathcal C^\op \to \widehat{\Cat}$ be a presheaf of categories on a pullback context $\mathcal C$. A sieve $\mathcal U$ on an object $S \in \mathcal C$ is a \emph{$D$-covering sieve} if it contains a family of quasi-admissible maps $\{X_i \to S\}_i$ such that
	\[
		D(S) \to \prod_i D(X_i)
	\]
	is conservative.

	As usual, a \emph{$D$-covering family} of $S \in \mathcal C$ is a family of maps $\{X_i \to S\}_i$ that generates a $D$-covering sieve.
\end{defn}

Our main result about $D$-covering sieves is the following:
\begin{thm} \label{thm:D-topology}
	% PERF: do I really need small?
	Let $D$ be a pullback formalism on a small pullback context $\mathcal C$. Then the $D$-covering sieves form a Grothendieck topology on $\mathcal C$, which we will call the \emph{$D$-covering topology}, and which satisfies the following:
	\begin{enumerate}

		% \item For any sieve $\mathcal U$ on an object $S \in \mathcal D$, $\mathcal U$ is a covering sieve for the $D$-topology if and only if it contains a family of quasi-admissible maps $\{X_i \to S\}_i$ such that
		% 	\[
		% 		D(S) \to \prod_i D(X_i)
		% 	\]
		% 	is conservative.

		\item For any morphism of pullback formalisms $D \to D'$, we have that $D'$ is a sheaf for the $D$-covering topology.

		\item For any $S \in \mathcal C$, and $M \in D(S)$, the presheaf $D(-;M) \in \Psh(\mathcal C_{/S})$ is a sheaf for the restriction of the $D$-covering topology to $\mathcal C_{/S}$.

		\item The $D$-covering sieves are precisely the sieves containing $\yo$-quasi-admissible $D$-acyclic sieves.

	\end{enumerate}
\end{thm}

Before addressing the proof of \Cref{thm:D-topology}, we will show the following corollary:
\begin{cor} \label{cor:dense slice}
	Let $D$ be a pullback formalism on a small pullback context $\mathcal C$. If $\mathcal B \subseteq \mathcal C$ is a full anodyne pullback subcontext such that every object $S \in \mathcal C$ admits a quasi-admissible $D$-covering family by objects of $\mathcal B$, then for any pullback formalism $D'$ on $\mathcal C$ admitting a morphism from $D$, the functor $\PF(\mathcal C)_{D'/} \to \PF(\mathcal B)|_{D'|_{\mathcal B}/}$ is an equivalence.
	\begin{proof}
		\Cref{thm:D-topology} says that descent for the $D$-covering topology is equivalent to descent along quasi-admissible $D$-covering sieves, so \Cref{prp:quasi-top basis} and \Cref{prp:local PF dense subctx} show that $\PF(\mathcal C) \to \PF(\mathcal B)$ restricts to an equivalence on the full subcategories of pullback formalisms that are sheaves for the $D$-covering topology.

		By \Cref{thm:D-topology}, we have that since $D'$ admits a morphism from $D$, every morphism of pullback formalisms $D' \to D''$ ($D'|_{\mathcal B} \to D''$) is a morphism between sheaves for the $D$-covering topology, so the result follows.
	\end{proof}
\end{cor}

The following result establishes the equivalence of various notions of ``quasi-admissibility'' for sieves in pullback contexts.
\begin{lem} \label{lem:sieve gen by qadm iff qadms are cofinal}
	% PERF: do I really need small?
	Let $\mathcal C$ be a small pullback context, and let $\mathcal U \subseteq \mathcal C_{/X}$ be a sieve on $X \in \mathcal C$. Write $\mathcal U^\qadm$ for the full subcategory of $\mathcal U$ consisting of those $X' \to X$ that are quasi-admissible. Then the following are equivalent:
	\begin{enumerate}

		\item $\mathcal U^\qadm \to \mathcal U$ is cofinal.

		\item $\mathcal U$ is generated by the quasi-admissible $X' \to X$ in $\mathcal U$.

		\item The map $\mathcal U \to \yo(X)$ is $\yo$-quasi-admissible.

		\item The map $\mathcal U \to \yo(X)$ is weakly $\yo$-quasi-admissible (see \Cref{defn:weak yo-qadm}).

	\end{enumerate}
	\begin{proof}
		Note that the third condition implies the fourth by \Cref{lem:Psh qadm source-local}, and the fourth implies the second by \cite[Lemma 6.2.3.13]{htt} since $\mathcal U \to \yo(X)$ is a monomorphism.

		By \cite[Theorem 4.1.3.1]{htt}, we have that $\mathcal U^\qadm \to \mathcal U$ is cofinal if and only if for every $f : Y \to X$ in $\mathcal U$, the category
		\[
			\mathcal U^\qadm \times_{\mathcal U} \mathcal U_{f/}
		\]
		is weakly contractible.

		% NOTE: need that $\mathcal U^\qadm$ is _full_ subcat of $\mathcal C_{/X}$, otherwise, this category won't have products
		% Although can form the product in $\mathcal C_{/X}$, and it would still by in $\mathcal U$, the object might not be a product in this subcategory, since the induced maps to it might not be in the subcategory.
		This category can be described as the full subcategory of $\mathcal U_{f/}$ consisting of those $Y \to X'$ over $X$ such that $X' \to X$ is quasi-admissible and in $\mathcal U$. Note that since quasi-admissible maps are stable under base change, this category admits binary products, which implies that it is cosifted if and only if it is nonempty. In particular, by \cite[Proposition 5.5.8.7]{htt}, this category is weakly contractible if and only if it is nonempty.

		Thus, $\mathcal U^\qadm \to \mathcal U$ is cofinal if and only if every $Y \to X$ in $\mathcal U$ factors through a quasi-admissible $X' \to X$ in $\mathcal U$, \ie, $\mathcal U$ is generated by the quasi-admissible $X' \to X$, which establishes the equivalence between the first two conditions.

		It only remains to show that the second condition implies the third. Indeed, by \Cref{thm:PF descent} it only remains to show that if $D \in \PF(\mathcal C)$, then $D$ has left base change for $\mathcal U \to \yo(X)$. But \Cref{thm:PF descent} already shows that $D$ has left base change for $\mathcal U \to \yo(X)$ against any map $\yo(X') \to \yo(X)$. Note that the base change of $\mathcal U \to \yo(X)$ along such a map is a sieve on $X'$ that is also generated by quasi-admissible maps, so we conclude by \Cref{prp:source locality of base change}.
	\end{proof}
\end{lem}

In order to prove \Cref{thm:D-topology}, we will first need the following:
\begin{rmk} \label{rmk:qadm pseudocover}
	If $\{X_i \to X\}_i$ is a family of quasi-admissible maps in a small pullback context $\mathcal C$, then \Cref{thm:PF descent} shows that for any $D \in \PF(\mathcal C)$, we have that $D(X) \to \prod_i D(X_i)$ is conservative if and only if for any morphism of pullback formalisms $D \to D'$, $D'$ has descent along every base change of this family.
\end{rmk}

\begin{proof}[Proof of \Cref{thm:D-topology}]
	First we will show that the $D$-covering sieves form a Grothendieck topology. \Cref{rmk:qadm pseudocover} shows that these are stable under base change, so (following \cite[Definition 6.2.2.1]{htt}) it only remains to show the following: if $\{X_i \to S\}_{i \in I}$ is a family of quasi-admissible maps such that
	\[
		D(S) \to \prod_{i \in I} D(X_i)
	\]
	is conservative, then for any sieve $\mathcal U$ on $S$, if the base change $\mathcal U_i$ of $\mathcal U$ to $X_i$ is a $D$-covering sieve for all $i \in I$, then $\mathcal U$ is a $D$-covering sieve. Indeed, this means that for any $i \in I$, there is a family of quasi-admissible maps $\{X_{ij} \to X_i\}_{j \in J_i}$ in $\mathcal U_i$ such that
	\[
		D(X_i) \to \prod_{j \in J_i} D(X_{ij})
	\]
	is conservative.

	It follows that for any $i \in I$ and $j \in J_i$, the map $X_{ij} \to S$ is a quasi-admissible map in $\mathcal U$. Furthermore, the composite
	\[
		D(S) \to \prod_{i \in I} D(X_i) \to \prod_{\substack{i \in I \\ j \in J_i}} D(X_{ij})
	\]
	is conservative, so we have shown that $\mathcal U$ is a $D$-covering sieve, which concludes the proof that the $D$-covering sieves form a Grothendieck topology.

	Now, by \Cref{lem:sieve gen by qadm iff qadms are cofinal}, we have that the $\yo$-quasi-admissible sieves are precisely the ones generated by quasi-admissible maps, so by \Cref{rmk:qadm pseudocover}, we have that the $D$-covering sieves are precisely the sieves that contain $\yo$-quasi-admissible $D$-acyclic sieves. Thus, a presheaf has descent for the $D$-covering topology if and only if it has descent along the $\yo$-quasi-admissible $D$-acyclic sieves, so we conclude by \Cref{rmk:qadm pseudocover} and \Cref{rmk:describe qadm acyclic}.
\end{proof}

\section{Localizing Pullback Formalisms} \label{S:invertibility}

In this section we will study how to freely invert morphisms in the categories coming from pullback formalisms. We will be particularly interested in doing this in order to freely impose the invariance properties considered in \Cref{S:acyclic}. This will be achieved in \Cref{S:imposing invariance} after establishing some general results about localizing pullback formalisms in \Cref{S:localizing pf}.
% First we will prove some general results about localization in \Cref{S:localizing pf}, and then in \Cref{S:acyclic} we will consider invariance properties of cohomology theories coming from pullback formalisms, and use the results of \Cref{S:localizing pf} to describe how to freely impose such invariance properties.

\subsection{General Localizations} \label{S:localizing pf}

In this \lcnamecref{S:localizing pf} we will study how to produce objectwise localizations of pullback formalisms in a way that is compatible with the structure of a pullback formalism.

Our first result is is a fully general statement that describes when we can construct localizations of pullback formalisms along prescribed families of morphisms:
\begin{prp} \label{prp:localizing formalisms}
	Let $D$ be a pullback formalism on a pullback context $\mathcal C$, and for each $X \in \mathcal C$, let $W_X$ be a collection of morphisms in $D(X)$ such that
	\begin{enumerate}

		\item The strongly saturated class $\bar W_X$ generated by $W_X$ is of small generation for all $X \in \mathcal C$.

		\item For each $f : X \to Y$ in $\mathcal C$, the functor $f^* : D(Y) \to D(X)$ sends $W_Y$ to $\bar W_X$.

		\item For each quasi-admissible map $f : X \to Y$ in $\mathcal C$, the functor $f_\sharp : D(X) \to D(Y)$ sends $W_X$ to $\bar W_Y$.

		\item For any $X \in \mathcal C$, and $A \in D(X)$ if $\phi \in W_X$, then $\phi \otimes A \in \bar W_X$.

	\end{enumerate}
	Then there is a morphism of pullback formalisms $L : D \to W^{-1} D$ such that for each $X \in \mathcal C$, the functor $L(X) : D(X) \to (W^{-1} D)(X)$ is an accessible localization along $W_X$, and furthermore, a morphism of pullback formalisms $D \to E$ extends through $L$ if and only if for each $X \in \mathcal C$, the map $D(X) \to E(X)$ sends maps in $W_X$ to equivalences, in which case the space of extensions is contractible.

	Conversely, if $L : D \to E$ is a morphism of pullback formalisms such that for each $X \in \mathcal C$, if we write $W_X$ for the collection of morphisms in $D(X)$ that are inverted by $L$, we have that the family $\{W_X\}_{X \in \mathcal C}$ satisfies the above conditions.
	\begin{proof}
		Theorem \ref{thm:monoidal objectwise localization} guarantees the existence of an $L : D \to W^{-1} D$ in $\Psh_{\CAlg(\PrL)}(\mathcal C)$ such that for each $X \in \mathcal C$, the functor $L(X) : D(X) \to (W^{-1} D)(X)$ is an accessible localization along $W_X$.

		If $f : X \to Y$ is quasi-admissible, since $f_\sharp$ preserves colimits and sends $W_X$ to $\bar W_Y$, it also sends $\bar W_X$ to $\bar W_Y$, so Lemma \ref{lem:localizing adjunctions} says that $(W^{-1} D)(f)$ admits a left adjoint and that the unit and counit of the adjunction are induced by the unit and counit of the adjunction $f_\sharp \dashv f^*$. In particular, $L$ respects quasi-admissibility.

		Thus, Proposition \ref{prp:map from PF is to PF} shows that $W^{-1} D$ is a pullback formalism, and that for any $E \in \PF(\mathcal C)$, the square
		\[
			\begin{tikzcd}
				\PF(\mathcal C)(W^{-1} D, E) \ar[d] \ar[r] & \PF(\mathcal C)(D, E) \ar[d] \\
				\Psh_{\CAlg(\PrL)}(\mathcal C)(W^{-1} D, E) \ar[r] & \Psh_{\CAlg(\PrL)}(\mathcal C)(D, E)
			\end{tikzcd}
		\]
		is Cartesian. Theorem \ref{thm:monoidal objectwise localization} already shows that if $D \to E$ is any transformation of functors $\mathcal C^\op \to \CAlg(\PrL)$, then the space of extensions is nonempty if and only if for each $X \in \mathcal C$, $D(X) \to E(X)$ sends maps in $W_X$ to equivalences, in which case the space of extensions is contractible. It follows that since the above square is Cartesian, if $D \to E$ respects quasi-admissibility, then the space of extensions through $L$ by maps that respect quasi-admissibility is contractible.

		We prove the properties required for the converse individually, noting that for each $X$, $W_X$ is already strongly saturated by Remark 5.5.4.10 of \cite{htt}:
		\begin{enumerate}

			\item This follows from Proposition 5.5.4.16 and Example 5.5.4.9 of \cite{htt}.

			\item This follows from commutativity of the square
				\[
					\begin{tikzcd}
						D(Y) \ar[d] \ar[r, "f^*"] & D(X) \ar[d] \\
						E(Y) \ar[r, "f^*"'] & E(X)
					\end{tikzcd}
				.\]

			\item This follows from commutativity of the following square (which is guaranteed by the fact that $L$ respects quasi-admissibility):
				\[
					\begin{tikzcd}
						D(X) \ar[d] \ar[r, "f_\sharp"] & D(Y) \ar[d] \\
						E(X) \ar[r, "f_\sharp"'] & E(Y)
					\end{tikzcd}
				.\]

			\item  This follows from commutativity of the following square (which is guaranteed by the fact that $L(X) : D(X) \to E(X)$ is symmetric monoidal):
				\[
					\begin{tikzcd}
						D(X) \times D(X) \ar[d] \ar[r, "\otimes"] & D(X) \ar[d] \\
						E(X) \times E(X) \ar[r, "\otimes"'] & E(X)
					\end{tikzcd}
				.\]

		\end{enumerate}
		
	\end{proof}
\end{prp}

The following result tells us how to construct a localization of a pullback formalism along a ``generating family'' of morphisms:
\begin{prp} \label{prp:localizing formalisms along generating family}
	Let $D$ be a pullback formalism on a quasi-small pullback context $\mathcal C$, and for each $X \in \mathcal C$, let $W_X$ be a collection of morphisms in $D(X)$ such that
	\begin{enumerate}

		\item For each $X \in \mathcal C$, the collection $W_X \subseteq \Fun(\Delta^1, D(X))$ is generated under colimits by some small subcollection.

		\item For each $f : X \to Y$ in $\mathcal C$, the map $f^* : D(Y) \to D(X)$ sends $W_Y$ to the strongly saturated class $\bar W_X$ generated by $W_X$.

	\end{enumerate}
	Then there is a morphism of pullback formalisms $L : D \to W^{-1} D$ such that for each $S \in \mathcal C$, the functor $L(S) : D(S) \to (W^{-1} D)(S)$ is an accessible localization along morphisms of the form $p_\sharp(f \otimes A)$, for $p : X \to S$ quasi-admissible, $f \in W_X$, and $A \in D(X)$. If $D$ is geometrically generated, this is also a localization along only morphisms of the form $p_\sharp(f)$ for $p : X \to S$ quasi-admissible, and $f \in W_X$.

	Furthermore, a morphism of pullback formalisms $D \to E$ extends through $L$ if and only if for each $X \in \mathcal C$, the map $D(X) \to E(X)$ sends maps in $W_X$ to equivalences, in which case the space of extensions is contractible.
	\begin{proof}
		For each $S \in \mathcal C$, write $\mathcal W_S$ for the collection of maps in $D(S)$ of the form $p_\sharp(f \otimes A)$ for $p : X \to S$ quasi-admissible, $f \in W_X$, and $A \in D(X)$. Since $\mathcal C_S$ is quasi-small, and for every quasi-admissible $p : X \to S$, $W_X$ is generated under colimits by a small subcollection, $D(X)$ is generated under colimits by a small collection of objects, and $p_\sharp(- \otimes -)$ preserves colimits in each variable, we have that $\mathcal W_S$ is generated under colimits by a small subcollection, and in particular the strongly saturated class $\bar{\mathcal W}_S$ that it generates is of small generation.

		It follows from quasi-admissible base change that if $f : X \to Y$ is a map in $\mathcal C$, then $f^*$ sends $\mathcal W_Y$ to $\mathcal W_X$. By construction, we also have that if $f$ is quasi-admissible, then $f_\sharp$ sends $\mathcal W_X$ to $\mathcal W_Y$.

		Given a quasi-admissible map $p : X \to S$, $f \in W_X$, $A \in D(X)$, and $B \in D(S)$, we have that
		\[
			p_\sharp(f \otimes A) \otimes B \simeq p_\sharp(f \otimes A \otimes p^* B)
		,\]
		which is in $\mathcal W_S$.

		Thus, \Cref{prp:localizing formalisms} shows that there is a localization $L : D \to W^{-1} D$ which satisfies that a morphism of pullback formalisms $D \to E$ extends through $L$ if and only if for each $X \in \mathcal C$, the functor $D(X) \to E(X)$ sends maps $\mathcal W_X$ to equivalences. In particular, $D(X) \to E(X)$ must send maps in $W_X \subseteq \mathcal W_X$ to equivalences, so it only remains to show that if $D \to E$ is morphism of pullback formalisms, and for each $X \in \mathcal C$, $D(X) \to E(X)$ sends maps in $W_X$ to equivalences, then for each $X \in \mathcal C$, $D(X) \to E(X)$ sends maps in $\mathcal W_X$ to equivalences. This also follows easily from the description of $\mathcal W_X$, and the fact that $D \to E$ respects quasi-admissibility and is objectwise monoidal.

		To show the statement about the case when $D$ is geometrically generated, we note that if $p : X \to S$ is quasi-admissible, and $f \in W_X$, then the functor $p_\sharp(f \otimes -)$ preserves colimits, so since $D(X)$ is generated under colimits by objects of the form $\sigma_\sharp(1)$, where $\sigma : X' \to X$ is quasi-admissible, we have that any map $p_\sharp(f \otimes A)$ is a colimit of maps of the form $p_\sharp(f \otimes \sigma_\sharp(1))$, but by the projection formula,
		\[
			p_\sharp(f \otimes \sigma_\sharp(1)) \simeq p_\sharp(\sigma_\sharp(\sigma^*(f) \otimes 1)) \simeq (p \circ \sigma)_\sharp(\sigma^*(f))
		,\]
		which is of the form $q_\sharp(g)$ for the quasi-admissible map $q = p \circ \sigma : X' \to S$, and $g = \sigma^*(f) \in W_{X'}$.
	\end{proof}
\end{prp}

\subsection{Imposing Invariance} \label{S:imposing invariance}

We will now study how to freely impose invariance conditions on pullback formalisms.

\begin{defn}
	Let $\mathcal C$ be a pullback context with a pseudotopology $\tau$.

	Define $\PF^\tau(\mathcal C) \subseteq \PF(\mathcal C)$ to be the full subcategory of pullback formalisms $D$ such that every $\tau$-acyclic pseudosieve is $D$-acyclic. Call such pullback formalisms \emph{$\tau$-invariant pullback formalisms}.
\end{defn}

\begin{defn} \label{defn:qadm pseudotop}
	A pseudotopology $\tau$ on a pullback context $\mathcal C$ is said to be \emph{quasi-admissible} if every $\tau$-acyclic pseudosieve is $\yo$-quasi-admissible.
\end{defn}

For the remainder of this section, $\mathcal C$ is a locally small quasi-small pullback context, and $\tau$ is a small quasi-admissible pseudotopology on $\mathcal C$.

\begin{prp} \label{prp:invariance localization}
	% NOTE: need quasi-small to apply result about localization
	% in particular, otherwise might be localizing $D(S)$ along too many things
	% since need all the $\sigma_\sharp(\epsilon)$, but maybe there are too many $\sigma$

	The inclusion $\PF^\tau(\mathcal C) \to \PF(\mathcal C)$ admits a left adjoint $(-)^\tau$, and for any $D \in \PF(\mathcal C)$, the unit $L_\tau : D \to D^\tau$ is a morphism of pullback formalisms such that for each $S \in \mathcal C$, the functor $D(S) \to D^\tau(S)$ is the localization along all maps $\sigma_\sharp(\epsilon)$, where $\sigma : S' \to S$ is quasi-admissible, and $\epsilon$ is a counit map $u_\sharp u^* \to \id_{D(S')}$ for $u$ a $\tau$-acyclic pseudosieve on $S'$.\footnote{By \Cref{prp:PF on Shv}, we have that $u$ is $\yo$-quasi-admissible, so the functors $u_\sharp, u^*$ exist and behave as expected.}

	In fact, the slice projection $\PF(\mathcal C)_{D^\tau/} \to \PF(\mathcal C)_{D/}$ is fully faithful, with essential image given by the morphisms from $D$ to $\tau$-invariant pullback formalisms.

	Furthermore, if $D$ is geometrically generated, the localization $D(S) \to D^\tau(S)$ is the same as the localization only along maps of the form $\sigma_\sharp(u_\sharp 1 \to 1)$ for $\sigma : S' \to S$ quasi-admissible, and $u$ a $\tau$-acyclic pseudosieve on $S'$.
	\begin{proof}
		For each $S \in \mathcal C$, and $D \in \PF(\mathcal C)$, let $W_S^D$ be the collection of maps in $D(S)$ that are counit maps $u_\sharp 1 \simeq u_\sharp u^* 1 \to 1$, where $u$ is a $\tau$-acyclic pseudosieve on $S$.
		
		Let $D \in \PF(\mathcal C)$. Since the pseudotopology $\tau$ is small, and $\mathcal C$ is quasi-small, the collection of $\tau$-acyclic pseudosieves on $S$ is small, so $W_S^D$ is small for every $S \in \mathcal C$. For any $f : X \to Y$ in $\mathcal C$, the fact that $f^* : D(Y) \to D(X)$ preserves monoidal units and all $\tau$-acyclic pseudosieves are $\yo$-quasi-admissible, implies that $f^* W_Y^D \subseteq W_X^D$.

		Thus, we can apply \Cref{prp:localizing formalisms along generating family} to produce a morphism of pullback formalisms $L_\tau : D \to D^\tau$ satisfying the following:
		\begin{itemize}

			\item For any morphism of pullback formalisms $D \to E$, if $D(S) \to E(S)$ sends $W_S^D$ to equivalences for all $S$, then there is a contractible space of extensions to $D^\tau \to E$. Note that every map in $W_S^D$ is of the form $u_\sharp(1) \to 1$ for $u$ a $\tau$-acyclic pseudoseive on $S$, so since $D \to E$ preserves monoidal units and all $\tau$-acyclic pseudosieves are $\yo$-quasi-admissible, it follows that $D(S) \to E(S)$ sends $W_S^D$ to the collection $W_S^E$. Since every map in $W_S^E$ is invertible if $E$ is $\tau$-invariant, it follows that there is a contractible space of extensions $D^\tau \to E$ if $E \in \PF^\tau(\mathcal C)$.

			\item For any $S \in \mathcal C$, $D(S) \to D^\tau(S)$ is the localization along all maps of the form $\sigma_\sharp(f \otimes A)$, where $\sigma : S' \to S$ is quasi-admissible, $f \in W_{S'}^D$, and $A \in D(S')$. We know that $f$ is of the form $u_\sharp(1) \to 1$ for $u$ a $\tau$-acyclic pseudosieve on $S'$, so by the projection formula (and since $u$ is $\yo$-quasi-admissible), $f \otimes A \simeq \epsilon(A)$, where $\epsilon$ is the counit of $u_\sharp \dashv u^*$. 

			\item If $D$ is geometrically generated, then for any $S \in \mathcal C$, $D(S) \to D^\tau(S)$ is the localization along all maps of the form $\sigma_\sharp(u_\sharp(1) \to 1)$, where $\sigma : S' \to S$ is quasi-admissible, and $u_\sharp(1) \to 1$ is in $W_{S'}^D$.

		\end{itemize}
		Thus, if $E \in \PF^\tau(\mathcal C)$, then the map
		\[
			\PF^\tau(\mathcal C)(D^\tau, E) \to \PF(\mathcal C)(D, E)
		\]
		is an equivalence, so that $L_\tau$ exhibits $D^\tau$ as a $\PF^\tau(\mathcal C)$-localization of $D$ (see \cite[Definition 5.2.7.6]{htt}). We may then use \cite[Proposition 5.2.7.8]{htt} to see that the inclusion $\PF^\tau(\mathcal C) \to \PF(\mathcal C)$ admits a left adjoint, and the unit of the adjunction admits the desired description.
		
		Finally, the statement about $\PF(\mathcal C)_{D^\tau/} \to \PF(\mathcal C)_{D/}$ follows from the fact that $D \to D^\tau$ is initial among morphisms from $D$ to $\tau$-invariant pullback formalisms, and the fact that any morphism from $D^\tau$ must be to a $\tau$-invariant pullback formalism by \Cref{lem:qadm acyclic}.
	\end{proof}
\end{prp}

\begin{cor} \label{cor:describe invariance localization of Huniv}
	The category $\PF^\tau(\mathcal C)$ has an initial object $H^\tau$, and furthermore, for each $S \in \mathcal C$, the functor $H^\univ(S) \to H^\tau(S)$ is the localization $\Psh(\mathcal C_S) \to \Psh^\tau(\mathcal C_S)$, where we may view $\tau$ as a pseudotopology on $\mathcal C_S$ using the procedure of \Cref{lem:pull back pseudotop}, so this is the localization along all maps that lie over $\tau$-acyclic pseudosieves in $\Psh(\mathcal C)$.

	Furthermore, the slice projection $\PF(\mathcal C)_{H^\tau/} \to \PF(\mathcal C)$ is a fully faithful functor whose essential image is given by $\PF^\tau(\mathcal C)$, and for any $S \in \mathcal C$, the functor $H^\univ(S) \to H^\tau(S)$ is an accessible locally Cartesian localization between presentable categories with universal colimits.
	\begin{proof}
		By \Cref{prp:invariance localization}, there is a left adjoint $(-)^\tau$ of the inclusion $\PF^\tau(\mathcal C) \to \PF(\mathcal C)$, and for any $D \in \PF(\mathcal C)$, the slice projection $\PF(\mathcal C)_{D^\tau/} \to \PF(\mathcal C)_{D/}$ is fully faithful with essential image given by the maps from $D$ to $\tau$-invariant pullback formalisms. Furthermore, since $\mathcal C$ is quasi-small, the pullback formalism $H^\univ$ exists, and \Cref{thm:Huniv is univ} says it is initial in $\PF(\mathcal C)$. It follows that $H^\tau = (H^\univ)^\tau$ is initial in $\PF^\tau(\mathcal C)$, and that $\PF(\mathcal C)_{H^\tau/} \to \PF(\mathcal C)_{H^\univ/} \simeq \PF(\mathcal C)$ is equivalent to the inclusion of $\tau$-invariant pullback formalisms.

		By \Cref{prp:invariance localization}, since $H^\univ$ is geometrically generated, for any $S \in \mathcal C$, the map $H^\univ(S) \to H^\tau(S)$ is the localization along maps of the form $\sigma_\sharp(u_\sharp 1 \to 1)$, where $\sigma : S' \to S$ is quasi-admissible, and $u$ is a $\tau$-acyclic pseudosieve on $S'$.

		Now, recall the pullback formalism $\bar H^\univ \coloneqq H^\slice_{\Psh(\mathcal C)}$ on $\Psh(\mathcal C)$ of \Cref{cnstr:Huniv}. Since $\bar H^\univ$ corresponds to the limit-preserving pullback formalism on $\Psh(\mathcal C)$ that is simply given by the slice presheaf $P \mapsto \Psh(\mathcal C)_{/P}$, we have that for any quasi-admissible $f : P \to Q$ in $\Psh(\mathcal C)$, there is a commutative triangle
		\[
			\begin{tikzcd}
				\bar H^\univ(P) \ar[dr] \ar[r, "f_\sharp"] & \bar H^\univ(Q) \ar[d] \\
										& \Psh(\mathcal C)
			\end{tikzcd}
		.\]
		Thus, in $\bar H^\univ(\yo(S)) = \Psh(\mathcal C)_{/\yo(S)}$, the maps of the form $\sigma_\sharp(u_\sharp 1 \to 1)$ for $\sigma : S' \to S$ quasi-admissible and $u$ a $\tau$-acyclic pseudosieve on $S'$, are precisely the maps in $\Psh(\mathcal C)_{/\yo(S)}$ lying over maps in $\Psh(\mathcal C)$ that are $\tau$-acyclic pseudosieves on objects $S'$ that admit quasi-admissible maps to $S$.

		Since $H^\univ(S) \to \bar H^\univ(\yo(S))$ is fully faithful, and $H^\univ \to \bar H^\univ|_{\mathcal C}$ respects quasi-admissibility, we have that the maps in $H^\univ(S)$ that are sent to maps of the form $\sigma_\sharp(u_\sharp 1 \to 1)$ as above in $\bar H^\univ(\yo(S))$, are precisely the maps of the form $\sigma_\sharp(u_\sharp 1 \to 1)$ in $H^\univ(S)$. Thus, we may conclude by our description of these maps in $\bar H^\univ(\yo(S))$.

		It follows that for any $S \in \mathcal C$, the functor $H^\univ(S) \to H^\tau(S)$ is the localization $\Psh(\mathcal C_S) \to \Psh^{\tau_S}(\mathcal C_S)$, where $\tau_S$ is the pseudotopology on $\mathcal C_S$ given by pulling $\tau$ back along $\mathcal C_S \to \mathcal C$ using \Cref{lem:pull back pseudotop}, noting that the hypotheses are satisfied by \Cref{lem:Psh qadm source-local}. Thus, \Cref{prp:pseudotop LCL} says that $H^\univ(S) \to H^\tau(S)$ is an accessible locally Cartesian localization between presentable categories with universal colimits.
	\end{proof}
\end{cor}

\begin{defn} \label{defn:local equivalences for pseudotop}
	For $\mathcal C$ a quasi-small pullback context with small quasi-admissible pseudotopology $\tau$, if $D$ is a pullback formalism on $\mathcal C$, and $S \in \mathcal C$ is an object, we define the \emph{$\tau$-local equivalences} of $D(S)$ to be the collection of maps in $D(S)$ that are inverted by $L_\tau : D(S) \to D^\tau(S)$.

	Note that by Corollary \ref{cor:describe invariance localization of Huniv}, this extends the definition of $\tau$-local equivalence from Definition \ref{defn:pseudotop}.
\end{defn}

\begin{rmk} \label{rmk:general realization}
	Let $F : \mathcal C \to \mathcal D$ be a morphism of pullback contexts. Let $D$ be a pullback formalism on $\mathcal D$, such that $F$ sends every $\tau$-acyclic pseudosieve to a $D$-acyclic map. Then there is a unique morphism of pullback formalisms
	\[
		H^\tau \to F^* D
	.\]

	Note that for any quasi-admissible map $X \to S$ in $\mathcal C$, the functor $H^\tau(S) \to D(F(S))$ sends $L_\tau \yo(X) \in H^\tau(S)$ to $\cls{F(X);F(S)}_D$ in $D(F(S))$.

	When $\mathcal D$ is quasi-small, it is easy to deduce that, in fact, we have a commutative diagram
	\[
		\begin{tikzcd}
			\mathcal C_S \ar[d] \ar[r] & \mathcal D_{F(S)} \ar[d] \\
			H^\tau(S) \ar[r] & D(F(S))
		\end{tikzcd}
	\]
	by using the commutative diagram
	\[
		\begin{tikzcd}
			H^\univ \ar[d] \ar[r] & H^\univ \circ F \ar[d] \\
			H^\tau \ar[r] & D \circ F
		\end{tikzcd}
	\]
	in $\PF(\mathcal C)$.
\end{rmk}

\begin{rmk} \label{rmk:abuse notation for univ PF}
	Let $\mathcal C' \subseteq \mathcal C$ be a full subcategory such that for any object $X \in \mathcal C$ admitting a quasi-admissible map to an object of $\mathcal C'$, $X \in \mathcal C'$. Then $\mathcal C'$ is a full anodyne pullback subcontext of $\mathcal C$ by \Cref{rmk:full anodyne subctx}, and the pseudotopology $\tau$ restricts to a pseudotopology $\tau'$ on $\mathcal C'$ by \Cref{lem:pull back pseudotop} and \Cref{lem:Psh qadm source-local}, and the morphism
	\[
		H^{\tau'} \to H^\tau|_{\mathcal C'}
	\]
	of Remark \ref{rmk:general realization} is an equivalence. Therefore, we will often abuse notation and write $H^\tau$ also for its restriction to any such subcategory $\mathcal C'$.
\end{rmk}

\section{Stabilization} \label{S:stabilize PF}

The category of spectra, thought of as a ``stabilization'' of the category of spaces, is obtained by stabilizing the action of the suspension functor on pointed spaces. It turns out that this can be thought of as formally adjoining $\wedge$-inverses of $S^1$ in the category of pointed spaces.

In this \namecref{S:stabilize PF}, we will be concerned with a way of producing ``stabilizations'' of pullback formalisms. First we will study how to formally adjoin $\otimes$-inverses in \Cref{S:tensor inversion}, and later we will study pointing in \Cref{S:pointing pf}. Note that the relation between stable categories and adjoining $\otimes$-inverses to pointed categories is made clear in \Cref{lem:auto stability}.

\subsection{Formal \texorpdfstring{$\otimes$}{⊗}-inversions} \label{S:tensor inversion}

In this \lcnamecref{S:tensor inversion}, we will study how to formally adjoin $\otimes$-inverses to pullback formalisms.

The following terminology will be useful for our discussion of formally adjoining $\otimes$-inverses. It corresponds to the terminology of generating strongly saturated classes of morphisms in the context of formally adjoining inverses of morphisms to presentable categories, as in \cite[\S 5.5.4]{htt}.
\begin{defn} \label{defn:tensor generation}
	% NOTE: it does not matter if I restrict $F$ to colimit-preserving functors to symmetric monoidal categories compatible with colimits
	% Since by \cite[Remark 4.8.1.9]{ha}, this gives the same notion of $\otimes$-generation
	% in particular, given an $F$ that sends $A$ to $\otimes$-invertible objects, composing it with the left adjoint from loc. cit gives...
	Given a symmetric monoidal category $\mathcal A \in \CAlg(\Cat)$, and a small collection of objects $A$ of $\mathcal A$, we say an object $a \in \mathcal A$ is \emph{$\otimes$-generated} by $A$ if for every symmetric monoidal functor $F : \mathcal A \to \mathcal B$, if $F$ sends every object of $A$ to a $\otimes$-invertible object of $\mathcal B$, then it sends $a$ to a $\otimes$-invertible object of $\mathcal B$.
\end{defn}

\begin{rmk}
	In the setting of Definition \ref{defn:tensor generation}, suppose that $a \in \mathcal A$ is an object such that there is some $a' \in \mathcal A$ where $a \otimes a'$ is equivalent to a tensor product of objects of $A$. Then $a$ is $\otimes$-generated by $A$.
	\begin{proof}
		For any symmetric monoidal $F : \mathcal A \to \mathcal B$ that sends objects of $A$ to $\Pic(\mathcal B)$,
		\[
			F(a \otimes a') \simeq F(a) \otimes F(a')
		\]
		is $\otimes$-invertible since it is equivalent to a tensor product of invertible objects, so let $b$ be a $\otimes$-inverse. It follows that $F(a)$ has an $\otimes$-inverse given by $F(a') \otimes b$.
	\end{proof}
\end{rmk}

We will be interesting in the following setting:

\begin{setng} \label{setng:stabilization}
	Let $\mathcal C$ be a pullback context, let $D$ be a pullback formalism on $\mathcal C$, and let $A = \{A_S\}_{S \in \mathcal C}$ be a family such that for each $S \in \mathcal C$, $A_S$ is a small collection of objects in $D(S)$, and for any map $f : X \to Y$ in $\mathcal C$, the functor $f^* : D(Y) \to D(X)$ sends every object in $A_Y$ to an object that is $\otimes$-generated by $A_X$.
\end{setng}

We would like to prove that in \Cref{setng:stabilization}, there is a universal morphism of pullback formalisms $\Sigma_A^\infty : D \to D[A^{\otimes -1}]$ such that for any $S \in \mathcal C$, $D(S) \to D[A^{\otimes -1}](S)$ sends $A_S$ to $\Pic D[A^{\otimes -1}](S)$.

We will need to introduce the following notion that will be useful for showing that $D[A^{\otimes -1}]$ is a pullback formalism:
\begin{defn} \label{defn:A-lifts}
	In \Cref{setng:stabilization}, say a map $f : X \to Y$ in $\mathcal C$ has $A$-lifts if every element of $A_X$ is $\otimes$-generated by $f^*(A_Y)$.
\end{defn}

\begin{rmk} \label{rmk:crit to check that obj sent to Pic}
	In \Cref{setng:stabilization}, if $D \to E$ is any transformation of presheaves $\mathcal C^\op \to \CAlg(\PrL)$, if $S' \to S$ is a map in $\mathcal C$ that has $A$-lifts, and the map $D(S) \to E(S')$ sends $A_S$ to $\Pic E(S')$, then $D(S') \to E(S')$ sends $A_{S'}$ to $\Pic E(S')$.

	In fact, we have more generally that if $S' \in \mathcal C$, and $a \in D(S')$, if there is a map $p : S' \to S$ such that $a$ is $\otimes$-generated by $p^*$ of the collection of objects in $D(S)$ that are sent to $\Pic E(S')$, then $a$ is sent to $\Pic E(S')$.
\end{rmk}

\begin{lem} \label{lem:crit for map from tensor inversion}
	In \Cref{setng:stabilization}, if $D \to E$ is a morphism of pullback formalisms, then the following conditions are equivalent:
	\begin{enumerate}

		\item For every $S \in \mathcal C$, the functor $D(S) \to E(S)$ sends $A_S$ to $\Pic E(S)$.

		\item For every $S \in \mathcal C$, there is a quasi-admissible $D$-covering family of $S$ by objects $S_0$ such that there is a map $S_0 \to S_1$ that has $A$-lifts, and $D(S_1) \to E(S_0)$ sends $A_{S_1}$ to $\Pic E(S_0)$.

	\end{enumerate}
	\begin{proof}
		It is clear that the first condition implies the second. Note that every quasi-admissible $D$-covering family is also a quasi-admissible $E$-covering family by \Cref{thm:D-topology}, so the converse follows from \Cref{lem:tensor invertibility is local} and \Cref{rmk:crit to check that obj sent to Pic}.
	\end{proof}
\end{lem}

\begin{defn} \label{defn:tensor inversion PF}
	In \Cref{setng:stabilization}, a morphism of pullback formalisms $\Sigma_A^\infty : D \to D[A^{\otimes -1}]$ is said to be \emph{a formal $\otimes$-inversion of $A$} if the slice projection $\PF(\mathcal C)_{D[A^{\otimes -1}]/} \to \PF(\mathcal C)_{D/}$ is fully faithful with essential image given by those morphisms $D \to E$ that satisfy the equivalent conditions of \Cref{lem:crit for map from tensor inversion}
\end{defn}

Before addressing our main results about the existence and properties of formal $\otimes$-inversions, we make some general remarks about \Cref{setng:stabilization}.

\begin{rmk} \label{rmk:pull back stab setng}
	In \Cref{setng:stabilization}, if $F : \mathcal C' \to \mathcal C$ is a morphism of pullback contexts, then we can define $A' = \{A_{F(S')}\}_{S' \in \mathcal C'}$, and $F^* D, A'$ is another instance of \Cref{setng:stabilization}. Furthermore, a map $f'$ in $\mathcal C'$ has $A'$-lifts if and only if $F(f')$ has $A$-lifts.
\end{rmk}

\begin{rmk} \label{rmk:extend stabilization setng}
	In \Cref{setng:stabilization}, suppose that $\mathcal C$ is a full pullback subcontext of a pullback context $\tilde{\mathcal C}$, and that $D$ extends to a pullback formalism on $\tilde{\mathcal C}$. For each $\tilde S \in \tilde{\mathcal C}$, we can define a collection of maps $\tilde A_{\tilde S}$ by
	\[
		\tilde A_{\tilde S} = \bigcup_{\substack{S \in \mathcal C \\ f : \tilde S \to S}} f^* A_S
	.\]

	It follows that for each $S \in \mathcal C$, $A_S \subseteq \tilde A_S$, but $\tilde A_S$ is $\otimes$-generated by $A_S$, and that for any map $f : X \to Y$ in $\tilde{\mathcal C}$, we have that $f^*(\tilde A_Y) \subseteq \tilde A_X$. Thus, $\tilde{\mathcal C}, \tilde D, \{\tilde A_S\}_{S \in \tilde{\mathcal C}}$ gives us a new instance of \Cref{setng:stabilization}.

	Furthermore, any map $f : X \to Y$ in $\mathcal C$ has $A$-lifts if and only if it has $\tilde A$-lifts.
\end{rmk}

The following facts about maps with $A$-lifts can sometimes be helpful:
\begin{lem} \label{lem:A-lifts}
	In \Cref{setng:stabilization}, let $f : X \to Y$ be a map in $\mathcal C$.
	\begin{enumerate}

		\item\label{itm:A-lifts/13}
			If there is a map $g : Y \to Z$ such that $g \circ f$ has $A$-lifts, then $f$ has $A$-lifts.

		\item\label{itm:A-lifts/source-local}
			Suppose there is a diagram $\tilde X : K^\triangleright \to \mathcal C$ that sends all maps to quasi-admissible maps, sends the cone point to $X$, and such that for all $p \in K$, the composite $\tilde X(p) \to X \to Y$ has $A$-lifts. If $D \circ \tilde X^\op$ is a limiting diagram, then $f$ has $A$-lifts.

	\end{enumerate}
	\begin{proof}
		The first fact is straightforward: if $A_X$ is $\otimes$-generated by $f^* g^* A_Z$, then since $g^* A_Z$ is $\otimes$-generated by $A_Y$, we have that $A_X$ is $\otimes$-generated by $f^* A_Y$.

		For the second fact, consider the following commutative diagram:
		\[
			\begin{tikzcd}
				D(X)[A_Y^{\otimes - 1}] \ar[d] \ar[r] & \varprojlim_{p \in K} D(\tilde X(p))[A_Y^{\otimes - 1}] \ar[d] \\
				D(X)[A_X^{\otimes - 1}] \ar[r] & \varprojlim_{p \in K} D(\tilde X(p))[A_X^{\otimes - 1}] \ar[d] \\
				& \varprojlim_{p \in K} D(\tilde X(p))[A_{\tilde X(p)}^{\otimes - 1}]
			\end{tikzcd}
		.\]
		Note that we can view $D \circ \tilde X^\op$ as a limiting diagram in $\Mod_{D(X)}(\PrL)$, and therefore also in $\Mod_{D(Y)}(\PrL)$, so using \cref{itm:stabilization/limits} of \Cref{prp:stabilization}, we find that the two horizontal arrows are equivalences. Furthermore, for every $p \in K$, the composite map $\tilde X(p) \to X \to Y$ has $A$-lifts, so by the first property we showed, it follows that $\tilde X(p) \to X$ has $A$-lifts. Thus, the composite of the two leftmost vertical arrows, as well as the bottom left vertical arrow, are equivalences, which implies that the top left vertical arrow is an equivalence. Therefore, the remaining map
		\[
			D(X)[A_Y^{\otimes-1}] \to D(X)[A_X^{\otimes-1}]
		\]
		is an equivalence, which implies that $f^*(A_Y)$ $\otimes$-generates $A_X$, \ie, that $f : X \to Y$ has $A$-lifts, as desired.
	\end{proof}
\end{lem}

We will need the following definition from \cite[Remark 2.20]{robalo}:
\begin{defn} \label{defn:symmetric}
	Let $\mathcal C$ be a symmetric monoidal category. An object $X \in \mathcal C$ is said to be symmetric if the cyclic permutation of $X^{\otimes n}$ is equivalent to the identity for some $n \geq 2$. A collection $A$ of objects in $\mathcal C$ is said to be symmetric if it is $\otimes$-generated by the symmetric objects in $A$.
\end{defn}

We now come to our first result establishing the existence of ``stabilizations'' of pullback formalisms. We will present the proof of this result later.
\begin{prp} \label{prp:PF stabilization}
	In \Cref{setng:stabilization}, assume that all quasi-admissible maps have $A$-lifts, and that for every $S \in \mathcal C$, every element of $A_S$ is $\otimes$-generated by the collection of symmetric objects in $A_S$.

	Then a formal $\otimes$-inversion $\Sigma_A^\infty : D \to D[A^{\otimes -1}]$ of $A$ exists, and it satisfies the following properties:
	\begin{enumerate}

		\item\label{itm:PF stabilization/adjoint}
			If $f$ is a map that has $A$-lifts, a transformation $D[A^{\otimes -1}] \to E$ is left adjointable at $f$ if $D \to E$ is left adjointable at $f$.

		\item\label{itm:PF stabilization/ptwise}
			For each $S \in \mathcal C$, the functor $(D \to D[A^{\otimes -1}])(S)$ is the functor $\Sigma_{A_S}^\infty : D(S) \to D(S)[A_S^{\otimes -1}]$ given by formally adjoining $\otimes$-inverses of elements of $A_S$. 

		\item\label{itm:PF stabilization/generation}
			For each $S \in \mathcal C$, the category $D[A^{\otimes -1}](S)$ is generated under filtered colimits by objects of the form $(\Sigma_A^\infty \alpha)^{\otimes -1} \otimes \Sigma_A^\infty M$ for $\alpha$ a tensor product of objects in $A_S$, and $M \in D(S)$.

		\item\label{itm:PF stabilization/descent}
			For any $X : K^\triangleright \to \mathcal C$ sending all edges in $K$ to quasi-admissible maps, and all edges $p \to \infty$ to maps with $A$-lifts, if
			\[
				D(X(\infty)) \to \varprojlim_{p \in K} D(X(p))
			\]
			is an equivalence, then
			\[
				D[A^{\otimes -1}](X(\infty)) \to \varprojlim_{p \in K} D[A^{\otimes -1}](X(p))
			\]
			is an equivalence.

		\item\label{itm:PF stabilization/closed pf}
			If $f : X \to Y$ has $A$-lifts, and $f_* : D(X) \to D(Y)$ preserves colimits, and is a linear right adjoint of $f^*$, then $f_* : D[A^{\otimes -1}](X) \to D[A^{\otimes -1}](Y)$ also preserves colimits, is a linear right adjoint of $f^*$, and the square
			\[
				\begin{tikzcd}
					D(Y) \ar[d] \ar[r] & D(X) \ar[d] \\
					D[A^{\otimes -1}](Y) \ar[r] & D[A^{\otimes -1}](X)
				\end{tikzcd}
			\]
			is right adjointable.

	\end{enumerate} 
\end{prp}

Sometimes we would like to apply \Cref{prp:PF stabilization} in cases where we do not know that all quasi-admissible maps have $A$-lifts, or that for every $S \in \mathcal C$, there are enough symmetric objects in $A_S$. In these cases, \Cref{thm:PF sheafy stabilization} shows that we can still define a good notion of stabilization $D \to D[A^{\otimes -1}]$, provided that we are able to reduce to the previous case. In particular, we need that in some sense, $\mathcal C$ is ``generated'' by a subcategory in which these conditions hold.

\begin{prp} \label{prp:dense stabilization}
	In \Cref{setng:stabilization}, suppose that $\mathcal C' \subseteq \mathcal C$ is a full anodyne pullback subcontext such that every object of $\mathcal C$ has a quasi-admissible $D$-covering family by objects of $\mathcal C'$, and let $A'$ be the restriction of $A$ to $\mathcal C'$. Then a formal $\otimes$-inversion of $A'$ exists if and only if a formal $\otimes$-inversion of $A$ exists, and a morphism of pullback formalisms $D \to \bar D$ is a formal $\otimes$-inversion of $A$ if and only if its restriction to $\mathcal C'$ is a formal $\otimes$-inversion of $A$.
	\begin{proof}
		Write $D' \coloneqq D|_{\mathcal C'}$. By \Cref{cor:dense slice}, we know that the functor $\PF(\mathcal C)_{D/} \to \PF(\mathcal C')_{D'/}$ is an equivalence. Thus, a formal $\otimes$-inversion $\Sigma_{A'}^\infty : D' \to D'[(A')^{\otimes-1}]$ of $A'$ exists if and only if there is a morphism of pullback formalisms $D \to \bar D$ that restricts to $\Sigma_{A'}^\infty$ on $\mathcal C'$. Hence, it suffices to show that $D \to \bar D$ is a formal $\otimes$-inversion of $A$ if and only if its restriction $D' \to \bar D'$ is a formal $\otimes$-inversion of $A'$.

		Indeed, by \Cref{cor:dense slice}, we have that the vertical arrows in the following diagram are equivalences:
		\[
			\begin{tikzcd}
				\PF(\mathcal C)_{\bar D/} \ar[d] \ar[r] & \PF(\mathcal C)_{D/} \ar[d] \\
				\PF(\mathcal C')_{\bar D'/} \ar[r] & \PF(\mathcal C')_{D'/}
			\end{tikzcd}
		.\]

		It follows that the top arrow is fully faithful if and only if the bottom arrow is fully faithful, so we conclude by noting that by \Cref{lem:tensor invertibility is local}, for any morphism $D \to E$ we have that $D(S) \to E(S)$ sends $A_S$ to $\Pic E(S)$ for all $S \in \mathcal C'$ if and only if it does for all $S \in \mathcal C$.
	\end{proof}
\end{prp}

\begin{defn} \label{defn:A-good}
	In \Cref{setng:stabilization}, say an object $S \in \mathcal C$ is $A$-good if every quasi-admissible map to $S$ has $A$-lifts, and $A_S$ is symmetric.
\end{defn}

\begin{rmk} \label{rmk:A-good anodyne}
	In \Cref{setng:stabilization}, any object admitting a quasi-admissible map to an $A$-good object is itself $A$-good, so the full subcategory of $\mathcal C$ consisting of $A$-good objects is a full anodyne pullback subcontext of $\mathcal C$.
\end{rmk}

Note that \Cref{lem:A-lifts} is sometimes helpful for verifying the hypotheses of this \namecref{thm:PF sheafy stabilization}.
\begin{thm} \label{thm:PF sheafy stabilization}
	In \Cref{setng:stabilization}, suppose that every object of $\mathcal C$ has a quasi-admissible $D$-covering family by $A$-good objects.
	Then there is a formal $\otimes$-inversion $\Sigma_A^\infty : D \to D[A^{\otimes -1}]$ of $A$ satisfying the following:
	\begin{enumerate}

		\item\label{itm:PF sheafy stabilization/descent}
			Descent: Let $K$ be a simplicial set, and let $X : K^\triangleright \to \mathcal C$ be a diagram that sends every edge to a quasi-admissible map, and such that for any quasi-admissible map $X' \to X(\infty)$, if $X'$ is $A$-good, then $D(X \times_{X(\infty)} X')^\op$ is limiting. Then $D[A^{\otimes -1}] X^\op$ is a limiting diagram.

		\item Let $S$ be an $A$-good object, and let $A'$ be any small collection of objects of $D(S)$ that $\otimes$-generate $A_S$, and that are $\otimes$-generated by $A_S$.
			\begin{enumerate}

				\item\label{itm:PF sheafy stabilization/ptwise}
					Objectwise description: The functor $(D \to D[A^{\otimes -1}])(S)$ is the functor $D(S) \to D(S)[A_B^{\otimes -1}]$ given by formally adjoining $\otimes$-inverses of elements of $A'$.

				\item\label{itm:PF sheafy stabilization/generation}
					Generation: The category $D[A^{\otimes -1}](S)$ is generated under filtered colimits by objects of the form $(\Sigma_A^\infty \alpha)^{\otimes -1} \otimes \Sigma_A^\infty M$ for $\alpha$ a tensor product of objects in $A'$, and $M \in D(S)$.

			\end{enumerate}
			
	\end{enumerate} 

	\begin{proof}
		Let $\mathcal B \subseteq \mathcal C$ be the full subcategory of $A$-good objects, so $\mathcal B$ is a full anodyne pullback subcontext by \Cref{rmk:A-good anodyne}. Thus, \Cref{prp:PF stabilization} and \Cref{prp:dense stabilization} show that there is a formal $\otimes$-inversion $\Sigma_A^\infty : D \to D[A^{\otimes -1}]$ of $A$ that restricts to the formal $\otimes$-inversion of $A|_{\mathcal B}$ given by \Cref{prp:PF stabilization}.

		We now address the listed statements:
		\begin{description}

			\item[\ref{itm:PF sheafy stabilization/descent}]
				Let $X : K^\triangleright \to \mathcal C$ be a diagram that sends all edges to quasi-admissible maps, and such that for any quasi-admissible map $X' \to X(\infty)$, the map
				\begin{equation} \label{eqn:descent hyp}
					D(X') \to \varprojlim_{p \in K} D(X' \times_{X(\infty)} X(p))
				\end{equation}
				is an equivalence. To show \cref{itm:PF sheafy stabilization/descent}, it suffices to show that
				\begin{equation} \label{eqn:descent conc}
					D[A^{\otimes -1}](X(\infty)) \to \varprojlim_{p \in K} D[A^{\otimes -1}](X(p))
				\end{equation}
				is an equivalence.

				Since $X(\infty)$ has a quasi-admissible $D$-covering family by objects of $\mathcal B$, and $\mathcal B$ is a full anodyne pullback subcontext of $\mathcal C$, by considering \v{C}ech nerves, there is a small diagram $X' : K' \to \mathcal C_{X(\infty)}$ that sends all vertices to quasi-admissible maps to $X(\infty)$ from objects of $\mathcal B$, and such that
				\[
					\varinjlim_{a' \in K'} \yo X'(a') \to \yo(X(\infty))
				\]
				is a local equivalence for the $D$-covering topology.

				Consider the commutative square:
				\[
					\begin{tikzcd}
						D[A^{\otimes -1}](X(\infty)) \ar[d] \ar[r] & \displaystyle\varprojlim_{a \in K} D[A^{\otimes -1}](X(a)) \ar[d] \\
						\displaystyle\varprojlim_{a' \in K'} D[A^{\otimes -1}](X'(a')) \ar[r] & \displaystyle\varprojlim_{\substack{a \in K \\ a' \in K'}} D[A^{\otimes -1}](X(a) \times_{X(\infty)} X'(a'))
					\end{tikzcd}
				.\]
				Since $D[A^{\otimes -1}]$ has descent for the $D$-covering topology, the vertical arrows are equivalences. For any $a' \in K'$, since $X'(a') \in \mathcal B$, we know that for any $a \in K$, the object $X(a) \times_{X(\infty)} X'(a)$ admits a quasi-admissible map to an object $X'(a')$ of $\mathcal B$, so it is in $\mathcal B$. Therefore, for any $a' \in K'$, $X \times_{X(\infty)} X'(a')$ is a diagram taking values in $\mathcal B$, and that sends all edges to quasi-admissible maps, so since \eqref{eqn:descent hyp} is an equivalence (since $X'(a') \to X(\infty)$ is quasi-admissible and $X'(a')$ is $A$-good), it follows from \cref{itm:PF stabilization/descent} of \Cref{prp:PF stabilization} that the bottom horizontal arrow is an equivalence. Thus, the top horizontal arrow is an equivalence, which concludes the proof of \cref{itm:PF sheafy stabilization/descent}.

			\item[\ref{itm:PF sheafy stabilization/ptwise} and \ref{itm:PF sheafy stabilization/generation}] follow immediately from \cref{itm:PF stabilization/ptwise,itm:PF stabilization/generation} of \Cref{prp:PF stabilization}, where we note that formally adjoining $\otimes$-inverses of objects in $A_S$ is equivalent to formally adjoining $\otimes$-inverses of objects of $A'$.

		\end{description}
		
	\end{proof}
\end{thm}

\begin{rmk} \label{rmk:stabilization compatible with restriction}
	In the setting of \Cref{thm:PF sheafy stabilization}, let $F : \mathcal C' \to \mathcal C$ be a morphism of pullback contexts, and write $A' = \{A'_{S'}\}_{S' \in \mathcal C'}$ for the family given by $A'_{S'} \coloneqq A_{F(S')}$ for $S' \in \mathcal C'$. If $\mathcal C'_{S'} \to \mathcal C_{F(S')}$ is essentially surjective for all $S' \in \mathcal C'$, then $F^* D, A'$ satisfy the hypotheses of \Cref{thm:PF sheafy stabilization}, and the morphism $F^*(\Sigma_A^\infty)$ factors as
	\[
		F^* D \xrightarrow{\Sigma_{A'}^\infty} (F^* D)[(A')^{\otimes -1}] \to F^*(D[A^{\otimes - 1}])
	,\]
	where the second map is an equivalence.

	In particular, if $F$ is the inclusion of an anodyne pullback subcontext, then $D|_{\mathcal C'}[(A')^{\otimes -1}]$ exists and is equivalent to $D[A^{\otimes -1}]|_{\mathcal C'}$.
	\begin{proof}
		We have already noted in \Cref{rmk:pull back stab setng} that $F^* D, A'$ is an instance of \Cref{setng:stabilization}, and that a map $f'$ in $\mathcal C'$ has $A'$-lifts if and only if $F(f')$ has $A$-lifts.

		Now, let $S' \in \mathcal C$ be an object such that $F(S')$ is $A$-good. Since every quasi-admissible map to $S'$ is sent to a quasi-admissible map to $F(S')$, which has $A$-lifts since $F(S')$ is $A$-good, we have that every quasi-admissible map to $S'$ has $A$-lifts. Similarly, since $A_{F(S')}$ is symmetric, we conclude that $S'$ is $A'$-good.

		Now, let $S' \in \mathcal C$. Then $F(S')$ has a quasi-admissible $D$-covering family by $A$-good objects, and $\mathcal C'_{S'} \to \mathcal C_{F(S')}$ is essentially surjective, we have that $S'$ has a quasi-admissible $F^* D$-covering family by $A'$-good objects.
		% PERF: qadm $F^*D$-covering of $S'$ = qadm $D$-covering of $F(S')$

		Thus we obtain a morphism $\Sigma_{A'}^\infty : F^* D \to (F^* D)[(A')^{\otimes -1}]$ as in \Cref{thm:PF sheafy stabilization}. It follows that $F^*(\Sigma_A^\infty)$ extends through $\Sigma_{A'}^\infty$ by a morphism $(F^* D)[(A')^{\otimes -1}] \to F^*(D[A^{\otimes -1}])$ that is an equivalence on $A'$-good objects by \cref{itm:PF sheafy stabilization/ptwise} of \Cref{thm:PF sheafy stabilization}, so it is an equivalence by \Cref{cor:dense slice}.
	\end{proof}
\end{rmk}

\begin{proof}[Proof of \Cref{prp:PF stabilization}]
	First we will construct $\Sigma_A^\infty$ as a transformation of functors $\mathcal C^\op \to \CAlg(\PrL)$. Indeed, by the discussion at the beginning of \cite[Section 6.1]{sixopsequiv}, and \cite[Proposition 2.9]{robalo}, for each $S \in \mathcal C$, there is a symmetric monoidal colimit-preserving functor $\Sigma_{A_S}^\infty : D(S) \to D(S)[A_S^{\otimes -1}]$ such that any morphism $F: D(S) \to \mathcal A \in \CAlg(\PrL)$ in $\CAlg(\PrL)$ extends along $\Sigma_{A_S}^\infty$ if and only if $F(a) \in \Pic(\mathcal A)$ for all $a \in A_S$, in which case the space of extensions is contractible.

	Thus, by \cite[Proposition A.1]{objwise-mono}, there is a morphism $\Sigma_A^\infty : D \to D[A^{\otimes -1}]$ in $\Fun(\mathcal C^\op, \CAlg(\PrL))$ that evaluates to $\Sigma_{A_S}^\infty$ for each $S \in \mathcal C$, as required in \cref{itm:PF stabilization/ptwise}.

	By \cite[Proposition A.11]{objwise-mono}, it follows that a morphism $D \to E$ in $\Fun(\mathcal C^\op, \CAlg(\PrL))$ extends through $\Sigma_A^\infty$ if and only if for each $S \in \mathcal C$, the functor $D(S) \to E(S)$ sends $A_S$ to $\Pic(E(S))$, and in this case the space of extensions is contractible.

	Let $f : X \to Y$ be a map that has $A$-lifts. By the left adjointability of the square \eqref{eqn:tensor loc square} of \Cref{S:stabilization}, we have that $f^* : D[A^{\otimes -1}](Y) \to D[A^{\otimes -1}](X)$ is the map 
	\[
		D(Y)[A_Y^{\otimes -1}] \to D(X)[A_Y^{\otimes -1}]
	\]
	in $\Mod_{D(Y)}(\PrL)$ induced from $f^* : D(Y) \to D(X)$ by formally inverting the action of objects in $A_Y$.

	If $A_Y$ is symmetric, we may apply \Cref{prp:stabilization}: if $f$ is quasi-admissible, since $f^* : D(Y) \to D(X)$ admits a $D(Y)$-linear left adjoint $f_\sharp$, \cref{itm:stabilization/linear} says that $f^* : D[A^{\otimes -1}](Y) \to D[A^{\otimes -1}](X)$ has a $D[A^{\otimes -1}](Y)$-linear left adjoint, and the square
	\[
		\begin{tikzcd}
			D(Y) \ar[d] \ar[r] & D(X) \ar[d] \\
			D[A^{\otimes -1}](Y) \ar[r] & D[A^{\otimes -1}](X)
		\end{tikzcd}
	\]
	is left-adjointable.

	Thus $\Sigma_A^\infty$ is left adjointable at $f$, and $D[A^{\otimes -1}]$ sends $f$ to a functor that has a linear left adjoint.

	Let us show that $\Sigma_A^\infty$ satisfies \Cref{itm:PF stabilization/adjoint}: if $f : X \to Y$ has $A$-lifts, and
	\[
		\begin{tikzcd}
			D(Y) \ar[d] \ar[r] & D(X) \ar[d] \\
			E(Y) \ar[r] & E(X)
		\end{tikzcd}
	\]
	is left adjointable, by \cref{itm:stabilization/adjointable} of \Cref{prp:stabilization}, we have that
	\[
		\begin{tikzcd}
			D(Y)[A_Y^{\otimes -1}] \ar[d] \ar[r] & D(X)[A_Y^{\otimes -1}] \ar[d] \\
			E(Y)[A_Y^{\otimes -1}] \ar[r] & E(X)[A_Y^{\otimes -1}]
		\end{tikzcd}
	\]
	is left adjointable, but $A_Y$ already acts invertibly on $E(Y)$ and $E(X)$, so by our previous observations, we find that this left adjointable square is equivalent to
	\[
		\begin{tikzcd}
			D[A^{\otimes -1}](Y) \ar[d] \ar[r] & D[A^{\otimes -1}](X) \ar[d] \\
			E(Y) \ar[r] & E(X)
		\end{tikzcd}
	,\]
	which shows that $D[A^{\otimes -1}] \to E$ is left adjointable at $f$, as desired.

	Thus, to see that $\Sigma_A^\infty$ is a formal $\otimes$-inversion of $A$, it only remains to show that $D[A^{\otimes -1}]$ has quasi-admissible base change: if $f : X \to Y$ is quasi-admissible, and
	\[
		\begin{tikzcd}
			X' \ar[d] \ar[r] & Y' \ar[d] \\
			X \ar[r, "f"'] & Y
		\end{tikzcd}
	\]
	is a Cartesian square in $\mathcal C$, then since
	\[
		\begin{tikzcd}
			D(Y) \ar[d] \ar[r] & D(X) \ar[d] \\
			D(Y') \ar[r] & D(X')
		\end{tikzcd}
	\]
	is left adjointable, \cref{itm:stabilization/adjointable} of \Cref{prp:stabilization} says that
	\[
		\begin{tikzcd}
			D(Y)[A_Y^{\otimes -1}] \ar[d] \ar[r] & D(X)[A_Y^{\otimes -1}] \ar[d] \\
			D(Y')[A_Y^{\otimes -1}] \ar[r] & D(X')[A_Y^{\otimes -1}]
		\end{tikzcd}
	\]
	is left-adjointable. By \cref{itm:stabilization/linear} of the same result, the square
	\[
		\begin{tikzcd}
			D(Y')[A_Y^{\otimes -1}] \ar[d] \ar[r] & D(X')[A_Y^{\otimes -1}] \ar[d] \\
			D(Y')[A_{Y'}^{\otimes -1}] \ar[r] & D(X')[A_{Y'}^{\otimes -1}]
		\end{tikzcd}
	\]
	is left adjointable, so the composite square
	\[
		\begin{tikzcd}
			D(Y)[A_Y^{\otimes -1}] \ar[d] \ar[r] & D(X)[A_Y^{\otimes -1}] \ar[d] \\
			D(Y')[A_{Y'}^{\otimes -1}] \ar[r] & D(X')[A_{Y'}^{\otimes -1}]
		\end{tikzcd}
	\]
	is left adjointable by \cite[Lemma F.6 (3)]{TwAmb}, but we have already observed that this square is equivalent to
	\[
		\begin{tikzcd}
			D[A^{\otimes -1}](Y) \ar[d] \ar[r] & D[A^{\otimes -1}](X) \ar[d] \\
			D[A^{\otimes -1}](Y') \ar[r] & D[A^{\otimes -1}](X')
		\end{tikzcd}
	,\]
	which concludes the proof that $D[A^{\otimes -1}]$ has quasi-admissible base change, so that $\Sigma_A^\infty$ is a formal $\otimes$-inversion of $A$.

	\Cref{itm:PF stabilization/generation} follows from \cref{itm:stabilization/generation} of \Cref{prp:stabilization}.

	Next, we will show \cref{itm:PF stabilization/descent}. Since $X$ sends every edge $p \to \infty$ to a map with $A$-lifts, we have already observed that
	\[
		D[A^{\otimes -1}] \circ X^\op
	\]
	is
	\[
		(D \circ X^\op)[A_{X(\infty)}^{\otimes -1}]
	.\]
	Thus, since we have that $X(p \to p')$ is quasi-admissible for $p \to p'$ in $K$, \cref{itm:stabilization/limits} of \Cref{prp:stabilization} says that if $D \circ X^\op$ is limiting, then so is $D[A^{\otimes -1}] \circ X^\op$.

	Finally, to see \ref{itm:PF stabilization/closed pf}, first note that as before, $f^* : D[A^{\otimes -1}](Y) \to D[A^{\otimes -1}](X)$ is the map $f^*[A_Y^{\otimes -1}] : D(Y)[A_Y^{\otimes -1}] \to D(X)[A_Y^{\otimes -1}]$ in $\Mod_{D(Y)[A_Y^{\otimes -1}]}(\PrL)$. Since $f_* : D(X) \to D(Y)$ is a map in $\Mod_{D(Y)}(\PrL)$ that has a $D(Y)$-linear left adjoint ($f^*$), part \ref{itm:stabilization/linear} of \Cref{prp:stabilization} says that the left adjoint of the induced map $f_*[A_Y^{\otimes -1}] : D(X)[A_Y^{\otimes -1}] \to D(Y)[A_Y^{\otimes -1}]$ in $\Mod_{D(Y)[A_Y^{\otimes -1}]}(\PrL)$ is $f^*[A_Y^{\otimes -1}] : D(Y)[A_Y^{\otimes -1}] \to D(X)[A_Y^{\otimes -1}]$, that
	\begin{equation} \label{eqn:radj stab sq}
		\begin{tikzcd}
			D(X) \ar[d] \ar[r] & D(Y) \ar[d] \\
			D(X)[A_Y^{\otimes -1}] \ar[r] & D(Y)[A_Y^{\otimes -1}]
		\end{tikzcd}
	\end{equation}
	is left adjointable, that $f^*[A_Y{^\otimes -1}]$ is a $D(Y)[A_Y^{\otimes -1}]$-linear left adjoint of the functor $f_*[A_Y^{\otimes -1}]$, and that the equivalence $f^* \Sigma_A^\infty \to \Sigma_A^\infty f^*$ making
	\[
		\begin{tikzcd}
			D(Y) \ar[d] \ar[r] & D(X) \ar[d] \\
			D[A^{\otimes -1}](Y) \ar[r] & D[A^{\otimes -1}](X)
		\end{tikzcd}
	\]
	commute is the left base change transformation of \eqref{eqn:radj stab sq}. It follows from \cite[Lemma F.4]{TwAmb} or \cite[Corollary 4.7.4.18(3)]{ha} that the latter square is right adjointable.
\end{proof}

\subsection{Pointing} \label{S:pointing pf}

Often before stabilizing a pullback formalism $D$, we will want to consider a ``pointed'' version $D_\bullet$ of $D$, in which $D_\bullet(S)$ is the category of pointed objects of $D(S)$.

First we recall some definitions:
\begin{defn}
	Following \cite[Definition 1.1.1.1]{ha}, say a category $\mathcal C$ is \emph{pointed} if it has a \emph{zero object}, \ie, an object which is both initial and terminal.

	Given a pullback context $\mathcal C$, say a pullback formalism $D$ on $\mathcal C$ is \emph{pointed} if it takes values in pointed categories, \ie, if $E(S)$ has a zero object for all $S \in \mathcal C$. Write $\PF_\bullet(\mathcal C) \subseteq \PF_\bullet(\mathcal C)$ for the full subcategory of pointed pullback formalisms on $\mathcal C$.
\end{defn}

The following result describes how to freely produce ``pullback formalisms of pointed objects''. This may be seen as a description of the left adjoint in \cite[Proposition 6.5]{UnivFF}.
\begin{prp} \label{prp:pointing}
	If $\mathcal C$ is a pullback context, the inclusion $\PF_\bullet(\mathcal C) \to \PF(\mathcal C)$ has a left adjoint. For $D \in \PF(\mathcal C)$, we write $D \to D_\bullet$ for the unit of the adjunction, which has the following properties:
	\begin{enumerate}

		\item \label{itm:pointing/generation}
			For each $S \in \mathcal C$, the functor $D(S) \to D_\bullet(S)$ is equivalent to the ``pointing'' functor $D(S) \to D(S)_{\pt/}$ given by $M \mapsto \pt \coprod M$, and $D_\bullet(S)$ is generated under simplicial colimits by the essential image of this functor.

		\item \label{itm:pointing/limits}
			If $F : K^\triangleright \to \mathcal C$ is a diagram such that $D \circ F^\op$ is limiting, and for all maps $f$ in $K^\triangleright$, $F(f)^*$ preserves terminal objects, then $D_\bullet \circ F^\op$ is also limiting. Note that this holds if $F$ sends all morphisms to quasi-admissible maps, or if $F$ sends all objects to objects $X \in \mathcal C$ such that the monoidal unit of $F(X)$ is terminal.

		\item \label{itm:pointing/colimits}
			If $K$ is a simplicial set, and $f : X \to Y$ is a map in $\mathcal C$ such that $f_* : D(X) \to D(Y)$ preserves $K^\triangleleft$-indexed colimits, then $f_* : D_\bullet(X) \to D_\bullet(Y)$ preserves $K$-indexed colimits. In particular, if $f_* : D(X) \to D(Y)$ preserves weakly contractible colimits, then $f_* : D_\bullet(X) \to D_\bullet(Y)$ preserves colimits.

	\end{enumerate}
\end{prp}

Using \Cref{prp:pointing}, we can also prove the following corollary of \Cref{thm:PF sheafy stabilization}. Note that in the statement of this \namecref{thm:ptd PF sheafy stabilization}, we use the fact that for any category $\mathcal C$ that has a terminal object and pushouts, and any map $a : a_0 \to a_1$ in $\mathcal C$, the cofibre $\cofib(a)$ naturally has the structure of a pointed object of $\mathcal C$ by considering it as the cobase change of $a$ along $a_0 \to \pt$.
\begin{thm} \label{thm:ptd PF sheafy stabilization}
	Let $\mathcal C$ be a pullback context, let $D$ be a pullback formalism on $\mathcal C$, and let $A = \{A_S\}_{S \in \mathcal C}$ be a family such that for each $S \in \mathcal C$, $A_S$ is a small collection of morphisms in $D(S)$, and for any map $f : X \to Y$ in $\mathcal C$, and $b \in A_Y$, we have that $\cofib(f^*(b)) \in D_\bullet(X)$ is $\otimes$-generated by objects of the form $\cofib(a)$ for $a \in A_X$.

	Say an object $S \in \mathcal C$ is \emph{$A_\bullet$-good} if
	\begin{enumerate}

		% NOTE: use strategy of \cite[Proposition 4.3.9]{TwAmb} to establish this
		\item for any quasi-admissible map $f : X \to S$, and $a \in A_X$, the object $\cofib(a) \in D_\bullet(X)$ is $\otimes$-generated by objects of the form $f^* \cofib(b)$ for $b \in A_S$, and

		\item the collection of objects $(A_\bullet)_S = \{\cofib(a) \in D_\bullet(S) : a \in A_S\}$ in $D_\bullet(S)$ is symmetric in the sense of \Cref{defn:symmetric}.

	\end{enumerate}

	Suppose that every object of $\mathcal C$ has a quasi-admissible $D$-covering family by $A_\bullet$-good objects. If we define $A_\bullet$ to be the family $\{(A_\bullet)_S\}_{S \in \mathcal C}$, then the hypotheses of \Cref{thm:PF sheafy stabilization} are satisfied for $A_\bullet$, which provides the morphism $D_\bullet \to D_\bullet[A_\bullet^{\otimes -1}]$, and if we define $\Sigma_{A,+}^\infty : D \to D[A^{\wedge -1}]$ to be the composite
	\[
		D \to D_\bullet \to D_\bullet[A_\bullet^{\otimes -1}]
	,\]
	then we have the following:
	\begin{enumerate}

		\item\label{itm:ptd PF sheafy stabilization/univ}
			Universal property:
			The functor $\PF_\bullet(\mathcal C)_{D[A^{\wedge -1}]/} \to \PF(\mathcal C)_{D/}$ is fully faithful with essential image given by those morphisms $\phi : D \to E$ such that $E$ is a pointed pullback formalism, and the following equivalent conditions are satisfied:
			\begin{enumerate}

				\item For every $S \in \mathcal C$, $D(S) \to E(S)$ sends $(A_\bullet)_S$ to $\Pic E(S)$.

				\item For every $S \in \mathcal C$, there is a quasi-admissible $D$-covering family by objects $S_0$ such that there is a map $f : S_0 \to S_1$ that has $A_\bullet$-lifts, and for any $a \in A_{S_1}$, $f^* \cofib(\phi(a)) \in \Pic E_0$.

			\end{enumerate}

		\item\label{itm:ptd PF sheafy stabilization/descent}
			Descent: Let $K$ be a simplicial set, and let $X : K^\triangleright \to \mathcal C$ be a diagram that sends every edge to a quasi-admissible map, and such that for any quasi-admissible map $X' \to X(\infty)$, if $X'$ is $A_\bullet$-good, then $D(X \times_{X(\infty)} X')^\op$ is limiting. Then $D[A^{\otimes -1}] X^\op$ is a limiting diagram.

		\item\label{itm:ptd PF sheafy stabilization/generation}
			Generation: If $S \in \mathcal C$ is $A_\bullet$-good, and $A'$ is a small symmetric collection of objects in $D_\bullet(S)$ that $\otimes$-generates and is $\otimes$-generated by $(A_\bullet)_S$, then the category $D[A^{\wedge -1}](S)$ is generated under sifted colimits by objects of the form
			\[
				(\Sigma_{(A_\bullet)_S}^\infty \alpha)^{\otimes -1} \otimes \Sigma_{A_S, +}^\infty M
			\]
			for $M \in D(S)$, and $\alpha$ a tensor product of objects in $A'$.

	\end{enumerate} 
\end{thm}

In fact, the stabilizations produced by \Cref{thm:ptd PF sheafy stabilization} often produce pullback formalisms that take values in stable categories:
\begin{lem} \label{lem:auto stability}
	Let $\mathcal C$ be a pullback context, and let $D$ be a pointed pullback formalism on $\mathcal C$. For any $S \in \mathcal C$, if $1/\cls{X}$ is $\otimes$-invertible in $D(S)$ for some $X \in \mathcal C_S$ such that $X \to S$ admits a section, then $D(S)$ is a stable category.
	\begin{proof}
		By \Cref{lem:monoidal crit for stability}, it suffices to show that $1/\cls{X}$ is the suspension of some object.

		Since $D$ takes values in pointed categories, \Cref{prp:pointing} and \Cref{thm:Huniv is univ} say that there is a morphism of pullback formalisms $\phi : H^\univ_\bullet \to D$, so $\phi(1/\cls{X}) \simeq 1/\cls{X}$. Since $H^\univ_\bullet(S) \to D(S)$ preserves colimits and zero objects, it preserves suspensions, so it suffices to show that $1/\cls{X}$ is the suspension of some object when $D = H^\univ_\bullet$.

		The square
		\[
			\begin{tikzcd}
				\cls{X} \ar[d] \ar[r] & 1 \ar[d] \\
				0 \ar[r] & 1/\cls{X}
			\end{tikzcd}
		\]
		in $H^\univ_\bullet(S) \simeq \Psh(\mathcal C_S)_{\pt/}$ forgets to the bottom square in the following commutative diagram in $\Psh(\mathcal C_S)$:
		\[
			\begin{tikzcd}
				\yo(X) \ar[d] \ar[r] & \pt \ar[d] \\
				\yo(X) \sqcup \pt \ar[d] \ar[r] & \pt \sqcup \pt \ar[d] \\
				\pt \ar[r] & 1/\cls{X}
			\end{tikzcd}
		\]
		in which all the squares are coCartesian. If $\yo(X)$ has a global section, then we have some $\bar X \in \Psh(\mathcal C_S)_{\pt/}$ lying over $\yo(X) \in \Psh(\mathcal C_S)$, so the outer square lifts to a square in $\Psh(\mathcal C_S)_{\pt/} \simeq H^\univ_\bullet(S)$:
		\[
			\begin{tikzcd}
				\bar X \ar[d] \ar[r] & \pt \ar[d] \\
				\pt \ar[r] & 1/\cls{X}
			\end{tikzcd}
		,\]
		which is coCartesian, since the functor $\Psh(\mathcal C_S)_{\pt/} \to \Psh(\mathcal C_S)$ is conservative and preserves pushouts.

		Thus,
		\[
			1/\cls{X} \simeq \Sigma \bar X
		.\]
	\end{proof}

\end{lem}

\begin{rmk} \label{rmk:describe pointing}
	For any category $\mathcal A$ with a terminal object $\pt$ and finite coproducts, every object $\pt \xrightarrow{x} X$ of $\mathcal A_{\pt/}$ is a cofibre $\cofib(\pt \coprod x)$ of the image of a map (namely $x$) in $\mathcal A$. We can use this observation to give concrete interpretations of some of the structure of $\mathcal A_{\pt/}$, given that $\mathcal A$ is a symmetric monoidal presentable category.

	\begin{itemize}

		\item If $F : \mathcal A \to \mathcal B$ is a functor, and $\mathcal B$ has a zero object, then for any extension $\bar{F} : \mathcal A_{\pt/} \to \mathcal B$ of $F$ that preserves finite colimits, we have that $\bar F$ sends the object $\pt \xrightarrow{x} X$ of $\mathcal A_{\pt/}$ to $\cofib(F(x))$.
			% NOTE: Indeed, we have a coCartesian square
			% \[
			% 	\begin{tikzcd}
			% 		\pt \coprod \pt \ar[d] \ar[r, "\pt \coprod x"] & \pt \coprod X \ar[d] \\
			% 		\pt \ar[r, "x"'] & X
			% 	\end{tikzcd}
			% \]
			% in $\mathcal A_{\pt/}$, so since $\bar F$ extends $F$ along $\pt \coprod - : \mathcal A \to \mathcal A_{\pt/}$, and preserves colimits, $\bar F$ sends the above coCartesian square to the following coCartesian square in $\mathcal B$:
			% \[
			% 	\begin{tikzcd}
			% 		F(\pt) \ar[d] \ar[r, "F(x)"] & F(X) \ar[d] \\
			% 		0 \ar[r] & \bar F(x)
			% 	\end{tikzcd}
			% .\]

		\item Given a symmetric monoidal presentable category $\mathcal A \in \CAlg(\PrL)$, the monoidal product $\otimes : \mathcal A_{\pt/} \times \mathcal A_{\pt/} \to \mathcal A_{\pt/}$ is given by a ``smash product'' built using the monoidal product of $\mathcal A$.
			% NOTE:
			% Details: use $\mathcal A \to \mathcal A_{\pt/}$ is monoidal, and that $\otimes$ preserves colimits in each variable
			% first note $\pt \otimes -$ is the constant functor to the zero object
			% next note that any ptd object $X$ is $\cofib(I \to X_+)$, where $I \to X_+$ is the map $\pt \coprod x$ if $x : \pt \to X$ is the pointing
			% ...
				% is determined by the conditions that it preserves colimits in each parameter, and that $\mathcal A \to \mathcal A_{\pt/}$ is monoidal. By realizing every object $\pt \xrightarrow{x} X$ of $\mathcal A_{\pt/}$ as a cofibre $\cofib(\pt \coprod x)$, we can deduce that the monoidal product on $\mathcal A_{\pt/}$ is a sort of ``smash product'' built on the monoidal product of $\mathcal A$.
		
	\end{itemize}
\end{rmk}

\begin{proof}[Proof of \Cref{thm:ptd PF sheafy stabilization}]
	Note that by \cref{itm:pointing/limits} of \Cref{prp:pointing}, or \Cref{thm:D-topology}, $D_\bullet$ has descent for the $D$-covering topology. Thus, By \Cref{rmk:describe pointing}, we may apply \Cref{thm:PF sheafy stabilization} to $D_\bullet$ and the family $\{(A_\bullet)_S\}_{S \in \mathcal C}$.

	\begin{enumerate}

		\item \Cref{itm:ptd PF sheafy stabilization/univ} follows immediately from \Cref{prp:pointing} and the fact that $\Sigma_{A_\bullet}^\infty$ is a formal $\otimes$-inversion of $A_\bullet$ by \Cref{thm:PF sheafy stabilization}.

		\item \Cref{itm:ptd PF sheafy stabilization/descent} follows immediately from \cref{itm:pointing/limits} of \Cref{prp:pointing} and \cref{itm:PF sheafy stabilization/descent} of \Cref{thm:PF sheafy stabilization}.

		\item \Cref{itm:ptd PF sheafy stabilization/generation} follows immediately from \cref{itm:pointing/generation} of \Cref{prp:pointing} and \cref{itm:PF sheafy stabilization/generation} of \Cref{thm:PF sheafy stabilization}.

	\end{enumerate}
\end{proof}

\begin{proof}[Proof of \Cref{prp:pointing}]
	We know that for any $\mathcal A \in \CAlg(\PrL)$, there is a map $\mathcal A \to \mathcal A_{\pt/}$ in $\CAlg(\PrL)$ that lies over the functor $M \mapsto \pt \coprod M$ in $\widehat{\Cat}$, and such that for any map $\mathcal A \to \mathcal B$ in $\CAlg(\PrL)$, if $\mathcal B$ has a zero object, then there is a contractible space of extensions $\mathcal A_{\pt/} \to \mathcal B$ in $\CAlg(\PrL)$.
	% PERF: find reference for this

	Note that if we consider the wide subcategory of $\CAlg(\PrL)$ where the morphisms send initial objects to terminal objects, then we find that $\mathcal A \to \mathcal A_{\pt/}$ defines a monomorphism in the opposite of this category\footnote{Indeed, if $\mathcal B$ has no zero object, then there are no colimit-preserving functors $\mathcal A \to \mathcal B$ that send the initial object to a terminal object, and otherwise, all colimit-preserving functors send the initial object to the terminal object}, so by \cite[Proposition A.1]{objwise-mono}, it follows that there is a map $D \to D_\bullet$ in $\Psh_{\CAlg(\PrL)}(\mathcal C)$ such that for each $S \in \mathcal C$, the functor $D(S) \to D_\bullet(S)$ is equivalent to $D(S) \to D(S)_{\pt/}$.

	If $f^* : D(Y) \to D(X)$ preserves terminal objects, $f^* : D_\bullet(Y) \to D_\bullet(X)$ is simply the induced functor $D(Y)_{\pt/} \to D(X)_{\pt/}$, so that
	\[
		\begin{tikzcd}
			D(Y) \ar[d] \ar[r] & D(Y)_{\pt/} \ar[d] \\
			D(X) \ar[r] & D(X)_{\pt/}
		\end{tikzcd}
	\]
	is right adjointable. Since the right adjoints preserve and reflect limits, it follows that if $D(Y) \to D(X)$ preserves limits, then so does $D_\bullet(Y) \to D_\bullet(X)$, so if $f$ is quasi-admissible, then $f^* : D_\bullet(Y) \to D_\bullet(X)$ has a left adjoint, and the square
	\[
		\begin{tikzcd}
			D(Y) \ar[d] \ar[r] & D(X) \ar[d] \\
			D_\bullet(Y) \ar[r] & D_\bullet(X)
		\end{tikzcd}
	\]
	is left adjointable since its transpose is right adjointable. This shows that $D \to D_\bullet$ respects quasi-admissibility. 

	Thus, Proposition \ref{prp:map from PF is to PF} says that $D_\bullet$ is a pullback formalism. Since we can produce $D \to D_\bullet$ for any $D \in \PF(\mathcal C)$, in order to show that $D \to D_\bullet$ is the unit for a left adjoint of the inclusion $\PF_\bullet(\mathcal C) \to \PF(\mathcal C)$, by \cite[Proposition 5.2.7.8]{htt}, it suffices to show that for any pointed pullback formalism $E$ on $\mathcal C$, the map
	\begin{equation} \label{eqn:pointing}
		\PF(\mathcal C)(D_\bullet, E) \to \PF(\mathcal C)(D, E)
	\end{equation}
	is an equivalence.

	By our construction of $D_\bullet$ and \cite[Proposition A.11]{objwise-mono}, we know that
	\[
		\Psh_{\CAlg(\PrL)}(\mathcal C)(D_\bullet, E) \to \Psh_{\CAlg(\PrL)}(\mathcal C)(D, E)
	\]
	is an equivalence of spaces, so in order to show that \eqref{eqn:pointing} is an equivalence, it suffices to show that
	\[
		\begin{tikzcd}
			\PF(\mathcal C)(D_\bullet, E) \ar[d] \ar[r] & \PF(\mathcal C)(D, E) \ar[d] \\
			\Psh_{\CAlg(\PrL)}(\mathcal C)(D_\bullet, E) \ar[r] & \Psh_{\CAlg(\PrL)}(\mathcal C)(D, E)
		\end{tikzcd}
	\]
	is Cartesian, but this follows from Proposition \ref{prp:map from PF is to PF}.

	Now we will address the listed statements:
	\begin{enumerate}

		\item For any $S \in \mathcal C$, the functor $D(S) \to D_\bullet(S)$ is equivalent to $D(S) \to D(S)_{\pt/}$, so the image of $D(S)$ generates $D_\bullet(S)$ under simplicial colimits.\footnote{This follows from \cite[Proposition 4.8.2.11]{ha}, which implies that $D(S) \to D(S)_{\pt/}$ is the left adjoint of a monadic adjunction.}

		\item Write $\PrL_*$ for the wide subcategory of $\PrL$ whose morphisms are those that preserve terminal objects. Then the inclusion $\PrL_* \to \PrL$ preserves and reflects limits, and the limit-preserving (see \cite[Proposition 5.5.3.13]{htt}) inclusion $\PrL \to \widehat{\Cat}$ restricts on $\PrL_*$ to a functor that lifts through $\widehat{\Cat}_{\Delta^0/} \to \widehat{\Cat}$, so we have a commutative square
			\[
				\begin{tikzcd}
					\PrL_* \ar[d] \ar[r] & \PrL \ar[d] \\
					\widehat{\Cat}_{\Delta^0/} \ar[r] & \widehat{\Cat}
				\end{tikzcd}
			\]
			of functors all of which preserve limits. Now, by our hypothesis, $D \circ F^\op$ can be seen as a functor that lands in $\PrL_*$, in which case $D_\bullet \circ F^\op$ is given by simply postcomposing $D \circ F^\op : (K^\op)^\triangleleft \to \widehat{\Cat}_{\Delta^0/}$ by the right adjoint of the functor $\Delta^0 \star - : \widehat{\Cat} \to \widehat{\Cat}_{\Delta^0/}$.\footnote{The fact that this is a right adjoint follows from \cite[Proposition 1.2.9.2]{htt} and \cite[Corollary 2.1.2.2]{htt}. Indeed, the first result gives an adjunction between model categories in which it is clear that the left adjoint preserves cofibrations, and the second result implies that the right adjoint preserves fibrations, so this is actually a Quillen adjunction, which implies that we have the desired adjunction of ($\infty$-)categories.}

			Since $\CAlg(\PrL) \to \PrL$ preserves limits, and $\PrL_* \to \PrL$ reflects limits, $D \circ F^\op$ is limiting, so since the right adjoint functor $\widehat{\Cat}_{\Delta^0/} \to \widehat{\Cat}$ preserves limits (as a right adjoint), the functor $D_\bullet \circ F^\op : (K^\op)^\triangleleft \to \widehat{\Cat}_{\Delta^0/}$ is limiting. Finally, since $\PrL_* \to \widehat{\Cat}_{\Delta^0/}$ reflects limits, $\PrL_* \to \PrL$ preserves limits, and $\CAlg(\PrL) \to \PrL$ reflects limits, it follows that $D_\bullet \circ F^\op$ is limiting also as a functor to $\CAlg(\PrL)$.

		\item We must show that if $f_* : D(X) \to D(Y)$ preserves $K^\triangleleft$-indexed colimits, then $f_* : D_\bullet(X) \to D_\bullet(Y)$ preserves $K$-indexed colimits.

			Indeed, since the latter functor is a limit-preserving functor between pointed categories, it is a functor between pointed categories that preserves zero objects, so it suffices to show that it preserves $K^\triangleleft$-indexed colimits.

			Consider the commutative diagram
			\[
				\begin{tikzcd}
					D(X)_{\pt/} \ar[d] \ar[r, "f_*"] & D(Y)_{\pt/} \ar[d] \\
					D(X) \ar[r, "f_*"'] & D(Y)
				\end{tikzcd}
			\]
			given by taking right adjoints in the commutative diagram
			\[
				\begin{tikzcd}
					D(Y) \ar[d] \ar[r, "f^*"] & D(X) \ar[d] \\
					D_\bullet(Y) \ar[r, "f^*"'] & D_\bullet(X)
				\end{tikzcd}
			.\]
			We note that the vertical arrows in the diagram of right adjoints are conservative and preserve weakly contractible colimits, so since $f_* : D(X) \to D(Y)$ preserves $K^\triangleleft$-indexed colimits, we have that the composite $D_\bullet(X) \to D(Y)$ preserves $K^\triangleleft$-indexed colimits, but since this factors as
			\[
				D_\bullet(X) \xrightarrow{f_*} D_\bullet(Y) \simeq D(Y)_{\pt/} \to D(Y)
			,\]
			and the last functor is conservative and preserves $K^\triangleleft$-indexed colimits, it follows that the first functor preserves $K^\triangleleft$-indexed colimits.

	\end{enumerate}

\end{proof}

\section{Applications to Motivic Homotopy Theories on Stacks} \label{S:apps}

We now give some examples of how to apply our abstract framework to produce both familiar and new examples of coefficient systems in the form of pullback formalisms. Our results will not only allow us to construct these examples, but also deduce important properties about each one.

We will be particularly interested in versions of ``motivic homotopy'' in different stacky contexts. The idea of emulating homotopy theory in settings outside of topology was introduced by Fabien Morel and Vladimir Voevodsky in \cite{Voe-A1} and \cite{A1htpysch}, and developed in the setting of algebraic geometry. For a category $\mathcal C$ of geometric objects, the general construction is given by choosing an interval for $\mathcal C$ that defines a notion of homotopy, as well as a suitable Grothendieck topology on $\mathcal C$, and defining the homotopy category over $S \in \mathcal C$ to be the category of ``homotopy-invariant'' sheaves on a suitable subcategory of $\mathcal C_{/S}$. The stable homotopy category is then defined by considering pointed objects and formally adjoining $\wedge$-inverses of Thom spaces of vector bundles.

In our setting, the relevant data from the interval object and Grothendieck topology are encoded in the form of a pseudotopology on $\mathcal C$, and the appropriate subcategories of the slice categories are determined by a choice of quasi-admissibility structure on $\mathcal C$. We can then construct the unstable homotopy categories using the construction of \Cref{cor:describe invariance localization of Huniv}, and stabilized versions using \Cref{thm:ptd PF sheafy stabilization}, or simply \Cref{prp:pointing} and \Cref{prp:PF stabilization}.

We will produce these generalized homotopy categories in three different settings.
\begin{enumerate}

	\item In \S\ref{S:alg}, we will recover the motivic homotopy theory of algebraic stacks, as developed in \cite{SixAlgSt}, and which specializes to the equivariant motivic homotopy theory of schemes, as developed in \cite{sixopsequiv}.

	\item In \S\ref{S:diff}, we will recover the motivic homotopy theory of differentiable stacks, as developed in \cite{TwAmb}, and which specializes to the usual topological equivariant homotopy theory.

	\item Finally, in \S\ref{S:hol}, we will produce a motivic homotopy theory for complex analytic stacks.

\end{enumerate}

\begin{rmk}[Some notes on the stacky setting:] \label{rmk:stacky challenges}

In classical non-stacky settings, the stabilization procedure often involves adjoining $\wedge$-inverses of only a single type of sphere. Indeed, since vector bundles are locally trivial, one can use descent properties to show that $\otimes$-invertibility of the Thom space of a trivial vector bundle implies $\otimes$-invertibility of Thom spaces of all vector bundles. In the stacky setting this is no longer the case, and we need to explicitly adjoin $\wedge$-inverses of Thom spaces of more general vector bundles.

This fact makes the stabilization process in the setting of stacks more difficult, and is the reason we need to use the more complicated \Cref{thm:PF sheafy stabilization} rather than the more straightforward \Cref{prp:PF stabilization}. Indeed, if we want to stabilize a pullback formalism $D$ on $\mathcal C$, and the collection of objects in $D(S)$ to which we adjoin $\wedge$-inverses varies significantly as $S$ ranges over the objects of $\mathcal C$, then it is not so easy to establish that quasi-admissible maps have ``$A$-lifts'', which is the condition that we use to show that the straightforward stabilization process of \Cref{prp:PF stabilization} produces a pullback formalism.

Another notable challenge that sometimes arises in the stacky case is that intervals may not generate the correct notion of homotopy invariance. For a suitable category $\mathcal C$ of geometric objects, it is generally the case that interval-homotopy is enough to make vector bundle projections into homotopy equivalences: indeed, if $V \to S$ is a vector bundle projection, then we can show that the zero section $S \to V$ defines a homotopy inverse by using the scalar multiplication of $V$ by the interval object to exhibit a homotopy between the composite $V \to S \to V$ and the identity. However, if $F \to S$ is a torsor under $V$, then we would generally like for $F \to S$ to also be a homotopy equivalence, but we no longer know that we have a section $S \to F$, nor a scalar multiplication of $F$ by the interval.

It is often the case that in non-stacky settings, or even sufficiently nice stacky settings, we have that $S$ can be covered by objects that have ``cohomological affineness'' properties which guarantee that $F \to S$ locally admits a section. Since $F \to S$ is a torsor under a vector bundle, this shows that $F \to S$ is locally equivalent to a vector bundle projection, so that $F \to S$ is locally a homotopy equivalence. In more general stacky settings, this may not be possible, and the fact that vector bundle torsors give equivalences in our homotopy categories is often quite useful, so we will need to impose it explicitly. Fortunately this is not a problem for us, as this sort of invariance property is still easily encoded by pseudotopologies.
% NOTE: vector bundle torsors giving equivalences is useful because of Jouanolou devices
% more generally, we also generally expect that (equivariant) maps that are homotopy equivalences non-equivariantly should also give equivalences equivariantly

\end{rmk}

\subsection{Algebraic Motivic Homotopy} \label{S:alg}

In this section, we will recover the usual notions motivic spaces and motivic spectra over algebraic stacks, as developed in \cite{SixAlgSt}, and in \cite{sixopsequiv} for quotient stacks.

We refer to \Cref{S:qproj} for notions of quasi-projectivity for algebraic stacks.

For the purposes of this section, we will use the following definition of the Nisnevich site of derived algebraic stacks.
\begin{defn}
	Let $\AlgStk$ be the category of (possibly derived) algebraic stacks, and define the Nisnevich topology on $\AlgStk$ to be the smallest Grothendieck topology such that the Zariski covering families cover, and families of the form $\{U,\tilde X \to X\}$ are covering when $U \to X$ is an open immersion, and $\tilde X \to X$ is an \'{e}tale map such that for any closed immersion $Z \to X$ that is disjoint form $U$, we have that $Z \times_{\tilde X} \tilde X \to Z$ is an equivalence.
\end{defn}
This definition of the Nisnevich topology may not be standard for (derived) algebraic stacks that are not qcqs, but in the qcqs case, the representable Nisnevich covers in the above definition recover the Nisnevich topology of \cite[Definition 2.4]{SixAlgSt}.
% Readers who are offended by this definition may simply restrict to qcqs derived algebraic stacks.

\begin{defn}
	Let $\AlgStk$ be the category of (possibly derived) algebraic stacks, and fix a quasi-admissibility structure on $\AlgStk$.

	We define a small quasi-admissible pseudotopology $\tau_\alg$ on $\AlgStk$ as the union of the Grothendieck topology $\tau_\Nis$ of quasi-admissible Nisnevich covers, along with the pseudotopology consisting of quasi-admissible vector bundle torsors. Also write $\tau_{\aff^1}$ for the pseudotopology consisting of projections $\aff^1_X \to X$, for $X \in \AlgStk$.

	Let $H^\alg$ be the pullback formalism $H^{\tau_\alg}$. As in Remark \ref{rmk:abuse notation for univ PF}, we will also write $H^\alg$ to denote its restriction to any full subcategory $\mathcal C$ of $\AlgStk$ such that for any quasi-admissible map $X \to Y$, if $Y \in \mathcal C$, then $X \in \mathcal C$.
\end{defn}

\begin{rmk}
	If every quasi-admissible map in $\AlgStk$ is representable by classical stacks, then the full subcategory of classical algebraic stacks is a full anodyne pullback subcontext of $\AlgStk$. Readers who wish to only consider classical stacks may use this fact to replace $\AlgStk$ by the category of classical algebraic stacks in what follows.
\end{rmk}

\begin{rmk}
	Generally, any quasi-admissibility structure we consider on $\AlgStk$ will satisfy the following:
	\begin{itemize}

		\item Every quasi-admissible map is qcqs, representable, and smooth.

		\item Every smooth quasi-projective map is quasi-admissible. In particular, qcqs open immersions are quasi-admissible, and vector bundle projections are quasi-admissible.

	\end{itemize}
	
\end{rmk}

\begin{rmk}
	Suppose that every quasi-admissible map in $\AlgStk$ is quasi-projective, and that every quasi-projective \'{e}tale map is quasi-admissible. Then for any qcqs algebraic stack $S$, the Grothendieck topology on $\AlgStk_S$ induced by the quasi-projective Nisnevich topology on $\AlgStk$ is the same as the ordinary Nisnevich topology on $\AlgStk_S$, since by Lemma \ref{lem:map of qproj is qproj}, any map in $\AlgStk_S$ is quasi-projective.
\end{rmk}

% NOTE: Nisnevich topology requires compatibility with stabilizers similar to the usual condition on residue fields

Before proceeding to our result on the properties of $H^\alg$, and the construction of its stabilization $\SH^\alg$, we discuss two examples of categories of algebraic stack that will be of central interest to us in this section:
\begin{exa}
	Let $\mathcal C^{\alg, \lred}$ be the category of qcqs algebraic stacks that admit quasi-projective Nisnevich covers by global quotients $X/G$, where $G$ is a linearly reductive group scheme over an affine scheme, and $X$ is a $G$-quasi-projective $G$-scheme.
	% NOTE: could allow $G$ to be over a qc scheme, because the qc scheme has a finite cover by affines

	Let $\mathcal C^{\alg, \nice}$ be the category of qcqs derived algebraic stacks with separated diagonal whose stabilizers are nice groups in the sense of \cite[Definition 2.1(i)]{SixAlgSt}, which we recall here for convenience: an fppf affine group scheme $G$ over an affine scheme $S$ is \emph{nice} if it is an extension of a finite \'{e}tale group scheme of order prime to the residue characteristic of $S$, by a group scheme of multiplicative type.
	
	Note that if $X$ is a qcqs derived algebraic stack with separated diagonal, and $X$ is a tame derived Deligne-Mumford stack, or a tame algebraic stack in the sense of \cite[\S3]{tame-stacks}, then $X \in \mathcal C^{\alg,\nice}$ by \cite[Examples 2.15 and 2.16]{SixAlgSt}. 

	Furthermore, by \cite[Theorem 2.12]{SixAlgSt}, any object of $\mathcal C^{\alg,\nice}$ that has affine diagonal is also in $\mathcal C^{\alg, \lred}$.
\end{exa}

\begin{rmk}
	Note that \cite[Theorem 2.12(ii)]{SixAlgSt} says that the objects of $\mathcal C^{\alg,\nice}$ are precisely the nicely scalloped stacks in the sense of \cite[Definition 2.9(iii)]{SixAlgSt}. Therefore, if $G$ is a nice group scheme over an affine scheme, and $G$ acts on a qcqs derived algebraic space $X$, then $X/G$ is in $\mathcal C^{\alg, \nice}$ by \cite[Theorem 2.14(ii)]{SixAlgSt}. In particular, for any qcqs map $X \to Y$ in $\AlgStk$ that is representable by derived algebraic spaces, if $Y \in \mathcal C^{\alg, \nice}$, then $X \in \mathcal C^{\alg, \nice}$.

	Using \Cref{rmk:full anodyne subctx}, it follows that if every quasi-admissible map in $\AlgStk$ is quasi-projective, then $\mathcal C^{\alg, \lred}$ is a full anodyne pullback subcontext of $\AlgStk$, and if every quasi-admissible map in $\AlgStk$ is qcqs and representable by algebraic spaces, then by \cite[Theorem 2.12(ii)]{SixAlgSt}, $\mathcal C^{\alg, \nice}$ is a full anodyne pullback subcontext of $\AlgStk$.
	% NOTE: use 2.12(ii)(c) + 2.14(ii)
\end{rmk}

\begin{rmk} \label{rmk:Halg is H}
	We compare $H^\alg$ to the ``$(*, \sharp, \otimes)$-formalisms'' considered in \cite{SixAlgSt}:
	\begin{itemize}

		\item Suppose the quasi-admissible maps in $\AlgStk$ are precisely the qcqs representable smooth morphisms. Then the pullback formalism $H^\alg$ on $\mathcal C^{\alg, \nice}$ is equivalent to the pullback formalism considered in \cite[\S3]{SixAlgSt} by \cite[A.2.2(a)]{SixAlgSt}.
			% NOTE: should be able to extend quasi-admissibility structure to include maps representable by algebraic spaces!

		\item Suppose the quasi-admissible maps in $\AlgStk$ are precisely the quasi-projective smooth morphisms. Then the pullback formalism $H^\alg$ is equivalent to the pullback formalism considered in \cite[A.2.2]{SixAlgSt} wherever it is defined.

	\end{itemize}

	Note also that by \cite[A.2.2]{SixAlgSt}, both of these quasi-admissibility structures define the same $\CAlg(\PrL)$-valued presheaf $H^\alg$ on $\mathcal C^{\alg, \nice} \cap \mathcal C^{\alg, \lred}$. In particular, this holds for the full subcategory of stacks in $\mathcal C^{\alg, \nice}$ that have affine diagonal.

	If $G$ is a linearly reductive group scheme over an affine scheme, and $X$ is a $G$-quasi-projective $G$-scheme then in the second case \cite[A.2.4]{SixAlgSt} says that
	\[
		H^\alg(X/G) \simeq H^G(X)
	,\]
	where the right-hand side is the construction considered in \cite{sixopsequiv}. This is the category of homotopy-invariant Nisnevich sheaves on the site of smooth $G$-quasi-projective schemes over $X$.

	In the first case, the same result holds by \cite[Remark 3.8]{SixAlgSt} if we assume $G$ is a nice group scheme over an affine scheme.
\end{rmk}

We now come to our main result outlining the properties of $H^\alg$:
\begin{prp}[Properties of $H^\alg$] \label{prp:univ prop of Halg}
	Suppose that every qcqs open immersion in $\AlgStk$ is quasi-admissible, and that $\mathcal C \subseteq \AlgStk$ is a full anodyne pullback subcontext consisting only of qcqs algebraic stacks.

	We record the following properties of the pullback formalism $H^\alg$ on $\mathcal C$:
	\begin{enumerate}

		\item $H^\alg$ has descent for the quasi-admissible Nisnevich topology.

		% \item Let $G$ be a linearly reductive embeddable group scheme over an affine scheme $B$, and $S$ a quasi-compact derived algebraic space over $B$ on which $G$ acts, and such that $S/G \in \mathcal C$. Then $H^\alg(S/G)$ is generated under sifted colimits by objects of the form $\cls{X/G}$, for $(X \to S)/G$ quasi-admissible and
		% 	\begin{enumerate}
		%
		% 		\item $X \to S$ is affine, under the condition that all quasi-admissible maps to $S/G$ are smooth and quasi-projective, and $X$ is $G$-quasi-projective;
		%
		% 		\item $X$ is affine, under the condition that all quasi-admissible maps to $S/G$ are smooth and quasi-projective, and $X$ is $G$-quasi-projective; or
		%
		% 		\item $X$ is quasi-affine, under the condition that all quasi-admissible maps to $S/G$ are qcqs and representable.
		%
		% 	\end{enumerate}

		\item Let $D$ be a pullback formalism on $\mathcal C$ that satisfies the following:
			\begin{description}

				\item[Reduced] $D(\emptyset)$ is contractible.

				\item[Excision] If $U \to S$ is an open immersion in $\mathcal C$, and $X \to S$ is a quasi-admissible \'{e}tale map that induces an equivalence away from $U$, then for any $M \in D(S)$,
					\[
						\begin{tikzcd}
							D(S;M) \ar[d] \ar[r] & D(U;M) \ar[d] \\
							D(X;M) \ar[r] & D(X \times_S U; M)
						\end{tikzcd}
					\]
					is Cartesian.

				% \item[Descent for nested covers]
				% 	If $\lambda$ is an ordinal, and $\{U_\alpha \subseteq S\}_{\alpha \in \lambda}$ is a Zariski-covering family of $S \in \mathcal C$ such that for $\alpha < \beta$, $U_\alpha \subseteq U_\beta$, then for any $M \in D(S)$, the map
				% 	\[
				% 		D(S;M) \to \varprojlim_{\alpha \in \lambda} D(U_\alpha; M)
				% 	\]
				% 	is an equivalence.

				\item[Homotopy invariance] If $V \to S$ is a quasi-admissible vector bundle torsor, and $M \in D(S)$, then
					\[
						D(S;M) \to D(V;M)
					\]
					is an equivalence.

			\end{description}
			Then the space of morphisms of pullback formalisms $H^\alg \to D$ is contractible. For any other pullback formalism $D$, there are no maps $H^\alg \to D$.

	\end{enumerate}
	\begin{proof} \hfill
		\begin{enumerate}

			\item This follows immediately from \Cref{thm:D-topology}.
			% \item This follows from Proposition \ref{prp:automatic admissible descent}, since the diagonal of any quasi-admissible \'{e}tale map is an open immersion between qcqs algebraic stacks, so it is quasi-admissible.

			\item Note that $\PF(\mathcal C)_{H^\alg/} \to \PF(\mathcal C)$ is fully faithful with essential image given by the $\tau_\alg$-invariant pullback formalisms. Thus, we may conclude by \Cref{rmk:describe qadm acyclic}, and the description of Nisnevich descent given by \cite[Proposition 2.5]{SixAlgSt}. In fact, we need to adapt \cite[Proposition 2.5]{SixAlgSt} since we are only considering the quasi-admissible Nisnevich topology, but the same argument using \cite[Theorem 2.2.7]{locspalg} applies in this case, since the diagonal of any quasi-admissible \'{e}tale map is a quasi-admissible \'{e}tale map.

		\end{enumerate}
		
	\end{proof}
\end{prp}

For algebraic stacks with nice stabilizers, the homotopy invariance condition on pullback formalisms reduces to the familiar condition of $\aff^1$-invariance.
\begin{lem} \label{lem:nice=>A1}
	If $X \in \mathcal C^{\alg, \nice}$, and qcqs representable \'{e}tale maps to $X$ are quasi-admissible, then any vector bundle torsor $V \to X$ is a $(\tau_\Nis \cup \tau_{\aff^1})$-acyclic map. In particular, in \Cref{prp:univ prop of Halg}, if $\mathcal C \subseteq \mathcal C^{\alg,\nice}$, and every qcqs representable smooth map is quasi-admissible, then the homotopy invariance condition can be replaced by the condition that for any $S \in \mathcal C$, and $M \in D(S)$, the map $D(S;M) \to D(\aff^1_S;M)$ is an equivalence.
	\begin{proof}
		Follows from \cite[A.2.2(a)]{SixAlgSt}, where the assumption on qcqs representable \'{e}tale maps is used to guarantee that the Nisnevich cover provided by \cite[Theorem 2.12(ii)]{SixAlgSt} is $\tau_\Nis$-acyclic.

		For the simplification of the statement of \Cref{prp:univ prop of Halg}, note that our assumption about quasi-admissible maps ensures both that qcqs representable \'{e}tale maps are quasi-admissible, and that $\tau_{\aff^1}$ consists of quasi-admissible vector bundle torsors, so that $\tau_\Nis \cup \tau_{\aff^1}$ and $\tau_\alg$ have the same acyclic maps on $\mathcal C$.
	\end{proof}
\end{lem}

We now come to our main result describing $\SH^\alg$, the stabilized version of $H^\alg$:
\begin{thm} \label{thm:SHalg}
	Let $\mathcal C$ be a full anodyne pullback subcontext of $\AlgStk$ and assume one of the following conditions:
	\begin{enumerate}

		\item \label{itm:SHalg lred}
			The quasi-admissible maps in $\AlgStk$ are the quasi-projective smooth morphisms, and $\mathcal C \subseteq \mathcal C^{\alg, \lred}$.

		\item \label{itm:SHalg nice}
			The quasi-admissible maps in $\AlgStk$ are the qcqs representable smooth morphisms, and $\mathcal C \subseteq \mathcal C^{\alg, \nice}$.

	\end{enumerate}
	
	There is a pullback formalism $\SH^\alg$ on $\mathcal C$ that is initial among pointed pullback formalisms $D$ satisfying the following:
	\begin{description}

		% \item[Descent] $D$ has descent for quasi-admissible Nisnevich covers.
		\item[Reduced] $D(\emptyset)$ is contractible.

		\item[Excision] If $U \to S$ is an open immersion in $\mathcal C$, and $X \to S$ is a quasi-admissible \'{e}tale map that induces an equivalence away from $U$, then for any $M \in D(S)$,
			\[
				\begin{tikzcd}
					D(S;M) \ar[d] \ar[r] & D(U;M) \ar[d] \\
					D(X;M) \ar[r] & D(X \times_S U; M)
				\end{tikzcd}
			\]
			is Cartesian.

		% \item[Stability] For all $S \in \mathcal C$, and vector bundle $V \to S$, the object $\cls{V}/\cls{V \setminus 0}$ is $\otimes$-invertible in $D(S)$. In particular, $D(S)$ is stable.
		\item[Stability] For any linearly reductive group $G$ over an affine scheme, $G$-scheme $S$ such that $S/G \in \mathcal C$, and $G$-equivariant vector bundle $V \to S$, the object $\cls{V/G}/\cls{V/G \setminus 0}$ is $\otimes$-invertible in $D(S)$.

		\item[Homotopy invariance] For any $S \in \mathcal C$, and $M \in D(S)$,
			\begin{enumerate}

				\item In case \ref{itm:SHalg lred}, for any vector bundle torsor $V \to S$, the map $D(S;M) \to D(V;M)$ is an equivalence.

				\item In case \ref{itm:SHalg nice}, the map $D(S;M) \to D(\aff^1_S;M)$ is an equivalence.

			\end{enumerate}

	\end{description}
	% In this case, we also have that $D$ takes values in stable categories, and for any $S \in \mathcal C$, and vector bundle $V \to S$, the object $\cls{V}/\cls{V \setminus 0}$ of $D(S)$ is $\otimes$-invertible, and for any torsor $F$ under $V$, the functor $D(S) \to D(F)$ is fully faithful

	We also have the following properties of $\SH^\alg$:
	\begin{enumerate}

		\item $\SH^\alg$ has descent for quasi-admissible Nisnevich covers.

		\item For every vector bundle $V \to S$ in $\mathcal C$, we have that $\cls{V}/\cls{V \setminus 0}$ is $\otimes$-invertible in $\SH^\alg(S)$, and for any torsor $F$ under $V$, the functor $\SH^\alg(S) \to \SH^\alg(F)$ is fully faithful. In fact, this also holds for the pullback formalisms $D$ as above.
			% PERF: also mention stability

		\item\label{itm:SHalg/excision}
			Excision: If $Z \to S$ is a closed immersion, and $X \to S$ is a quasi-admissible \'{e}tale neighborhood of $Z$, then for any $M \in \SH^\alg(S)$,
			\[
				\SH^\alg_Z(S;M) \to \SH^\alg_Z(X;M)
			\]
			is an equivalence, where for any $Y \to S$, pointed pullback formalism $D$, and $M \in D(S)$, $D_Z(M)$ is defined as the fibre of the map of pointed spaces
			\[
				D(Y;M) \to D(Y \times_S (S \setminus Z); M)
			.\]

		\item\label{itm:SHalg/ptwise}
			Objectwise description: Let $G$ be a linearly reductive embdeddable\footnotemark{} group scheme over an affine scheme, and let $S$ be a derived qcqs $G$-scheme such that $S/G \in \mathcal C$. In case \ref{itm:SHalg lred}, also assume $S$ is $G$-quasi-projective, and in case \ref{itm:SHalg nice}, also assume $G$ is nice. Then the natural functor $\Sigma^\infty_\alg : H^\alg_\bullet(S/G) \to \SH^\alg(S/G)$ provided by \Cref{prp:univ prop of Halg} is given by formally adjoining $\otimes$-inverses of objects of the form $p^*(\cls{V}/\cls{V \setminus 0})$, where $p$ is the map $S/G \to \B G$, and $V$ ranges over vector bundles on $\B G$.

			\footnotetext{Recall from \cite[Definition 2.1(ii)]{SixAlgSt} that an fppf affine group scheme $G$ over an affine scheme $S$ is embeddable if it is a closed subgroup of $\GL_S(\mathcal E)$ for some locally free sheaf $\mathcal E$ on $S$.}

		\item\label{itm:SHalg/generation}
			Generation: Let $G$ be a linearly reductive embeddable group scheme over an affine scheme $B$, and let $S$ be a derived qcqs $G$-scheme such that $S/G \in \mathcal C$, and we write $p : S \to B$ is the structure map. Then $\SH^\alg(S/G)$ is generated under sifted colimits by objects of the form
			\[
				\cls{X/G} \otimes (p/G)^*(\cls{V/G}/\cls{(V \setminus 0)/G})^{\otimes -1}
			,\]
			where $V$ ranges over $G$-equivariant vector bundles on $B$, and quasi-admissible maps $(X \to S)/G$, under the following conditions:
			\begin{enumerate}

				\item In case \ref{itm:SHalg lred}, we must assume that $S$ is $G$-quasi-projective (so it must be a classical stack).\footnote{In fact, \cite[Proposition A.8]{SixAlgSt} lets us reduce either to affine $X$, or to affine maps $X \to S$.}
					% and $X \to S$ ranges either over $G$-quasi-projective smooth maps such that $X$ is affine, or over $G$-equivariant smooth affine maps;

				\item In case \ref{itm:SHalg nice}, we must assume that $G$ is nice.\footnote{In fact, \cite[Proposition 3.7(ii)]{SixAlgSt} lets us reduce to quasi-affine $X$.}
					% and $X \to S$ ranges over smooth morphisms such that $X$ is quasi-affine.

			\end{enumerate}

	\end{enumerate}
	\begin{proof}
		For each $S \in \AlgStk$, let $A_S$ be the collection of morphisms in $H^\alg(S)$ of the form $\cls{V \setminus 0} \to \cls{V}$ for $V$ a vector bundle on $S$.

		We will define a collection $R$ of maps in $\mathcal C$, and a full subcategory $\mathcal B \subseteq \mathcal C$:
		\begin{enumerate}

			\item In case \ref{itm:SHalg lred}, $R$ is the collection of quasi-projective maps, and $\mathcal B$ is the category of linearly quasi-fundamental stacks (see \cite[Definition A.1(ii)]{SixAlgSt}).

			\item In case \ref{itm:SHalg nice}, $R$ is the collection of representable maps, and $\mathcal B$ is the category of nicely quasi-fundamental stacks (see \cite[Definition 2.9(ii)]{SixAlgSt}).

		\end{enumerate}
		
		We will verify the hypotheses of \Cref{thm:ptd PF sheafy stabilization}. Note that \cite[Lemma 4.13 or A.3.1]{SixAlgSt} show that maps in $R$ between objects of $\mathcal B$ have $A_\bullet$-lifts, and the proof of \cite[Lemma 6.3]{sixopsequiv} shows that for any $S \in \mathcal B$, every element of $(A_\bullet)_S$ is symmetric.

		Recall that $H^\alg$ has $\tau_\Nis$-descent by \Cref{prp:univ prop of Halg}.
		\begin{enumerate}

			% \item In case \ref{itm:SHalg lred}, we have that by \Cref{lem:G-qproj vs qproj} and \Cref{lem:desc lin qfund}, every quasi-admissible map to an object of $\mathcal B_0$ is in $\mathcal B_0$, and by \cite[Example A.4]{SixAlgSt} as well as the definition of $\mathcal C^{\alg,\lred}$, every object of $\mathcal C$ has a $\tau_\Nis$-cover by objects of $\mathcal B_0$, so the hypotheses of \Cref{thm:ptd PF sheafy stabilization} are satisfied.
			\item In case \ref{itm:SHalg lred}, note that \Cref{lem:map of qproj is qproj} says that quasi-projective maps between qcqs stacks are closed under composition, so every map in $R$ to an object of $\mathcal B$ is also from an object in $\mathcal B$, so it has $A$-lifts. By \cite[Example A.4]{SixAlgSt} as well as the definition of $\mathcal C^{\alg,\lred}$, every object of $\mathcal C$ has a $\tau_\Nis$-cover by objects of $\mathcal B$, so the hypotheses of \Cref{thm:ptd PF sheafy stabilization} are satisfied.

			\item In case \ref{itm:SHalg nice}, we still have that since any object $S$ of $\mathcal C$ is nicely scalloped by \cite[Theorem 2.12(ii)]{SixAlgSt}, the argument of \cite[Proposition 3.7(i)]{SixAlgSt} shows that there is a diagram $X : K^\triangleright \to \mathcal C$ that sends the cone point to $S$, all maps to quasi-admissible maps, and $K$ to $\mathcal B$. It follows that since every quasi-admissible map is in $R$, which is closed under composition, \cref{itm:A-lifts/source-local} of \Cref{lem:A-lifts} shows that for $B \in \mathcal B$, any map $S \to B$ in $R$ has $A$-lifts. Thus, for any nicely embeddable group scheme $G$ over an affine scheme, and $S$ a qcqs $G$-scheme, since $S/G$ has a map in $R$ to an object of $\mathcal B$, by \cref{itm:A-lifts/13} of \Cref{lem:A-lifts}, every map in $R$ to $S/G$ has $A$-lifts. In particular, the objects of $\mathcal B$ are $A$-good, so since every object of $\mathcal C$ has a quasi-admissible $\tau_\Nis$-cover by nicely quasi-fundamental stacks, the hypotheses of \Cref{thm:ptd PF sheafy stabilization} are satisfied.

		\end{enumerate}

		Thus, we may produce $H^\alg \to \SH^\alg$ by applying \Cref{thm:ptd PF sheafy stabilization}.
		\begin{itemize}

			\item The universal property follows from \Cref{prp:univ prop of Halg}, and \cref{itm:ptd PF sheafy stabilization/univ} of \Cref{thm:ptd PF sheafy stabilization}. We use \Cref{lem:nice=>A1} to get the version of the homotopy invariance statement for case \ref{itm:SHalg nice}. This also shows that vector bundle torsors are $\SH^\alg$-acyclic in either case, and \cref{itm:ptd PF sheafy stabilization/univ} of \Cref{thm:ptd PF sheafy stabilization} also shows that for any $S \in \mathcal C$, all objects of the form $\cls{V}/\cls{V \setminus 0}$ of $\SH^\alg(S)$ are $\otimes$-invertible.

			\item Since $H^\alg$ has descent for quasi-admissible Nisnevich covers, \Cref{thm:D-topology} or \cref{itm:ptd PF sheafy stabilization/descent} of \Cref{thm:ptd PF sheafy stabilization} says that $\SH^\alg$ does as well.

			\item \Cref{itm:SHalg/excision} follows immediately from the fact that $\SH^\alg$ is $\tau_\Nis$-invariant.

			\item \Cref{itm:SHalg/ptwise} follows from \cref{itm:PF sheafy stabilization/ptwise} of \Cref{thm:PF sheafy stabilization} by noting that $S/G \to \B G$ is a map that has $A_\bullet$-lifts, and $S/G$ is $A_\bullet$-good in the sense of \Cref{thm:PF sheafy stabilization}.

			\item Finally, we explain how to deduce \cref{itm:SHalg/generation}. Indeed, the assumptions guarantee that $\B G \in \mathcal B$, and $(p/G) : S/G \to \B G$ is in $R$, so \cref{itm:ptd PF sheafy stabilization/generation} of \Cref{thm:ptd PF sheafy stabilization} says that $\SH^\alg(S/G)$ is generated under sifted colimits by objects of the form
				\[
					\Sigma^\infty_\alg M \otimes (p/G)^*(\cls{V/G}/\cls{(V \setminus 0)/G})^{\otimes -1}
				,\]
				where $V$ ranges over $G$-equivariant vector bundles on $B$, and $M$ ranges over objects of $H^\alg_\bullet(S/G)$. Thus, we may conclude since $H^\alg(S/G)$ is generated under sifted colimits by objects $\cls{X/G}$ for quasi-admissible $(X \to S)/G$.
				% \begin{enumerate}
				%
				% 	\item In case \ref{itm:SHalg lred}, \cite[Proposition A.8]{SixAlgSt} shows that $H^\alg(S/G)$ is generated under sifted colimits by objects of the form $\cls{X/G}$, where $X \to S$ ranges either over $G$-equivariant smooth affine maps, or over $G$-equivariant smooth maps such that $X$ is affine.
				% 		% PERF: explain how to deduce that these maps are $G$-quasi-projective
				% 		% follows from \cite[Lemma 2.12]{sixopsequiv} and \ref{lem:G-qproj vs qproj}}
				%
				% 	\item In case \ref{itm:SHalg nice}, \cite[Proposition 3.7(ii)]{SixAlgSt} shows that $H^\alg(S/G)$ is generated under sifted colimits by objects $\cls{X/G}$ for
				%
				% \end{enumerate}

		\end{itemize}

	\end{proof}
\end{thm}

\begin{rmk}
	The pullback formalism $\SH^\alg$ is equivalent to the pullback formalism $\SH$ considered in \cite[A.3]{SixAlgSt} in case \ref{itm:SHalg lred}, and to the version considered in \cite[\S4]{SixAlgSt} in case \ref{itm:SHalg nice}. We will only explain the comparison with case \ref{itm:SHalg nice} and \cite[\S4]{SixAlgSt}, as the other comparison is completely analogous. Indeed, if $S$ is a nicely quasi-fundamental stack, then using \Cref{rmk:Halg is H}, the map $H^\alg_\bullet(S) \to \SH^\alg(S)$ is given simply by formally adjoining $\otimes$-inverses of objects of the form $\cls{V}/\cls{V \setminus 0}$ for vector bundles $V \to S$, and this is also what is done in \cite{SixAlgSt}. By \cite[Proposition A.11]{objwise-mono}, it follows that the transformation $H_\bullet \to \SH$ of \cite{SixAlgSt} coincides with the transformation $H^\alg_\bullet \to \SH^\alg$ when restricted to full subcategory nicely quasi-fundamental stacks, and we may conclude by Nisnevich descent.

	Given a linearly reductive group $G$ over an affine scheme $B$, recall that in \cite[\S6]{sixopsequiv}, the category $\SH^G(X)$ is constructed for $X$ a $G$-quasi-projective $G$-scheme by formally adjoining $\otimes$-inverses of objects of the form $f^*(\cls{V}/\cls{V \setminus 0})$ to the category $H^G_\bullet(X)$ defined in \cite[Definition 3.12]{sixopsequiv}, for $V$ a $G$-equivariant vector bundle on $B$, and $f$ the map $X \to B$. In \Cref{rmk:Halg is H}, we saw that in case \ref{itm:SHalg lred} of \Cref{thm:SHalg}, there is a natural equivalence $H^\alg_\bullet(X/G) \simeq H^G_\bullet(X)$, and the same holds in case \ref{itm:SHalg nice} if we assume that $G$ is nice.

	Thus, by \cref{itm:SHalg/ptwise} of \Cref{thm:SHalg}, if $G$ is embeddable, we have that as long as $\SH^G(X)$ is defined, it is equivalent to $\SH^\alg(X/G)$. See also \cite[Remark 4.9]{SixAlgSt}.
\end{rmk}

\subsection{Differentiable Motivic Homotopy} \label{S:diff}

In this section, we will recover the usual notions of genuine sheaves of animae and genuine sheaves of spectra for differentiable stacks, as developed in \cite[II.4]{TwAmb}.

We will need to recall some basic notions about stacks on manifolds from \cite[\S II.2.1]{TwAmb}. Note that we will not make use of notion of ``differentiable stack'' as defined \cite[Definition 2.2.1]{TwAmb} as it will not be necessary for our purposes. We only make reference to it when comparing to constructions in \cite[\S II]{TwAmb}.
\begin{defn} \label{defn:diff}
	Let $\Mfld$ be the site of manifolds and open covers. Since this is a subcanonical site, we can and will view $\Mfld$ as a full subcategory of $\Shv(\Mfld)$. We consider the following properties of a map $f : X \to Y$ in $\Shv(\Mfld)$.
	\begin{enumerate}

		\item $f$ is said to be a \emph{representable (surjective) submersion} if for any manifold $N$ and map $N \to Y$, the base change $X \times_Y N \to N$ is a (surjective) submersion.

		\item $f$ is said to be an \emph{open embedding} if for any manifold $N$ and map $N \to Y$, the base change $X \times_Y N \to N$ is an open embedding.

		\item $f$ is said to be \emph{representable} if for any representable submersion $N \to Y$, the pullback $N \times_Y X$ is a manifold.

	\end{enumerate}
	An \emph{open cover} of $X \in \Shv(\Mfld)$ is a family of maps $\{U_i \to X\}_i$ such that for any manifold $M$ and map $M \to X$, the base change $\{U_i \times_X M \to M\}_i$ is an open cover.
\end{defn}

\begin{rmk}
	The maps considered in \Cref{defn:diff} are closed under compositions by \cite[Corollary 2.1.11]{TwAmb}.
\end{rmk}

% NOTE: use $\Shv(\Mfld)$ instead of differentiable stacks, but then don't get the countable nested union thing ([Proposition 4.1.5]{TwAmb}) and also get terminology that is not as convenient, eg to do with separated diff stacks 

\begin{defn}
	We equip $\Shv(\Mfld)$ with the quasi-admissibility structure of representable submersions, and define a small quasi-admissible pseudotopology $\tau_\diff$ on $\Shv(\Mfld)$ as the union of the Grothendieck topology of open covers, along with the pseudotopology consisting of projections $\reals \times X \to X$.

	Let $H^\diff$ be the pullback formalism $H^{\tau_\diff}$. As in Remark \ref{rmk:abuse notation for univ PF}, we will also write $H^\diff$ to denote its restriction to any full subcategory $\mathcal C$ of $\Shv(\Mfld)$ such that for any representable submersion $X \to Y$, if $Y \in \mathcal C$, then $X \in \mathcal C$.
\end{defn}

\begin{exa}
	Let $\mathcal C^\diff$ be the full subcategory of $\Shv(\Mfld)$ consisting of objects that have an open cover by global quotients of the form $M/G$ for $G$ a compact Lie group, and $M$ a $G$-manifold. Note that if $X \to Y$ is a representable submersion in $\Shv(\Mfld)$, and $Y \in \mathcal C^\diff$, then $X \in \mathcal C^\diff$. Also note that by \cite[Theorem 3.7.2]{TwAmb}, every separated differentiable stack in the sense of \cite[Definition 3.3.1]{TwAmb} is in $\mathcal C^\diff$. Spelling this out, we have that if $X \in \Shv(\Mfld)$ satisfies that there exists a representable surjective submersion $M \to X$, where $M$ is a manifold, and the diagonal map $X \to X \times X$ is proper, then $X \in \mathcal C^\diff$.
\end{exa}

\begin{rmk}
	The pullback formalism $H^\diff$ is equivalent to the presheaf $\mathrm{H}$ constructed in \cite[Construction 4.2.23]{TwAmb} whenever it is defined. In particular, for any Lie group $G$, and $G$-manifold $M$, $H^\diff(M/G)$ is the category of $\reals$-invariant sheaves on the site of $G$-equivariant submersions to $M$.

	Alternatively, we can use the following \Cref{prp:univ prop of Hdiff} and \cite[Proposition 4.5.21]{TwAmb} to quickly see that $H^\diff$ and $\mathrm{H}$ define equivalent pullback formalisms on the pullback context of separated differentiable stacks (see \cite[Definition 3.3.1]{TwAmb}). Note that in \cite[\S II.4.5]{TwAmb}, pullback formalisms are always assumed to be sheaves.
\end{rmk}

\begin{rmk} \label{rmk:Hdiff on mflds}
	We can make $\Mfld$ into a pullback context in which the quasi-admissible maps are the open embeddings, so that the Grothendieck topology on $\Mfld$ is quasi-admissible, and defines a pullback formalism $H^\shf$ on $\Mfld$ given by sending a manifold $M$ to the usual category $\Shv(M)$ of sheaves on $\Op(M)$. The inclusion $\Mfld \to \Shv(\Mfld)$ is a morphism of pullback contexts, and since covering sieves in $\Mfld$ are sent to $\tau_\diff$-acyclic maps, \Cref{rmk:general realization} says that there is an essentially unique morphism of pullback formalisms $H^\shf \to H^\diff|_\Mfld$, which evaluates on $M \in \Mfld$ to the colimit-preserving functor
	\[
		\Shv(M) \to H^\diff(M)
	\]
	that sends an open subspace $U \subseteq M$ to $\cls{U;M} \in H^\diff(M)$. In fact, \cite[Proposition 4.2.27]{TwAmb} shows that $H^\shf \to H^\diff|_\Mfld$ is an equivalence.
\end{rmk}

\begin{prp}[Properties of $H^\diff$] \label{prp:univ prop of Hdiff}
	We record the following properties of the pullback formalism $H^\diff$:
	\begin{enumerate}

		\item $H^\diff$ has descent for the topology of open covers.

		\item\label{itm:Hdiff/top}
			For any Lie group $G$, the category $H^\diff(\B G)$ is equivalent to the usual $G$-equivariant unstable homotopy category.

		\item\label{Hdiff/mfld}
			The restriction of $H^\diff$ to $\Mfld$ is the natural functor sending a manifold $M$ to the category $\Shv(M)$.
		
		\item For any full anodyne pullback subcontext $\mathcal C$ of $\Shv(\Mfld)$, let $D$ be a pullback formalism on $\mathcal C$ that satisfies the following:
			\begin{description}

				\item[Reduced] $D(\emptyset)$ is contractible.

				\item[Excision] If $U,V \subseteq S$ are opens that form an open cover, then for any $M \in D(S)$,
					\[
						\begin{tikzcd}
							D(S;M) \ar[d] \ar[r] & D(U;M) \ar[d] \\
							D(V;M) \ar[r] & D(U \cap V; M)
						\end{tikzcd}
					\]
					is Cartesian.

				\item[Descent for nested covers]
					If $\lambda$ is an ordinal, and $\{U_\alpha \subseteq S\}_{\alpha \in \lambda}$ are opens of $S$ such that for $\alpha < \beta$, $U_\alpha \subseteq U_\beta$, and
					\[
						S = \bigcup_{\alpha \in \lambda} U_\alpha
					,\]
					then for any $M \in D(S)$, the map
					\[
						D(S;M) \to \varprojlim_{\alpha \in \lambda} D(U_\alpha; M)
					\]
					is an equivalence.

				\item[Homotopy invariance] For any $S \in \mathcal C$, and $M \in D(S)$,
					\[
						D(S;M) \to D(S \times \reals; M)
					\]
					is an equivalence.

			\end{description}
			Then the space of morphisms of pullback formalisms $H^\diff \to D$ is contractible. For any other pullback formalism $D$, there are no maps $H^\diff \to D$.

	\end{enumerate}
	\begin{proof}\hfill
		\begin{enumerate}

			% \item Follows from Proposition \ref{prp:automatic admissible descent}, since the diagonal of open embedding is a representable submersion.
			\item This follows from \Cref{thm:D-topology}.

			\item This follows from \cite[Theorem 4.4.16]{TwAmb}.

			\item This is \Cref{rmk:Hdiff on mflds}.

			\item Note that $\PF(\mathcal C)_{H^\diff/} \to \PF(\mathcal C)$ is fully faithful with essential image given by the $\tau_\diff$-invariant pullback formalisms. Thus, we conclude by \Cref{rmk:describe qadm acyclic}, and \Cref{prp:descent by open covers} (see also \cite[Proposition 4.1.5]{TwAmb}).

		\end{enumerate}
		
	\end{proof}
\end{prp}

\begin{rmk}
	In the setting of \Cref{prp:univ prop of Hdiff}, if we assume that $\mathcal C$ only contains differentiable stacks, \ie, objects $X \in \Shv(\Mfld)$ for which there is a manifold $M$, and a representable surjective submersion $M \to X$, then in the condition about nested open covers, it suffices to only consider countable nested covers, which is the case $\lambda = \omega$. This follows from \cite[Proposition 4.1.5]{TwAmb}.
\end{rmk}

Next, we will consider the stabilized version of $H^\diff$:
\begin{thm} \label{thm:SHdiff}
	% Let $\mathcal C$ be a full anodyne pullback subcontext of $\Shv(\Mfld)$ such that every object of $\mathcal C$ admits an open cover by objects of the form $M/G$ for $G$ a compact Lie group such that $\B G \in \mathcal C$, and $M$ a $G$-manifold. Then there is a pullback formalism $\SH^\diff$ on $\mathcal C$ that is initial among pointed pullback formalisms satisfying the following:
	For any full anodyne pullback subcontext $\mathcal C$ of $\mathcal C^\diff$ there is a pullback formalism $\SH^\diff$ on $\mathcal C$ that is initial among pointed pullback formalisms $D$ satisfying the following:
	\begin{description}

		% \item[Descent] $D$ has descent for the topology of open covers.
		\item[Reduced] $D(\emptyset)$ is contractible.

		\item[Excision] If $U,V \subseteq S$ are opens that form an open cover, then for any $M \in D(S)$,
			\[
				\begin{tikzcd}
					D(S;M) \ar[d] \ar[r] & D(U;M) \ar[d] \\
					D(V;M) \ar[r] & D(U \cap V; M)
				\end{tikzcd}
			\]
			is Cartesian.

		\item[Descent for nested covers]
			If $\lambda$ is an ordinal, and $\{U_\alpha \subseteq S\}_{\alpha \in \lambda}$ are opens of $S$ such that for $\alpha < \beta$, $U_\alpha \subseteq U_\beta$, and
			\[
				S = \bigcup_{\alpha \in \lambda} U_\alpha
			,\]
			then for any $M \in D(S)$, the map
			\[
				D(S;M) \to \varprojlim_{\alpha \in \lambda} D(U_\alpha; M)
			\]
			is an equivalence.

		% \item[Stability] For all $S \in \mathcal C$, and vector bundle $V \to S$, the object $\cls{V}/\cls{V \setminus 0}$ is $\otimes$-invertible in $D(S)$. In particular, $D(S)$ is stable.
		% \item[Stability] For any compact Lie group $G$ such that $\B G \in \mathcal C$, and any real $G$-representation $V$, the object $\cls{V/G}/\cls{V/G \setminus 0}$ is $\otimes$-invertible in $D(\B G)$.
		\item[Stability] For any compact Lie group $G$ and $G$-manifold $M$ such that $M/G \in \mathcal C$, if $V$ is a finite-dimensional real $G$-representation, the object $\cls{(M \times V)/G}/\cls{(M \times V \setminus 0)/G}$ is $\otimes$-invertible in $D(M/G)$.

		\item[Homotopy invariance] For any $S \in \mathcal C$, and $M \in D(S)$,
			\[
				D(S;M) \to D(S \times \reals; M)
			\]
			is an equivalence.

	\end{description}
	% In this case, we also have that $D$ takes values in stable categories, and for all $S \in \mathcal C$, the functor $D(S) \to D(S \times \reals)$ is fully faithful, and for any vector bundle $V \to S$, the object $\cls{V}/\cls{V \setminus 0}$ is $\otimes$-invertible in $D(S)$.

	We also have the following properties of $\SH^\diff$:
	\begin{enumerate}

		\item $\SH^\diff$ has descent for open covers.

		\item For every vector bundle $V \to S$ in $\mathcal C$, we have that $\cls{V}/\cls{V \setminus 0}$ is $\otimes$-invertible in $\SH^\diff(S)$. In fact, this also holds for the pullback formalisms $D$ as above.
			% PERF: also mention stability

		\item\label{itm:SHdiff/top}
			For any compact Lie group $G$ such that $\B G \in \mathcal C$, the category $\SH^\diff(\B G)$ is equivalent to the genuine $G$-equivariant stable homotopy category.

		\item\label{itm:SHdiff/excision}
			Excision: If $Z \to S$ is a map in $\mathcal C$ that is representable by closed embeddings, and $U \to S$ is an open neighborhood of $Z$, then for any $M \in \SH^\diff(S)$,
			\[
				\SH^\diff_Z(S;M) \to \SH^\diff_Z(U;M)
			\]
			is an equivalence, where for any $Y \to S$, pointed pullback formalism $D$ on $\mathcal C$, and $M \in D(S)$, $D_Z(M)$ is defined as the fibre of the map of pointed spaces
			\[
				D(Y;M) \to D(Y \times_S (S \setminus Z); M)
			.\]

		\item\label{itm:SHdiff/ptwise}
			Objectwise description: For any compact Lie group $G$ and $G$-manifold $M$ such that $M/G \in \mathcal C$, we have that the natural functor $H^\diff_\bullet(M/G) \to \SH^\diff(M/G)$ provided by \Cref{prp:univ prop of Hdiff} is given by formally adjoining $\otimes$-inverses of objects of the form $\cls{(M \times V)/G}/\cls{(M \times V \setminus 0)/G}$, where $V$ is a finite-dimensional real $G$-representation. 

		\item\label{itm:SHdiff/generation}
			Generation: For any compact Lie group $G$ and $G$-manifold $M$ such that $M/G \in \mathcal C$, we have that $\SH^\diff(M/G)$ is generated under sifted colimits by objects of the form
			\[
				\cls{X/G} \otimes (\cls{M \times V/G}/\cls{(M \times V \setminus 0)/G})^{\otimes -1}
			,\]
			where $X \to M$ is a $G$-equivariant submersion of $G$-manifolds, and $V$ is a finite-dimensional real $G$-representation.

	\end{enumerate}
	\begin{proof}
		For each $S \in \Shv(\Mfld)$, let $A_S$ be the collection of morphisms in $H^\diff(S)$ of the form $\cls{V \setminus 0} \to \cls{V}$ for $V$ a vector bundle on $S$. By \Cref{prp:univ prop of Hdiff}, $H^\diff$ is a sheaf for the topology of open covers, and since $\mathcal C \subseteq \mathcal C^\diff$, every object of $\mathcal C$ has an open cover by objects of the form $M/G$ for $G$ a compact Lie group acting on a manifold $M$.

		If $S$ is of the form $M/G$ for $G$ a compact Lie group acting on a manifold $M$, and $V$ is a vector bundle on $S$, then $\cls{V}/\cls{V \setminus 0}$ is a symmetric object of $H^\diff_\bullet(S)$ by \cite[Lemma 4.3.8]{TwAmb} or the argument of \cite[Lemma 6.3]{sixopsequiv}. Furthermore, every representable map to $S$ has $A_\bullet$-lifts by \cite[Lemma 4.3.9]{TwAmb}.

		Thus we may produce $H^\diff \to \SH^\diff$ by applying \Cref{thm:ptd PF sheafy stabilization}.
		\begin{itemize}

			\item Note that for any compact Lie group $G$, and $G$-manifold $M$, the map $M/G \to \B G$ has $A_\bullet$-lifts. Thus, the universal property follows from \Cref{rmk:describe qadm acyclic}, \Cref{prp:univ prop of Hdiff}, and \cref{itm:ptd PF sheafy stabilization/univ} of \Cref{thm:ptd PF sheafy stabilization}, which also shows that for any vector bundle $V \to S$ in $\mathcal C$, the object $\cls{V}/\cls{V \setminus 0}$ of $\SH^\diff(S)$ is $\otimes$-invertible.

			\item Since $H^\diff$ has descent for open covers, $\SH^\diff$ does too by \Cref{thm:D-topology} or \cref{itm:ptd PF sheafy stabilization/descent} of \Cref{thm:ptd PF sheafy stabilization}.

			\item \Cref{itm:SHdiff/excision} follows immediately from the fact that $\SH^\diff$ has descent for open covers.

			\item \Cref{itm:SHdiff/top} follows from \cref{itm:SHdiff/ptwise}, and \cref{itm:Hdiff/top} of \Cref{prp:univ prop of Hdiff}, or simply \cite[Proposition 4.4.17]{TwAmb}.

			\item \Cref{itm:SHdiff/generation} follows from \cref{itm:ptd PF sheafy stabilization/generation} of \Cref{thm:ptd PF sheafy stabilization}, the fact that $M/G \to \B G$ has $A_\bullet$-lifts, and the fact that $H^\diff(M/G)$ is generated under sifted colimits by objects of the form $\cls{X/G}$ for $X \to M$ a $G$-equivariant submersion of $G$-manifolds.

		\end{itemize}
	\end{proof}
\end{thm}

\begin{rmk}
	The pullback formalism $\SH^\diff$ is equivalent to the pullback formalism $\mathrm{SH}$ considered in \cite[4.3.3]{TwAmb}. Indeed, by descent, it suffices to show this fact after restricting to global quotient stacks of the form $M/G$, for $G$ a compact Lie group and $M$ a $G$-manifold, but this follows immediately from \cite[Definition 4.3.4]{TwAmb}.

	Alternatively, the universal property given by \cite[Proposition 4.5.27]{TwAmb} immediately shows that $\SH^\diff$ and $\mathrm{SH}$ are equivalent pullback formalisms on the pullback context of separated differentiable stacks (see \cite[Definition 3.3.1]{TwAmb}). Note that in \cite[\S II.4.5]{TwAmb}, pullback formalisms are always assumed to be sheaves.
\end{rmk}

\subsection{Holomorphic Motivic Homotopy} \label{S:hol}

In this section, we will define complex analytic versions of motivic spaces and spectra.

We refer to \cite{Fischer} for basic notions about complex spaces. Other useful references are \cite{CohAnalSh} and \cite{SteinSpaces}.

Let us recall some basic definitions:
\begin{defn}
	The category $\An_\comps$ of complex spaces is the full subcategory of the category of locally $\comps$-ringed spaces that are Hausdorff and second countable, and that admit an open cover by subspaces $W \subseteq U$, where $U$ is an open subspace of the locally $\comps$-ringed space given by $\comps^n$ and the sheaf of holomorphic functions $\comps$-valued functions, and $W$ is the closed subspace given by the vanishing of a coherent ideal sheaf. See \cite[0.14]{Fischer}. The maps in $\An_\comps$ are called holomorphic maps.

	Given a complex group $G$ -- a group object in $\An_\comps$ -- we say \emph{complex $G$-space} to refer to a complex space $X$ equipped with a holomorphic $G$-action $G \times X \to X$.

	See \cite[2.18]{Fischer} for the definition of a submersion of complex spaces, and \cite[0.23]{Fischer} for the definitions of embeddings, immersions, and local biholomorphisms. It is worth noting that the word ``embedding'' refers to a map of locally ringed spaces that is usually called a ``closed immersion'' in the setting of algebraic geometry. We will say \emph{open subspace} to refer to the usual notion of open immersions of locally ringed spaces.
\end{defn}

\begin{rmk}
	Oka's coherence theorem states that the structure sheaf of any complex space is coherent.
\end{rmk}

\begin{defn}
	Equip $\An_\comps$ with the Grothendieck topology of open covers. We say that a map $X \to Y$ in $\Shv(\An_\comps)$ is an open subspace, embedding, representable submersion, representable local biholomorphism, if for any $Y' \in \An_\comps$ and map $Y' \to Y$, the map $Y' \times_Y X \to Y'$ is an open subspace, embedding, submersion, local biholomorphism, respectively, of complex spaces.

	We make $\Shv(\An_\comps)$ into a quasi-small pullback context by equipping it with the quasi-admissibility structure of representable submersions.

	Define $\tau_\NisC$ to be the Grothendieck topology on $\Shv(\An_\comps)$ given by open covers and \'{e}tale excision: it is the Grothendieck topology given by declaring that the following are covering families:
	\begin{enumerate}

		\item Families $\{U_i \to X\}_i$ of open subspaces in $\Shv(\An_\comps)$ such that for any $X' \in \An_\comps$ and map $X' \to X$, the family $\{U_i \times_X X' \to X'\}_i$ is an open covering family.

		\item Families $\{U, \tilde X \to X\}$, where $U \to X$ is an open subspace and $\tilde X \to X$ is a representable local biholomorphism in $\Shv(\An_\comps)$ such that $\tilde X \to X$ is invertible away from $U$, \ie, for any embedding $Z \to X$ such that $Z \times_X U$ is empty, we have that $\tilde X \times_X Z \to Z$ is an equivalence.
		% WARN: this is not the same as just having descent for open covers if we allow quotients by non-discrete groups

	\end{enumerate}

	We define $\tau_\htpy$ to be the small quasi-admissible pseudotopology on $\Shv(\An_\comps)$ consisting of maps $\yo(V \to S)$ in $\Psh(\Shv(\An_\comps))$ where $V \to S$ is a torsor under a vector bundle on $S$. % PERF: define

	Define $\tau_\hol$ to be the small quasi-admissible pseudotopology on $\Shv(\An_\comps)$ given as the union $\tau_\NisC \cup \tau_\htpy$.

	Let $H^\hol$ be the pullback formalism $H^{\tau_\hol}$ of Corollary \ref{cor:describe invariance localization of Huniv}. As in Remark \ref{rmk:abuse notation for univ PF}, we will also write $H^\hol$ to denote its restriction to any full subcategory $\mathcal C$ of $\Shv(\An_\comps)$ such that for any representable submersion $X \to Y$, if $Y \in \mathcal C$, then $X \in \mathcal C$.
\end{defn}

\begin{nota}
	Given a complex group $G$ and a complex $G$-space $X$, we write $X/G$ for the object of $\Shv(\An_\comps)$ given as the colimit in $\Shv(\An_\comps)$ of the action groupoid of the action of $G$ on $X$.
\end{nota}

\begin{exa}
	We define $\mathcal C^\hol$ to be the full subcategory of $\Shv(\An_\comps)$ consisting of those objects that admit a $\tau_\NisC$-cover by global quotients of the form $X/G$ for $G$ a complex reductive group acting on a complex space $X$. Note that if $X \to Y$ is a representable submersion, and $Y \in \mathcal C^\hol$, then $X \in \mathcal C^\hol$, so $\mathcal C^\hol$ is a full anodyne pullback subcontext of $\Shv(\An_\comps)$.
\end{exa}

\begin{prp}[Properties of $H^\hol$] \label{prp:univ prop of Hhol}
	We record the following properties of the pullback formalism $H^\hol$:
	\begin{enumerate}

		\item $H^\hol$ satisfies \'{e}tale excision and descent for open covers.

		\item For any full anodyne pullback subcontext $\mathcal C \subseteq \Shv(\An_\comps)$, let $D$ be a pullback formalism on $\mathcal C$ that satisfies the following:
			\begin{description}

				\item[Reduced] $D(\emptyset)$ is contractible.

				\item[Excision] If $U \to S$ is an open subspace, and $X \to S$ is a local biholomorphism that is invertible away from $U$, then for any $M \in D(S)$,
					\[
						\begin{tikzcd}
							D(S;M) \ar[d] \ar[r] & D(U;M) \ar[d] \\
							D(X;M) \ar[r] & D(X \times_S U; M)
						\end{tikzcd}
					\]
					is Cartesian.

				\item[Descent for nested covers] If $\lambda$ is an ordinal, and $\{U_\alpha \subseteq S\}_{\alpha \in \lambda}$ are opens of $S$ such that for $\alpha < \beta$, $U_\alpha \subseteq U_\beta$, and
					\[
						S = \bigcup_{\alpha \in \lambda} U_\alpha
					,\]
					then for any $M \in D(S)$, the map
					\[
						D(S;M) \to \varprojlim_{\alpha \in \lambda} D(U_\alpha; M)
					\]
					is an equivalence.

				\item[Homotopy invariance] For any $S \in \mathcal C$, vector bundle torsor $V \to S$, and $M \in D(S)$,
					\[
						D(S;M) \to D(V; M)
					\]
					is an equivalence.

			\end{description}
			Then the space of morphisms of pullback formalisms $H^\hol \to D$ is contractible. For any other pullback formalism $D$, there are no maps $H^\hol \to D$.

	\end{enumerate}
	\begin{proof}\hfill
		\begin{enumerate}

			% \item Follows from Proposition \ref{prp:automatic admissible descent}, since the diagonal of any representable local biholomorphism is a representable submersion.
			\item This follows from \Cref{thm:D-topology}.

			\item Note that $\PF(\mathcal C)_{H^\hol/} \to \PF(\mathcal C)$ is fully faithful with essential image given by the $\tau_\hol$-invariant pullback formalisms. Thus, we may conclude by \Cref{rmk:describe qadm acyclic}, and \Cref{prp:descent by open covers}. Compare to \cite[Proposition 4.1.5]{TwAmb}.

		\end{enumerate}
		
	\end{proof}
\end{prp}

In order to define the stabilization $\SH^\hol$ of $H^\hol$, we will need to use the notion of \emph{Stein spaces}. A complex space $X$ is said to be a Stein space if it is ``cohomologically affine for coherent sheaves'', \ie, taking global sections defines an exact functor $\Coh(X) \to \Mod_\comps$. This is not the standard definition, but follows from Cartan's Theorem B -- see \cite[Corollary 2 of 0.37]{Fischer}. 

\begin{rmk}
	We collect some remarks about Stein spaces:
	\begin{enumerate}

		\item Any complex space that is a closed subspace of a Stein space is itself Stein.

		\item For $n \geq 0$ any integer, any domain of holomorphy in $\comps^n$ is Stein.

		\item Products and fibred products of Stein spaces are Stein (see \cite[0.32]{Fischer}).

		\item If $X$ is a complex space, then Stein open subspaces of $X$ are closed under finite intersections (\cite[bottom of page 33]{CohAnalSh}).

	\end{enumerate}
	We may use these facts to show that any complex space has an open cover by Stein spaces:
	\begin{itemize}

		\item Any disc in $\comps$ is a domain of holomorphy, so it is Stein.
			% NOTE: see example in https://encyclopediaofmath.org/wiki/Domain_of_holomorphy

		\item Any product of discs is Stein, so for any integer $n \geq 0$, the topology of $\comps^n$ has a basis consisting of Stein open subspaces.

		\item Thus, if $U \subseteq \comps^n$ is open it has an open covering by Stein open subspaces.

		\item If $W \subseteq U$ is a closed complex subspace of an open $U \subseteq \comps^n$, then $W$ has an open covering by closed subspaces of Stein spaces, so $W$ has an open covering by Stein spaces.

		\item Every complex space has an open cover by complex subspaces that are closed subspaces of open subspaces of $\comps^n$ for $n \geq 0$, so we are done.

	\end{itemize}
\end{rmk}

One challenge for defining $\SH^\hol$ involves verifying the hypotheses about ``$A$-lifts'' when applying \Cref{thm:ptd PF sheafy stabilization}. This ends up to reducing to the following problem: we must show that for certain complex Lie groups $G$, and certain complex spaces $X$ with $G$-action, every $G$-equivariant vector bundle on $X$ is a retract of a ``trivial'' $G$-equivariant vector bundle, \ie, one that is pulled back from a $G$-equivariant vector bundle on $\pt$.

Here we outline one possible strategy for showing this:
\begin{rmk} \label{rmk:hol resol fin gp}
	Let $X$ be a reduced Stein space with an action of a complex reductive group $G$, and let $V \to X$ be a $G$-equivariant vector bundle on $X$. If $K$ is a maximal compact subgroup of $G$, then \cite[Proposition 2.4]{equivK} says that there is a $K$-equivariant topological vector bundle $V' \to X$ and a finite-dimensional $K$-representation $W$ such that $V \oplus V' \simeq X \times W$ as $K$-equivariant topological vector bundles.

	Since $K$ is a maximal compact subgroup of the reductive group $G$, we have that $G$ is isomorphic to the complexification of $K$, so that the action of $K \to \GL(W)$ extends to $G$, so that $W$ has the structure of a $G$-representation. Furthermore, by \cite[Theorem 2.3(a)]{equivOkaBundles}, we have a (holomorphic) $G$-equivariant vector bundle $V'' \to X$ and an equivalence of $K$-equivariant topological vector bundles $V'' \simeq V'$. Thus we have an equivalence of $K$-equivariant topological vector bundles $V \oplus V'' \simeq X \times W$, but by \cite[Theorem 1.5]{equivOkaBundles}, this means that there is also a equivalence of $K$-equivariant \emph{holomorphic} vector bundles $V \oplus V'' \simeq X \times W$.

	When $G$ is finite so $G = K$, this shows that for every vector bundle $V \to X/G$, there is a vector bundle $V' \to X/G$, and a vector bundle $W \to \pt/G$ such that $V \oplus V' \simeq p^* W$, where $p$ is the map $X/G \to \pt/G$.
\end{rmk}

Thus, it will be particularly convenient to consider complex analytic stacks that have finite stabilizers. We record some properties of such stacks:
\begin{lem} \label{lem:hol pseudotop for finite groups}
	Let $\mathcal C^{\hol,\fin}$ be the full subcategory of $S \in \Shv(\An_\comps)$ such that $S$ admits a $\tau_\NisC$-cover by objects of the form $X/G$ for $G$ a finite group acting on a complex space $X$.
	\begin{enumerate}

		\item $\mathcal C^{\hol,\fin}$ is a full anodyne pullback subcontext of $\Shv(\An_\comps)$.

		\item If $G$ is a finite group acting on a complex space $X$, then any $\tau_\NisC$-cover of $X/G$ is covering for the topology $\tau_\open$ of open covers, and in fact, it contains an open covering by objects of the form $X'/G$, where $X'$ is a Stein space.
		% NOTE: I don't think you can show that for a general object of $\mathcal C^{\hol,\fin}$, any $\tau_\NisC$-cover is an open cover
		% If knew the object had AN open cover by $X/G$s, then could do it
		% Alternatively, might be able to do it if knew it had a FINITE $\tau_\NisC$-covering family by $X/G$s
		% It might be the case that this works for geometric stacks, that is, ones that have an atlas, that is, a representable surjective submersion from a complex space 

		\item For any $S \in \mathcal C^{\hol,\fin}$, any vector bundle torsor $V \to S$ is a $(\tau_\NisC \cup \tau_\comps)$-acyclic map, where $\tau_\comps$ is the pseudotopology on $\Shv(\An_\comps)$ whose acyclic pseudosieves are maps of the form $\comps \times X \to X$ for $X \in \Shv(\An_\comps)$.

	\end{enumerate}
	
	\begin{proof}\hfill
		\begin{enumerate}

			\item To see that the inclusion $\mathcal C^{\hol,\fin} \to \Shv(\An_\comps)$ is anodyne, it suffices to note that if $G$ is a complex Lie group acting on a complex space $X$, then any representable map to $X/G$ is of the form $(Y \to X)/G$, where $Y \to X$ is a map in $\An_\comps$. Thus, we may conclude by \Cref{rmk:full anodyne subctx}.

			\item
				Next, we show that if $G$ is a finite group acting on a complex space $X$, then any $\tau_\NisC$-covering sieve of $X/G$ is also $\tau_\open$-covering. Indeed, by the definition of $\tau_\NisC$, it suffices to show that if $U$ is a $G$-invariant open subspace of $X$, and $X' \to X$ is a $G$-equivariant \'{e}tale map that is invertible away from $U$, then the sieve generated by $U/G,X'/G \to X/G$ is $\tau_\open$-covering, but this follows from \Cref{lem:existence of equivariant local sections Hausdorff}.

				By \Cref{lem:invariant nbhd basis Hausdorff}, we can then further refine this sieve so that it is generated by open subspaces of the form $(Y \to X)/G$, for $Y$ a $G$-invariant Stein open subspace of $X$.

			\item Since every object of $\mathcal C^{\hol,\fin}$ has a $\tau_\NisC$-cover by global quotients of the form $X/G$ for $G$ a finite group acting on a Stein space $X$, it suffices to show that any vector bundle torsor on $X/G$ is a $\tau_\comps$-acyclic map.
				
				Indeed, any vector bundle torsor on $X/G$ is of the form $(Y \to X)/G$, for $Y$ a torsor under a $G$-equivariant holomorphic vector bundle $V \to X$, such that the map $Y \times_X V \to Y$ is $G$-equivariant. As in the discussion following \cite[Definition 2.18]{sixopsequiv}, we know that $Y$ corresponds to the sheaf on $X$ that classifies sections of some $G$-equivariant surjection $\mathcal E \to \mathcal O_X$, where $\mathcal E$ is finite locally free.

				Since $X$ is Stein, any such surjection $\mathcal E \to \mathcal O_X$ admits a section, and since $G$ is finite, this section can further be chosen to be $G$-equivariant, so that $\mathcal E \to \mathcal O_X$ is $G$-equivariantly isomorphic to a product projection $\mathcal O_X \oplus \mathcal F \to \mathcal O_X$, where $\mathcal F$ is a $G$-equivariant finite locally free sheaf. Thus, $Y \to X$ is isomorphic to the vector bundle projection for the $G$-equivariant vector bundle corresponding to $\mathcal F$.

				In this case, scalar multiplication defines a $\comps$-indexed homotopy between $Y/G \to X/G \to Y/G$ and the identity, whence $Y/G \to X/G$ is $\tau_\comps$-acyclic.

				% PERF: details

		\end{enumerate}
	\end{proof}
\end{lem}

Finally we come to our main result describing $\SH^\hol$, the stabilized version of $H^\hol$:
\begin{thm} \label{thm:SHhol}
	Let $\mathcal C$ be a full anodyne pullback subcontext of $\Shv(\An_\comps)$ such that every object $S \in \mathcal C$ admits a $\tau_\NisC$-cover by global quotients of the form $X/G$ for $G$ a finite group acting on a \emph{reduced} complex space $X$. Then there is a pullback formalism $\SH^\hol$ on $\mathcal C$ that is initial among pointed pullback formalisms $D$ satisfying the following:
	\begin{description}

		% \item[Descent] $D$ \'{e}tale excision and descent for open covers.
		\item[Reduced] $D(\emptyset)$ is contractible.

		\item[Excision] If $U \to S$ is an open subspace, and $X \to S$ is a local biholomorphism that is invertible away from $U$, then for any $M \in D(S)$,
			\[
				\begin{tikzcd}
					D(S;M) \ar[d] \ar[r] & D(U;M) \ar[d] \\
					D(X;M) \ar[r] & D(X \times_S U; M)
				\end{tikzcd}
			\]
			is Cartesian.

		\item[Descent for nested covers] If $\lambda$ is an ordinal, and $\{U_\alpha \subseteq S\}_{\alpha \in \lambda}$ are opens of $S$ such that for $\alpha < \beta$, $U_\alpha \subseteq U_\beta$, and
			\[
				S = \bigcup_{\alpha \in \lambda} U_\alpha
			,\]
			then for any $M \in D(S)$, the map
			\[
				D(S;M) \to \varprojlim_{\alpha \in \lambda} D(U_\alpha; M)
			\]
			is an equivalence.

		% \item[Stability] For any $S \in \mathcal C$, and vector bundle $V \to S$, the object $\cls{V}/\cls{V \setminus 0}$ is $\otimes$-invertible in $D(S)$.
		% \item[Stability] For any finite group $G$ such that $\B G \in \mathcal C$, and complex $G$-representation $V$, the object $\cls{V/G}/\cls{V/G \setminus 0}$ is $\otimes$-invertible in $D(\B G)$.
		\item[Stability] For any finite group $G$ acting on a reduced Stein space $X$, and finite-dimensional complex $G$-representation $V$, the object $\cls{(X \times V)/G}/\cls{(X \times V \setminus 0)/G}$ is $\otimes$-invertible in $D(X/G)$.

		\item[Homotopy invariance] For any $S \in \mathcal C$, and $M \in D(S)$,
			\[
				D(S;M) \to D(S \times \comps; M)
			\]
			is an equivalence.

	\end{description}
	% In this case, we also have that $D$ takes values in stable categories, and for all $S \in \mathcal C$, and vector bundle $V \to S$, the object $\cls{V}/\cls{V \setminus 0}$ is $\otimes$-invertible in $D(S)$, and for any torsor $F$ under $V$, the functor $D(S) \to D(F)$ is fully faithful.

	We also have the following properties of $\SH^\hol$:
	\begin{enumerate}

		\item $\SH^\hol$ has \'{e}tale excision and descent for open covers.

		\item For every vector bundle $V \to S$ in $\mathcal C$, we have that $\cls{V}/\cls{V \setminus 0}$ is $\otimes$-invertible in $\SH^\hol(S)$, and for any torsor $F$ under $V$, the functor $\SH^\hol(S) \to \SH^\hol(F)$ is fully faithful. In fact, this also holds for the pullback formalisms $D$ as above.
			% PERF: also mention stability

		\item\label{itm:SHhol/excision}
			If $Z \to S$ is a map in $\mathcal C$ that is representable by embeddings\footnote{See \cite[0.23]{Fischer}.}, and $X \to S$ is an \'{e}tale neighborhood of $Z$,\footnotemark{} then for any $M \in \SH^\hol(S)$,
			\[
				\SH^\hol_Z(S;M) \to \SH^\hol_Z(X;M)
			\]
			is an equivalence, where for any $Y \to S$, pointed pullback formalism $D$ on $\mathcal C$, and $M \in D(S)$, $D_Z(M)$ is defined as the fibre of the following map of pointed spaces
			\[
				D(Y;M) \to D(Y \times_S (S \setminus Z); M)
			.\]
			\footnotetext{This means that $X \times_S Z \to Z$ is an equivalence.}

		\item \label{itm:SHhol/ptwise}
			Objectwise description: For any finite group $G$ acting on a reduced complex space $X$ such that $X/G \in \mathcal C$, we have that the natural functor $H^\hol_\bullet(X/G) \to \SH^\hol(X/G)$ provided by \Cref{prp:univ prop of Hhol} is given by formally adjoining $\otimes$-inverses of objects of the form $\cls{(X \times V)/G}/\cls{(X \times V \setminus 0)/G}$, where $V$ is a finite-dimensional complex $G$-representation. 

		\item \label{itm:SHhol/generation}
			Generation: For any finite group $G$ acting on a reduced complex space $S$ such that $S/G \in \mathcal C$, we have that $\SH^\hol(S/G)$ is generated under sifted colimits by objects of the form
			\[
				\cls{X/G} \otimes (\cls{(X \times V)/G}/\cls{(X \times V \setminus 0)/G})^{\otimes -1}
			,\]
			where $X \to S$ is a $G$-equivariant submersion of complex $G$-spaces, $X$ is Stein, and $V$ is a finite-dimensional complex $G$-representation.

	\end{enumerate}
	\begin{proof}
		For each $S \in \Shv(\An_\comps)$, let $A_S$ be the collection of morphisms in $H^\hol(S)$ of the form $\cls{V \setminus 0} \to \cls{V}$ for $V$ a vector bundle on $S$, and write $(A_\bullet)_S$ for the collection of objects of the form $\cls{V}/\cls{V \setminus 0}$, as in \Cref{thm:ptd PF sheafy stabilization}.

		Then \Cref{rmk:hol resol fin gp} shows that for any finite group $G$ acting on a reduced Stein space $X$, the map $X/G \to \B G$ has $A_\bullet$-lifts, since for any $S \in \Shv(\An_\comps)$, and vector bundles $V,V'$ on $S$, we have that
		\[
			\cls{V \oplus V'}/\cls{V \oplus V' \setminus 0} \simeq \cls{V}/\cls{V \setminus 0} \wedge \cls{V'}/\cls{V' \setminus 0}
		\]
		in $H^\univ_\bullet(S)$.

		If $G$ is a finite group acting on a reduced complex space $X$, then \Cref{lem:hol pseudotop for finite groups} says that $X/G$ has an open cover by open subspaces of the form $(U \to X)/G$, where $U \to X$ is a $G$-invariant Stein open subspace. Since $G$-invariant Stein open subspaces are closed under finite intersections, it follows from \cref{itm:A-lifts/source-local} of \Cref{lem:A-lifts} that the map $X/G \to \B G$ has $A_\bullet$-lifts in the sense of \Cref{thm:PF sheafy stabilization}.

		Thus, if we let $\mathcal B \subseteq \mathcal C$ be the full subcategory of $\mathcal C$ consisting of objects of the form $X/G$ for $G$ a finite group acting on a reduced complex space $X$, then by \cref{itm:A-lifts/13} of \Cref{lem:A-lifts}, any representable map to an object of $\mathcal B$ has $A_\bullet$-lifts. Using the argument of \cite[Lemma 6.3]{sixopsequiv}, we can show that if $V$ is a vector bundle on an object $S$ of $\mathcal B$, then $\cls{V}/\cls{V \setminus 0}$ is symmetric in $H^\hol_\bullet(S)$, which shows that the objects of $\mathcal B$ are $A_\bullet$-good in the sense of \Cref{thm:PF sheafy stabilization}.

		By \Cref{prp:univ prop of Hhol} and \Cref{thm:D-topology}, $H^\hol_\bullet$ has $\tau_\NisC$-descent, so since every object of $\mathcal C$ has a quasi-admissible $\tau_\NisC$-cover by objects in $\mathcal B$, we may apply \Cref{thm:ptd PF sheafy stabilization} to produce $H^\hol \to \SH^\hol$ as the map $\Sigma_{A,+} : H^\hol \to H^\hol[A^{\wedge -1}]$.

		\begin{itemize}

			\item The universal property of $\SH^\hol$ follows from \Cref{prp:univ prop of Hhol}, and \cref{itm:ptd PF sheafy stabilization/univ} of \Cref{thm:ptd PF sheafy stabilization}, where we use \Cref{lem:hol pseudotop for finite groups} to see that the homotopy invariance condition given in this result in terms of projection maps $S \times \comps \to S$, is equivalent to the homotopy invariance condition given in \Cref{prp:univ prop of Hhol} in terms of vector bundle torsors. \Cref{itm:ptd PF sheafy stabilization/univ} of \Cref{thm:ptd PF sheafy stabilization} also shows that for any vector bundle $V \to S$ in $\mathcal C$, the object $\cls{V}/\cls{V \setminus 0}$ of $\SH^\hol(S)$ is $\otimes$-invertible.

			\item The descent statement follows either from \Cref{thm:D-topology} or \cref{itm:ptd PF sheafy stabilization/descent} of \Cref{thm:ptd PF sheafy stabilization}.

			\item \Cref{itm:SHhol/excision} follows immediately from the fact that $\SH^\hol$ has \'{e}tale excision.

			\item \Cref{itm:SHhol/ptwise} follows from \cref{itm:PF sheafy stabilization/ptwise} of \Cref{thm:PF sheafy stabilization} since the map $X/G \to \B G$ has $A_\bullet$-lifts.

			\item For \cref{itm:SHhol/generation}, first note that $H^\hol(S/G)$ is generated under sifted colimits by objects of the form $\cls{X/G}$ for $X \to S$ a $G$-equivariant submersion, but since $S$ is reduced, so is $X$. Since $G$ is finite, $X$ admits a $G$-invariant Stein open cover by \Cref{lem:hol pseudotop for finite groups}, so since Stein open subspaces are closed under finite intersections, $H^\hol(S/G)$ is generated under sifted colimits by objects of the form $\cls{X/G}$ for $X \to S$ a $G$-equivariant submersion, and $X$ Stein. Thus, we may conclude by \cref{itm:ptd PF sheafy stabilization/generation} of \Cref{thm:ptd PF sheafy stabilization}, since $S/G \to \B G$ has $A_\bullet$-lifts.

		\end{itemize}

	\end{proof}
\end{thm}

\appendix
\section{Quasi-projective Morphisms of Algebraic Stacks} \label{S:qproj}

The following definition of quasi-projectivity agrees with \cite[\S 14.3.4]{ChampsAlg}. Also see \cite[Theorem 8.6]{Rydh_2015}.
\begin{defn} \label{defn:qproj}
	A map $X \to Y$ of algebraic stacks is \emph{(quasi-)projective} if there is a finite type $\mathcal F \in \QC(Y)$, and a map over $Y$ from $X$ to $\proj_Y(\mathcal F)$ which is representable by (quasi-compact) closed immersions.

	% A map $X \to Y$ of derived algebraic stacks is (quasi-)projective if for any (classical) algebraic stack $Y' \to Y$, the map $X \times_Y Y' \to Y'$ is a (quasi-)projective map of algebraic stacks.
\end{defn}

The following lemma shows that any quasi-projective morphism can be factored as an open immersion followed by a projective map.
\begin{lem} \label{lem:qc imm}
	Let $f : X \to Y$ be a quasi-compact immersion of algebraic stacks. Then $f$ factors as a quasi-compact open immersion followed by a closed immersion.
	\begin{proof}
		By \stackscite{0CPU}, $f$ has a scheme-theoretic image, giving us a factorization $X \to Z \to Y$. Let $\tilde Y \to Y$ be a surjective smooth map, where $\tilde Y$ is a scheme. Thus, we have Cartesian squares
		\[
			\begin{tikzcd}
				\tilde X \ar[d] \ar[r] & \tilde Z \ar[d] \ar[r] & \tilde Y \ar[d] \\
				X \ar[r] & Z \ar[r] & Y
			\end{tikzcd}
		,\]
		and since $f : X \to Y$ is quasi-compact, \stackscite{0CMK} says that $\tilde Z \to \tilde Y$ is the scheme-theoretic image of the base change $\tilde f : \tilde X \to \tilde Y$ of $f$.

		By the definition of scheme-theoretic image, we know that $Z \to Y$ is a closed immersion. Note that since the diagonal of $Z \to Y$ is quasi-compact, and $X \to Y$ is quasi-compact, we have that $X \to Z$ is quasi-compact, so it only remains to show that it is an open immersion. By \stackscite{0503}, it suffices to show that $\tilde X \to \tilde Z$ is an open immersion. This follows from \stackscite{01RG}.
	\end{proof}
\end{lem}

% NOTE: also see eg 0C4N, or \cite[2.6]{sixopsequiv}
\begin{lem} \label{lem:map of qproj is qproj}
	Let
	\[
		X \xrightarrow{f} Y \xrightarrow{g} S
	\]
	be morphisms of qcqs algebraic stacks, and assume $g$ is (quasi-)projective. Then $f$ is (quasi-)projective if and only if $g \circ f$ is (quasi-)projective.
	\begin{proof}
		% NOTE: by 02KZ, finite type is fpqc local
		% NOTE: judging by 04XA, representable means representable by schemes

		Since $g$ is quasi-projective, its diagonal is a closed immersion, so it is projective. Thus, using the usual argument of factoring $f$ as the composite $X \to X \times_S Y \to Y$, we see that once we show that (quasi-)projective maps are closed under composition (the ``if'' direction), we will also have that if $g \circ f$ is (quasi-)projective, then so is $f$.

		By \cite[Theorem 8.6 (i)]{Rydh_2015}, we have that a representable morphism $A \to B$ of qcqs algebraic stacks is quasi-projective if and only if it is of finite type, and there exists an invertible $\mathcal O_A$-module that is ample on $A/B$. Furthermore, by \cite[Theorem 8.6 (ii)]{Rydh_2015}, a map of qcqs algebraic stacks is projective if and only if it is proper and quasi-projective. Thus, since proper maps are closed under composition, it suffices to show that if $f$ is quasi-projective, then $g \circ f$ is quasi-projective (where $g$ is already assumed to be quasi-projective). Since we already know that representable maps of finite type are closed under composition, we just need to show that if there exists a $g$-ample invertible $\mathcal O_Y$-module, and an $f$-ample invertible $\mathcal O_X$-module, then there is a $g \circ f$-ample invertible $\mathcal O_X$-module.

		% As noted in \cite[Definition 8.5]{Rydh_2015}, relative ampleness is fppf-local \cite[Corollary 2.7.2]{EGA4}, its definition for representable morphisms of stacks is deduced from the same notion for morphisms of schemes.
		Since $S$ is quasi-compact algebraic stack, by \stackscite{04YC}, there exist Cartesian squares
		\[
			\begin{tikzcd}
				\tilde X \ar[d, "p"'] \ar[r, "\tilde f"] & \tilde Y \ar[d, "q"] \ar[r, "\tilde g"] & \tilde S \ar[d] \\
				X \ar[r, "f"'] & Y \ar[r, "g"'] & S
			\end{tikzcd}
		,\]
		where the vertical arrows are surjective smooth maps from quasi-compact schemes. Since $f$ and $g$ are quasi-projective, there is an $f$-ample invertible $\mathcal O_X$-module $\mathcal L$, and a $g$-ample invertible $\mathcal O_Y$-module $\mathcal M$. By \stackscite{0C4K}, there is an $n \gg 0$ such that
		\[
			p^* \mathcal L \otimes \tilde f^* q^* \mathcal M^{\otimes n}
		\]
		is $\tilde g \circ \tilde f$-ample. Thus, $\mathcal N \coloneqq \mathcal L \otimes f^* \mathcal M^{\otimes n}$ is an invertible $\mathcal O_X$-module such that
		\[
			p^*(\mathcal N) \simeq p^* \mathcal L \otimes \tilde f^* q^* \mathcal M^{\otimes n}
		\]
		is $\tilde g \circ \tilde f$-ample, so $\mathcal N$ is $g \circ f$-ample by fpqc-locality of relative ampleness (see \cite[Corollary 2.7.2]{EGA4}).

	\end{proof}
\end{lem}

We recall the following definitions from \cite{SixAlgSt}:
\begin{defn}
	Following \cite[Definition 2.1(ii)]{SixAlgSt}, given an affine scheme $S$, say an fppf affine group scheme $G$ over $S$ is \emph{embeddable} if it is a closed subgroup of $\GL_S(\mathcal E)$ for some locally free sheaf $\mathcal E$ on $S$.

	Following \cite[Definition A.1(ii)]{SixAlgSt}, say an algebraic stack $X$ is \emph{linearly quasi-fundamental} if it admits a quasi-projective map to a stack $\B G$, for $G$ a linearly reductive embeddable group scheme over an affine scheme.
\end{defn}

Hoyois introduces a different notion of (quasi-)projectivity for global quotient stacks in \cite[Definition 2.5]{sixopsequiv}:
\begin{defn} \label{defn:G-qproj}
	Let $B$ be a scheme, and $G$ be a flat finitely presented group scheme over $B$. A \emph{$G$-scheme} is a scheme over $B$ equipped with a $G$-action. A $G$-equivariant morphism of $G$-schemes $f : X \to Y$ is said to be \emph{$G$-(quasi-)projective} if there is a finite locally free $G$-equivariant $\mathcal O_X$-module $\mathcal E$ and a closed (quasi-compact) immersion $X \to \proj(\mathcal E)$ over $Y$.
\end{defn}

We use the following definition from \cite[Definition 2.18]{SixAlgSt}:
\begin{defn} \label{defn:resolution property}
	Say a stack $X \in \AlgStk$ has the \emph{resolution property} if for any quasi-coherent module $\mathcal M$ on $X$, there is a surjection
	\[
		\bigoplus_\alpha \mathcal E_\alpha \to \mathcal M
	,\]
	where for each $\alpha$, $\mathcal E_\alpha$ is locally free of finite rank, and the direct sum is indexed over a small collection.
\end{defn}

\begin{rmk} \label{rmk:finite resolution property}
	If $X \in \AlgStk$ is quasi-compact and has the resolution property, and $\mathcal M$ is a finitely generated quasi-coherent module on $X$, then there is an $\mathcal O_X$-module $\mathcal E$ that is locally free of finite rank and a surjection $\mathcal E \to \mathcal M$.
	\begin{proof}
		Since $X$ is quasi-compact, \stackscite{04YC} says that there is a smooth surjective morphism $p : \Spec A \to X$. Since $X$ has the resolution property, we have a surjection
		\[
			\bigoplus_\alpha \mathcal E_\alpha \to \mathcal M
		,\]
		where for each $\alpha$, $\mathcal E_\alpha$ is locally free of finite rank.

		Applying $p^*$ to this surjection gives a surjection $\bigoplus_\alpha E_\alpha \to M$ of $A$-modules, where $M$ is finitely generated since $\mathcal M$ is. By picking a finite generating set for $M$, we can choose a finite collection of indices $\alpha_1, \dotsc, \alpha_n$ such that the restriction
		\[
			E_{\alpha_1} \oplus \dotsb \oplus E_{\alpha_n} \to M
		\]
		is surjective.

		Since $p$ is smooth surjective, and $p^*$ sends $\mathcal E_{\alpha_1} \oplus \dotsb \oplus \mathcal E_{\alpha_n} \to \mathcal M$ to the above surjection, this map is also a surjection.
	\end{proof}
\end{rmk}

\begin{lem} \label{lem:G-qproj vs qproj}
	Let $G$ be a linearly reductive group over an affine scheme $B$. Then a $G$-scheme $Y$ is $G$-(quasi-)projective in the sense of Definition \ref{defn:G-qproj} if and only if $Y/G \to B/G$ is (quasi-)projective, and in this case, if a $G$-equivariant map $X \to Y$ is $G$-(quasi-)projective then $(X \to Y)/G$ is (quasi-)projective.

	If $G$ is embeddable, then the converse of the last statement holds, so any (quasi-)projective map to a linearly quasi-fundamental stack is given by a closed (quasi-compact) immersion followed by the projection of a projective bundle.
	\begin{proof}
		Since $B$ is trivially $G$-projective, it suffices to show the statements about maps $X \to Y$ and $(X \to Y)/G$.

		Under the identification $\QC^G(Y) = \QC(Y/G)$, we have that $(\proj_Y(\mathcal E) \to Y)/G$ is $\proj_{Y/G}(\mathcal E) \to Y/G$. Furthermore, since closed (quasi-compact) immersions are local on the target by \stackscite{02L6} (\stackscite{02L8}), if $X \to \proj_Y(\mathcal E)$ is a $G$-equivariant closed (quasi-compact) immersion, then $(X \to \proj_Y(\mathcal E))/G$ is a closed (quasi-compact) immersion, so that $(X \to Y)/G$ is indeed (quasi-)projective.

		For the converse when $G$ is embeddable, note that if $(X \to Y)/G$ is (quasi-)projective, then we have that it factors as $X/G \to \proj_{Y/G}(\mathcal F) \to Y/G$ for some finite type $\mathcal F \in \QC(Y/G)$, where the first map is a (quasi-compact) closed immersion. Since $G$ is embeddable \cite[Remark A.2]{SixAlgSt} says that $Y$ has the resolution property, so by \Cref{rmk:finite resolution property}, there is a surjection $\mathcal E \to \mathcal F$ in $\QC(Y/G)$, where $\mathcal E$ is finite locally free. This gives a closed immersion $\proj_{Y/G}(\mathcal F) \to \proj_{Y/G}(\mathcal E)$ over $Y/G$, so we actually have a factorization $X/G \to \proj_{Y/G}(\mathcal E) \to Y/G$.

		The base change of this along $Y \to Y/G$ gives a factorization as required in Definition \ref{defn:G-qproj}.
	\end{proof}
\end{lem}

\begin{lem} \label{lem:desc lin qfund}
	An algebraic stack $S$ is linearly quasi-fundamental if and only if it is of the form $X/G$ for $G$ a linearly reductive embeddable group scheme over an affine scheme $S$, and $X$ a $G$-quasi-projective $G$-scheme.
	\begin{proof}
		By definition, we have that $S$ is linearly quasi-fundamental if and only if it admits a quasi-projective map to $\B G$ for some linearly reductive embeddable group scheme $G$ over an affine scheme $B$, so $S \simeq X/G$ for some $G$-scheme $X$ such that the quasi-projective map $S \to \B G$ is equivalent to $(X \to B)/G$. Since $G$ is embeddable, \Cref{lem:G-qproj vs qproj} says that the map $X \to B$ is $G$-quasi-projective, so $X$ is a $G$-quasi-projective $G$-scheme.

		For the converse, note that if $G$ is a linearly reductive embeddable group scheme over an affine scheme $B$, and $X$ is a $G$-quasi-projective $G$-scheme, so that the map $X \to B$ is $G$-quasi-projective, then \Cref{lem:G-qproj vs qproj} says that the map $(X \to B)/G$ is a quasi-projective map $X/G \to \B G$.
	\end{proof}
\end{lem}

\section{Finite Groups Acting on Hausdorff Spaces}

\begin{cnstr} \label{cnstr:equiv nbhd}
	Let $G$ be a topological group acting on a topological space $X$. Let $H \leq G$ be a subgroup, and let $\{N_{gH} \subseteq X\}_{gH \in G/H}$ be a family of subspaces.

	Define
	\[
		N \coloneqq \bigcap_{g \in G} g^{-1} N_{gH}
	,\]
	so for all $g \in G$, $gN \subseteq N_{gH}$. If $h \in H$, then (since the actions of $h$ on $G,X$ are invertible)
	\[
		hN = \bigcap_{g \in G} hg^{-1} N_{gH} = \bigcap_{g \in G} h(gh)^{-1} N_{ghH} = \bigcap_{g \in G} g^{-1} N_{gH} = N
	,\]
	so $N$ is $H$-invariant, and we can define a map
	\[
		\imath : G \times_H N \to X
	\]
	such that $\imath(g,x) = gx$.

	This construction satisfies some properties:
	\begin{enumerate}

		\item If $N_{g_1H} \cap N_{g_2H} = \emptyset$, then $g_1N \cap g_2N = \emptyset$.

		\item If $x \in N_1$ is stabilized by $H$, and $gx \in N_{gH}$ for all $g \in G$, then $x \in N$.

		\item If $G$ is finite, and $N_{gH}$ is open for all $g \in G$, then $N$ is open. If $G$ is discrete, then $\imath : G \times_H N \to X$ is $\coprod_{gH \in G/H} N \to X$ such that the map on the disjoint summand corresponding to $gH \in G/H$ is the map $N \xrightarrow{g} gN \to X$.

	\end{enumerate}
\end{cnstr}

\begin{lem} \label{lem:invariant nbhd basis Hausdorff}
	Let $G$ be a finite group acting on a topological space $X$. Let $\mathcal B$ be a neighborhood basis of a point $x \in X$ such that $\mathcal B$ is closed under finite intersections. If $x \in X$ is stabilized by $H \leq G$, then the collection of $H$-invariant neighborhoods of $x$ in $\mathcal B$ is a neighborhood basis of $x$ in $X$.

	If $X$ is also Hausdorff, $G$ is discrete, and $H$ is equal to the stabilizer of $x$, then the collection of those $H$-invariant neighborhoods $U$ of $x$ in $\mathcal B$ such that $G \times_H U \to X$ is an open embedding, is a neighborhood basis.
	\begin{proof}
		Suppose $x \in X$ is stabilized by $H \leq G$. Since $\mathcal B$ is a neighborhood basis of $x$, it suffices to show that if $U' \in \mathcal B$ is a neighborhood of $x$, then there is a $U \in \mathcal B$ that is an $H$-invariant neighborhood of $x$ contained in $U'$. Since $H$ is finite, and $\mathcal B$ is closed under finite intersections, we can take
		\[
			U = \bigcap_{h \in H} hU'
		.\]
		Note this contains $x$ since $x$ is fixed by the action of $H$.

		For the second statement, when $X$ is Hausdorff, and $H$ is the stabilizer of $x$, since $G/H$ is finite, we may choose a (finite) pairwise disjoint family $\{U'_{gH}\}_{gh \in G/H}$ such that $U'_{gH}$ is an open neighborhood of $gx$ for all $g \in G$, and $U'_H \subseteq U'$ is an open subneighborhood.

		Since $\mathcal B$ is a neighborhood basis of $x$, we have that $g \mathcal B$ is a neighborhood basis of $gx$ for each $g \in G$, so for $gH \in G/H$, we can pick an open subneighborhood $U_{gH} \subseteq U'_{gH}$ of $gx$ such that $g^{-1} U_{gH} \in \mathcal B$.

		Now, since $\mathcal B$ is closed under finite intersections, applying Construction \ref{cnstr:equiv nbhd} to the family $\{U_{gH}\}_{gH \in G/H}$ produces an $H$-invariant neighborhood $U$ of $x$ that is in $\mathcal B$, and such that $\{gU\}_{gH \in G/H}$ are pairwise disjoint. Since $G$ is discrete, the map $G \times_H U \to X$ is equivalent to
		\[
			\coprod_{gH \in G/H} U \to X
		,\]
		which is equivalent to
		\[
			\bigcup_{gH \in G/H} gU \to X
		,\]
		which is an open embedding.
	\end{proof}
\end{lem}

\begin{lem} \label{lem:existence of equivariant local sections Hausdorff}
	Let $G$ be a finite discrete group acting on Hausdorff topological spaces $X,S$, and let $p : X \to S$ be a $G$-equivariant continuous map. Suppose $p$ is a local homeomorphism at $x \in X$, and the stabilizer of $x$ equals the stabilizer of $p(x)$. Then there is a $G$-invariant open neighborhood of $x$ where $p$ is an open embedding.
	\begin{proof}
		% NOTE: need local homeomorphism in this case
		% if only take submersion, then get subspace $U'$ of $X$ on which $p$ is open embedding
		% but then we need to intersect it with $hU'$, which gives a not-necessarily open subspace of $U'$, so the restriction of $p$ to this may stop being an open embedding

		Let $H \leq G$ be the stabilizer of $x$. Since $p$ is a local homeomorphism at $x$, we have a neighborhood basis $\mathcal B$ of $x$ consisting of those open neighborhoods $U$ of $x$ such that $p|_U$ is an open embedding. Note that $\mathcal B$ is closed under finite intersections. By Lemma \ref{lem:invariant nbhd basis Hausdorff}, we get an $H$-invariant open neighborhood $U$ of $x$ such that $p|_U$ is an open embedding.

		In particular, for $g \in G \setminus H$, we have that $gU$ is disjoint from $U$. Furthermore, since $p$ is $G$-equivariant, we have a commutative diagram
		\[
			\begin{tikzcd}
				U \ar[d, "g"'] \ar[r] & X \ar[d, "g"] \ar[r, "p"] & S \ar[d, "g"] \\
				gU \ar[r] & X \ar[r, "p"'] & S
			\end{tikzcd}
		,\]
		so since the top composite is an open embedding, and the vertical arrows are homeomorphisms, we have that the bottom composite, which is $p|_{gU}$ is an open embedding.

		By Lemma \ref{lem:invariant nbhd basis Hausdorff}, since $H$ is also the stabilizer of $p(x)$ we have an $H$-invariant open neighborhood $V$ of $p(x)$ that is contained in $p(U)$, and such that $G \times_H V \to S$ is an open embedding.

		Let $U' = U \cap p^{-1}(V)$, so $U'$ is an $H$-invariant open neighborhood $x$ that is contained in $U$, whence $G \times_H U' \to X \to S$ is an open embedding. In particular $G \times_H U'$ is a $G$-invariant open neighborhood of $x$ on which $p$ is an open embedding.
	\end{proof}
\end{lem}

\section{Localizations}

\begin{lem} \label{lem:localizing adjunctions}
	Let $F : \mathcal C \rightleftharpoons \mathcal D : G$ be functors, and $S,T$ classes of morphisms in $\mathcal C, \mathcal D$ respectively such that $F(S) \subseteq T$ and $G(T) \subseteq S$. Write $L_S, L_T$ for the localization functors, and $\bar F : \mathcal S^{-1} \mathcal C \rightleftharpoons T^{-1} \mathcal D : \bar G$ for the induced functors on the localizations (so $\bar F \circ L_S \simeq L_T \circ F$, and similarly for $\bar G$). If $\eta, \epsilon$ are the unit and counit of an adjunction $F \dashv G$, write $\bar \eta$ and $\bar \epsilon$ for the corresponding $\id \to \bar G \bar F$ and $\bar F \bar G \to \id$ (so $\bar \eta L_S \simeq L_S \bar \eta$ and $\bar \epsilon L_T \simeq L_T \bar \epsilon$). Then $\bar \eta, \bar \epsilon$ are the unit and counit of an adjunction $\bar F \dashv \bar G$.
	\begin{proof}
		We need to show that if $\eta$ is a unit of an adjunction $F \dashv G$, then $\bar \eta$ is a unit of an adjunction $\bar F \dashv \bar G$, and similarly for $\epsilon$ and $\bar \epsilon$.

		First, note that since $F(S) \subseteq T$, it follows that $G$ sends $T$-local objects to $S$-local objects, and indeed, the restriction of $G$ to $T$-local objects coincides with $\bar G$.

		Now, let $\epsilon : FG \to \id$ with corresponding $\bar \epsilon : \bar F \bar G \to \id$, and consider the following commutative diagram for $X \in S^{-1}(\mathcal C)$ and $Y \in T^{-1} \mathcal D$:
		\[
			\begin{tikzcd}
				\mathcal C(X,GY) \ar[d] \ar[r] & \mathcal D(FX, FGY) \ar[d] \ar[r] & \mathcal D(FX, Y) \ar[d] \\
				S^{-1} \mathcal C(X, GY) \ar[d, equals] \ar[r] & T^{-1} \mathcal D(L_T FX,  L_T FGY) \ar[d, equals] \ar[r] & T^{-1} \mathcal D(L_T FX, L_T Y) \ar[d, equals] \\
				S^{-1} \mathcal C(X, \bar G Y) \ar[r] & T^{-1} \mathcal D(\bar F X,  \bar F \bar G Y) \ar[r] & T^{-1} \mathcal D(\bar F X, \bar F \bar G Y)
			\end{tikzcd}
		.\]
		Assuming the top row composes to an equivalence, we would like to show the bottom row does too. Indeed, since $Y$ is $T$-local, the top right arrow is invertible, and since $X$ and $G Y$ are $S$-local, so is the top left arrow. It follows that the middle row composes to an equivalence, whence the bottom row does too.

		For the unit, consider
		\[
			\begin{tikzcd}
				\mathcal D(FX, Y) \ar[d] \ar[r] & \mathcal C(GFX, GY) \ar[d] \ar[r] & \mathcal C(X, GY) \ar[d] \\
				T^{-1} \mathcal D(L_T FX, Y) \ar[d, equals] \ar[r] & S^{-1} \mathcal C(L_S G FX,  G Y) \ar[d, equals] \ar[r] & S^{-1} \mathcal C(L_S X, GY) \ar[d, equals] \\
				T^{-1} \mathcal D(\bar F X, Y) \ar[r] & S^{-1} \mathcal C(\bar G \bar F X,  \bar G Y) \ar[r] & S^{-1} \mathcal C(X, \bar GY )
			\end{tikzcd}
		.\]
		Since $Y$ is $T$-local, $GY$ is $S$-local, so the top middle and top right arrows are invertible. Since $Y$ is $T$-local, the top left arrow is also invertible, so if the top row is an equivalence, so is the bottom row.
	\end{proof}
\end{lem}

\begin{lem} \label{lem:compatible localizations}
	Let $F : \mathcal C \to \mathcal D$ be a functor, let $\mathcal C' \subseteq \mathcal C$ be a reflective subcategory,
	and let $\mathcal D' \subseteq \mathcal D$ be a reflective subcategory containing $F(\mathcal C')$. The following are equivalent:
	\begin{enumerate}

		\item The square
			\[
				\begin{tikzcd}
					\mathcal C' \ar[d] \ar[r] & \mathcal C \ar[d, "F"] \\
					\mathcal D' \ar[r] & \mathcal D
				\end{tikzcd}
			\]
			is left adjointable.

		\item The functor $F$ sends $\mathcal C'$-equivalences to $\mathcal D'$-equivalences.

		\item For every $X \in \mathcal C$,
		if $X \to X'$ is the localization map for $X$, then $F(X \to X')$ is the localization map for $F(X)$.

	\end{enumerate}
	\begin{proof}
		Write $L : \mathcal C \to \mathcal C'$ and $M : \mathcal D \to \mathcal D'$ for the localization maps.

		Suppose the square is left adjointable. If $\phi$ is an $L$-equivalence in $\mathcal C$,
		then $MF(\phi)$ is $FL(\phi)$ is $F$ of an equivalence,
		so it is an equivalence,
		so $F(\phi)$ is an $M$-equivalence as desired.

		Next, suppose $F$ sends $L$-equivalences to $M$-equivalences. Then $F(X \to X')$ is an $M$-equivalence, and we already know that $F$ sends $\mathcal C'$ to $\mathcal D'$, so $F(X') \in \mathcal D'$ as desired.

		Finally, suppose that for $X \in \mathcal C$, if $X \to X'$ is a localization map for $X$,
		then $F(X \to X')$ is the localization map for $F(X)$. Consider the left base change morphism of the square
		\[
			\begin{tikzcd}
				\mathcal C' \ar[d, "F'"'] \ar[r, "\imath"] & \mathcal C \ar[d, "F"] \\
				\mathcal D' \ar[r, "\jmath"'] & \mathcal D
			\end{tikzcd}
		\]
		which is given by
		\[
			MF \to MF\imath L \simeq M\jmath F' L \to F' L
		.\]
		At $X \in \mathcal C$, this is
		\[
			MF(X) \xrightarrow{MF(X \to X')} MFX' \simeq MF'X' \simeq F'X'
		.\]
		The first arrow is invertible since $F(X \to X')$ is an $M$-equivalence.
	\end{proof}
\end{lem}

\subsection{Locally Cartesian Localizations} \label{S:LCL}

In this section we will prove some results about locally Cartesian localizations. See \cite[\S1]{LCC} for a more in depth discussion of this notion. This is a generalization of the notion of left-exact localizations such as sheafifications, which also includes other important examples, for example, if $\mathcal C$ is a small category and $\mathcal A$ is a set of morphisms in $\mathcal C$ that is stable under base change, then \cite[Proposition 3.4]{sixopsequiv} says that the full subcategory of $\Psh(\mathcal C)$ consisting of presheaves $P$ such that $P$ sends $\mathcal A$ to equivalences is a locally Cartesian localization of $\Psh(\mathcal C)$.

In fact, we will give a very general criterion to check that localizations are locally Cartesian in \Cref{prp:crit for LCL}. In particular, a localization of presentable categories is locally Cartesian if it is given by localizing along a small collection of morphisms that is stable under base change.

We begin by recalling the definition. Although we will generally only consider these localizations in the context of presentable categories, it is not always necessary to assume our categories are presentable for the purposes of this section, so we will drop this assumption from the definition considered in \cite[\S1]{LCC}.
% NOTE: could make some generalization where don't assume that the base changes always exist, but it's complicated and I probably will only consider this for presentable caetgories.
\begin{defn}
	A \emph{locally Cartesian localization} of a category $\mathcal C$ is a functor $L : \mathcal C \to L\mathcal C$ that is a left adjoint of the inclusion of a full subcategory $L \mathcal C \to \mathcal C$, and such that $L$-equivalence are stable under base change along maps between local objects, \ie, if $S' \to S$ is a map in $L\mathcal C$, and
	\[
		X \xrightarrow{f} Y \to S
	\]
	are maps in $\mathcal C$ such that $L(f)$ is an equivalence, then $L(S' \times_S f)$.
\end{defn}
% PERF: give other characterization, in which $Y \to S$ is the identity

Note that every left exact localization of a category that has finite limits is locally Cartesian by \cite[Proposition 6.2.1.1]{htt}, so sheafification functors are locally Cartesian localizations.

\begin{lem} \label{lem:char LCL}
	Given a localization functor $L : \mathcal C \to L\mathcal C$, if $\mathcal C$ has pullbacks, then the following are equivalent
	\begin{enumerate}

		\item For any map $Y' \to Y$ in $L\mathcal C$, and map $X \to Y$ in $\mathcal C$, the square
			\[
				\begin{tikzcd}
					L(X \times_Y Y') \ar[d] \ar[r] & Y' \ar[d] \\
					LX \ar[r] & Y
				\end{tikzcd}
			\]
			is Cartesian.

		\item For any map $Y' \to Y$ in $L\mathcal C$, and map $X \to Y$ in $\mathcal C$, the base change of $X \to LX$ along $Y' \to Y$ is an $L$-equivalence.

		\item $L$ is locally Cartesian.

	\end{enumerate}
	
	\begin{proof}
		For any maps $Y' \to Y \gets X$ as above, we have the following commutative diagram
		\begin{equation} \label{eqn:lcl rect}
			\begin{tikzcd}
				X \times_Y Y' \ar[d] \ar[r] & L(X \times_Y Y') \ar[d] \ar[r] & Y' \ar[d] \\
				X \ar[r] & LX \ar[r] & Y
			\end{tikzcd}
		\end{equation}
		Since the outer rectangle is always Cartesian, if the right square is Cartesian, then so is the left. This means that $L$-equivalences of the form $X \to LX$ are preserved by base change along maps between objects of $L \mathcal C$, so the first condition implies the second.

		Now, for any $L$-equivalence $X_0 \to X_1$ in $\mathcal C$, we have that the composite $X_0 \to X_1 \to LX_1$ is an $L$-equivalence to an object of $L\mathcal C$, so it is equivalent to $X_0 \to LX_0$. Now, for any map $Y' \to Y$ in $L\mathcal C$, we have that any map $X_1 \to Y$ factors through $X_1 \to LX_1$, so we have that $(X_0 \to X_1) \times_Y Y'$ composed with $(X_1 \to LX_1) \times_Y Y'$ is equivalent to $(X_0 \to LX_0) \times_Y Y'$, and these latter two maps are $L$-equivalences under the second condition, so that $(X_0 \to X_1) \times_Y Y'$ is an $L$-equivalence. Thus, the second condition implies that $L$ is locally Cartesian.

		Finally, assume that $L$ is locally Cartesian, and let $Y' \to Y \gets X$ be as above. Then we have a commutative diagram
		\[
			\begin{tikzcd}
				X \times_Y Y' \ar[d] \ar[r] & LX \times_Y Y' \ar[d] \ar[r] & Y' \ar[d] \\
				X \ar[r] & LX \ar[r] & Y
			\end{tikzcd}
		\]
		where all of the squares are Cartesian. Now $LX \times_Y Y'$ is in $L\mathcal C$, and the map $X \times_Y Y' \to LX \times_Y Y'$ is an $L$-equivalence since it is a base change of $X \to LX$ along $Y' \to Y$, and $L$ is locally Cartesian. Therefore the map $X \times_Y Y' \to LX \times_Y Y'$ lifts to an equivalence $L(X \times_Y Y') \to LX \times_Y Y'$, which shows that the third condition implies the first, as required.
	\end{proof}
\end{lem}

For a general localization functor $L : \mathcal C \to L\mathcal C$, we always have that $L$ is essentially surjective, but it is not clear if it is surjective on morphisms in any sense. Even if $L$ is a sheafification functor, it may not induce effective epimorphisms on mapping spaces. The following result shows that locally Cartesian localizations are ``locally essentially surjective'' in that they induce essentially surjective functors on slice categories:
\begin{lem} \label{lem:LCL implies ess surj on slices}
	Let $L : \mathcal C \to L\mathcal C$ be a locally Cartesian localization functor. Then for any $S \in \mathcal C$, the functor $\mathcal C_{/S} \to L\mathcal C_{/LS}$ is essentially surjective.
	\begin{proof}
		If $S' \to LS$ is any object of $L\mathcal C_{/LS}$, then since $S \to LS$ is an $L$-equivalence, we have a Cartesian square
		\[
			\begin{tikzcd}
				S'' \ar[d] \ar[r] & S \ar[d] \\
				S' \ar[r] & LS
			\end{tikzcd}
		\]
		where the leftmost vertical map is an $L$-equivalence. It follows that $S' \to LS$ is equivalent to $L(S'' \to S)$.
	\end{proof}
\end{lem}

Although locally Cartesian localizations need not preserve all finite limits, we can give an easy criterion for when they preserve finite products:
\begin{lem} \label{lem:crit for LCL to be Cart mon}
	Let $L : \mathcal C \to L\mathcal C$ be a locally Cartesian localization between categories that admit small colimits, and assume that $L\mathcal C$ has a terminal object. Then $L$ preserves finite products if and only if for any objects $X,Y \in \mathcal C$, the map $L(X \times (Y \to LY))$ is an equivalence, \ie, $X \times -$ preserves $L$-equivalences. 
	\begin{proof}
		Since $L\mathcal C \to \mathcal C$ preserves limits, we have that $\mathcal C$ has a terminal object which is already in $L\mathcal C$, so $L$ preserves terminal objects, and it only remains to consider the condition that $L$ preserves binary products.

		Indeed, $L$ preserves binary products if and only if for any $X,Y \in \mathcal C$, the map $X \times Y \to LX \times LY$ is an $L$-equivalence. Note that this map is equivalent to the following composite:
		\[
			X \times Y \to X \times LY \to LX \times LY
		.\]
		The second map is the base change of $X \to LX$ along $LY \to \pt$, which is a map between objects of $L\mathcal C$. Since $L$ is locally Cartesian, it follows that this map is an $L$-equivalence, so $L$ preserves binary products if and only if
		$X \times (Y \to LY)$ is an $L$-equivalence for all $X,Y \in \mathcal C$.
	\end{proof}
\end{lem}

It can be difficult to check that the localization along some prescribed set of morphisms is left-exact. Indeed, even if the set of morphisms is stable under base change, the strongly saturated class it generates might not be. The following result shows that even though such a localization might not be left-exact, it is always locally Cartesian:
\begin{prp} \label{prp:crit for LCL}
	Let $\mathcal C$ be a presentable category with universal colimits, and let $W$ be a small collection of morphisms in $\mathcal C$.

	Suppose that base changes along maps between $W$-local objects send maps in $W$ to $W$-equivalences. More precisely, we assume that for any map $f : X \to Y$ in $W$, if $S' \to S$ is a map between local objects, then for any map $Y \to S$, we have that $f \times_S S'$ is a $W$-equivalence.

	Then $W^{-1} \mathcal C$ is a presentable category that has universal colimits, and the inclusion $W^{-1} \mathcal C \subseteq \mathcal C$ admits an accessible left adjoint $L$ which is a locally Cartesian localization.

	Furthermore $L$ preserves finite products if and only if for any object $X \in \mathcal C$, and map $f \in W$, we have that $X \times f$ is a $W$-equivalence.
\end{prp}

Before addressing the proof of \Cref{prp:crit for LCL}, we will need the following \namecref{lem:LC strongly sat}:
\begin{lem} \label{lem:LC strongly sat}
	Let $\mathcal C$ be a category with universal colimits,\footnote{In fact, we only need that colimits are preserved by base changes along maps between $W$-local objects.}
	and let $W$ be a strongly saturated class of morphisms in $\mathcal C$ such that the inclusion of $W$-local objects $W^{-1} \mathcal C \subseteq \mathcal C$ admits a left adjoint $L$.

	Write $W^\circ$ for the collection of maps $f : X \to Y$ in $\mathcal C$ such that for any map $S' \to S$ between $W$-local objects, and map $Y \to S$, we have that $f \times_S S' \in W$. Then $W^\circ$ is a strongly saturated class contained in $W$.
	\begin{proof}
		First note that $W^\circ \subseteq W$, since if $f : X \to Y$ is in $W^\circ$, then the base change of $f$ along $\id_{L(Y)}$ is in $W$.

		We show that $W^\circ$ satisfies \cite[Definition 5.5.4.5]{htt}:
		\begin{enumerate}

			\item Suppose that $f' : X' \to Y'$ is a cobase change of a map $f \in W^\circ$. Let $T \to S$ be a map between $W$-local objects, and let $Y' \to S$ be a map. Since $\mathcal C$ has universal colimits, we have that $f' \times_S T$ is a cobase change of $f \times_S T$. Since $f \in W^\circ$, we have that $f \times_S T \in W$, so $f' \times_S T \in W$ since $W$ is strongly saturated.

			\item To show that $W^\circ$ is stable under small colimits in $\Fun(\Delta^1, \mathcal C)$, we must show that if $f : X \to Y$ is a transformation between small diagrams $X,Y : K \to \mathcal C$, and for all $a \in K$, $f(a) \in W^\circ$, then $\varinjlim f \in W^\circ$. Indeed, for any map $S' \to S$ between $W$-local objects, if $\varinjlim Y \to S$ is any map, then since $\mathcal C$ has universal colimits, we have that $\varinjlim f \times_S S'$ is equivalent to $\varinjlim (f \times_S S')$, but for all $a \in K$, since $f(a) \in W^\circ$, it follows that $f(a) \times_S S' \in W$, so since $W$ is strongly saturated, it follows that $\varinjlim f \times_S S' \in W$.

			\item Let
				\[
					X \xrightarrow{f} Y \xrightarrow{g} Z
				\]
				be maps in $\mathcal C$. We must show that if 2 out of the 3 maps $f,g, g \circ f$ are in $W^\circ$, then so is the third.

				Let $S' \to S$ be a map between $W$-local objects. If $Z \to S$ is any map then $(g \circ f) \times_S S'$ is a composite of $f \times_S S'$ with $g \times_S S'$. This immediately show that if $f,g$ or $f, g \circ f$ are in $W^\circ$, then the third map is also in $W^\circ$, since $W$ is strongly saturated.

				Now, suppose that $g$ and $g \circ f$ are in $W^\circ$. We have already observed that then $g \in W$, so we have a commutative square
				\[
					\begin{tikzcd}
						Y \ar[d] \ar[r, "g"] & Z \ar[d] \\
						L(Y) \ar[r, "L(g)"'] & L(Z)
					\end{tikzcd}
				,\]
				where the bottom map is an equivalence, so $Y \to L(Y)$ factors through $g$ as
				\[
					Y \xrightarrow{g} Z \to L(Z) \xrightarrow{L(g)^{-1}} L(Y)
				.\]

				If $Y \to S$ is any map, then since $S \in W^{-1} \mathcal C$, it extends through $Y \to L(Y)$, and therefore through $g$. Thus, as before, we find that the $f \times_S S' \in W$ since $(g \times_S S') \circ (f \times_S S') \simeq (g \circ f) \times_S S'$, and $g \times_S S'$ and $(g \circ f) \times_S S'$ are in $W$, which is strongly saturated.

		\end{enumerate}
		
	\end{proof}
\end{lem}

\begin{proof}[Proof of \Cref{prp:crit for LCL}]
	It follows from \cite[Proposition 5.5.4.15]{htt} that the left adjoint $L$ exists, that $W^{-1} \mathcal C$ is presentable, and that the collection $\bar W$ of $W$-equivalences, \ie maps $f$ such that $L(f)$ is an equivalence, is the smallest strongly saturated class containing $W$. Furthermore, \cite[Proposition 5.5.4.2]{htt} shows that $L$ is accessible.

	\Cref{lem:LC strongly sat} shows that if $\bar W^\circ$ is the collection of maps $f : X \to Y$ such that for any map $S' \to S$ in $W^{-1} \mathcal C$, and map $Y \to S'$, $L(f \times_S S')$ is an equivalence, then $\bar W^\circ$ is a strongly saturated class contained in $\bar W$. By our assumptions on $W$, we have that $W \subseteq \bar W^\circ$, so $\bar W \subseteq \bar W^\circ$. Thus, $\bar W = \bar W^\circ$, so $\bar W$ is stable under base change along maps between local objects, whence $L$ is locally Cartesian.

	The fact that $W^{-1} \mathcal C$ has universal colimits then follows from \cite[Proposition 1.4]{LCC} and the fact that $\mathcal C$ has universal colimits.

	For the statement about finite products, note that \Cref{lem:crit for LCL to be Cart mon} says that $L$ preserves finite products if and only if for all $X \in \mathcal C$, $X \times -$ preserves $W$-equivalences. Clearly this implies that $X \times -$ sends maps in $W$ to $W$-equivalences. For the converse, note that if $X \times -$ sends maps in $W$ to $W$-equivalences, then since the collection of $W$-equivalences is a strongly saturated class, \cite[Remark 5.5.4.10]{htt} says that $X \times -$ sends the strongly saturated class generated by $W$ to $W$-equivalences, but this class is precisely the collection of $W$-equivalences.
\end{proof}

\subsection{Monoidal structures}

The main goal of this section is to prove the following result:
\begin{thm} \label{thm:monoidal objectwise localization}
	Let $\mathcal O$ be an operad, $F : K \to \Alg_{\mathcal O} \PrL$ be a functor, and for each $p \in K$ and $X \in \mathcal O$ let $W_p(X)$ be a class of morphisms in $F(p)(X)$ that contains all identities and such that the strongly saturated class generated by $W_p(X)$ is of small generation. Suppose that
	\begin{itemize}

		\item for $\phi : p \to q$ in $K$, $X \in \mathcal O$, and $f \in W_p(X)$, we have that $F(\phi)(X)(f)$ is in the strongly saturated class generated by $W_q(X)$, and

		\item for any $f \in \Mul_{\mathcal O}(\{X_i\}_{i = 1}^n, Y)$, and $\{g_i \in W_p(X_i)\}_{i = 1}^n$, we have that $\otimes_f \{g_i\}_{i = 1}^n$ is in the strongly saturated class generated by $W_p(Y)$.
			
	\end{itemize}
	Then there is a map $F \to W^{-1} F$ in $\Fun(K, \Alg_{\mathcal O} \PrL)$ such that for each $p \in K$ and $X \in \mathcal O$, the morphism $(F \to W^{-1} F)(p)(X)$ is an accessible localization functor along the collection $W_p(X)$ of morphisms in $F(p)(X)$.

	Furthermore, for any other map $F \to G$, the space of extensions of this map through $F \to W^{-1} F$ is nonempty if and only if for each $p \in K$, and $X \in \mathcal O$, the functor $(F \to G)(p)(X)$ sends maps in $W_p(X)$ to equivalences in $G(p)(X)$, in which case this space is contractible.
\end{thm}

\begin{rmk} \label{rmk:coCart fib of operads}
	Recall that if $\mathcal O$ is an operad, then a coCartesian fibration $\mathcal C^\otimes \to \mathcal O^\otimes$ is a coCartesian fibration of operads (\cite[Definition 2.1.2.13]{ha}) if the following equivalent conditions are satisfied:
	\begin{itemize}

		\item The composite $\mathcal C^\otimes \to \mathcal O^\otimes \to N(\Fin_*)$ is an operad

		\item For any object
			\[
				T \simeq T_1 \oplus \dotsb \oplus T_n \in \mathcal O^\otimes_{\langle n \rangle}
			,\]
			the inert morphisms $T \to T_1, \dotsc, T_n$ induce an equivalence
			\[
				\mathcal C^\otimes_T \to \mathcal C^\otimes_{T_1} \times \dotsb \times \mathcal C^\otimes_{T_n}
			.\]
			
	\end{itemize}
\end{rmk}

\begin{prp} \label{prp:monoidal localization}
	Let $\mathcal O$ be an operad, and $L : \mathcal C \to \mathcal D$ a map in $\Mon_{\mathcal O}(\Cat)$. Then $L^\otimes : \mathcal C^\otimes \to \mathcal D^\otimes$ admits a fully faithful right adjoint relative $\mathcal O^\otimes$ (\cite[Definition 7.3.2.2]{ha}) if and only if for every $X \in \mathcal O$, the functor $L(X) : \mathcal C(X) \to \mathcal D(X)$ admits a fully faithful right adjoint.

	Furthermore, in this case, for any $\mathcal E \in \Mon_{\mathcal O}(\Cat)$, we have that\footnote{See \cite[Definition 2.1.3.7]{ha} for notation.}
	\[
		\Fun_{\mathcal O}^\otimes(L, \mathcal E) : \Fun_{\mathcal O}^\otimes(\mathcal D, \mathcal E) \to \Fun_{\mathcal O}^\otimes(\mathcal C, \mathcal E)
	\]
	is fully faithful with essential image given by those $F : \mathcal C \to \mathcal E$ such that for all $X \in \mathcal O$ and $g \in \mathcal C(X)$, if $L(X)(g)$ is invertible, so is $F(X)(g)$. 
	\begin{proof}
		Note that if $f \in \Mul_{\mathcal O}(\{X_1, \dotsc, X_n\}, Y)$ (see \cite[Definition 2.1.1.16]{ha} for notation), and $\{g_i \in \mathcal C(X_i)\}_{i = 1}^n$ such that $L(X_i)(g_i)$ is invertible for each $i$, we have that
		\[
			L(Y)(\otimes_f \{g_i\}_{i = 1}^n) \simeq \otimes_f \{L(X_i)(g_i)\}_{i = 1}^n
		\]
		is invertible. It follows that if $L(X)$ admits a fully faithful right adjoint for every $X \in \mathcal O$, \cite[Proposition 2.2.1.9 (1)]{ha} says $L^\otimes$ admits a fully faithful right adjoint relative $\mathcal O^\otimes$.

		Conversely, if $L^\otimes$ admits a fully faithful right adjoint relative $\mathcal O^\otimes$, then \cite[Proposition 7.3.2.5]{ha} says that for each $X \in \mathcal O \subseteq \mathcal O^\otimes$, the functor $L(X)$ admits a right adjoint given by a base change of the right adjoint of $L^\otimes$. Since the right adjoint of $L^\otimes$ is fully faithful, so is the right adjoint of $L(X)$.

		Now, assume the equivalent conditions hold. Since $L^\otimes$ is a localization functor, \cite[Proposition 5.2.7.12]{htt} says that
		\[
			\Fun_{\mathcal O^\otimes}(L^\otimes, \mathcal E^\otimes)
		\]
		is fully faithful with essential image given by those $F : \mathcal C^\otimes \to \mathcal E^\otimes$ such that for $g \in \mathcal C^\otimes$ if $L^\otimes(g)$ is invertible, so is $F(g)$.

		Recall from \cite[Definition 2.1.3.7]{ha} that for $\mathcal O$-monoidal categories $\mathcal E_0^\otimes, \mathcal E_1^\otimes$, $\Fun_{\mathcal O}^\otimes(\mathcal E_0, \mathcal E_1)$ is the full subcategory of $\Fun_{\mathcal O^\otimes}(\mathcal E_0^\otimes, \mathcal E_1^\otimes)$ of those maps preserving inert morphisms and coCartesian morphisms. Since $L^\otimes$ is $\mathcal O$-monoidal, the above fully faithful functor restricts to a fully faithful functor
		\[
			\Fun_{\mathcal O}^\otimes(\mathcal D, \mathcal E) \to \Fun_{\mathcal O}^\otimes(\mathcal C, \mathcal E)
		.\]
		Now, if $g \in \mathcal C^{\otimes}$ lies over $\{X_1, \dotsc, X_n\}$ in $\mathcal O^\otimes$, so that $g = g_1 \oplus \dotsb \oplus g_n$, we have
		\[
			L^\otimes(g) \simeq L_{X_1}(g_1) \oplus \dotsb \oplus L_{X_n}(g_n)
		,\]
		and this is invertible if $L_{X_i}(g_i)$ is invertible for each $i$. Using a similar description for $F \in \Fun_{\mathcal O}^\otimes(\mathcal C, \mathcal E)$, we find that $F$ is in the essential image if and only if for any $X \in \mathcal O$ and $g \in \mathcal C(X)$, if $L_X(g)$ is invertible, then $F(X)(g)$ is invertible.
	\end{proof}
\end{prp}

\begin{defn} \label{defn:monoidal localization}
	Let $\mathcal O$ be an operad, and $L : \mathcal C \to \mathcal D$ a morphism in $\Mon_{\mathcal O}(\Cat)$. Say $L$ is an $\mathcal O$-monoidal localization if it satisfies the equivalent conditions of Proposition \ref{prp:monoidal localization}.

	Given a family $W = \{W_X\}_{X \in \mathcal O}$ such that for each $X \in \mathcal O$, $W_X$ is a collection of morphisms in $\mathcal C(X)$, say $L$ is an $\mathcal O$-monoidal localization along $W$ if for each $X \in \mathcal O$, the functor $L(X)$ is a localization along $W_X$, that is, it is a localization functor and the class of morphisms inverted by $L(X)$ is the strongly saturated class of morphisms generated by $W_X$ (see \cite[Remark 5.5.4.10]{htt}).
\end{defn}

% NOTE: difference between operad maps and monoidal functors is HA 2.1.2.7 vs 2.1.3.7
% also see 2.4.2.5, 2.4.2.6

The following definition follows \cite[Definition 3.1.1.18]{ha}:
\begin{defn} \label{defn:colimit-compatibility for monoidal categories}
	Let $\mathcal O$ be an operad, and $K$ a simplicial set. Say an $\mathcal O$-monoidal category $\mathcal C$ is compatible with $K$-indexed colimits if for any $f \in \Mul_{\mathcal O}(\{X_1, \dotsc, X_n\}, Y)$, the functor
	\[
		\otimes_f : \mathcal C(X_1) \times \dotsb \times \mathcal C(X_n) \to \mathcal C(Y)
	\]
	of \cite[Remark 2.1.2.16]{ha} preserves $K$-indexed colimits in each index. This means that for any
	\[
		\bar \sigma : K^\triangleright \to \mathcal C(X_1) \times \dotsb \times \mathcal C(X_n)
	\]
	such that $\bar \sigma$ is colimiting in one index and constant in all others, the composite $\otimes_f \circ \bar \sigma$ is colimiting.
\end{defn}

See \cite[Definition 3.4.4.1]{ha}.
\begin{rmk} \label{rmk:PrL algebras}
	Recall that $\PrL$ has the symmetric monoidal structure of \cite[Proposition 4.8.1.15]{ha}. The inclusion $\PrL \to \widehat{\Cat}$ is lax symmetric monoidal (where $\widehat{\Cat}$ has the Cartesian monoidal structure), that is, it gives a map of operads $\PrL^\otimes \to \widehat{\Cat}^\otimes$. This induces a forgetful functor
	\[
		\Alg_{\mathcal O}(\PrL) \to \Alg_{\mathcal O}(\widehat{\Cat}) \xrightarrow{\simeq} \Mon_{\mathcal O}(\widehat{\Cat})
	,\]
	where the last equivalence is given by \cite[Proposition 2.4.2.5]{ha}. This forgetful functor is an inclusion corresponding to the subcategory of $\Mon_{\mathcal O}(\widehat{\Cat})$ whose objects are the $\mathcal O$-monoidal categories $\mathcal C : \mathcal O^\otimes \to \widehat{\Cat}$ such that for each $X \in \mathcal O$, $\mathcal C(X)$ is presentable, and $\mathcal C$ is compatible with $K$-indexed colimits as in Definition \ref{defn:colimit-compatibility for monoidal categories}, and whose morphisms are the $\mathcal C \to \mathcal D$ such that for each $X \in \mathcal O$, the functor $\mathcal C(X) \to \mathcal D(X)$ has a right adjoint.
\end{rmk}

\begin{lem} \label{lem:colimit-compatible monoidal localization}
	Let $\mathcal O$ be an operad, and let $L : \mathcal C \to \mathcal D$ be an $\mathcal O$-monoidal localization (Definition \ref{defn:monoidal localization}) of $\mathcal O$-monoidal categories. Let $K$ be a simplicial set, and suppose that $\mathcal C$ is compatible with $K$-indexed colimits as an $\mathcal O$-monoidal category (Definition \ref{defn:colimit-compatibility for monoidal categories}). Then $\mathcal D$ is also compatible with $K$-indexed colimits as an $\mathcal O$-monoidal category.
	\begin{proof}
		Let $f \in \Mul_{\mathcal O}(\{X_1, \dotsc, X_n\}, Y)$, and consider the following commutative diagram
		\[
			\begin{tikzcd}
				\mathcal C_{X_1} \times \dotsb \times \mathcal C_{X_n} \ar[d] \ar[r] & \mathcal C_Y \ar[d] \\
				\mathcal D_{X_1} \times \dotsb \times \mathcal D_{X_n} \ar[r] & \mathcal D_Y
			\end{tikzcd}
		,\]
		where the vertical arrows are induced by $L$, and are therefore colimit-preserving, and the horizontal arrows are $\otimes_f$.

		Given
		\[
			\bar \tau : K^\triangleright \to \mathcal D_{X_1} \times \dotsb \times \mathcal D_{X_n}
		\]
		which is colimiting in one index $j$ and constant in all others, since $L^\otimes$ admits a fully faithful right adjoint, we have that this actually factors through the left vertical arrow of the above square by a 
		\[
			\bar \sigma : K^\triangleright \to \mathcal C_{X_1} \times \dotsb \times \mathcal C_{X_n}
		\]
		which is colimiting in the same index $j$, and constant in all others. This is because we can consider a colimiting extension $\bar \sigma_j : K^\triangleright \to \mathcal C_{X_j}$ of
		\[
			K \to K^\triangleright \xrightarrow{\bar \tau_j} \mathcal D_{X_j} \to \mathcal C_{X_j}
		;\]
		the resulting composite $K^\triangleright \to \mathcal C_{X_j} \to \mathcal D_{X_j}$ coincides with $\bar \tau_j$ since they are both colimiting cones of the same functor $K \to D_{X_j}$.

		Given $\bar \sigma$, we find that $\otimes_f \circ \bar \tau$ is colimiting since it is equivalent to $L_Y \circ \otimes_f \circ \bar \sigma$, which is colimiting since $\otimes_f \circ \bar \sigma$ is colimiting by hypothesis.
	\end{proof}
\end{lem}

% NOTE: need compatibility with colimits to remove strong saturation hypotheses on the $W_X$
% need presentability to get existence of localizations
% could have made a more complicated less appealing statement without these hypotheses, but it wouldn't be useful anyway
The following reformulation of \cite[Proposition 2.2.1.9]{ha} is useful for producing monoidal localizations.
\begin{prp} \label{prp:existence of monoidal localizations}
	Let $\mathcal O$ be an operad, and $\mathcal C \in \Alg_{\mathcal O}(\PrL)$. Suppose that for each $X \in \mathcal O$, we have a collection of morphisms $W_X$ in $\mathcal C_X$ that contains all identity maps and such that the strongly saturated class generated by $W_X$ is of small generation. Suppose that for any $f \in \Mul_{\mathcal O}(\{X_1, \dotsc, X_n\}, Y)$,\footnote{See \cite[Definition 2.1.1.16]{ha} for the definition of $\Mul$.} and $\{g_i \in \mathcal W_{X_i}\}_{i = 1}^n$, we have that $\otimes_f \{g_i\}_{i = 1}^n$\footnote{See \cite[Remark 2.1.2.16]{ha} for notation.} is in the strongly saturated class generated by $W_Y$. Then there is an $L : \mathcal C \to W^{-1} \mathcal C$ in $\Alg_{\mathcal O}(\PrL)$ which is an $\mathcal O$-monoidal localization along $\{W_X\}_{X \in \mathcal O}$ (Definition \ref{defn:monoidal localization}).
	\begin{proof}
		For each $X \in \mathcal O$, write $\bar W_X$ for the strongly saturated class generated by $W_X$. Since this is of small generation, \cite[Proposition 5.5.4.15]{htt} says that there is a localization functor $L_X : \mathcal C_X \to W_X^{-1} \mathcal C_X$, where a morphism $f \in \mathcal C_X$ is in $\bar W_X$ if and only if $L_X(f)$ is invertible.

		Using Remark \ref{rmk:PrL algebras}, we have that $\mathcal C$ corresponds to an $\mathcal O$-monoidal category compatible with all small colimits. For any $f \in \Mul_{\mathcal O}(\{X_1, \dotsc, X_n\}, Y)$, since
		\[
			\otimes_f : \mathcal C_{X_1} \times \dotsb \times \mathcal C_{X_n} \to \mathcal C_Y
		\]
		preserves small colimits in each variable, and sends $W_{X_1} \times \dotsb \times W_{X_n}$ to $\bar W_Y$, it follows from \cite[Remark 5.5.4.10]{htt} that $\otimes_f$ sends $\bar W_{X_1} \times \dotsb \times \bar W_{X_n}$ to $\bar W_Y$: if $g_i : A_i \to B_i \in W_{X_i}$, we have that
		\[
			\otimes_f \{g_i\}_{i = 1}^n \simeq \otimes_f(g_1, A_2, \dotsc, A_n) \circ \dotsb \circ \otimes_f(B_1, \dotsc, B_{n - 1}, g_n)
		,\]
		and for each $i$, $\id_{A_i}$ and $\id_{B_i}$ are in $W_{X_i}$, so this is a composite of maps in the strongly saturated class generated by $W_Y$, so it is in this strongly saturated class.

		Thus, the family of localization $\{L_X\}_{X \in \mathcal O}$ is compatible with the $\mathcal O$-monoidal structure on $\mathcal C$ as in \cite[Definition 2.2.1.6]{ha}, so \cite[Proposition 2.2.1.9]{ha} says that there is a morphism $L^\otimes : \mathcal C^\otimes \to W^{-1} \mathcal C^{\otimes}$ of $\mathcal O$-monoidal categories which admits a fully faithful right adjoint relative $\mathcal O^\otimes$, and that the fibre over $X \in \mathcal O \subseteq \mathcal O^\otimes$ is the localization $L_X : \mathcal C_X \to W_X^{-1} \mathcal C_X$.

		This shows that $\mathcal C \to W^{-1} \mathcal C$ is an $\mathcal O$-monoidal localization in $\Mon_{\mathcal O}(\widehat{\Cat})$, and it only remains to show that this actually lifts to a morphism in $\Alg_{\mathcal O}(\PrL)$. Indeed, as in Remark \ref{rmk:PrL algebras}, this amounts to showing that for each $X \in \mathcal O$, $(W^{-1} \mathcal C)(X)$ is presentable, $(\mathcal C \to W^{-1} \mathcal C)(X)$ is colimit preserving, and that $W^{-1} \mathcal C$ is compatible with all small colimits. First two properties follow from the fact that $(\mathcal C \to W^{-1} \mathcal C)(X)$ is the accessible localization functor $\mathcal C(X) \to W_X^{-1} \mathcal C(X)$ provided by \cite[Proposition 5.5.4.15]{htt}, and compatibility of $W^{-1} \mathcal C$ follows from Lemma \ref{lem:colimit-compatible monoidal localization} since $\mathcal C \to W^{-1} \mathcal C$ is a monoidal localization and $\mathcal C$ is compatible with all small colimits.
	\end{proof}
\end{prp}

Finally, we come to the proof of our main result:
\begin{proof}[Proof of \Cref{thm:monoidal objectwise localization}]
	By Proposition \ref{prp:existence of monoidal localizations}, we have that for each $p \in K$, there is an $\mathcal O$-monoidal localization $F(p) \to W_p^{-1} F(p)$ in $\Alg_{\mathcal O} \PrL$ of $F(p)$ along $W_p = \{W_p(X)\}_{X \in \mathcal O}$.

	Recall that Proposition \ref{prp:monoidal localization} says that for all $p$, the map in $(\Alg_{\mathcal O} \PrL)^\op$ corresponding to $F(p) \to W_p^{-1} F(p)$ is a monomorphism, and that for each map $p \to q$ in $K$, the composite
	\[
		F(p) \to F(q) \to W_q^{-1} F(q)
	\]
	extends through $F(p) \to W_p^{-1} F_p$, so \cite[Proposition A.1]{objwise-mono} says that there is a map $F \to W^{-1} F$ that corresponds to an objectwise monomorphism in $\Fun(K^\op, (\Alg_{\mathcal O} \PrL)^\op) \simeq \Fun(K, \Alg_{\mathcal O} \PrL)^\op$, and that satisfies that $(F \to W^{-1} F)(p)$ is $F_p \to W_p^{-1} F(p)$ for each $p \in K$.

	We obtain the statement about extending through $F \to W^{\-1} F$ by \cite[Proposition A.11]{objwise-mono} and \Cref{prp:monoidal localization}.
\end{proof}

\section{Adjointability}

\begin{defn} \label{defn:Beck Chevalley}
	Given a square of categories
	\[
		\begin{tikzcd}
			\mathcal C \ar[d, "F"'] \ar[r, "\phi_*"] & \mathcal C' \ar[d, "F'"] \\
			\mathcal D \ar[r, "\psi_*"'] & \mathcal D'
		\end{tikzcd}
	\]
	along with left adjoints $\phi^* \dashv \phi_*$ and $\psi^* \dashv \psi_*$, and an equivalence $\sigma : F' \phi_* \simeq \psi_* F$, define the \emph{left base change transformation} of $\sigma$ to be the composite
	\[
		\psi^* F' \to \psi^* F' \phi_* \phi^* \simeq \psi^* \psi_* F \phi^*  \to F \phi^*
	,\]
	where the first map comes from the unit of $\phi^* \dashv \phi^*$, the second from $\sigma$, and the last from the counit of $\psi^* \dashv \psi_*$.

	Dually, if $\phi_*, \psi_*$ have right adjoints $\phi^!, \psi^!$, define the \emph{right base change transformation} of $\sigma$ to be the composite
	\[
		F \phi^! \to \psi^! \psi_* F \phi^! \simeq \psi^! F \phi_* \phi^! \to \psi^! F
	.\]

	% Finally, if $\phi_*, \psi_*$ have left adjoints $\phi^*, \psi^*$, and $F,F'$ have right adjoints $F_*, F'_*$, and the left base change transformation of $\sigma$ is invertible, define the \emph{left exchange transformation} of $\sigma$ to be the right base change transformation of the left base change transformation of $\sigma$, which is the composite
	% \[
	% 	\phi^* F'_* \to F_* F \phi^* F'_* \simeq F_* \psi^* F' F'_* \to F_* \psi^*
	% ,\]
	% where the middle equivalence is induced by the inverse of the left base change transformation of $\sigma$, $\psi^* F' \to F \phi^*$.  Also see \cite[Definition F.12]{TwAmb}.
	%
	% Say the square is left-right adjointable if this transformation is invertible.
	%
	% Similarly, we may define a right exchange transformation for $\sigma$ when $F,F'$ have left adjoints and $\phi_*, \psi_*$ have right adjoints, and we say that the square is right-left adjointable if this transformation is invertible.
\end{defn}

\begin{lem} \label{lem:LAd PrL pres colim}
	Let $S$ be a simplicial set. Then $\Fun^\LAd(S, \PrL)$ has all small colimits, and the inclusion $\Fun^\LAd(S, \PrL) \to \Fun(S, \PrL)$ preserves small colimits.
	\begin{proof}
		Recall (as in \cite[Remark 4.7.4.14]{ha}) that for a square whose columns admit right adjoints and whose rows admit left adjoints, the square is left adjointable if and only if the transposed square is right adjointable. In fact, if all the maps admit right adjoints, we find that the square is left adjointable if and only the corresponding square of right adjoints is left adjointable. Thus, the equivalence $\Fun(S, \PrL)^\op \simeq \Fun(S^\op, \PrR)$ restricts to an equivalence $\Fun^\LAd(S, \PrL)^\op \simeq \Fun^\LAd(S^\op, \PrR)$, since a square in $\PrL$ is right adjointable if and only if the opposite square in $\PrR$ is left adjointable.
		
		Now, \cite[Corollary 4.7.4.18]{ha} shows that $\Fun^\LAd(S^\op, \widehat{\Cat})$ has small limits and the inclusion into $\Fun(S^\op, \widehat{\Cat})$ preserves small limits. Therefore, by \cite[Theorem 5.5.3.18]{htt}, we find that $\Fun^\LAd(S^\op, \PrR)$ has small limits, and the inclusion into $\Fun(S^\op, \PrR)$ preserves small limits. It follows that $\Fun^\LAd(S, \PrL) \simeq \Fun^\LAd(S^\op, \PrR)^\op$ admits small colimits, and the inclusion $\Fun^\LAd(S, \PrL) \to \Fun(S, \PrL)$ preserves small colimits, since
		\[
			(\Fun^\LAd(S^\op, \PrR) \to \Fun(S^\op, \PrR))^\op
		\]
		does.
	\end{proof}
\end{lem}

% The following is a reformulation of \cite[Corollary 4.7.4.19]{ha}:
% \begin{lem} \label{lem:radj comm lims}
% 	Given $\chi : T \to \Fun^\RAd(S, \PrL)$, there is a functor $\bar \chi : T^\triangleleft \to \Fun^\RAd(S^\triangleright, \PrL)$ which is limiting, and takes values in colimiting functors $S^\triangleright \to \PrL$. In particular, the map
% 	\[
% 		\varinjlim_{s \in S} \varprojlim_{t \in T} \chi(t)(s) \to \varprojlim_{t \in T} \varinjlim_{s \in S} \chi(t)(s)
% 	\]
% 	is an equivalence.
% \end{lem}

\subsection{Transformations and Adjointability}

Fix a category $\mathcal C$.

The following definition will give us a convenient language with which to discuss adjointability.
\begin{defn} \label{defn:compat}
	Given a transformation $\phi : F \to G$ of functors $F,G : \mathcal C \to \Cat$, say that $\phi$ is left (right) adjointable at a map $X \to Y$ in $\mathcal C$ if the square
	\[
		\begin{tikzcd}
			F(X) \ar[d] \ar[r] & F(Y) \ar[d] \\
			F'(X) \ar[r] & F'(Y)
		\end{tikzcd}
	\]
	is left (right) adjointable.

	We will simply say that $\phi$ is left (right) adjointable if it is left (right) adjointable at all maps in $\mathcal C$.
\end{defn}

In this section, we will establish results that may be seen as sorts of ``locality'' properties of adjointability.

First we establish a version of \cite[Corollary 4.7.4.18]{ha} that will be convenient in certain situations.
\begin{lem} \label{lem:adjointability limits}
	Let $K$ be a simplicial set, and let $G : K \to \Fun(\mathcal C, \Cat)$ be a functor. Let $f$ be a map in $\mathcal C$, and suppose that for each $p \to q$ in $K$, the transformation $G(p) \to G(q)$ is left (right) adjointable at $f$. Then for any $p$, $\varprojlim G \to G(p)$ is left (right) adjointable at $f$, and a transformation $F \to \varprojlim G$ is left (right) adjointable at $f$ if and only if for every $p \in K$, the composite $F \to G(p)$ is left (right) adjointable at $f$.
	\begin{proof}
		We may assume $\mathcal C = \Delta^1$, and that $f$ is the 1-cell, so we assumed that $G$ is a functor $K \to \Fun^\LAd(\mathcal C, \Cat)$ ($\Fun^\RAd(\mathcal C, \Cat)$). Therefore, \cite[Corollary 4.7.4.18(1) and (2)]{ha} show the ``if'' direction, and that in general, for any $p$, $\varprojlim G \to G(p)$ is left (right) adjointable at $f$, so \cite[Lemma F.6(2) and (3)]{TwAmb} shows the ``only if'' direction.
	\end{proof}
\end{lem}

\begin{lem} \label{lem:adjointability vert composites}
	Let $\lambda$ be an ordinal, and let $G : (\lambda^\op)^\triangleleft \cong (\lambda + 1)^\op \to \Fun(\mathcal C, \Cat)$ be a diagram such that for every limit ordinal $\alpha \in \lambda$, $G$ restricts to a limiting diagram on $(\alpha^\op)^\triangleleft \cong (\alpha + 1)^\op$. Let $f$ be a map in $\mathcal C$, and suppose that for each $\alpha \in \lambda$, the transformation $G(\alpha + 1) \to G(\alpha)$ is left (right) adjointable at $f$. Then for all $\alpha \leq \beta \leq \lambda$, the transformation $G(\beta) \to G(\alpha)$ is left (right) adjointable at $f$, and if $\lambda' \leq \lambda$ is a limit ordinal, a transformation $F \to G(\lambda')$ is left (right) adjointable at $f$ if and only if for all $\alpha \in \lambda'$, the transformation $F \to G(\alpha)$ is left (right) adjointable at $f$.
	\begin{proof}
		We proceed by transfinite induction on $\lambda$. Indeed, when $\lambda$ is empty, this is tautological, and when $\lambda = \lambda' + 1$, and we know the result for $\lambda'$, this follows from \cite[Lemma F.6(3) or (2)]{TwAmb}, so it only remains to check the case where $\lambda$ is a limit ordinal, and we know the result for all $\lambda' \in \lambda$.

		Since $\lambda$ is a limit ordinal, we know that for all $\lambda' \in \lambda$, we have that $\lambda' + 1 \in \lambda$. In particular, by our inductive hypothesis, we know that for all $\alpha \leq \beta < \lambda$, the map $G(\beta) \to G(\alpha)$ is left (right) adjointable at $f$, and if $\lambda' \in \lambda$ is a limit ordinal, a transformation $F \to G(\lambda')$ is left (right) adjointable at $f$ if and only if for all $\alpha \in \lambda'$, the transformation $F \to G(\alpha)$ is left (right) adjointable at $f$.

		Thus, it only remains to show that for all $\alpha \in \lambda$, the map $G(\lambda) \to G(\alpha)$ is left (right) adjointable at $f$, and that a transformation $F \to G(\lambda)$ is left (right) adjointable at $f$ if and only if for all $\alpha \in \lambda$, the transformation $F \to G(\alpha)$ is left (right) adjointable at $f$, but this follows immediately from \Cref{lem:adjointability limits} since $G$ is a limiting diagram.
	\end{proof}
\end{lem}

\begin{cor} \label{cor:adjointability horiz composites}
	Let $\phi : D \to D'$ be a transformation of presheaves $D,D' : \mathcal C^\op \to \widehat{\Cat}$, let $\lambda$ be an ordinal, and let $X : \lambda + 1 \to \mathcal C$ be a diagram such that for any limit ordinal $\lambda' \leq \lambda$, both $D$ and $D'$ restrict to limiting diagrams on $(\lambda' + 1)^\op$, and for any $\alpha \leq \lambda$, $\phi X(\alpha)$ has a right (left) adjoint. Then $\phi$ is left (right) adjointable at $X(\alpha) \to X(\beta)$ for all $\alpha \leq \beta \leq \lambda + 1$ if and only if for all $\alpha \in \lambda$, $\phi$ is left (right) adjointable at $X(\alpha) \to X(\alpha + 1)$, in which case for any limit ordinal $\lambda' \leq \lambda$, $\phi$ is left (right) adjointable at a map $X(\lambda') \to Y$ if $E(Y) \to E X(\lambda')$ has a left (right) adjoint for $E \in \{D,D'\}$, $\phi(Y)$ has a right (left) adjoint, and for all $\alpha \in \lambda'$, $\phi$ is left (right) adjointable at $X(\alpha) \to Y$.
	\begin{proof}
		The ``only if'' direction is clear by taking $\beta = \alpha + 1$. For the converse, we may use $\phi$ to define functor $\Xi : (\lambda + 1)^\op \to \Fun(\Delta^1, \widehat{\Cat})$, such that for each map $\alpha \to \beta$ in $\lambda + 1$, the map $\Xi(\beta) \to \Xi(\alpha)$ is the commutative square
		\[
			\begin{tikzcd}
				DX(\beta) \ar[d] \ar[r] & D'X(\beta) \ar[d] \\
				DX(\alpha) \ar[r] & D'X(\alpha)
			\end{tikzcd}
		.\]
		If $\beta = \alpha + 1$, then since $\phi$ is left (right) adjointable at $X(\alpha) \to X(\alpha + 1)$, the transpose of this square is left (right) adjointable, so since we have assumed that the horizontal arrows have right (left) adjoints, it follows that $\Xi(\alpha + 1) \to \Xi(\alpha)$ is right (left) adjointable.

		Note that since $D,D'$ restrict to limiting diagrams on $\lambda' + 1$ for all limit ordinals $\lambda' \leq \lambda$, it follows that $\Xi$ restricts to a limiting diagram on $\lambda' + 1$ for all limit ordinals $\lambda' \leq \lambda$. Thus, \Cref{lem:adjointability vert composites} shows that $\Xi$ is right (left) adjointable at all maps in the image of $X$, and that if $\lambda' \leq \lambda$ is a limit ordinal, a transformation $F \to \Xi(\lambda')$ is right (left) adjointable if and only if for all $\alpha \in \lambda'$, the transformation $F \to \Xi(\alpha)$ is right (left) adjointable.

		Thus, we may conclude by taking transposes again.
	\end{proof}
\end{cor}

\begin{lem} \label{lem:vert cons adj}
	Let $K$ be a simplicial set, and let $G : K^\triangleleft \to \Fun(\mathcal C, \Cat)$ be a functor. Let $f : X \to Y$ be a map in $\mathcal C$, and suppose that for each $p \to q$ in $K^\triangleleft$, the transformation $G(p) \to G(q)$ is right (left) adjointable at $f$, and that $G(-\infty)(X) \to \varprojlim_{p \in K} G(p)(X)$ is conservative. Then a transformation $F \to G(\infty)$ is right (left) adjointable at $f$ if and only if for every $p \in K$, the composite $F \to G(p)$ is right (left) adjointable at $f$.
	\begin{proof}
		We have a commutative diagram
		\[
			\begin{tikzcd}
				F(X) \ar[d] \ar[r] & F(Y) \ar[d] \\
				G(-\infty)(X) \ar[d] \ar[r] & G(-\infty)(Y) \ar[d] \\
				\varprojlim_{p \in K} G(p)(X) \ar[r] & \varprojlim_{p \in K} G(p)(Y)
			\end{tikzcd}
		.\]
		We know that the bottom square is right (left) adjointable by \Cref{lem:adjointability limits}, and $F \to G(p)$ is right (left) adjointable at $f$ for all $p \in K$ if and only if the outer square is right (left) adjointable. By \cite[Lemma F.6(2)]{TwAmb} (\cite[Lemma F.6(3)]{TwAmb}), we have that this is equivalent to the condition that the right (left) base change transformation of the top square is sent to an equivalence by $G(-\infty)(X) \to G(p)(X)$ for all $p \in K$. Since $G(-\infty)(X) \to \varprojlim_{p \in K} G(p)(X)$ is conservative, the functors $\{G(-\infty)(X) \to G(p)(X)\}_{p \in K}$ are jointly conservative, so this is equivalent to the condition that the top square is right (left) adjointable.
	\end{proof}
\end{lem}

\begin{prp} \label{prp:Pr adj lims}
	Assume $\mathcal C$ is small, and let $\phi : F \to G$ be a transformation of functors $F,G : \Delta^1 \star \mathcal C \to \widehat{\Cat}$, where $\phi(c)$ is a left (right) adjoint functor between presentable categories for all $c \in \{0\} \star \mathcal C$, and $\phi$ is left (right) adjointable at all maps in $\mathcal C$.
	\begin{enumerate}

		\item If $F, G$ restrict to limiting diagrams on $\{1\} \star \mathcal C$, then $\phi$ is left (right) adjointable at all maps $1 \to c$, and it is left (right) adjointable at $0 \to 1$ if and only if it is left (right) adjointable at all maps $0 \to c$ for $c \in \mathcal C$.

			% NOTE: seems like this might be subsumed by the "surj" results, but actually, would only be able to apply those here for the left adjointable case, since only know that right adj is cons iff image of left adj gens under colimits, don't know the opposite (that left adj is cons iff image of right adj gens under colimits)
			% So keep this
		\item Suppose that $H(1) \to \varprojlim H|_{\mathcal C}$ is conservative for $H \in \{F,G\}$, that $\phi(1)$ is a left (right) adjoint functor between presentable categories, and that $\phi$ is left (right) adjointable at all maps $1 \to c$. Then $\phi$ is left (right) adjointable at $0 \to 1$ if and only if it is left (right) adjointable at all maps $0 \to c$ for $c \in \mathcal C$.

	\end{enumerate}
	\begin{proof}
		We may use $\phi$ to define functor $\Xi : \Delta^1 \star \mathcal C \to \Fun(\Delta^1, \widehat{\Cat})$, such that for each map $c \to d$ in $\Delta^1 \star \mathcal C$, the map $\Xi(c) \to \Xi(d)$ is the commutative square
		\[
			\begin{tikzcd}
				F(c) \ar[d] \ar[r] & G(c) \ar[d] \\
				F(d) \ar[r] & G(d)
			\end{tikzcd}
		.\]
		In fact, since $\phi$ is left (right) adjointable at all maps in $\mathcal C$, the transpose of this square is left (right) adjointable for $c,d \in \mathcal C$, so since the horizontal arrows have right (left) adjoints, this square is right (left) adjointable, \ie, $\Xi(c) \to \Xi(d)$ is right (left) adjointable.

		\begin{enumerate}

			\item \Cref{lem:adjointability limits} says that $\Xi(1) \to \Xi(c)$ is right (left) adjointable for all $c \in \mathcal C$, and that $\Xi(0) \to \Xi(1)$ is right (left) adjointable if and only if $\Xi(0) \to \Xi(c)$ is right (left) adjointable for all $c \in \mathcal C$.

				Since $F,G$ send any map in $\mathcal C$ to a functor that has a left (right) adjoint, we have that $F|_{\mathcal C}$ and $G|_{\mathcal C}$ land in $\PrR$ ($\PrL$), so by
				\cite[Proposition 5.5.3.13 and 5.5.3.18]{htt}, we have that for $H \in \{F,G\}$ and $c \in \mathcal C$, the functor $H(1) \to H(c)$ has a left (right) adjoint, and that $H(0) \to H(1)$ has a left (right) adjoint if and only if $H(0) \to H(c)$ has a left (right) adjoint for all $c \in \mathcal C$, thus we conclude by taking transposes.

			\item The map $\Xi(1)(0) \to \varprojlim_{c \in \mathcal C} \Xi(c)(0)$ is $F(1) \to \varprojlim_{c \in \mathcal C} F(c)$, which is conservative, so \Cref{lem:vert cons adj} says that $\Xi(0) \to \Xi(1)$ is right (left) adjointable if and only if for all $c \in \mathcal C$, the composite $\Xi(0) \to \Xi(c)$ is right (left) adjointable.

				Now, since $F,G$ send every map in $\mathcal C$ to a functor that has a left (right) adjoint, we have that $F|_{\mathcal C}$ and $G|_{\mathcal C}$ land in $\PrR$ ($\PrL$), so by
				\cite[Proposition 5.5.3.13 and 5.5.3.18]{htt}, we have that for $H \in \{F,G\}$ and $c \in \mathcal C$,
				the functor $\varprojlim H|_{\mathcal C} \to H(c)$ has a left (right) adjoint, and
				by \Cref{lem:2/3 for adjoints}, we have the functor $H(0) \to H(1)$ has a left (right) adjoint if and only if for all $c \in \mathcal C$, the functor $H(0) \to H(c)$ has a left (right) adjoint.

				Thus, we may conclude by taking transposes again.

		\end{enumerate}
	\end{proof}
\end{prp}

\begin{lem} \label{lem:horiz cons adj}
	Let $\phi : F \to G$ be a transformation in $\Fun(K, \Cat)$, and let $f : X \to Y$, and $\{f_i : X_i \to X\}$ be maps in $\mathcal C$.

	Suppose that for each $i$, $\phi$ is right (left) adjointable at $f_i$ and $f \circ f_i$. If the right (left) adjoints of $\{G(X) \to G(X_i)\}_i$ are jointly conservative, then $\phi$ is right (left) adjointable at $f$.
	\begin{proof}
		For each $i$, we can consider the commutative diagram
		\[
			\begin{tikzcd}
				F(X_i) \ar[d] \ar[r] & F(X) \ar[d] \ar[r] & F(Y) \ar[d] \\
				G(X_i) \ar[r] & G(X) \ar[r] & G(Y)
			\end{tikzcd}
		.\]
		We know that the outer rectangle and the left square are right (left) adjointable, so by \cite[Lemma F.6(4)]{TwAmb} (\cite[Lemma F.6(1)]{TwAmb}), we have that the right (left) base change transformation of the left square is sent to an equivalence by the right (left) adjoint of $G(X_i) \to G(X)$ for each $i$.
	\end{proof}
\end{lem}

\begin{lem} \label{lem:ff adj crit}
	Let $G : \Delta^1 \star K \to \Cat$ be a diagram such that $G(0 \to 1)$ has a fully faithful left (right) adjoint. For any transformation $\phi : F \to G$, if
	\begin{enumerate}

		\item $F(0 \to 1)$ has a left (right) adjoint,

		\item the functors $\{F(1 \to a)\}_{a \in K}$ are jointly conservative,

		\item $\phi$ is left (right) adjointable at $0 \to a$ for all $a \in K$, and

		\item $\phi : F(a) \to G(a)$ has a right (left) adjoint $\psi$ for all $a \in \{0\} \star K$,

	\end{enumerate}
	then $\phi$ is left (right) adjointable at $0 \to 1$, and at $1 \to a$ for all $a \in K$ such that $F(1 \to a)$ has a left (right) adjoint.
	\begin{proof}
		First note that since $G(0 \to 1)$ has a fully faithful left (right) adjoint, and $G(0 \to a)$ has a left (right) adjoint for all $a \in K$, we may use \cite[Proposition 5.5.4.2]{htt} to show that the left (right) adjoint of $G(0 \to a)$ lands in the essential image of the left (right) adjoint of $G(0 \to 1)$. It then follows that $G(1 \to a)$ admits a left (right) adjoint.

		Furthermore, using that $G(0 \to 1)$ has a fully faithful left (right) adjoint again, by \cite[Lemma F.6(1) or (4)]{TwAmb}, it suffices to show that $\phi$ is left (right) adjointable at $0 \to 1$.

		We need to show that the square
		\[
			\begin{tikzcd}
				F(0) \ar[d] \ar[r, "\alpha^*"] & F(1) \ar[d] \\
				G(0) \ar[r, "\beta^*"] & G(1)
			\end{tikzcd}
		\]
		is left (right) adjointable. If we write $\alpha_!, \beta_!$ for the left (right) adjoints of $\alpha^*, \beta^*$, this means that we need to show the composite
		\begin{gather*}
			\beta_! \phi \to \beta_! \phi \alpha^* \alpha_! \simeq \beta_! \beta^* \phi \alpha_! \to \phi \alpha_! \\
			(\phi \alpha_! \to \beta_! \beta^* \phi \alpha_! \simeq \beta_! \phi \alpha^* \alpha_! \to \beta_! \phi)
		\end{gather*}
		is an equivalence. Since $\beta_!$ is a fully faithful left (right) adjoint of $\beta^*$, we know that the first (last) map is an equivalence, so it suffices to show that the last (first) map is an equivalence. Once again using that $\beta_!$ is a fully faithful left (right) adjoint of $\beta^*$, for this it suffices to show that for any $A \in F(1)$, the object $\phi \alpha_! A$ is in the essential image of $\beta_!$, and by \cite[Proposition 5.5.4.2(1)]{htt}, this is equivalent to showing that for all maps $w$ in $G(0)$ if $\beta^*(w)$ is an equivalence, then $G(0)(\phi \alpha_! A, w)$ ($G(0)(w, \phi \alpha_! A)$) is an equivalence.

		Indeed, this map is equivalent to $F(1)(A, \alpha^* \psi w)$ ($F(1)(\alpha^* \psi w, A)$), so we can reduce to showing that $\alpha^* \psi w$ is an equivalence. Since the functors $\{F(1 \to a)\}_{a \in K}$ are jointly conservative, it suffices to show that for each $a \in K$, the functor $F(1 \to a)$ sends $\alpha^* \psi w$ to an equivalence, but since $F(1 \to a) \circ \alpha^* \simeq F(0 \to a)$, and $\phi$ is left (right) adjointable at $0 \to a$, this is the same as $\psi G(0 \to a) w$, which is as equivalence since $G(0 \to a) \simeq G(1 \to a) \beta^*$, and $\beta^* w$ is an equivalence.
	\end{proof}
\end{lem}

% \begin{lem} \label{lem:vert cons adj}
% 	Let $\psi : F \to G$ and $\{\psi_i : G \to G_i\}_i$ be transformations in $\Fun(\mathcal C, \Cat)$. Let $f : X \to Y$ be a map in $\mathcal C$ such that $\psi_i$ and $\psi_i \circ \psi$ are right (left) adjointable at $f$.
%
% 	If the functors $\{G(X) \to G_i(X)\}_i$ are jointly conservative, then $\psi$ is right (left) adjointable at $f$.
% 	\begin{proof}
% 		For each $i$, we can consider the commutative diagram
% 		\[
% 			\begin{tikzcd}
% 				F(X) \ar[d] \ar[r] & F(Y) \ar[d] \\
% 				G(X) \ar[d] \ar[r] & G(Y) \ar[d] \\
% 				G_i(X) \ar[r] & G_i(Y)
% 			\end{tikzcd}
% 		.\]
% 		We know that the outer rectangle and the bottom square are right (left) adjointable, so by \cite[Lemma F.6(2)]{TwAmb} (\cite[Lemma F.6(3)]{TwAmb}), we have that the right (left) base change transformation of the top square is sent to an equivalence by $G(X) \to G_i(X)$ for each $i$.
% 	\end{proof}
% \end{lem}

% \subsubsection{Surjectivity} \label{S:adj surj}

For the remainder of this section, we fix a collection $\mathcal K$ of simplicial sets. Recall from \cite[Definition 4.8.1.1]{ha} that $\Cat(\mathcal K)$ is the subcategory of $\Cat$ given by categories that admit colimits indexed by elements of $\mathcal K$, and where the morphisms are the functors that preserve such colimits.

\begin{lang} \label{lang:I-colims}
	The term ``$\mathcal K$-colimits'' refers to ``colimits indexed by elements of $\mathcal K$''.
\end{lang}

\begin{rmk} \label{rmk:lims of cats compat w colims}
	By \kerodoncite{06B6}, we have that $\Cat(\mathcal K)$ has small limits, and the evident inclusion $\Cat(\mathcal K) \to \Cat$ preserves small limits.
\end{rmk}

The following result is useful for checking the hypotheses of some of the remaining results of this section -- see \eg{} \Cref{cor:locality of adjointability}.
\begin{lem} \label{lem:lim gen under colim}
	Let $K \in \mathcal K$, and let $\mathcal C : K^\triangleleft \to \Cat(\mathcal K)$ be a diagram such that for each $p \in K$, the functor $F_p : \mathcal C(-\infty) \to \mathcal C(p)$ admits a left adjoint $L_p$. Then the functor $F : \mathcal C(-\infty) \to \varprojlim \mathcal C|_K$ admits a left adjoint $L$, and the image of $L$ generates $\mathcal C(-\infty)$ under $\mathcal K$-colimits if and only if the images of the functors $\{L_p\}_{p \in K}$ generate $\mathcal C(-\infty)$ under $\mathcal K$-colimits.
	\begin{proof}
		Since $K \in \mathcal K$, \cite[Lemma D.4.7]{HM6FF} shows that $F$ admits a left adjoint $L$, and that this left adjoint sends any object to an object that is a $K$-indexed colimit of objects in the images of $\{L_p\}_{p \in K}$. Since $K \in \mathcal K$, this means that the subcategory of $\mathcal C(-\infty)$ generated by the image of $L$ under $\mathcal K$-colimits is the same as the subcategory generated by the images of $\{L_p\}_{p \in K}$.
	\end{proof}
\end{lem}

\begin{lem} \label{lem:vert surj adj}
	Let $\phi : F \to G$ and $\{\phi_i : F_i \to F\}_i$ be transformations in $\Fun(\mathcal C, \Cat(\mathcal K))$. Let $f : X \to Y$ be a map in $\mathcal C$ such that $\phi$ and $\phi \circ \phi_i$ are right (left) adjointable at $f$, and such that the images of $\{F_i(Y) \to F(Y)\}_i$ generate $F(Y)$ under $\mathcal K$-colimits. Suppose that for all $i$, $\phi_i$ and $\phi \circ \phi_i$ are adjointable at $f$, and that the right (left) adjoints of $F(f)$ and $G(f)$ preserve $\mathcal K$-colimits. Then $\phi$ is right (left) adjointable at $f$.
	\begin{proof}
		For each $i$, we can consider the commutative diagram
		\[
			\begin{tikzcd}
				F_i(X) \ar[d] \ar[r] & F_i(Y) \ar[d] \\
				F(X) \ar[d] \ar[r] & F(Y) \ar[d] \\
				G(X) \ar[r] & G(Y)
			\end{tikzcd}
		.\]
		We know that the outer rectangle and the top square are right (left) adjointable, so by \cite[Lemma F.6(2)]{TwAmb} (\cite[Lemma F.6(3)]{TwAmb}), we have that the right (left) base change transformation of the bottom square is an equivalence when evaluated at objects in the image of $F_i(Y) \to F(Y)$. Since the vertical arrows preserve $\mathcal K$-colimits, and the right (left) adjoints of the horizontal arrows in the bottom square do too, it follows that the right (left) base change transformation of the bottom square preserves $\mathcal K$-colimits, so since the images of the functors $\{F_i(Y) \to F(Y)\}_i$ generate $F(Y)$ under $\mathcal K$-colimits, it follows that the bottom square is right (left) adjointable. 
	\end{proof}
\end{lem}

\begin{lem} \label{lem:horiz surj adj}
	Let $\phi : F \to G$ be a transformation in $\Fun(\mathcal C, \Cat(\mathcal K))$, and let $g : X \to Y$, and $\{g_i : Y \to Y_i\}$ be maps in $\mathcal C$.

	Suppose that for each $i$, $\phi$ is right (left) adjointable at $g_i$ and $g_i \circ g$. If the images of the right (left) adjoints of $\{F(Y) \to F(Y_i)\}_i$ generate $F(Y)$ under $\mathcal K$-colimits, and the right (left) adjoints of $F(g), G(g)$ preserve $\mathcal K$-colimits, then $\phi$ is right (left) adjointable at $g$.
	\begin{proof}
		For each $i$, we can consider the commutative diagram
		\[
			\begin{tikzcd}
				F(X) \ar[d] \ar[r] & F(Y) \ar[d] \ar[r] & F(Y_i) \ar[d] \\
				G(X) \ar[r] & G(Y) \ar[r] & G(Y_i)
			\end{tikzcd}
		.\]
		We know that the outer rectangle and the right square are right (left) adjointable, so by \cite[Lemma F.6(4)]{TwAmb} (\cite[Lemma F.6(1)]{TwAmb}), we have that the right (left) base change transformation of the left square is an equivalence when evaluated at objects in the image of the right (left) adjoint of $G(Y) \to G(Y_i)$. Since the right (left) adjoints of the horizontal arrows in the left square preserves $\mathcal K$-colimits, the right (left) base change transformation of the left square preserves colimits indexed by the elements of $\mathcal K$. Thus, we may conclude since the images of the right (left) adjoints of the functors $G(Y) \to G(Y_i)$ generate $G(Y)$ under $\mathcal K$-colimits.
	\end{proof}
\end{lem}

\begin{cor} \label{cor:locality of adjointability}
	Let $\phi : F \to G$ be a transformation in $\Fun(\mathcal C, \Cat(\mathcal K))$, let $K \in \mathcal K$, let $Y : K^\triangleleft \to \mathcal C$ be a diagram, and let $f : X \to Y(\infty)$ be a map.

	Suppose that for each $i \in K$, $\phi$ is left adjointable at $Y(-\infty) \xrightarrow{g_i} Y(i)$ and $g_i \circ f$. If the image of the left adjoint of $F(Y(-\infty)) \to \varprojlim_{i \in K} F(Y(i))$ generates $F(Y)$ under $\mathcal K$-colimits, then $\phi$ is left adjointable at $g$.
	\begin{proof}
		By \Cref{lem:lim gen under colim}, we have that the images of the left adjoints of the maps $\{F(Y) \to F(Y(i))\}_{i \in K}$ generate $F(Y)$ under $\mathcal K$-colimits, so this follows from \Cref{lem:horiz surj adj}. 
	\end{proof}
\end{cor}

\subsection{Linear Adjoints and the Projection Formula}

\begin{defn} \label{defn:proj form}
	% Let $\mathcal O$ be an operad. Say a morphism $f^* : \mathcal D \to \mathcal C$ in $\Mon_{\mathcal O}(\Cat)$ has a linear left adjoint if the following equivalent conditions hold:
	% \begin{enumerate}
	%
	% 	\item For any $\phi \in \Mul_{\mathcal O}(\{X_0, \dotsc, X_k\}, Y)$, and $B \in \mathcal D(X_0)$, $A_1 \in \mathcal C(X_1), \dotsc, A_k \in \mathcal C(X_k)$
	% 		\[
	% 			f_\sharp \otimes_\phi\{f^* B, A_1, \dotsc, A_k\} \to f_\sharp \otimes_\phi\{f^*B, f^* f_\sharp A_1, \dotsc, f^* f_\sharp A_k\} \simeq f_\sharp f^* \otimes_\phi \{B, f_\sharp A_1, \dotsc, f_\sharp A_k\} \to \otimes_\phi \{B, f_\sharp A_1, \dotsc, f_\sharp A_k\}
	% 		\]
	% 		is invertible.
	%
	% 	\item For any $B \in \mathcal D(X_0)$, and $\phi \in \Mul_{\mathcal O}(\{X_0, \dotsc, X_n\}, Y)$, the square
	% 		\[
	% 			\begin{tikzcd}
	% 				\mathcal D(X_1) \times \dotsb \times \mathcal D(X_n) \ar[d, "\otimes_\phi B"'] \ar[r, "f^*"] & \mathcal C(X_1) \times \dotsb \times \mathcal C(X_n) \ar[d, "\otimes_\phi f^* B"] \\
	% 				\mathcal D(Y) \ar[r, "f^*"'] & \mathcal C(Y)
	% 			\end{tikzcd}
	% 		\]
	% 		is left-adjointable.
	%
	% 	\item For any $\phi \in \Mul_{\mathcal O}(\{X_0, \dotsc, X_n\}, Y)$, the square
	% 		\[
	% 			\begin{tikzcd}
	% 				\mathcal D(X_0) \times \dotsb \times \mathcal D(X_n) \ar[d, "\otimes_\phi"'] \ar[r] & \mathcal D(X_0) \times \mathcal C(X_1) \times \dotsb \times \mathcal C(X_n) \ar[d, "\otimes_\phi \circ (f^* \times \id)"] \\
	% 				\mathcal D(Y) \ar[r, "f^*"'] & \mathcal C(Y)
	% 			\end{tikzcd}
	% 		\]
	% 		is left-adjointable.
	%
	% \end{enumerate}

	Let $\mathcal A \in \CAlg(\Cat)$. Say a morphism $f^* : \mathcal D \to \mathcal C$ in $\Mod_{\mathcal A}(\Cat)$ has a \emph{$\mathcal A$-linear left (right) adjoint} if the following equivalent conditions hold:
	% Say a morphism $f^* : \mathcal D \to \mathcal C$ in $\CAlg(\Cat)$ has a \emph{linear left adjoint} if the following equivalent conditions hold, in which case we also say that its left adjoint $f_\sharp$ \emph{satisfies the projection formula}:
	\begin{enumerate}

		% \item For any $M \in \mathcal D$, the square
		\item For any $a \in \mathcal A$, the square
			\[
				\begin{tikzcd}
					% \mathcal D \ar[d, "- \otimes M"'] \ar[r, "f^*"] & \mathcal C \ar[d, "- \otimes f^* M"] \\
					\mathcal D \ar[d, "- \otimes a"'] \ar[r, "f^*"] & \mathcal C \ar[d, "- \otimes f^* a"] \\
					\mathcal D \ar[r, "f^*"'] & \mathcal C
				\end{tikzcd}
			\]
			is left (right) adjointable.

		\item The square
			\[
				\begin{tikzcd}
					% \mathcal D \times \mathcal D \ar[d, "\otimes"'] \ar[r, "f^* \times \mathcal D"] & \mathcal C \times \mathcal D \ar[d, "- \otimes f^* -"] \\
					\mathcal D \times \mathcal A \ar[d, "\otimes"'] \ar[r, "f^* \times \mathcal A"] & \mathcal C \times \mathcal A \ar[d, "\otimes"] \\
					\mathcal D \ar[r, "f^*"'] & \mathcal C
				\end{tikzcd}
			\]
			is left (right) adjointable.

		% \item The functor $f^*$ has a left adjoint $f_\sharp$, and for any $X \in \mathcal C$ and $Y \in \mathcal D$, the map
		\item The functor $f^*$ has a left adjoint $f_\sharp$ (\resp{} a right adjoint $f_*$), and for any $X \in \mathcal C$ and $a \in \mathcal A$, the composite
			\[
				% f_\sharp(X \otimes f^* Y) \to f_\sharp(f^* f_\sharp X \otimes f^* Y) \simeq f_\sharp f^* (f_\sharp X \otimes Y) \to f_\sharp X \otimes Y
				f_\sharp(X \otimes a) \to f_\sharp(f^* f_\sharp X \otimes a) \simeq f_\sharp f^* (f_\sharp X \otimes a) \to f_\sharp X \otimes a
			\]
			(\resp{} the composite
			\[
				f_* X \otimes a \to f_* f^*(f_* X \otimes a) \simeq f_*(f^* f_* X \otimes a) \to f_*(X \otimes a)
			\])
			is an equivalence.

	\end{enumerate}
	When $f^* : \mathcal D \to \mathcal C$ is a symmetric monoidal functor, we simply say $f^*$ has a linear left (right) adjoint if it has a $\mathcal D$-linear left (right) adjoint, or that the left (right) adjoint satisfies the (dual) projection formula.

	If $\mathcal C, \mathcal D$ are also monoidal closed, and $f^*$ admits a left adjoint, then $f^*$ has a linear left adjoint if and only if for any $X,Y \in \mathcal D$, the map
	\begin{multline}
		f^* \underline\Hom_{\mathcal D}(X,Y) \to \underline\Hom_{\mathcal C}(f^* X, f^* \underline\Hom_{\mathcal D}(X, Y) \otimes f^* X) \\
		\simeq \underline\Hom_{\mathcal C}(f^* X, f^*(\underline\Hom_{\mathcal D}(X,Y) \otimes X)) \to \underline\Hom_{\mathcal C}(f^* X, f^* Y)
	\end{multline}
	is invertible.
\end{defn}

\begin{rmk} \label{rmk:linear adjoints}
	Let $\mathcal A \in \CAlg(\Cat)$, and let $f^* : \mathcal M \to \mathcal N$ be a map in $\Mod_{\mathcal A}(\Cat)$ that has a right adjoint $f_* : \mathcal N \to \mathcal M$. If $f_*$ is an $\mathcal A$-linear right adjoint, then the right base change transformation of
	\[
		\begin{tikzcd}
			\mathcal M \times \mathcal A \ar[d, "\otimes"'] \ar[r, "f^* \times \mathcal A"] & \mathcal N \times \mathcal A \ar[d, "\otimes"] \\
			\mathcal M \ar[r, "f^*"'] & \mathcal N
		\end{tikzcd}
	\]
	defines a commutative square
	\[
		\begin{tikzcd}
			\mathcal N \times \mathcal A \ar[d, "\otimes"'] \ar[r, "f_* \times \mathcal A"] & \mathcal M \times \mathcal A \ar[d, "\otimes"] \\
			\mathcal N \ar[r, "f_*"'] & \mathcal M
		\end{tikzcd}
	\]
	which gives $f_*$ the structure of a map in $\Mod_{\mathcal A}(\Cat)$.

	Conversely, given the existence of a commutative square of the latter form, if it is left adjointable, and the left base change transformation recovers the first commutative square, then the left base change transformation of the first commutative square is invertible, so that $f_*$ is a linear right adjoint of $f^*$. 
\end{rmk}

\begin{lem} \label{lem:proj form implies ff equiv triv fund class}
	Let $f^* : \mathcal D \to \mathcal C$ be a monoidal functor with a linear left adjoint $f_\sharp$. Then $f^*$ is fully faithful if and only if $f_\sharp(1) \simeq 1$.
	\begin{proof}
		Suppose that $f_\sharp(1) \simeq 1$. Then since $f^*$ has a linear left adjoint, we have that
		\[
			f_\sharp f^* \simeq - \otimes f_\sharp(1) \simeq - \otimes 1 \simeq \id
		,\]
		so $f_\sharp f^* \simeq \id$. Thus, by \cite[Lemma 3.3.1]{CARMELI2021107763}, it follows that $f^*$ is fully faithful.

		Conversely, if $f^*$ is fully faithful, then the counit of $f_\sharp \dashv f^*$ is an equivalence, so
		\[
			f_\sharp 1 \simeq f_\sharp f^* 1 \to 1
		\]
		is an equivalence.
	\end{proof}
\end{lem}

\begin{thm} \label{thm:Mod LAd}
	% Let $\CAlg^\LAd(\PrL)$ be the wide subcategory of $\CAlg(\PrL)$ where the morphisms are the functors admitting linear left adjoints. Then
	Let $\mathcal A \in \CAlg(\PrL)$, and let $\Mod^\LAd_{\mathcal A}(\PrL)$ be the wide subcategory of $\Mod_{\mathcal A}(\PrL)$ where the morphisms are the functors admitting $\mathcal A$-linear left adjoints.
	\begin{enumerate}

		\item\label{itm:proj form/transf comp}
			Let $\lambda$ be an ordinal, and let $F : (\lambda + 1)^\op \to \Mod_{\mathcal A} \PrL$ be a diagram such that for every limit ordinal $\lambda' \leq \lambda$, $F$ restricts to a limiting diagram on $(\lambda' + 1)^\op \cong ((\lambda')^\op)^\triangleleft$. Suppose that for every $\alpha \in \lambda$, $F(\alpha) \to F(\alpha + 1)$ has a linear left adjoint. Then $F$ lands in $\Mod_{\mathcal A}^\LAd \PrL$, and for any limit ordinal $\lambda' \leq \lambda$, and $\mathcal M \in \Mod_{\mathcal A} \PrL$, a map $F(\lambda') \to \mathcal M$ has a linear left adjoint if and only if for all $\alpha \in \lambda'$, the map $F(\alpha) \to \mathcal M$ has a linear left adjoint.

		\item\label{itm:proj form/source-local}
			% Given an $F : K^\triangleleft \to \CAlg^\LAd(\PrL)$, if
			Given an $F : K^\triangleleft \to \Mod^\LAd_{\mathcal A}(\PrL)$, if
			\[
				F(\infty) \to \varprojlim_{p \in K} F(p)
			\]
			% is conservative, then for any $\phi : \mathcal C \to F(\infty)$ in $\CAlg(\PrL)$, we have that $\phi$ has a linear left adjoint if and only if for all $p \in K$, the functor $\phi_p : \mathcal C \to F(p)$ has a linear left adjoint.
			is conservative, then for any $\phi : \mathcal C \to F(\infty)$ in $\Mod_{\mathcal A}(\PrL)$, we have that $\phi$ has an $\mathcal A$-linear left adjoint if and only if for all $p \in K$, the functor $\phi_p : \mathcal C \to F(p)$ has an $\mathcal A$-linear left adjoint.

		\item\label{itm:proj form/limits}
			% $\CAlg^\LAd(\PrL)$ admits small limits, and the inclusion $\CAlg^\LAd(\PrL) \to \CAlg(\PrL)$ preserves small limits.
			The category $\Mod^\LAd_{\mathcal A}(\PrL)$ admits small limits, and the forgetful functor $\Mod_{\mathcal A}^\LAd(\PrL) \to \widehat{\Cat}$ preserves small limits.

	\end{enumerate}
	\begin{proof}\hfill
		\begin{enumerate}

			\item This follows immediately from \cite[Proposition 5.5.3.13]{htt} and \Cref{cor:adjointability horiz composites}, where $\phi$ is $\otimes : F \times \mathcal A \to F$.

			\item For any $a \in \mathcal A$, $\otimes a$ defines an endomorphism of the forgetful functor $\Mod_{\mathcal A} \PrL \to \widehat{\Cat}$, which is left adjointable at every map in the image of $F$ because $F$ lands in $\Mod^\LAd_{\mathcal A}(\PrL)$, and furthermore, for any $\mathcal C \in \Mod_{\mathcal A} \PrL$, the functor $\otimes a : \mathcal C \to \mathcal C$ admits a right adjoint. Thus, we conclude by \Cref{prp:Pr adj lims}.

			\item Note that by \cite[Corollary 3.4.3.6]{ha} and \cite[Proposition 5.5.3.13]{htt}, we have that $\Mod_{\mathcal A} \PrL$ admits small limits, and the forgetful functor $\Mod_{\mathcal A} \PrL \to \widehat{\Cat}$ preserves small limits, so we just need to show that if $F : K \to \Mod_{\mathcal A}(\PrL)$ is a small diagram that lands in $\Mod^\LAd_{\mathcal A}(\PrL) \subseteq \Mod_{\mathcal A}(\PrL)$, then a map
				\[
					\mathcal C \to \varprojlim F
				\]
				admits an $\mathcal A$-linear left adjoint if and only if $\mathcal C \to F(p)$ admits an $\mathcal A$-linear left adjoint for all $p \in K$, but this follows immediately from \Cref{prp:Pr adj lims} applied to the transformations induced by $a \otimes -$ for each $a \in \mathcal A$.

		\end{enumerate}
	\end{proof}
\end{thm}

\begin{lem} \label{lem:tensor invertibility is local}
	Let $F : \mathcal C \to \mathcal D$ be a map in $\CAlg(\PrL)$ that has a linear left adjoint, and assume $F$ is a conservative functor. For any object $X \in \mathcal C$, if $F(X)$ is $\otimes$-invertible, then $X$ is $\otimes$-invertible.

	This also holds if $F$ is a map in $\CAlg(\widehat{\Cat})$ as long as $\mathcal C$ and $\mathcal D$ have closed monoidal structures.
	\begin{proof}
		Write $\underline{\mathcal C} : \mathcal C^\op \times \mathcal C \to \mathcal C$ and $\underline{\mathcal D} : \mathcal D^\op \times \mathcal D \to \mathcal D$ for the internal hom functors associated to the closed\footnote{Recall the monoidal structure is automatically closed if $\mathcal C, \mathcal D \in \CAlg(\PrL)$.} monoidal structures of $\mathcal C$ and $\mathcal D$ respectively.

		If $F(X)$ is $\otimes$-invertible, then
		\[
			\mathcal D(-, F(X)^{\otimes -1}) \simeq \mathcal D(- \otimes F(X), 1) \simeq \mathcal D(-, \underline{\mathcal D}(F(X), 1))
		,\]
		so the map $F(X) \otimes \underline{\mathcal D}(F(X), 1) \to 1$ is invertible.

		Since $F$ has a linear left adjoint, it sends $X \otimes \underline{\mathcal C}(X,1) \to 1$ to the invertible map $F(X) \otimes \underline{\mathcal D}(F(X), 1) \to 1$, so since $F$ is conservative, the former map is invertible, which shows that $X$ is $\otimes$-invertible.
	\end{proof}
\end{lem}

\subsection{Base Change}

\begin{defn} \label{defn:base change}
	Say a presheaf $D : \mathcal C^\op \to \Cat$ satisfies left (right) base change for a map $f : X \to Y$ against a map $q : Y' \to Y$ if the pullback $X \times_Y Y'$ exists, and the square
	\[
		\begin{tikzcd}
			D(Y) \ar[d, "q^*"'] \ar[r, "f^*"] & D(X) \ar[d] \\
			D(Y') \ar[r] & D(X')
		\end{tikzcd}
	\]
	is left (right) adjointable.

	Say $D$ has left (right) base change for $f$ if it has left (right) base change for $f$ against all maps.
\end{defn} 

\begin{prp} \label{prp:transf comp base change}
	Let $D : \mathcal C^\op \to \widehat{\Cat}$ be a presheaf on a category $\mathcal C$, let $\lambda$ be an ordinal, and let $f : X \to Y$ be a Cartesian transformation of diagrams $X,Y : \lambda + 1 \to \mathcal C$ such that for any limit ordinal $\lambda' \in \lambda$, $DX^\op$ and $DY^\op$ restrict to limiting diagrams on $(\lambda' + 1)^\op \cong ((\lambda')^\op)^\triangleleft$.
	\begin{enumerate}

		\item Suppose that for each $\alpha \in \lambda$, $D$ has left (right) base change for $f(\alpha + 1)$ against $Y(\alpha) \to Y(\alpha + 1)$. Then for any $\alpha \leq \beta \leq \lambda$, $D$ has left (right) base change for $f(\beta)$ against $Y(\alpha) \to Y(\beta)$.

		\item Suppose that for each $\alpha \in \lambda$, $D$ has left (right) base change for $Y(\alpha) \to Y(\alpha + 1)$ against $f(\alpha + 1)$, and that for all $\alpha \leq \lambda$, $f(\alpha)^*$ has a right (left) adjoint. Then for any $\alpha \leq \beta \leq \lambda$, $D$ has left (right) base change for $Y(\alpha) \to Y(\beta)$ against $f(\beta)$.

	\end{enumerate}
	\begin{proof}\hfill
		\begin{enumerate}

			\item This follows from \Cref{lem:adjointability vert composites}.

			\item This follows from \Cref{cor:adjointability horiz composites}.

		\end{enumerate}
		
	\end{proof}
\end{prp}

\begin{prp} \label{prp:source locality of base change}
	Let $D : \mathcal C^\op \to \widehat{\Cat}$ be a presheaf on a category $\mathcal C$, and let $f : X \to Y$ be a Cartesian transformation of small diagrams $X,Y : K \star \Delta^1 \to \mathcal C$.

	\begin{enumerate}

		\item Suppose that $D X^\op$ and $D Y^\op$ restrict to limiting diagrams on $(K \star \{0\})^\op$, and for every map $a \to b$ in $K$, $D$ has left (right) base change for $f(b)$ against $Y(a \to b)$. Then $D$ has left (right) base change for $f(0)$ against $Y(a \to 0)$ for all $a \in K$, and $D$ has left (right) base change for $f(1)$ against $Y(0) \to Y(1)$ if and only if it has left (right) base change for $f(1)$ against $Y(a \to 1)$ for all $a \in K$.

		\item Suppose that for every $a \in K \star \{1\}$, $f(a)^*$ is a left (right) adjoint functor between presentable categories, and suppose that for any map $a \to b$ in $K$, $D$ has left (right) base change for $Y(a \to b)$ against $f(b)$.
			\begin{enumerate}

				\item Suppose that $D X^\op$ and $D Y^\op$ restrict to limiting diagrams on $(K \star \{0\})^\op$. Then $D$ has left (right) base change for $Y(0) \to Y(1)$ against $f(1)$ if and only if it has left (right) base change for $Y(a) \to Y(1)$ against $f(1)$ for all $a \in K$.

				\item Suppose that $D Z(0) \to \varprojlim_{a \in K} D Z(a)$ is conservative for $Z \in \{X,Y\}$, that $f(0)^*$ is a left (right) adjoint functor between presentable categories, and $D$ has left (right) base change for $Y(a) \to Y(0)$ against $f(0)$ for all $a \in K$. Then $D$ has left (right) base change for $Y(0) \to Y(1)$ if and only if it has left (right) base change for $Y(a) \to Y(1)$ against $f(1)$ for all $a \in K$.

			\end{enumerate}
		
	\end{enumerate}
	\begin{proof}\hfill
		\begin{enumerate}

			\item Follows immediately from \Cref{lem:adjointability limits}.
		
			\item Follows immediately from \Cref{prp:Pr adj lims}.

		\end{enumerate}
	\end{proof}
\end{prp}

\begin{rmk} \label{rmk:Cech nerve}
	Let $\mathcal C$ be a category, and let $X : \Delta_+^\op \to \mathcal C$ be an augmented simplicial object. Recall from the discussion following \cite[Proposition 6.1.2.11]{htt} that $X$ is said to be a \v{C}ech nerve if it satisfies the equivalent conditions of this result. In fact, the following are equivalent:
	\begin{enumerate}

		\item $X$ is a \v{C}ech nerve.

		\item For all $n \geq 0$, the square
			\[
				\begin{tikzcd}
					X_{n + 1} \ar[d] \ar[r, "d_0"] & X_n \ar[d] \\
					X_0 \ar[r] & X_{-1}
				\end{tikzcd}
			\]
			is Cartesian.

		\item For any map $[m] \to [n]$ in $\Delta_+$, and $k \geq 0$ the square
			\begin{equation} \label{eqn:Cech nerve square}
				\begin{tikzcd}
					X_{n + k} \ar[d] \ar[r, "(d_0)^{\circ k}"] & X_n \ar[d] \\
					X_{m + k} \ar[r, "(d_0)^{\circ k}"'] & X_m
				\end{tikzcd}
			\end{equation}
			is Cartesian.

	\end{enumerate}
	\begin{proof}
		The equivalence of the last two conditions follows easily by observing that $[-1]$ is an initial object of $\Delta_+$, and pasting Cartesian squares (see \cite[Lemma 4.4.2.1]{htt}).

		If $X$ is a \v{C}ech nerve, then by \cite[Proposition 6.1.2.11]{htt}, the underlying simplicial object of $X$ is a groupoid, and the square \eqref{eqn:Cech nerve square} is Cartesian for $m,n = 0$ and $k = 1$. Since $X$ is a groupoid, \cite[Proposition 6.1.2.6(4'')]{htt} says that the top square in
		\[
			\begin{tikzcd}
				X_{n + 1} \ar[d, "d_2 \circ \dotsb \circ d_{n + 1}"'] \ar[r, "d_0"] & X_n \ar[d, "d_1 \circ \dotsb \circ d_n"] \\
				X_1 \ar[d, "d_1"'] \ar[r, "d_0"] & X_0 \ar[d] \\
				X_0 \ar[r] & X_{-1}
			\end{tikzcd}
		\]
		is Cartesian for $n \geq 0$, so since the bottom square is also Cartesian, the outer rectangle is Cartesian, as required.

		Conversely, if the square \eqref{eqn:Cech nerve square} is always Cartesian, then the bottom square in the above rectangle is Cartesian (by setting $m = n = 0$ and $k = 1$). By \cite[Proposition 6.1.2.11]{htt}, it suffices to show that the underlying simplicial object of $X$ is a groupoid, but by \cite[Proposition 6.1.2.6(4'')]{htt}, this is equivalent to showing that the square
		\[
			\begin{tikzcd}
				X_{k + n} \ar[d] \ar[r] & X_n \ar[d] \\
				X_k \ar[r] & X_0
			\end{tikzcd}
		\]
		is Cartesian for all $m,n \geq 0$, but this follows easily from the fact that \eqref{eqn:Cech nerve square} is Cartesian for $m = 0$.
	\end{proof}
\end{rmk}

\begin{prp} \label{prp:nerve adj}
	Let $\mathcal C$ be a category, and let $\{X_i\}_{i \in I}$ be a family of objects of $\mathcal C$ such that for all $n \geq 0$ and indices $\bar i = (i_1, \dotsc, i_n) \in I^n$, the product
	\[
		X_{\bar i} \coloneqq \prod_{k = 1}^n X_{i_k}
	\]
	exists.

	For any presheaf $D : \mathcal C^\op \to \widehat{\Cat}$ such that for all $n \geq 0$ and $\bar i \in I^n$,
	\begin{enumerate}

		\item $D(X_{\bar i})$ admits $I$-indexed coproducts.

		\item For any $j \in I$, the functors
			\[
				D(X_j \times X_{\bar i} \to X_{\bar i}) \quad\text{and}\quad
				D(X_{\bar i} \times (X_j \to X_j \times X_j))
			\]
			preserve $I$-indexed coproducts.

		\item For any $j,j' \in I$, $D$ has left base change for
			\begin{align*}
				X_j \times X_{\bar i} \to X_{\bar i} &\text{ against } X_{j'} \times X_{\bar i} \to X_{\bar i} \\
				(X_j \times X_{j'} \times X_{j'} \to X_{j'} \times X_{j'}) \times X_{\bar i} &\text{ against } (X_{j'} \to X_{j'} \times X_{j'}) \times X_{\bar i}
			\end{align*}

	\end{enumerate}
	
	Let $\mathcal U \in \tau_{\leq -1} \Psh(\mathcal C)$ be the sieve generated by the family $\{X_i\}_{i \in I}$. Then for any $m,n \geq 0$, and indices $\bar i \in I^{m + 1}$, and $\bar j \in I^{n + 1}$, the square
	\begin{equation} \label{eqn:nerve ladj lim}
		\begin{tikzcd}
			D(\mathcal U) \ar[d] \ar[r] & D(X_{\bar i}) \ar[d] \\
			D(X_{\bar j}) \ar[r] & D(X_{\bar i} \times X_{\bar j})
		\end{tikzcd}
	\end{equation}
	is left adjointable.

	For each $i \in I$, write $G_i : D(\pt) \to D(X_i)$ for the functor induced by $X_i \to \pt$, and assume that $D(\pt)$ preserves geometric realizations of simplicial objects that are $G_i$-split for all $i \in I$, and that $G_i$ preserves these geometric realizations for all $i \in I$.

	Then $D$ has descent along $\mathcal U$ if and only if the functors $\{G_i\}_{i \in I}$ are jointly conservative. In fact, if we view $D$ as a limit-preserving presheaf on $\Psh(\mathcal C)$, then we have the following:
	\begin{enumerate}

		\item $D(\pt) \to D(\mathcal U)$ has a fully faithful left adjoint.

		\item Let $\mathcal C'$ be a full subcategory of $\mathcal C$ that contains all objects of the form $X_{\bar i}$ for $\bar i \in I^n$ and $n \geq 0$. Then $\mathcal U \in \Psh(\mathcal C') \subseteq \Psh(\mathcal C)$. Let $\phi : D' \to D|_{\mathcal C'}$ be a transformation such that $D'(\pt)$ admits geometric realizations and $|I|$-small coproducts, $\phi : D'(\pt) \to D(\pt)$ and $\phi : D'(X_i) \to D(X_i)$ admit right adjoints for all $i \in I$, and $\phi$ is left adjointable at $X_{\bar i} \to \pt$ for all $\bar i \in I^n$ and $n \geq 0$. Then for all $n \geq 1$ and $\bar i \in I^n$, we have that the left square in
			\[
				\begin{tikzcd}
					D'(\pt) \ar[d] \ar[r] & D'(\mathcal U) \ar[d] \ar[r] & D'(X_{\bar i}) \ar[d] \\
					D(\pt) \ar[r] & D(\mathcal U) \ar[r] & D(X_{\bar i})
				\end{tikzcd}
			\]
			is left adjointable, and the right square is left adjointable if the top right horizontal arrow admits a left adjoint.\footnote{We obtain the middle vertical arrow by viewing $\phi$ as a transformation of limit-preserving presheaves on $\Psh(\mathcal C') \subseteq \Psh(\mathcal C)$.}

	\end{enumerate}
	\begin{proof}
		By \cite[Proposition 6.2.3.4]{htt}, we have that $\mathcal U$ is the geometric realization of the underlying simplicial object of the \v{C}ech nerve of $\coprod_{i \in I} \yo(X_i)$. The \v{C}ech nerve is the augmented simplicial object $\tilde X : \Delta_+^\op \to \Psh(\mathcal C)$ such that for each $n \geq 0$, the face map $d_0 : \tilde X_n \to \tilde X_{n - 1}$ is equivalent to
		\[
			\coprod_{i_0, \dotsc, i_n \in I} \yo(X_{(i_0, \dotsc, i_n)}) \simeq \coprod_{i_0, \dotsc, i_n \in I} \yo(X_{i_0}) \times \yo(X_{(i_1, \dotsc, i_n)}) \to \coprod_{i_1, \dotsc, i_n \in I} \yo(X_{(i_1, \dotsc, i_n)})
		,\]
		which is equivalent to
		\[
			\coprod_{\bar i \in I^n} \coprod_{j \in I} \yo(X_j \times X_{\bar i} \to X_{\bar i})
		.\]
		In fact, by permuting factors we find that all face maps are equivalent to $d_0$.

		Similarly, every degeneracy map $s_k : \tilde X_{n - 1} \to \tilde X_n$ is equivalent to the composite
		\[
			\coprod_{i_0, \dotsc, i_{n - 1} \in I} X_{(i_0, \dotsc, i_{n - 1})} \to \coprod_{\substack{i_0, \dotsc, i_n \in I \\ i_k = i_{k + 1}}} X_{(i_0, \dotsc, i_n)} \hookrightarrow \coprod_{i_0, \dotsc, i_n \in I} X_{(i_0, \dotsc, i_n)}
		,\]
		in which the second map is the inclusion of disjoint summands, and the first map is equivalent to
		\[
			\coprod_{\bar i \in I^{n - 1}} \coprod_{j \in I} \yo(X_{\bar i} \times (X_j \xrightarrow{\text{diagonal}} X_j \times X_j))
		.\]

		We will show that for every map $[m] \to [n]$ in $\Delta_+$, the square
		\begin{equation} \label{eqn:nerve ladj}
			\begin{tikzcd}
				D(\tilde X_m) \ar[d] \ar[r] & D(\tilde X_{m + 1}) \ar[d] \\
				D(\tilde X_n) \ar[r] & D(\tilde X_{n + 1})
			\end{tikzcd}
		\end{equation}
		is left adjointable.

		Recall from \Cref{rmk:Cech nerve} that for any map $[m] \to [n]$ in $\Delta_+$, since $\tilde X$ is a \v{C}ech nerve, the square
		\[
			\begin{tikzcd}
				\tilde X_{n + 1} \ar[d] \ar[r, "d_0"] & \tilde X_n \ar[d] \\
				\tilde X_{m + 1} \ar[r, "d_0"'] & \tilde X_m
			\end{tikzcd}
		\]
		is Cartesian. By \cite[Lemma F.6(3)]{TwAmb}, since every map in $\Delta_+$ is a composite of codegeneracy maps and coface maps, to show that the square \eqref{eqn:nerve ladj} is left adjointable for all morphsims $[m] \to [n]$ in $\Delta_+$, it suffices to show that $D$ has left base change for $d_0$ against $d_k$ and $s_k$.

		\begin{description}

			\item[$d_0$ against $d_k$]
				We have an equivalence of Cartesian squares
				\[
					\begin{tikzcd}
						\tilde X_{n + 1} \ar[d, "d_{k + 1}"'] \ar[r, "d_0"] & \tilde X_n \ar[d, "d_k"] \\
						\tilde X_n \ar[r, "d_0"'] & \tilde X_{n - 1}
					\end{tikzcd}
					\simeq \coprod_{\bar i \in I^n} \left(
						\begin{tikzcd}
							\displaystyle\coprod_{j,j' \in I} \yo(X_j \times X_{j'} \times X_{\bar i}) \ar[d] \ar[r] & \displaystyle\coprod_{j' \in I} \yo(X_{j'} \times X_{\bar i}) \ar[d] \\
							\displaystyle\coprod_{j \in I} \yo(X_j \times X_{\bar i}) \ar[r] & \yo(X_{\bar i})
					\end{tikzcd} \right)
				.\]
				Since $D$ is limit-preserving on $\Psh(\mathcal C)^\op$, and by \cite[Corollary 4.7.4.18(2)]{ha} and \Cref{lem:adjointability limits}, to show that $D$ has left base change for $d_0$ against $d_0$, it suffices to show that for each $\bar i \in I^n$ and $j' \in I$, the square
				\[
					\begin{tikzcd}
						D(X_{\bar i}) \ar[d] \ar[r] & \prod_{j \in I} D(X_j \times X_{\bar i}) \ar[d] \\
						D(X_{j'} \times X_{\bar i}) \ar[r] & \prod_{j \in I} D(X_j \times X_{j'} \times X_{\bar i})
					\end{tikzcd}
				\]
				is left adjointable.

				We have assumed that for any $j \in I$, $D$ has left base change for $X_j \times X_{\bar i} \to X_{\bar i}$ against $X_{j'} \times X_{\bar i} \to X_{\bar i}$. In particular, $D(X_j \times X_{\bar i} \to X_{\bar I})$ and $D(X_j \times X_{j'} \times X_{\bar i} \to X_{j'} \times X_{\bar i})$ have left adjoints. It follows from \cite[Lemma D.4.7(ii)]{HM6FF} that the horizontal arrows in the above squares have left adjoints that can be expressed in terms of $I$-indexed coproducts of the left adjoints of the two aforementioned functors. We have also assumed that the vertical arrows preserve $I$-indexed coproducts, so the fact that this square is left adjointable follows from the fact that $D$ has the left base change properties we just recalled.

			\item[$d_0$ against $s_k$]
				Note that
				\[
					\begin{tikzcd}
						\tilde X_n \ar[d, "s_{k + 1}"'] \ar[r, "d_0"] & \tilde X_{n - 1} \ar[d, "s_k"] \\
						\tilde X_{n + 1} \ar[r, "d_0"'] & \tilde X_n
					\end{tikzcd}
				\]
				is equivalent to the same square where $k = 0$, and this is equivalent to the coproduct over $\bar i \in I^{n - 1}$ and $j' \in I$ of the following composite of Cartesian squares:
				\[
					\begin{tikzcd}
							\displaystyle\coprod_{j \in I} \yo(X_j \times X_{j'} \times X_{\bar i}) \ar[d] \ar[r] & \yo(X_{j'} \times X_{\bar i}) \ar[d] \\
							\displaystyle\coprod_{j \in I} \yo(X_j \times X_{j'} \times X_{j'} \times X_{\bar i}) \ar[d] \ar[r] & \yo(X_{j'} \times X_{j'} \times X_{\bar i}) \ar[d] \\
							\displaystyle \coprod_{j,j'' \in I} \yo(X_j \times X_{j'} \times X_{j''} \times X_{\bar i}) \ar[r] & \displaystyle\coprod_{j'' \in I} \yo(X_{j'} \times X_{j''} \times X_{\bar i})
					\end{tikzcd}
				.\]
				As before, using \cite[Corollary 4.7.4.18(2)]{ha} and \cite[Lemma D.4.7]{HM6FF}, the fact that $D$ sends the vertical arrows of the top square to functors that preserve $I$-indexed coproducts, and the fact that $D$ has left base change for $(X_j \times X_{j'} \times X_{j'} \to X_{j'} \times X_{j'}) \times X_{\bar i}$ against $(X_{j'} \to X_{j'} \times X_{j'}) \times X_{\bar i}$ for all $j,j' \in I$ and $\bar i \in I^{n - 1}$, it follows that $D$ sends the top square to a left adjointable square.

				Note that since the vertical arrows of the bottom square are inclusions of disjoint summands, and $D$ is a limit-preserving presheaf on $\Psh(\mathcal C)$, it follows that $D$ sends this square to a left adjointable square since it sends the horizontal arrows to functors that admit left adjoints. Thus, we conclude \cite[Lemma F.6(3)]{TwAmb} that $D$ sends this composite of Cartesian squares to a left adjointable square.

				Finally, since $D$ is limit-preserving on $\Psh(\mathcal C)^\op$, it follows from \cite[Corollary 4.7.4.18(2)]{ha} that $D$ sends the coproduct over $\bar i \in I^{n - 1}$ and $j' \in I$ of these Cartesian squares to a left adjointable square.

		\end{description}
		This concludes the proof that the square \eqref{eqn:nerve ladj} is left adjointable for all maps $[m] \to [n]$ in $\Delta_+$.

		To see that the square \eqref{eqn:nerve ladj lim} is left adjointable, consider the following diagram:
		\[
			\begin{tikzcd}
				D(\mathcal U) \ar[d] \ar[r] & D(\tilde X_0) \ar[d] \ar[r] & \cdots \ar[r] & D(\tilde X_\ell) \ar[dd] \ar[r] & \cdots \ar[r] & D(\tilde X_m) \ar[dd] \\
				D(\tilde X_0) \ar[d] \ar[r] & D(\tilde X_1) \ar[d] \\
				D(\tilde X_n) \ar[r] & D(\tilde X_{n + 1}) \ar[r] & \cdots \ar[r] & D(\tilde X_{\ell + n + 1}) \ar[r] & \cdots \ar[r] & D(\tilde X_{m + n + 1})
			\end{tikzcd}
		.\]
		Note that since $\mathcal U \simeq \realise{\tilde X_\bullet}$, the top left square is left adjointable by \cite[Theorem 4.7.5.2]{htt}, and the bottom left square is left adjointable since \eqref{eqn:nerve ladj} is left adjointable for $m = 0$, so by \cite[Lemma F.6(3)]{TwAmb}, the outer left rectangle is left adjointable. Using the left adjointability of \eqref{eqn:nerve ladj} again, the right rectangles are left adjointable by \cite[Lemma F.6(1)]{TwAmb}, so the outermost rectangle is also left adjointable by the same result.

		Write $G$ for the functor
		\[
			D(\tilde X_{-1}) \simeq D(\pt) \to \prod_{i \in I} D(X_i) \simeq D(\tilde X_0)
		.\]
		Note that a simplicial object of $D(\pt)$ is $G$-split if and only if it is $G_i$-split for all $i \in I$, and, if it exists, its geometric realization is preserved by $G$ if and only if it is preserved by $G_i$ for all $i$. Thus, under our additional hypotheses, we find that $D(\pt)$ admits geometric realizations of $G$-split simplicial objects, and $G$ preserves these realizations. Thus, we may apply \cite[Corollary 4.7.5.3]{ha} to see that $D(\pt) \to D(\mathcal U)$ admits a fully faithful left adjoint.

		For the last statement, first note that $\mathcal U$ is in $\Psh(\mathcal C')$, seen as a full subcategory of $\Psh(\mathcal C)$ with colimit-preserving inclusion induced by the inclusion $\mathcal C' \to \mathcal C$, since $\Psh(\mathcal C') \subseteq \Psh(\mathcal C)$ is closed under colimits, and $\mathcal U$ is a colimit of objects that we know are in $\Psh(\mathcal C')$. It follows that the functors $\{D'(\mathcal U) \to D'(X_i)\}_{i \in I}$ are jointly conservative, and $D'(\pt) \to D'(\mathcal U)$ admits a left adjoint by \cite[Lemma D.4.7(ii)]{HM6FF} since $D'(\pt)$ admits geometric realizations and $|I|$-small coproducts. Thus, we may conclude by \Cref{lem:ff adj crit}.
	\end{proof}
\end{prp}

\section{Stabilization} \label{S:stabilization}

Given a commutative ring $R$, an $R$-module $M$, and an element $x \in R$, one can construct the localization $M[x^{-1}]$ as the colimit
\[
	M \xrightarrow{x} M \xrightarrow{x} \cdots
,\]
which can be thought of as an ``$x$-stabilization of $M$''.

Sometimes this construction of $M[x^{-1}]$ can allow us to prove nice behaviours of the functor $\Mod_R \to \Mod_{R[x^{-1}]}$.

In this section, we will be interested in the case that $R$ is a symmetric monoidal presentable category $\mathcal A$, and $\Mod_R$ is replaced by $\Mod_{\mathcal A}(\PrL)$, and we will establish in \Cref{prp:stabilization} some nice $2$-categorical behaviour of the functor on $\Mod_{\mathcal A}(\PrL)$ that ``freely inverts'' the action of some set $A$ of elements in $\mathcal A$.

Before continuing our discussion of this matter, we record the following result that relates the usual notion of stability for categories to the notions considered in this section.
\begin{lem} \label{lem:monoidal crit for stability}
	Let $\mathcal C$ be a symmetric monoidal category that has a zero object, and such that the monoidal product $\otimes : \mathcal C \times \mathcal C \to \mathcal C$ preserves finite colimits in each parameter. Then the following are equivalent, where $\Sigma : \mathcal C \to \mathcal C$ denotes the suspension functor:\footnote{See \cite[\S1.1.2]{ha}.}
	\begin{enumerate}

		\item There is some $M \in \mathcal C$ such that $\Sigma M$ is $\otimes$-invertible.

		\item The object $\Sigma 1$ is $\otimes$-invertible.

		\item $\mathcal C$ is a stable category.

	\end{enumerate}
	\begin{proof}
		Note that since $\otimes$ preserves empty colimits in each variable, if $0$ is a zero object of $\mathcal C$, we have that $0 \otimes X \simeq 0$ for all $X \in \mathcal C$, and since $\otimes$ preserves pushouts in each variable, we have that $\Sigma 1 \otimes - \simeq \Sigma(1 \otimes -) \simeq \Sigma$, where $1 \in \mathcal C$ is the monoidal unit, so \cite[Proposition 1.4.2.11]{ha} says that $\mathcal C$ is stable if and only if $\Sigma 1$ is $\otimes$-invertible.

		Finally, note that for any $M \in \mathcal C$, $\Sigma M \simeq \Sigma 1 \otimes M$, so if $\Sigma M$ is $\otimes$-invertible, then so is $\Sigma 1$.
	\end{proof}
\end{lem}

First we will need to recall the following general fact from \cite{ha}:
\begin{lem} \label{lem:ext scal}
	Let $\mathcal A \to \mathcal B$ be a map in $\CAlg(\PrL)$. Then there is a symmetric monoidal functor $\Mod_{\mathcal A}(\PrL) \to \Mod_{\mathcal B}(\PrL)$ that is a left adjoint of the canonical map $\Mod_{\mathcal B}(\PrL) \to \Mod_{\mathcal A}(\PrL)$ given by \cite[Corollary 3.4.3.3]{ha}.
	\begin{proof}
		Follows immediately from \cite[Theorem 4.5.3.1]{ha}.
		% NOTE: details
		% By \cite[Theorem 4.5.3.1]{ha}, the map
		% \[
		% 	\Mod(\PrL)^\otimes \to \CAlg(\PrL) \times \Fin_*
		% \]
		% is a coCartesian fibration, which establishes that there is a map $\Mod_{\mathcal A}(\PrL) \to \Mod_{\mathcal B}(\PrL)$, but \cite[Corollary 3.4.3.3]{ha} says that the map $\phi : \Mod(\PrL) \to \CAlg(\PrL)$ given as the composite
		% \[
		% 	\phi : \Mod(\PrL) \to \Mod(\PrL)^\otimes \to \CAlg(\PrL) \times \Fin_* \to \CAlg(\PrL)
		% \]
		% is a Cartesian fibration, and that a map in $\Mod(\PrL)$ is $\phi$-Cartesian if and only if its image in $\PrL$ is an equivalence, so it follows that $\Mod_{\mathcal A}(\PrL) \to \Mod_{\mathcal B}(\PrL)$ has a right adjoint that is a map in $\widehat{\Cat}_{/\PrL}$.
	\end{proof}
\end{lem}

Following \cite[\S6.1]{sixopsequiv} and \cite[\S2.1]{robalo}, we have that for any symmetric monoidal presentable category $\mathcal A \in \CAlg(\PrL)$, and small collection $A$ of objects in $\mathcal A$, there is a map $\Sigma_A^\infty : \mathcal A \to \mathcal A[A^{\otimes -1}]$ in $\CAlg(\PrL)$ which is the initial map that sends every element of $A$ to a $\otimes$-invertible object.

Using the ideas of \cite[\S6.1]{sixopsequiv} and \cite[Proposition 2.9]{robalo}, we find that the natural commutative square
\begin{equation} \label{eqn:tensor loc square}
	\begin{tikzcd}
		\CAlg(\PrL)_{\mathcal A[A^{\otimes -1}]/} \ar[d] \ar[r] & \CAlg(\PrL)_{\mathcal A/} \ar[d] \\
		\Mod_{\mathcal A[A^{\otimes -1}]}(\PrL) \ar[r] & \Mod_{\mathcal A}(\PrL)
	\end{tikzcd}
\end{equation}
has fully faithful rows, and in fact it is Cartesian, and the essential image of the bottom arrow is the full subcategory of $\mathcal A$-modules on which the elements of $A$ act by automorphisms.

By \Cref{lem:compatible localizations} and the discussion preceding \cite[Proposition 2.9]{robalo}, this square is actually left adjointable.

We establish the following notation:
\begin{nota}
	Given $\mathcal A \in \CAlg(\PrL)$, and $A$ a small collection of objects in $\mathcal A$, we denote $-[A^{\otimes -1}]$ the left adjoint of either horizontal inclusion in \eqref{eqn:tensor loc square}, and we denote the unit of the adjunction by $\Sigma_A^\infty : \id \to -[A^{\otimes -1}]$.
\end{nota}

Now we come to the main result of this section.

\begin{prp} \label{prp:stabilization}
	Let $\mathcal A \in \CAlg(\PrL)$ be symmetric monoidal presentable category, and let $A$ be a small collection of symmetric (\Cref{defn:symmetric}) objects in $\mathcal A$.

	\begin{enumerate}

		\item\label{itm:stabilization/generation}
			For any $\mathcal M \in \Mod_{\mathcal A}(\PrL)$, the functor $\Sigma_A^\infty : \mathcal M \to \mathcal M[A^{\otimes -1}]$ satisfies that the category $\mathcal M[A^{\otimes -1}]$ is generated under filtered colimits by objects of the form $\alpha^{\otimes -1} \otimes \Sigma_A^\infty M$, for $M \in \mathcal M$, and $\alpha$ a tensor product of objects in $A$.

		\item\label{itm:stabilization/linear}
			For any map $\mathcal M \to \mathcal N$ in $\Mod_{\mathcal A}(\PrL)$ that admits an $\mathcal A$-linear left adjoint, the induced functor $\mathcal M[A^{\otimes -1}] \to \mathcal N[A^{\otimes -1}]$ has an $\mathcal A[A^{\otimes -1}]$-linear left adjoint, the square of categories
			\begin{equation} \label{eqn:stabilization/linear}
				\begin{tikzcd}
					\mathcal M \ar[d] \ar[r] & \mathcal N \ar[d] \\
					\mathcal M[A^{\otimes -1}] \ar[r] & \mathcal N[A^{\otimes -1}]
				\end{tikzcd}
			\end{equation}
			is left adjointable, and the commutative square
			\[
				\begin{tikzcd}
					\mathcal N \ar[d] \ar[r] & \mathcal M \ar[d] \\
					\mathcal N[A^{\otimes -1}] \ar[r] & \mathcal M[A^{\otimes -1}]
				\end{tikzcd}
			\]
			given by its left base change transformation is the morphism $f_\sharp \to f_\sharp[A^{\otimes -1}]$ in $\Fun(\Delta^1, \Mod_{\mathcal A}(\PrL))$, where $f_\sharp$ is the left adjoint of $\mathcal M \to \mathcal N$.

			% Furthermore, the left adjoint $\mathcal N[A^{\otimes -1}] \to \mathcal M[A^{\otimes -1}]$ is the map induced by the left adjoint $\mathcal N \to \mathcal M$.
			% there is a map $f_! : \mathcal N \to \mathcal M$ that is a left adjoint of $f^*$, then $f_![A^{\otimes -1}]$ is a left adjoint of $f^*[A^{\otimes -1}]$.
		
		\item\label{itm:stabilization/adjointable}
			For any square
			\[
				\begin{tikzcd}
					\mathcal N \ar[d] \ar[r] & \mathcal N' \ar[d] \\
					\mathcal M \ar[r] & \mathcal M'
				\end{tikzcd}
			\]
			in $\Mod_{\mathcal A}(\PrL)$, if the underlying square of categories is left adjointable, then the induced square
			\[
				\begin{tikzcd}
					\mathcal N[A^{\otimes -1}] \ar[d] \ar[r] & \mathcal N'[A^{\otimes -1}] \ar[d] \\
					\mathcal M[A^{\otimes -1}] \ar[r] & \mathcal M'[A^{\otimes -1}]
				\end{tikzcd}
			\]
			is also left adjointable.

		\item For any functor $F : K \to \Mod_{\mathcal A}(\PrL)$, there is a transformation $F \to F[A^{\otimes -1}]$ such that for each $p \in K$, $(F \to F[A^{\otimes - 1}])(p)$ is $\Sigma_A^\infty : F(p) \to F[A^{\otimes -1}]$.
			\begin{enumerate}

				\item\label{itm:stabilization/univ} For any transformation $F \to G$ in $\Fun(K, \Mod_{\mathcal A}(\PrL))$, the space of extensions of this transformation along $F \to F[A^{\otimes -1}]$ is empty, unless every object of $A$ acts invertibly on $G(p)$ for all $p$, in which case it is contractible.

				\item\label{itm:stabilization/limits}
					If for every $p \to p'$ in $K$, and $a \in A$, the square
					\[
						\begin{tikzcd}
							F(p) \ar[d, "a \otimes -"'] \ar[r] & F(p') \ar[d, "a \otimes -"] \\
							F(p) \ar[r] & F(p')
						\end{tikzcd}
					\]
					is left adjointable, then the map
					\[
						(\varprojlim F)[A^{\otimes -1}] \to \varprojlim (F[A^{\otimes -1}])
					\]
					is an equivalence in $\Mod_{\mathcal A}(\PrL)$.

			\end{enumerate}

	\end{enumerate}
	\begin{proof}
		First we note that by \cite[Corollary 2.22]{robalo}, and \cite[\S6.1]{sixopsequiv}, there is a filtered category $L(A)$, and a diagram
		\[
			\theta_A : L(A) \to \Fun(\Mod_{\mathcal A}(\PrL), \Mod_{\mathcal A}(\PrL))
		\]
		such that for each map $x \to y$ in $L(A)$, the map $\theta_A(x \to y)$ is the transformation
		\[
			\id \xrightarrow{\alpha \otimes -} \id
		\]
		where $\alpha$ is some tensor product of elements of $A$, and there is an initial object $0$ of $L(A)$ such that the unit $\Sigma_A^\infty : \id \to -[A^{\otimes -1}]$ is the map
		\[
			\theta_A(0) \to \varinjlim \theta_A
		.\]

		Note also that by \cite[Corollaries 3.4.3.6 and 3.4.4.6]{ha}, the forgetful functor $\Mod_{\mathcal A}(\PrL) \to \PrL$ preserves and detects small limits and small colimits.

		Now we proceed to the proofs of the individual statements:
		\begin{enumerate}

			\item The following argument is based on the proof of \cite[Proposition 6.4 (2)]{sixopsequiv}.

				Let $p : \tilde{\mathcal M_A} \to L(A)$ be the presentable fibration corresponding to the functor $\theta_A(-)(M) : L(A) \to \PrL$. By \cite[Corollary 3.3.3.2]{htt}, we have that the limit of the corresponding functor $L(A)^\op \to \PrR$ is equivalent to the category of Cartesian sections of $p$, so $\mathcal M[A^{\otimes -1}]$ is equivalent to the category of coCartesian sections of $p$.

				Now, for each $x \in L(A)$, if $0 \in L(A)$ is the initial object, then the functor $\mathcal M \to \mathcal M$ corresponding to $0 \to x$ is given by $\alpha \otimes -$, for some tensor product $\alpha$ of elements of $A$. If we write $\Sigma^{\infty - x}_A : \mathcal M \to \mathcal M[A^{\otimes -1}]$ for the morphism in $\PrL$ corresponding to the inclusion into the colimit for $x \in L(A)$, we find that $\psi(x)^*$ sends $M \in \mathcal M$ to $\alpha^{\otimes -1} \otimes \Sigma^\infty_A M$.

				Thus, by \cite[Lemma 6.3.3.6]{htt}, we have that any object of $\mathcal M[A^{\otimes -1}]$ is an $L(A)$-indexed colimit of objects of the form $\alpha^{\otimes -1} \otimes \Sigma^\infty_A M$, for $\alpha$ a tensor product of elements of $A$, and $M \in \mathcal M$. Thus, we conclude since $L(A)$ is filtered.

			\item Let $f : \mathcal M \to \mathcal N$ be a map in $\Mod_{\mathcal A}(\PrL)$ that admits an $\mathcal A$-linear left adjoint. Then the induced map $f[A^{\otimes -1}] : \mathcal M[A^{\otimes -1}] \to \mathcal N[A^{\otimes -1}]$ is an $L(A)$-indexed colimit in $\Fun(\Delta^1, \PrL)$, where for each $x \to y$ in $L(A)$, the corresponding map in $\Fun(\Delta^1, \PrL)$ is of the form
				\[
					\begin{tikzcd}
						\mathcal M \ar[d, "a \otimes -"'] \ar[r, "f"] & \mathcal N \ar[d, "a \otimes -"] \\
						\mathcal M \ar[r, "f"'] & \mathcal N
					\end{tikzcd}
				.\]
				Since $f$ has an $\mathcal A$-linear left adjoint, these squares are left adjointable, so we conclude by \Cref{lem:LAd PrL pres colim} that the square \eqref{eqn:stabilization/linear} is left adjointable.

				Similarly, for any $a_0 \in \mathcal A$, $a_0 \otimes -$ defines a transformation in $\Fun^\LAd(\Delta^1, \PrL)$ from $f$ to itself, and in fact since $\mathcal A$ is \emph{symmetric} monoidal, it defines a transformation on the level of $L(A)$-indexed diagrams, so the transformation from $f[A^{\otimes -1}]$ to itself given by $a_0 \otimes -$ is the $L(A)$-indexed colimit of a morphism in $\Fun^\LAd(\Delta^1, \PrL)$, so by \Cref{lem:LAd PrL pres colim}, it is also in $\Fun^\LAd(\Delta^1, \PrL)$, that is,
				\[
					\begin{tikzcd}
						\mathcal M[A^{\otimes -1}] \ar[d, "a_0 \otimes -"'] \ar[r] & \mathcal N[A^{\otimes -1}] \ar[d, "a_0 \otimes -"] \\
						\mathcal M[A^{\otimes -1}] \ar[r] & \mathcal N[A^{\otimes -1}]
					\end{tikzcd}
				\]
				is left adjointable for all $a_0 \in \mathcal A$, which shows that $f[A^{\otimes -1}]$ has an $A$-linear left adjoint.

				Finally, we have that the left base change transformation of \eqref{eqn:stabilization/linear} gives a morphism $f_\sharp \to g$, where $g$ is a left adjoint of $f[A^{\otimes -1}]$. Since $g$ is an object of $\Fun(\Delta^1, \Mod_{\mathcal A[A^{\otimes -1}]}(\PrL))$, there is an extension of the map $f_\sharp \to g$ along $f_\sharp \to f_\sharp[A^{\otimes -1}]$ to a map $f_\sharp[A_{\otimes -1}] \to g$, but this map must be an equivalence, since it is an equivalence on the domain and codomain by the universal property of $\Sigma_A^\infty : \id \to (-)[A^{\otimes -1}]$.

				% Finally, to show that the left adjoint of $f[A^{\otimes -1}]$ is the functor $\mathcal N[A^{\otimes -1}] \to \mathcal M^{\otimes -1}]$ induced by the left adjoint $f_\sharp : \mathcal N \to \mathcal M$ of $f$, note that by the universal property of $\mathcal N \to \mathcal N[A^{\otimes -1}]$, it suffices to show that the composite
				% \[
				% 	\mathcal N \to \mathcal N[A^{\otimes -1}] \to \mathcal M[A^{\otimes -1}]
				% \]
				% is equivalent to the composite
				% \[
				% 	\mathcal N \xrightarrow{f_\sharp} \mathcal M \to \mathcal M[A^{\otimes -1}]
				% ,\]
				% but this follows from the fact that the square \eqref{eqn:stabilization/linear} is left adjointable.

				% NOTE: might be able to show this last fact even without asking for elements of $A$ to be symmetric (and instead just considering the stabilization, even if it doesn't have the universal property)
				% to do this, take right adjoints of everything so that considering limits instead of colimits (useful since the forgetful functor to $\widehat{\Cat}$ preserves limits), and use the equivalence of \cite[Corollary 4.7.4.18(3)]{ha}

			\item Note that the left adjointable square
				\[
					\begin{tikzcd}
						\mathcal N \ar[d] \ar[r, "g"] & \mathcal N' \ar[d] \\
						\mathcal M \ar[r, "f"'] & \mathcal M'
					\end{tikzcd}
				\]
				is a morphism $g \to f$ in $\Fun^\LAd(\Delta^1, \PrL)$, and 
				\[
					\begin{tikzcd}
						\mathcal N[A^{\otimes -1}] \ar[d] \ar[r, "g\lbrack A^{\otimes -1}\rbrack"] & \mathcal N'[A^{\otimes -1}] \ar[d] \\
						\mathcal M[A^{\otimes -1}] \ar[r, "f\lbrack A^{\otimes -1}\rbrack"'] & \mathcal M'[A^{\otimes -1}]
					\end{tikzcd}
				\]
				is a morphism $g[A^{\otimes -1}] \to f[A^{\otimes -1}]$ in $\Fun(\Delta^1, \PrL)$ that is an $L(A)$-indexed colimit of morphisms $g \to f$. This means that it is an $L(A)$-indexed colimit of morphisms in $\Fun^\LAd(\Delta^1, \PrL)$. By Lemma \ref{lem:LAd PrL pres colim}, it follows that the colimit of these morphisms is also a morphism in $\Fun^\LAd(\Delta^1, \PrL)$, as desired.

			\item The transformation $F \to F[A^{\otimes -1}]$ is given by composing $F$ by the unit transformation $\Sigma_A^\infty : \id \to -[A^{\otimes -1}]$.
				% Use \cite[Proposition A.1]{objwise-mono} and \cite[Corollary 2.22]{robalo} to produce $F \to F[A^{\otimes -1}]$ of the desired form.
				\begin{enumerate}

					\item Follows immediately from \cite[Proposition A.11]{objwise-mono}.

					\item We may consider the composite
						\[
							L(A) \xrightarrow{\theta_A} \Fun(\Mod_{\mathcal A}(\PrL), \Mod_{\mathcal A}(\PrL)) \xrightarrow{- \circ F} \Fun(K, \Mod_{\mathcal A}(\PrL))
						,\]
						which corresponds to a functor $\chi : L(A) \times K \to \Mod_{\mathcal A}(\PrL)$, such that for $x \to y$ in $L(A)$, and $p \to p'$ in $K$, the square
						\[
							\begin{tikzcd}
								\chi(x,p) \ar[d] \ar[r] & \chi(x,p') \ar[d] \\
								\chi(y,p) \ar[r] & \chi(y,p')
							\end{tikzcd}
						\]
						is
						\[
							\begin{tikzcd}
								F(p) \ar[d, "a \otimes -"'] \ar[r, "F(p \to p')"] & F(p') \ar[d, "a \otimes -"] \\
								F(p) \ar[r, "F(p \to p')"'] & F(p')
							\end{tikzcd}
						\]
						for $a$ some tensor product of elements of $A$. It follows from \cite[Lemma F.6 (3)]{TwAmb} that this square is left adjointable, since it is a composite of squares that we have assumed are left adjointable, so its transpose is right adjointable. Thus, \cite[Proposition 4.7.4.19]{ha} shows that the map
						\[
							\varinjlim_{x \in L(A)} \varprojlim_{p \in K} \chi(x,p) \to \varprojlim_{p \in K} \varinjlim_{x \in L(A)} \chi(x,p)
						\]
						is an equivalence, but this is precisely the map
						\[
							(\varprojlim F)[A^{\otimes -1}] \to \varprojlim (F[A^{\otimes -1}])
						.\]

				\end{enumerate}

		\end{enumerate}

	\end{proof}
\end{prp}

\section{Miscellaneous Categorical Results}

In this section we record some unsorted abstract results about categories.

\begin{lem} \label{lem:radj cons iff ladj gens}
	Let $\{f_i : \mathcal C \to \mathcal D_i\}_i$ be a small collection of right adjoint functors of presentable categories, and for each $i$, let $f_i^*$ be a left adjoint of $f_i$. Then the family $\{f_i\}_i$ is jointly conservative if and only if the union of the images of the functors $\{f_i^*\}_i$ generate $\mathcal C$ under colimits.
	\begin{proof}
		For each $i$, since $\mathcal D_i$ is presentable, we have a small set $S_i$ of objects of $\mathcal D_i$ that generate $\mathcal D_i$ under small colimits. Since $f_i^*$ preserves small colimits for all $i$, we have that the union of the images of $\{f_i^*\}$ generates $\mathcal C$ under colimits if and only if
		\[
			S \coloneqq \bigcup_i f_i^*(S_i)
		\]
		generates $\mathcal C$ under small colimits. Note that $S$ is a small union of small collections, so $S$ is small. Therefore, \cite[Corollary 2.5]{MonTow} says that $S$ generates $\mathcal C$ under colimits if and only if the functors $\{\mathcal C(f_i^*(Y), -)\}_{i, Y \in S_i}$ are jointly conservative, and since $f_i^* \dashv f_i$, this is equivalent to the functors $\{\mathcal D_i(Y, f_i(-))\}_{i, Y \in S_i}$ being jointly conservative, or equivalently, that the composite
		\[
			\mathcal C \to \prod_i \mathcal D_i \xrightarrow{\prod_i (\mathcal D_i(Y,-))_{Y \in S_i}} \prod_{i, Y \in S_i} \spaces
		\]
		is conservative.

		By \cite[Corollary 2.5]{MonTow}, we know already that the second functor is conservative, so the first functor is conservative if and only if the composite is conservative, as desired.
	\end{proof}
\end{lem}

\begin{lem} \label{lem:cofinal ordinal diagram}
	Let $\lambda$ be an ordinal seen as a poset category, and let $J : \lambda \to \mathcal C$ be a functor that sends every object of $\lambda$ to a $(-1)$-truncated object. Let $\mathcal U$ be the full subcategory of $\mathcal C$ consisting of objects that admit a map to an object of the form $J(a)$ for some $a \in \lambda$. Then $J$ factors through a cofinal functor $\lambda \to \mathcal U$.
	\begin{proof}
		Note that $J$ is fully faithful, so since its essential image is contained in $\mathcal U$, it factors through the inclusion $\mathcal U \subseteq \mathcal C$ by a full faithful functor $\lambda \to \mathcal U$.

		By \cite[Theorem 4.1.3.1]{htt}, this functor is cofinal if and only if for each $X \in \mathcal C$, if $X$ admits a map to $J(a)$ for some $a \in \lambda$, then the category
		\[
			\lambda \times_{\mathcal U} \mathcal U_{X/} \simeq \lambda \times_{\mathcal C} \mathcal C_{X/}
		\]
		is weakly contractible. This category is the full subcategory of $\lambda$ of those $a \in \lambda$ such that there exists a map $X \to J(a)$. We have assumed that this is nonempty, so since $\lambda$ is well-ordered, this category has an initial object, whence it is weakly contractible.
	\end{proof}
\end{lem}

\begin{prp} \label{prp:descent by open covers}
	Let $\mathcal C$ be a category with a collection of monomorphisms $\mathcal I$ that is stable under base change, composition, contains all equivalences, and such that for any family of maps $\{U_i \to X\}_i$ in $\mathcal I$, there is a map $U \to X$ in $\mathcal I$ that is initial among maps to $X$ through which $U_i \to X$ factors for each $i$. We call $U \to X$ the \emph{union} of the family $\{U_i \to X\}_i$.

	Suppose $\mathcal C$ is equipped with a Grothendieck topology such that for any $X \in \mathcal C$, and sieve $\mathcal U$ of $X$, $\mathcal U$ is covering if and only if it contains a small family of maps $\{U_i \to X\}_i$ in $\mathcal I$ whose union $U \to X$ is an equivalence.

	Assume $\mathcal C$ has a strict initial object $\emptyset$.\footnote{An initial object $\emptyset$ of a category $\mathcal C$ is said to be a \emph{strict initial} object if every map in $\mathcal C$ to $\emptyset$ is an equivalence.}
	Then for any presheaf $P : \mathcal C^\op \to \spaces$, $P$ is a sheaf if and only if
	\begin{enumerate}

		\item $P(\emptyset) \simeq \pt$.

		\item For any maps $U,V \to X$ that generate a covering sieve, the square
			\[
				\begin{tikzcd}
					P(X) \ar[d] \ar[r] & P(U) \ar[d] \\
					P(V) \ar[r] & P(U \times_X V)
				\end{tikzcd}
			\]
			is Cartesian.

		\item For any ordinal $\lambda$, and diagram $J : \lambda \to \mathcal C_{/X}$ that sends every morphism to a morphism in $\mathcal I$, if $\{J(\alpha) \to X\}_{\alpha \in \lambda}$ is a covering sieve, then 
			\[
				P(X) \to \varprojlim_{\alpha \in \lambda^\op} P(J(\alpha))
			\]
			is an equivalence.

	\end{enumerate}
	\begin{proof}
		For any $J$ as in the last condition, Lemma \ref{lem:cofinal ordinal diagram} says that $J$ factors through the sieve $\mathcal U \subseteq \mathcal C_X$ generated by $\{J(\alpha) \to X\}_{\alpha \in \lambda}$ by a cofinal map $\lambda \to \mathcal U$, so if $P$ is a sheaf, then it satisfies the last condition.

		Note that for any object $X \in \mathcal C$, the union of the empty family of maps to $X$ is an initial object of $\mathcal C_{/X}$, so it is the map $\emptyset \to X$. Thus, the empty sieve covers $X$ if and only if $\emptyset \to X$ is invertible.

		Since every map in $\mathcal I$ is a monomorphism, and $\mathcal I$ is stable under base change and contains all equivalences, \cite[Theorem 2.2.7]{locspalg} says that the first two conditions are equivalent to the condition that $P$ has descent along covering sieves that are generated by \emph{finite} families of maps in $\mathcal I$.

		Thus, it only remains to show that if $P$ satisfies the third condition and has descent along covering sieves generated by finite families of maps in $\mathcal I$, then $P$ is a sheaf.

		We will prove by transfinite induction that for any ordinal $\lambda$, $P$ has descent for covering sieves generated by $\lambda$-small families of maps in $\mathcal I$. We have already assumed that this holds when $\lambda$ is finite.

		Assume $\lambda$ is a limit ordinal, and that $P$ satisfies descent for covering sieves generated by $\kappa$-small families of maps in $\mathcal I$ for $\kappa \in \lambda$. Then if $\{V_\alpha \to X\}_{\alpha \in \lambda}$ is a family of maps in $\mathcal I$ that generates a covering sieve, for each $\alpha \in \lambda$, define $J(\alpha) \to X$ to be the union of $\{V_\beta \to X\}_{\beta \in \alpha}$. 

		By our assumptions about unions and the Grothendieck topology, this defines a functor $J : \lambda \to \mathcal C_{/X}$, and for every $\alpha \in \lambda$, the $(\alpha + 1)$-small family $\{V_\beta \to J(\alpha)\}_{\beta \in \alpha}$ generates a covering sieve. Since $\lambda$ is a limit ordinal, $\alpha + 1 < \lambda$, so $P$ has descent along this sieve.

		By our cofinality argument from before, we also know that $P$ has descent along the sieve generated by $\{J(\alpha) \to X\}_{\alpha \in \lambda}$, so since it also has descent along the sieves generated by $\{V_\beta \to J(\alpha)\}_{\beta \in \alpha}$ for all $\alpha \in \lambda$, we get that $P$ has descent along the sieve generated by $\{V_\alpha \to X\}_{\alpha \in \lambda}$.

		Finally we consider the case of successor ordinals: given any ordinal $\lambda$, we will show that $P$ has descent along covering sieves generated by $(\lambda + 1)$-small families of maps in $\mathcal I$ assuming that it has descent along covering sieves generated by $\lambda$-small families of maps.

		Let $\{V_\alpha \to X\}_{\alpha \in \lambda + 1}$ be a family of maps in $\mathcal I$ that generate a covering sieve of $X$, and let $V \to X$ be the map in $\mathcal I$ that is the union of $\{V_\alpha \to X\}_{\alpha \in \lambda}$. The sieve generated by $\{V, V_\lambda \to X\}$ is a covering sieve since it contains the covering sieve generated by $\{V_\alpha \to X\}_{\alpha \in \lambda + 1}$, so $P$ has descent along the sieve generated by $\{V, V_\lambda \to X\}$ since it is generated by a finite family of maps in $\lambda$. We know that $P$ has descent along the sieve generated by $\{V_\alpha \to V\}_{\alpha \in \lambda}$ since it is a covering sieve generated by a $\lambda$-small family of maps in $\mathcal I$, so it also has descent along the refinement $\{V_\alpha \to X\}_{\alpha \in \lambda + 1}$, as desired.

		Thus, since every covering sieve of an object $X \in \mathcal C$ contains a covering sieve generated by a small family of maps in $\mathcal I$, it follows that $P$ is a sheaf.
	\end{proof}
	
\end{prp}

\begin{lem} \label{lem:locality in psh}
	Given a small category $\mathcal C$, a collection of morphisms $S$ in $\Psh(\mathcal C)$ that is stable under pullback is local (in the sense of \cite[Definition 6.1.3.8]{htt}) if and only if a map $P \to Q$ is in $S$ exactly when for all $\yo(X) \to Q$, the base change $P \times_Q \yo(X) \to \yo(X)$ is in $S$.

	% In particular, given a collection $S$ of maps in $\mathcal C$ that is stable under base change, and a map $P \to Q$ in $\Psh(\mathcal C)$, the following conditions are equivalent:
	% \begin{enumerate}
	%
	% 	\item We have an equivalence $Q \simeq \varinjlim_i Q_i$ such that for each $i$, the base change $P \times_Q Q_i \to Q_i$ is of the form
	% 		\[
	% 			\varinjlim_{p \in K} \yo(A(p)) \to \yo(A(\infty))
	% 		\]
	% 		for some small diagram $A : K^\triangleright \to \mathcal C$ sending every edge to a morphism in $S$.
	%
	% 	\item For all $\yo(X) \to Q$, the base change $P \times_Q \yo(X) \to \yo(X)$ is of the form
	% 		\[
	% 			\varinjlim_{p \in K} \yo(A(p)) \to \yo(A(\infty))
	% 		\]
	% 		for some small diagram $A : K^\triangleright \to \mathcal C$ sending every edge to a morphism in $S$, and such that $A(\infty) = X$.
	%
	% \end{enumerate}
	\begin{proof}
		The ``only if'' direction is immediate from stability of $S$ under base change, and \cite[Lemma 6.1.3.5(3)]{htt}.

		For the converse, we use \cite[Lemma 6.1.3.5(3)]{htt} again, noting that the Cartesianness is automatic since $\Psh(\mathcal C)$ is a topos. Thus, we need to show that if $\bar \alpha : \bar p \to \bar q$ is a Cartesian transformation of colimit diagrams $\bar p, \bar q : K^\triangleright \to \Psh(\mathcal C)$, and for all $x \in K$, $\bar \alpha(x)$ is in $S$, then $\bar \alpha(\infty)$ is in $S$. Indeed, for any $C \in \mathcal C$, we have that (by \cite[Proposition 5.1.6.8]{htt})
		\[
			\Psh(\mathcal C)(\yo(C), \bar q(\infty)) \simeq \varinjlim_{x \in K} \Psh(\mathcal C)(\yo(C), \bar q(x))
		,\]
		so any $\yo(C) \to \bar q(\infty)$ is equivalent to a composite $\yo(C) \to \bar q(x) \to \bar q(\infty)$ for some $x \in K$. Thus, by Cartesianness of $\bar \alpha$, the base change $\bar p(\infty) \times_{\bar q(\infty)} \yo(C) \to \yo(C)$ is equivalent to the base change $\bar p(x) \times_{\bar q(x)} \yo(C) \to \yo(C)$ of $\bar \alpha(x)$, which is in $S$ since $\bar \alpha(x)$ is in $S$. This shows that for every $\yo(C) \to \bar q(\infty)$, the base change $\bar p(\infty) \times_{\bar q(\infty)} \yo(C) \to \yo(C)$ of $\bar \alpha(\infty)$ is in $S$, so $\bar \alpha(\infty)$ is in $S$ as required. 

		% Now, let $S$ is a collection of maps in $\mathcal C$ that is stable by base change. It is clear that the second condition on maps in $\Psh(\mathcal C)$ implies the first. To see the converse, let $\tilde S$ be the collection of maps satisfying the second condition. By universality of colimits in $\Psh(\mathcal C)$, $\tilde S$ is stable under pullback since $S$ is stable by base change, and a map $P \to Q$ is in $\tilde S$ if and only if for all $\yo(X) \to Q$, the base change $P \times_Q \yo(X) \to \yo(X)$ is in $\tilde S$. Thus, $\tilde S$ is a local class, so the first condition holds.
	\end{proof}
\end{lem}

\begin{rmk} \label{rmk:mult limit}
	Let $\mathcal A \in \CAlg(\Cat)$, and let $F : K^\triangleleft \to \Mod_{\mathcal A}(\Cat)$ be a limiting diagram. Then for any $a \in \mathcal A$, the map $F(-\infty) \xrightarrow{a \otimes -} F(-\infty)$ is the limit 
	\[
		\varprojlim_{p \in K} (F(p) \xrightarrow{a \otimes -} F(p))
	.\]
	To make sense of this statement, we note that $a \otimes -$ defines an endomorphism of the identity functor of $\Mod_{\mathcal A}(\Cat)$, so it induces an endomorphism of $F$, which induces a diagram $K^\triangleleft \to \Fun(\Delta^1, \Mod_{\mathcal A}(\Cat))$, and this diagram is limiting since it is after evaluating on the points of $\Delta^1$.
\end{rmk}

% PERF: could get rid of this and a bunch of other stuff if just keep first version of Pr adj lims
\begin{lem} \label{lem:2/3 for adjoints}
	Let $\mathcal D : K^\triangleleft \to \PrR$ ($\PrL$) be a small diagram of presentable categories such that $\mathcal D(-\infty) \to \varprojlim D|_K$ is conservative. For any presentable category $\mathcal C$ and functor $F : \mathcal C \to \mathcal D(-\infty)$, we have that $F$ admits a left (right) adjoint if and only if $\mathcal C \to \mathcal D(p)$ admits a left (right) adjoint for all $p \in K$.
	% Let
	% \[
	% 	\mathcal D \xrightarrow{g} \mathcal C \xrightarrow{f} \mathcal C'
	% \]
	% be functors of presentable categories, where $f, g \circ f$ have left (right) adjoints, and $f$ is conservative. Then $g$ has a left (right) adjoint.
	\begin{proof}
		The argument for right adjoints is analogous to the argument for left adjoints, so we only present the argument for left adjoints.

		By \cite[Proposition 5.5.3.18]{htt}, we have that for all $p \in K$, the functor $\mathcal D(-\infty) \to \varprojlim \mathcal D|_K$ admits a left adjoint, and that $\mathcal C \to \mathcal D(p)$ admits a left adjoint for all $p \in K$ if and only if $\mathcal C \to \varprojlim D|_K$ admits a left adjoint.

		Since $\mathcal D(-\infty) \to \varprojlim \mathcal D|_K$ and $\mathcal C \to \varprojlim \mathcal D|_K$ have left adjoints, by \cite[Corollary 5.5.2.9]{htt}, they are accessible and preserve small limits. In particular, there is a regular cardinal $\kappa$ such that these two functors preserve small limits and $\kappa$-filtered colimits. Since $\mathcal D(-\infty) \to \varprojlim \mathcal D|_K$ is conservative, it follows that $F$ also preserves small limits and $\kappa$-filtered colimits, so by \cite[Corllary 5.5.2.9]{htt}, it has a left adjoint.
	\end{proof}
\end{lem}

% \clearpage
\phantomsection
% \addcontentsline{toc}{part}{References}
\bibliography{refs}

\providecommand{\bysame}{\leavevmode\hbox to3em{\hrulefill}\thinspace}
\providecommand{\MR}{\relax\ifhmode\unskip\space\fi MR }
\providecommand{\MRhref}[2]{%
  \href{http://www.ams.org/mathscinet-getitem?mr=#1}{#2}
}
\providecommand{\href}[2]{#2}
\begin{thebibliography}{{Man}22}

\bibitem[AOV08]{tame-stacks}
Dan Abramovich, Martin Olsson, and Angelo Vistoli, \emph{Tame stacks in positive characteristic}, Annales de l'Institut Fourier \textbf{58} (2008), no.~4, 1057--1091 (en). \MR{2427954}

\bibitem[Ayo07a]{Ayoub6I}
Joseph Ayoub, \emph{Les six op\'erations de {G}rothendieck et le formalisme des cycles \' evanescents dans le monde motivique. {I}}, Ast\'erisque (2007), no.~314, x+466. \MR{2423375}

\bibitem[Ayo07b]{Ayoub6II}
\bysame, \emph{Les six op\'erations de {G}rothendieck et le formalisme des cycles \'evanescents dans le monde motivique. {II}}, Ast\'erisque (2007), no.~315, vi+364. \MR{2438151}

\bibitem[Ayo10]{AyoubBetti}
\bysame, \emph{Note sur les op\'erations de {G}rothendieck et la r\'ealisation de {B}etti}, J. Inst. Math. Jussieu \textbf{9} (2010), no.~2, 225--263. \MR{2602027}

\bibitem[Ayo14]{AyoubICM}
\bysame, \emph{A guide to (\'etale) motivic sheaves}, Proceedings of the {I}nternational {C}ongress of {M}athematicians---{S}eoul 2014. {V}ol. {II}, Kyung Moon Sa, Seoul, 2014, pp.~1101--1124. \MR{3728654}

\bibitem[BS20]{C2equiv}
Mark Behrens and Jay Shah, \emph{{$C_2$}-equivariant stable homotopy from real motivic stable homotopy}, Ann. K-Theory \textbf{5} (2020), no.~3, 411--464. \MR{4132743}

\bibitem[Cno23]{TwAmb}
Bastiaan Cnossen, \emph{Twisted ambidexterity in equivariant homotopy theory: Two approaches}, Ph.D. thesis, Universit{\"a}ts-und Landesbibliothek Bonn, 2023.

\bibitem[CSY21]{CARMELI2021107763}
Shachar Carmeli, Tomer~M. Schlank, and Lior Yanovski, \emph{Ambidexterity and height}, Advances in Mathematics \textbf{385} (2021), 107763.

\bibitem[DG22]{UnivFF}
Brad Drew and Martin Gallauer, \emph{The universal six-functor formalism}, Ann. K-Theory \textbf{7} (2022), no.~4, 599--649. \MR{4560376}

\bibitem[Fis76]{Fischer}
Gerd Fischer, \emph{Complex analytic geometry}, Lecture Notes in Mathematics, vol. Vol. 538, Springer-Verlag, Berlin-New York, 1976. \MR{430286}

\bibitem[GK17]{LCC}
David Gepner and Joachim Kock, \emph{Univalence in locally cartesian closed {$\infty$}-categories}, Forum Math. \textbf{29} (2017), no.~3, 617--652. \MR{3641669}

\bibitem[GR71]{SGA1}
Alexander Grothendieck and Michèle Raynaud, \emph{Rev\^etements \'etales et groupe fondamental}, Springer Berlin Heidelberg, 1971.

\bibitem[GR84]{CohAnalSh}
Hans Grauert and Reinhold Remmert, \emph{Coherent analytic sheaves}, Grundlehren der mathematischen Wissenschaften [Fundamental Principles of Mathematical Sciences], vol. 265, Springer-Verlag, Berlin, 1984. \MR{755331}

\bibitem[GR04]{SteinSpaces}
\bysame, \emph{Theory of {S}tein spaces}, Classics in Mathematics, Springer-Verlag, Berlin, 2004, Translated from the German by Alan Huckleberry, Reprint of the 1979 translation. \MR{2029201}

\bibitem[Gro65]{EGA4}
Alexander Grothendieck, \emph{{\'E}l\'ements de g\'eom\'etrie alg\'ebrique : {IV.} {\'etude} locale des sch\'emas et des morphismes de sch\'emas, {Seconde} partie}, Publications Math\'ematiques de l'IH\'ES \textbf{24} (1965), 5--231 (fr).

\bibitem[HM24]{HM6FF}
Claudius {Heyer} and Lucas {Mann}, \emph{{6-Functor Formalisms and Smooth Representations}}, arXiv e-prints (2024), arXiv:2410.13038.

\bibitem[Hoy14]{quadratic-refinement-GLV-trace}
Marc Hoyois, \emph{A quadratic refinement of the {G}rothendieck–{L}efschetz–{V}erdier trace formula}, Algebraic \& Geometric Topology \textbf{14} (2014), no.~6, 3603–3658.

\bibitem[Hoy17]{sixopsequiv}
\bysame, \emph{The six operations in equivariant motivic homotopy theory}, Adv. Math. \textbf{305} (2017), 197--279, Some corrections have been made to the arXiv version as recently as 2024, and references in the text should actually be interpreted as referring to the version of this paper found at \url{ arxiv.org/abs/1509.02145v5}. \MR{3570135}

\bibitem[Kha19]{locspalg}
Adeel~A. Khan, \emph{The {M}orel-{V}oevodsky localization theorem in spectral algebraic geometry}, Geom. Topol. \textbf{23} (2019), no.~7, 3647--3685. \MR{4046969}

\bibitem[Kha21]{SixAlgSp}
Adeel~A. Khan, \emph{{V}oevodsky’s criterion for constructible categories of coefficients}, \url{https://www.preschema.com/papers/six.pdf}, 2021.

\bibitem[KLS18]{equivOkaBundles}
Frank Kutzschebauch, Finnur L\'arusson, and Gerald~W. Schwarz, \emph{An equivariant parametric {O}ka principle for bundles of homogeneous spaces}, Math. Ann. \textbf{370} (2018), no.~1-2, 819--839. \MR{3747503}

\bibitem[KR24]{SixAlgSt}
Adeel~A. Khan and Charanya Ravi, \emph{Generalized cohomology theories for algebraic stacks}, Advances in Mathematics \textbf{458} (2024), 109975.

\bibitem[LMB00]{ChampsAlg}
Gérard Laumon and Laurent Moret-Bailly, \emph{Champs algébriques}, Springer Berlin Heidelberg, 2000.

\bibitem[Lur09a]{dag5}
Jacob Lurie, \emph{Derived algebraic geometry {V}: Structured spaces}.

\bibitem[Lur09b]{htt}
\bysame, \emph{Higher topos theory}, Annals of Mathematics Studies, vol. 170, Princeton University Press, Princeton, NJ, 2009. \MR{2522659}

\bibitem[Lur17]{ha}
\bysame, \emph{Higher algebra}, \url{http://www.math.harvard.edu/~lurie/papers/HA.pdf}, September 2017.

\bibitem[Lur18]{SAG}
\bysame, \emph{Spectral algebraic geometry}, \url{https://www.math.ias.edu/~lurie/papers/SAG-rootfile.pdf}, 2018.

\bibitem[Lur25]{kerodon}
Jacob Lurie, \emph{Kerodon}, \url{https://kerodon.net}, 2025.

\bibitem[LZ12]{LZ6FF}
Yifeng {Liu} and Weizhe {Zheng}, \emph{{Enhanced six operations and base change theorem for higher {A}rtin stacks}}, arXiv e-prints (2012), arXiv:1211.5948.

\bibitem[Magona]{6FF}
Roy Magen, \emph{Geometric criteria for 6-functor formalisms in the setting of pullback formalisms}, in preparation.

\bibitem[Magonb]{Gluing}
\bysame, \emph{The gluing property of pullback formalisms}, in preparation.

\bibitem[{Man}22]{Mann6FF}
Lucas {Mann}, \emph{{A $p$-Adic 6-Functor Formalism in Rigid-Analytic Geometry}}, arXiv e-prints (2022), arXiv:2206.02022.

\bibitem[MV99]{A1htpysch}
Fabien Morel and Vladimir Voevodsky, \emph{{$A^1$}-homotopy theory of schemes}, Publications mathématiques de l’IHÉS \textbf{90} (1999), no.~1, 45–143.

\bibitem[MW25]{internal-cats}
Louis Martini and Sebastian Wolf, \emph{Presentability and topoi in internal higher category theory}, 2025.

\bibitem[Ram23]{objwise-mono}
Maxime Ramzi, \emph{An elementary proof of the naturality of the {Y}oneda embedding}, Proc. Amer. Math. Soc. \textbf{151} (2023), no.~10, 4163--4171. \MR{4643310}

\bibitem[Rob15]{robalo}
Marco Robalo, \emph{{$K$}-theory and the bridge from motives to noncommutative motives}, Adv. Math. \textbf{269} (2015), 399--550. \MR{3281141}

\bibitem[Ryd15]{Rydh_2015}
David Rydh, \emph{Approximation of sheaves on algebraic stacks}, International Mathematics Research Notices \textbf{2016} (2015), no.~3, 717–737.

\bibitem[Seg68]{equivK}
Graeme Segal, \emph{Equivariant k-theory}, Publications mathématiques de l’IHÉS \textbf{34} (1968), no.~1, 129–151.

\bibitem[{Sta}24]{stacks-project}
The {Stacks Project Authors}, \emph{\textit{Stacks Project}}, \url{https://stacks.math.columbia.edu}, 2024.

\bibitem[Voe98]{Voe-A1}
Vladimir Voevodsky, \emph{{$A^1$}-homotopy theory}, EMS Press, January 1998.

\bibitem[Voe00]{voe-view}
Valdimir Voevodsky, \emph{My view of the current state of the motivic homotopy theory}, \url{ https://www.math.ias.edu/Voevodsky/files/files-original/Dropbox/Published_papers/Motives/Open/Stage1/ann.pdf }, 2000.

\bibitem[Voe01]{voe-crit}
\bysame, \emph{{V}oevodsky’s lectures on cross functors}, \url{https://www.math.ias.edu/vladimir/node/94}, 2001.

\bibitem[Yan22]{MonTow}
Lior Yanovski, \emph{The monadic tower for {$\infty$}-categories}, Journal of Pure and Applied Algebra \textbf{226} (2022), no.~6, 106975.

\end{thebibliography}
\bibliographystyle{amsalpha}

\end{document}